\documentclass[10pt]{article}

\usepackage[utf8]{inputenc}
\usepackage[russian,english]{babel} 
\usepackage{amssymb,amsmath,amsfonts,amsthm}
\usepackage{mydef}
\usepackage{boldfonts}
\usepackage{xspace}
\usepackage{epigraph}
\usepackage{mathrsfs}
\usepackage{enumerate}
\usepackage{hyperref}
\usepackage{esint}
\usepackage{graphicx,eucal}
\usepackage{eucal}
\usepackage{booktabs}
\usepackage{multicol}
\usepackage{color}
\usepackage{bm}
\usepackage{longtable,lipsum}
\usepackage[leqno]{mathtools}
\usepackage{pgfplots}
\usepackage{tikz}
\usetikzlibrary{arrows,shapes}
\usetikzlibrary{arrows, decorations.markings}
\usetikzlibrary{calc,3d}
\usepackage{soul}
\usepackage[linesnumbered,boxruled,algosection]{algorithm2e}
\usepackage{authblk}

\begin{document}

\title{Numerical analysis of strongly nonlinear PDEs}

\author[1]{Michael Neilan\thanks{neilan@pitt.edu}}

\author[2]{Abner J. Salgado\thanks{asalgad1@utk.edu}}

\author[3]{Wujun Zhang\thanks{wujun@math.rutgers.edu}}

\affil[1]{Department of Mathematics, University of Pittsburgh}
\affil[2]{Department of Mathematics, The University of Tennessee}
\affil[2]{Department of Mathematics, Rutgers University}

\date{\today}

\maketitle


\begin{abstract}{\normalsize
We review the construction and analysis of numerical methods for
strongly nonlinear PDEs,
with an emphasis on convex
and nonconvex fully nonlinear 
equations
and the convergence to viscosity solutions. 
We begin by describing a fundamental result in this area which states that
stable, consistent,
and monotone schemes
converge as the discretization 
parameter tends to zero. 
We review methodologies 
to construct finite difference,
finite element, and semi-Lagrangian
schemes that satisfy these criteria, and,
in addition, 
discuss some rather novel tools that have paved the way to
derive rates of convergence within this framework. 
}
\end{abstract}

\section{Introduction}

\epigraph{\emph{\selectlanguage{russian}Все счастливые семьи похожи друг на друга, каждая несчастливая семья несчастлива по-своему.}
}{L.~Tolstoy}

The quote above from L.~Tolstoy \cite{TOLSTOJ:ANNAKARE1965}, which roughly translates to ``All happy families resemble each other, but each unhappy one is unhappy in its own way'' was used in \cite[Section 8.1]{MR2478556} to point out that the numerical approximation of partial differential equations substantially differs from that of ordinary differential equations. The same quote was used in the Preface of \cite{MR3443431} to say that the theory for nonlinear equations is very different than the one for linear problems, and that each nonlinearity needs to be treated in its own way. For these reasons we feel compelled to begin our discussion with the same quote, since we have chosen the unhappiest of all possible families for numerical approximation: Fully nonlinear equations.


The goal of this paper is to 
summarize recent advances and trends
in the numerical approximation and theory
of strongly second-order nonlinear PDEs
with an emphasis on fully nonlinear convex
and nonconvex PDEs.
Such PDEs  appear in diverse
applications such as 
 weather and climate modeling, stochastic optimal control, determining 
 the initial shape of the universe, 
 optimal reflector design, differential geometry, optimal transport, 
 mathematical finance, image processing, and mesh generation. 
Despite their importance in these application areas,
and in contrast to the PDE and solution theory,
numerical methods for fully nonlinear problems
is still an emerging field in numerical analysis.
The reasons for the delayed development 
are plentiful. 
Besides the strong nonlinearity,  the fundamental difficulties
in computing solutions of fully nonlinear
problems are  the lack of regularity of solutions, the conditional
uniqueness of solutions, and most importantly,
the notion of viscosity solutions.  Similar to 
weak solutions for PDEs in divergence-form,
the viscosity solution concept relaxes the pointwise
meaning of the PDE, and while doing so, 
broadens the class of admissible functions in which to seek a solution. 
However, unlike the weak solution framework, the definition of viscosity
solutions is not based on variational principles, but rather
comparison principles. 
While viscosity solution theory and the PDE theory of fully nonlinear 
problems has made incredible progress
during the last 25 years, the numerical results for such problems
has been slow to catch up due to pointwise nature
and nonvariational structure found in the theory. 

A breakthrough occurred in 1991 
with
\cite{BarlesSoug91}
which
roughly speaking, asserts that a consistent, stable and {\it monotone}
numerical method (or general approximation scheme)
converges to the viscosity solution as the discretisation or regularization parameter
tends to zero.   The first two conditions in this framework, consistency
and stability, are expected; they are the cornerstone of 
any convergence theory of numerical PDEs and
 is recognized as the basis of the Lax-Richtmyer equivalence theorem.
While arguably less well-known, monotonicity 
of numerical methods is also a long-established
area of study, for example, its importance
 in the context of linear finite difference schemes has been
 realized  (at least) 80 years ago (see, e.g., \cite{Gerschy30,MotzkinWasow53}).  
 On the other hand, the construction
of numerical methods 
that satisfy all three criteria,
at least for fully nonlinear problems, is not immediately
obvious.
 
Around the same time, 
and complementary to the Barles-Souganidis framework,
\cite{KuoTrudinger90,KuoTrudinger92} gave 
a methodology to construct consistent, stable,
and monotone finite difference schemes
for uniformly elliptic fully nonlinear operators.
In addition, they showed that 
such discrete schemes
satisfy properties found in the viscosity solution theory,
including Alexandrov-Bakelman-Pucci (ABP) maximum principles, Harnack inequalities,
and H\"older estimates.  While these results
gave a somewhat practical guide to compute viscosity
solutions, and the theory paved the way for future
advancements, the fundamental issue
of convergence {\it rates} was explicitly stated as an open 
and elusive problem.


For the next 15 years progress of numerical fully
nonlinear second-order PDEs was relatively limited 
and mostly constrained to the theory
and convergence rates for {\it convex} PDEs, in particular,
the Hamilton-Jacobi-Bellman equation.  In this direction, \cite{Krylov97}
introduced the groundbreaking idea of ``shaking the coefficients''
to obtain sufficiently smooth subsolutions, which along
with comparison principles, yield rates of convergence
even for degenerate problems.  These techniques
were later refined under various scenarios and assumptions
of the problem and discretisation (e.g., \cite{BarlesJakobsen05,MR1916291,MR1759507}), 
although convexity
of the equation always played an essential role in the analysis.

The last ten years has seen an explosion of results
for numerical nonlinear PDEs,
including a variety of discretization
types and convergence results.
These include the construction of relatively simple and practical
wide-stencil  finite difference schemes
tailored  to specific PDEs \cite{Oberman08,FroeseOberman11,BFA14,Froese16},
and the emergence
of Galerkin methods for fully nonlinear problems \cite{FengNeilan11,FengNeilan09b,FengLewis14B,GlowinskiDean06B,GlowinskiDean06A}.
With regard to the convergence theory,
\cite{CaffarelliSoug08}, using intricate regularity results, 
 derived algebraic rates of convergence for finite difference approximations
for nonconvex PDEs with constant coefficients, and
these results were quickly extended to problems
with variable coefficients and lower-order terms
by \cite{Krylov15} and \cite{Tura15}. 
On the Galerkin front, 
\cite{NochettoZhang16} extended
the Kuo-Trudinger theory to finite element methods
and derived  ABP maximum principles 
for linear elliptic problems.  These results
were soon extended in several directions,
including rates of convergence for a discrete \MA equation \cite{NochettoZhang16A} and wide-stencil
finite difference schemes \cite{NochettoNtogkasZhang},
and the construction and analysis of finite element
methods for nonconvex fully nonlinear problems
\cite{SalgadoZhang16}.

The intention of this survey 
is to summarize the
25 years of development
of fully nonlinear numerical PDEs.
Let us describe the organization and the problems 
we consider in the paper.
After setting the notation and stating some instances
of fully nonlinear problems,
we review some of the basic
 theory and analysis of elliptic PDEs
 in Section \ref{sec:PDEs}. Here
different notions of solutions 
are introduced and the underlying properties
of the solutions and operators are discussed.
Besides being of independent interest
these fundamental results motivate
both the construction and analysis
of the numerical methods.
We develop a general framework to 
compute second-order elliptic problems
in Section \ref{sec:numericmono}.
Basic properties of numerical
methods, namely, consistency,
stability, and monotonicity
are given, which will lay 
the groundwork for future developments
in the paper.  Special attention
will be on discrete ABP maximum principles
in both the finite difference and finite element setting.
Section \ref{sec:FEMnondiv}
concerns finite element approximations
for linear problems in non-divergence
form and with nonsmooth coefficients.
In Section \ref{sec:convex}
we combine the ideas of the previous
sections and consider finite element
and finite difference approximations
for fully nonlinear convex PDEs.
Besides the construction and convergence
of the schemes, a focus
of this section is the rates of convergence
and the techniques to obtain these results.
These results are extended to a particular  
convex PDE, the Monge-Amp\`ere equation,
in Section \ref{sec:MongeAmpere}.
We discuss recent results
of the numerical approximations for 
fully nonlinear nonconvex PDEs
in Section \ref{sec:nonconvex}.
Finally we give some concluding
remarks and state some open
problems in Section \ref{sec:Outlook}.

Before starting our discussion
let us first briefly outline the derivation of some fully nonlinear PDEs that
we focus on in the paper
and illustrate their connection with some applications
and other areas of mathematics.

\subsection{Convex PDEs}\label{subsec:IntroConvex}

This section obtains an instance
of the Hamilton-Jacobi-Bellman (HJB) equation, a prototypical
fully nonlinear second-order {\it convex} PDE,
and shows how such problems arise in stochastic optimal 
control problems. In addition, we show below
that {\it every} uniformly elliptic, convex operator
with bounded gradient is an implicit HJB problem.

Following  \cite[Chapter 11]{Oksendal03} and \cite{FlemingSoner06},
we 
consider a stochastic process $X_\tau$ governed 
by the  differential equation
\begin{equation}
\label{eqn:stateequation}
\begin{dcases}
  d X_\tau = \ell(X_\tau,\tau,\alpha_\tau)\diff\tau + { \sigma}(X_\tau,\tau,\alpha_\tau)\diff W_\tau\quad \tau> 0,\\
X_0  = x_0 \in \mathbb{R}^d.
\end{dcases}
\end{equation}
%
%
Here, $W_\tau$ is a Brownian motion of dimension $d$, $\sigma$ is an $d\times d$
matrix-valued function,
$\alpha_\tau\in \calA$ is a (Markov) control
and $\mathcal{A}$ is the control space. 
Problem \eqref{eqn:stateequation} describes
a dynamical system driven by additive white noise with 
diffusion coefficient
(or volatility) $\sigma$ and non-stochastic drift $\ell$.  
Under appropriate smoothness and growth conditions
on $\ell$ and $\sigma$, and for a fixed $\alpha(\cdot)\in \mathcal{A}$, 
there is a path-wise unique solution to \eqref{eqn:stateequation}.
Associated with problem \eqref{eqn:stateequation}
is the family of stochastic processes 
that satisfy \eqref{eqn:stateequation}
but with initial time $t\ge 0$:
\begin{equation}
\label{eqn:stateequationFam}
\begin{dcases}
  \diff \Xx_\tau = \ell(\Xx_\tau,\tau,\alpha_\tau)\diff\tau + { \sigma}(\Xx_\tau,\tau,\alpha_\tau)\diff W_\tau\quad \tau> t,\\
\Xx_t  = x \in \mathbb{R}^d,
\end{dcases}
\end{equation}

Now let $\Omega\subset \bbR^d$ be an open 
bounded domain, and set $Q = \Omega\times (0,T)$
for some $T\in (0,\infty]$.
Denote by $\hat{T}$ the first exit time  
for the process $\Xx_\tau$ satisfying \eqref{eqn:stateequationFam}, i.e.,
\[
\hat{T} = \hat{T}^{x,t} = \inf \{\tau >t; (\Xx_\tau,\tau) \not\in Q\}.
\]
%
We then define the performance function
(or cost functional) 
\begin{align*}
J(x,t,\alpha) = \mathbb{E}_{x,t}\Big\{\int_t^{\hat{T}} f(\Xx_\tau,\tau,\alpha_\tau)\diff \tau + \chi_{{\hat{T}}<\infty}  \psi (\Xx_{\hat{T}},\hat T)\Big\}
\end{align*}
over all $\alpha\in \mathcal{A}$
 under the constraint \eqref{eqn:stateequationFam}.
Here, $f:\bbR^d\times \bbR \times \mathcal{A}\to \bbR$ is 
the profit rate function, $\psi:\bbR^d\times \bbR\to \bbR$ is
the bequest function, 
$\chi_{\hat{T} < \infty}$ is the indicator function, which equals one if $\hat{T}<\infty$ and equals zero otherwise,
and $\mathbb{E}_{x,t}$ represents the expectation conditional on $\Xx_t =x$. 
 We then consider
the problem of finding the value function $u:\bar Q\to \bbR$ and optimal control $\alpha^*\in \mathcal{A}$ 
such that
\begin{align*}
u(x,t) = \sup_{\alpha\in \mathcal{A}} J(x,t,\alpha) = J(x,t,\alpha^*).
\end{align*}

Let us now give an heuristic derivation of the Hamilton-Jacobi-Bellman equation based on 
Bellman's principle of optimality and It\^o's lemma. 
%
First, the Bellman's principle of optimality states that
\cite{Bellman57}:
\begin{quote}
Whatever the initial state $\{\Xx_s\}_{s<\tau}$ and initial decision $\{\alpha_s\}_{s<\tau}$ are, 
the remaining controls $\{\alpha_s\}_{s>\tau}$ must constitute an optimal policy with regard to the state $\Xx_\tau$  resulting from the initial decision.
\end{quote}
More precisely, this principle equates to
\begin{align}\label{eqn:principleoptimality}
  u(x, t) = \sup_{\alpha \in \mathcal {A}} \mathbb{E}_{x,t} \left\{  u(\Xx_{t+\delta t},t+\delta t) + \int_t^{t+\delta t} f(\Xx_\tau, \tau, \alpha_\tau) \diff \tau \right\}.
\end{align}
Next, recalling that $\Xx_t = x$, by It\^{o}'s lemma we obtain
\begin{align*}
  u(\Xx_{t+\delta t}, t+ \delta t)  
  &=  u(x, t) + \int_t^{t+\delta t} \frac{\partial u}{\partial t} \diff \tau + \int_t^{t+\delta t} D u \cdot \diff \Xx_\tau\\
  &\qquad  + \frac 12 \int_t^{t+\delta t} \diff \Xx_\tau \cdot D^2 u \diff \Xx_\tau.
\end{align*}
Since the state variable $\Xx_\tau$ is governed by the stochastic equation \eqref{eqn:stateequationFam}, 
and by using the
identity
 $\diff W_\tau \otimes \diff W_\tau = I \diff \tau$ 
we obtain
\begin{align*}
  &\; u(\Xx_{t+\delta t}, t+ \delta t) - u(x, t) 
\\
= &\; \int_t^{t+\delta t} \left( \frac{\partial u}{\partial t}  + D u \cdot \ell +  \frac 12 (\sigma \sigma^\intercal):D^2 u \right) \diff \tau
%
+ \int_t^{t+\delta t} (\sigma^\intercal Du)\cdot \diff W_\tau. 
\end{align*}
Therefore, from  principle of optimality \eqref{eqn:principleoptimality}, 
\[
   \sup_{\alpha \in \mathcal {A}} \mathbb{E}_{x,t} \left\{  u(\Xx_{t+\delta t}, t+\delta t) - u(x, t) + \int_t^{t+\delta t} f(\Xx_{\tau}, \tau, \alpha) \diff \tau  \right\} 
   = 0
\]
and It\^{o}'s formula above, we obtain 
\[
   \sup_{\alpha \in \mathcal {A}} \mathbb{E}_{x,t} \left\{ \int_t^{t+\delta t} \left(\frac{\partial u}{\partial t} +Du \cdot \ell + \frac12 (\sigma \sigma^\intercal):D^2 u+f\right)\diff \tau\right\}
   = 0,
\]
where we used that
\[
  \polE_{x,t} \left\{ \int_t^{t+\delta t} (\sigma^\intercal Du)\cdot \diff W_\tau \right\} = 0 .
\]

Dividing by $\delta t$ and formally taking the limit $\delta t \to 0$, we then obtain a deterministic equation, namely, the Hamilton-Jacobi-Bellman equation
\[
\frac{\partial u} {\partial t}(x,t) +  \sup_{\alpha \in \mathcal A} \left(  \mathcal{L}^\alpha u(x,t) + f (x, t, \alpha) \right) = 0,
\]
where
\[
  \mathcal{L}^{\alpha} u(x,t) = 
D u(x,t) \cdot \ell(x,t,\alpha) + \frac 12 \sigma (x,t,\alpha) \sigma(x,t,\alpha)^\intercal : D^2 u(x,t).
\]

To derive the equation, we assumed that the value function $u(x,t)$ has continuous second order derivatives. 
However, the PDE theory reveals that the solution of the HJB equation in general does not satisfy this regularity assumption. 
To justify the derivation above, the concept of viscosity solutions was introduced in \cite{CrandallLions83} for first order 
Hamilton-Jacobi equations and later generalized to second order Hamilton-Jacobi-Bellman equations \cite{Lions82,Lions83}. 
Viscosity solutions, which is an essential concept in the PDE theory, will also play an important role in this paper.

\subsection{Nonconvex PDEs}\label{subsec:IntroNonConvex}


The  Isaacs  equation, a prototypical nonconvex PDE,  describes  zero  sum  stochastic  games  with  two  players.
Each  player  has  one  control  and  they  have  opposite  objectives.   
The  first  player  chooses  a  control to maximize the expected payoff, whereas the second player chooses a control to minimize it. 
Stochastic game theory has wide applications in engineering and mathematical finance.  
Here, we follow the ideas in \cite{FlemingSouganidis89} to derive the equation. 

The dynamics of the stochastic differential game we investigate are given by the controlled stochastic differential equation
\begin{equation}
\label{eqn:StateEqnIs}
\begin{dcases}
  \diff \Xx_\tau  =  \ell(\Xx_\tau, \tau, \alpha_\tau, \beta_\tau) \diff \tau + \sigma(\Xx_\tau, \tau, \alpha_\tau, \beta_\tau) \diff W_\tau\quad \tau > t,\\
   \Xx_t = x \in \Real^d,
  \end{dcases}
\end{equation}
and the expected payoff is
\[
  J(x, t, \alpha, \beta) = \polE_{x,t} \left\{ \int_t^{\hat{T}} f(\Xx_\tau, \tau, \alpha_\tau, \beta_\tau) \diff \tau + \chi_{\hat{T}<\infty} \psi(\Xx_{\hat{T}},\hat{T}) \right\},
\]
where $\Xx_\tau$ satisfies the differential equation \eqref{eqn:StateEqnIs} for $\tau>t$ 
and has initial condition $\Xx_t = x$.



The Isaacs equation can be derived in a similar fashion as the HJB equation.
If, at time $t$, the first player chooses a control $\alpha$ to maximize the expected payoff $J$, and the second player, based on the decision of first player, chooses the control $\beta$ to minimize it, then we set
\[
u^+(x, t) = \inf_{\beta \in \calB} \sup_{\alpha \in \calA} \polE_{x,t} \{ J(x, t, \alpha, \beta) \},
\]
and call this the upper value function. On the other hand, if the second player makes the decision first and the first player reacts accordingly, then we set
the lower value function as
\[
u^-(x, t) =  \sup_{\alpha \in \calA} \inf_{\beta \in \calB} \polE_{x,t} \{ J(x, t, \alpha, \beta) \}.
\] 
By the principle of optimality we have 
\[
  u^+(x, t) = \inf_{\beta \in \calB} \sup_{\alpha \in \calA}  \polE_{x,t} 
  \left\{ u^+( \Xx_{t+\delta t} , t+\delta t) + \int_{t}^{t+\delta t} f(\Xx_\tau, \tau, \alpha, \beta) \diff \tau \right\}.
\]
Using a similar derivation as the HJB equation, we obtain that for all $(x,t) \in Q$ the upper value function must satisfy
\[
  \frac{\partial u^+}{\p t}(x,t) + H^+(x,t,D u^+, D^2 u^+) = 0,
\]
where
\[
  H^+(x,t,\bp,M) = \inf_{\beta \in \calB} \sup_{\alpha \in \calA} \left\{
    \frac12 \sigma\sigma^\intercal :M + \ell(x,t,\alpha,\beta) \cdot \bp + f(x,t,\alpha,\beta)
  \right\}.
\]
Similarly we obtain, for $(x,t) \in Q$,
\begin{align*}
  \frac{\partial u^-}{\p t}(x,t) + H^-(x,t,D u^-, D^2 u^-) = 0,
\end{align*}
with
\[
  H^-(x,t,\bp,M) = \sup_{\alpha \in \calA} \inf_{\beta \in \calB} \left\{
    \frac12 \sigma\sigma^\intercal :M + \ell(x,t,\alpha,\beta) \cdot \bp + f(x,t,\alpha,\beta)
  \right\}.
\]

We finally comment that if the Isaacs' condition holds, \ie we have that for all $x,t,\bp,M$,
\[
  H^+(x,t,\bp,M) = H^-(x,t,\bp,M)
\]
then we can conclude that
$u^+(x,t) = u^-(x,t)$ for all $(x,t)  \in Q$. In this case, we say that an optimal policy exists.

\subsection{Characterizations of elliptic PDEs}
\label{subsec:Equivalence}
Let us show that there is no loss in generality in confining our considerations to these two equations by following the construction proposed in \cite[Lemma 2.2]{Evans1980}. Let $\Omega \subset \Real^d$ be a bounded domain and let $F: \Omega \times \polS^d \to \Real$ be continuously differentiable, nondecreasing with respect to its second argument and with bounded gradient.
Here, $\polS^d$ denotes the space of symmetric $d\times d$ matrices. 
We comment that, as we will see below (\cf Definition~\ref{def:FLelliptic}), these conditions guarantee that the operator $F$ is elliptic. We have the following representation: for every $x \in \Omega$ and $M \in \polS^d$, we have

\begin{equation}
\label{eq:PDEeqIsaacs}
  F(x,M) = \inf_{\beta \in \polS^d} \sup_{\alpha \in \polS^d} \left[ \int_0^1 \frac{\p F}{\p M}(x,(1-t) \beta + t \alpha ) : (M-\beta) + F(x,\beta)  \right].
\end{equation}
To see this, let us denote by $IS$ the right hand side of \eqref{eq:PDEeqIsaacs} and notice that, by setting $\beta = M$ we obtain
\begin{align*}
  IS &\leq \sup_{\alpha \in \polS^d} \left[ \int_0^1 \frac{\p F}{\p M}(x,(1-t) M + t \alpha )  : (M-M) + F(x,M)  \right] \\
     &= F(x,M).
\end{align*}
On the other hand, setting $\alpha = M$ we obtain that
\begin{multline*}
  \int_0^1 \frac{\p F}{\p M}(x,(1-t) \beta + t M )  : (M-\beta) + F(x,\beta) \leq  \\
  \sup_{\alpha \in \polS^d} \left[ \int_0^1 \frac{\p F}{\p M}(x,(1-t) \beta + t \alpha ) : (M-\beta) + F(x,\beta)  \right].
\end{multline*}
Under the given assumptions on $F$ the left hand side of this inequality can be rewritten as
\begin{multline*}
  \int_0^1 \frac{\p F}{\p M}(x,(1-t) \beta + t M )  : (M-\beta) + F(x,\beta) = \\
  \int_0^1 \frac{\diff}{\diff t} F(x,(1-t) \beta + t M )   + F(x,\beta) =
  F(x,M) - F(x,\beta) + F(x,\beta),
\end{multline*}
and, consequenlty,
\[
  F(x,M) \leq \sup_{\alpha \in \polS^d} \left[ \int_0^1 \frac{\p F}{\p M}(x,(1-t) \beta + t \alpha )  : (M-\beta) + F(x,\beta)  \right],
\]
or $F(x,M) \leq IS$.

With representation \eqref{eq:PDEeqIsaacs} at hand we define
\begin{align*}
  A^{\alpha,\beta}(x) &= \int_0^1 \frac{\p F}{\p M}(x,(1-t) \beta + t \alpha )  , \\
  f^{\alpha,\beta}(x) &= A^{\alpha,\beta}(x):\beta -F(x,\beta), \\
  \calA &= \calB = \polS^d,
\end{align*}
and note that, since $F(x,\cdot)$ is nondecreasing, $A^{\alpha,\beta}(x)\geq0$ for all $\alpha \in \calA$, $\beta \in \calB$ and $x \in \Omega$. In conclusion, $F(x,\cdot)$ can be represented as the inf--sup of a family of affine maps, \ie
\[
  F(x,M) = \inf_{\beta \in \calB} \sup_{\alpha \in \calA} \left[ A^{\alpha,\beta}(x):M - f^{\alpha,\beta}(x) \right].
\]

If we, in addition, assume that $F$ is convex in its second argument, then we have
\[
  F(x,M) - F(x,\alpha) \geq \frac{\p F}{\p M}(x,\alpha):(M-\alpha), \quad \forall \alpha \in \polS^d.
\]
Setting $\beta = \alpha$ in \eqref{eq:PDEeqIsaacs} yields
\[
  F(x,M) \leq \sup_{\alpha \in \polS^d}\left[ \frac{\p F}{\p M}(x,\alpha):(M-\alpha) + F(x,\alpha) \right],
\]
so that
\[
  F(x,M) = \sup_{\alpha \in \calA} \left[ A^\alpha(x):M - f^\alpha(x) \right],
\]
with $\calA = \polS^d$ and
\[
  A^\alpha(x) = \frac{\p F}{\p M}(x,\alpha) \geq 0, \qquad f^\alpha(x) = \frac{\p F}{\p M}(x,\alpha):\alpha - F(x,\alpha).
\]

We conclude by remarking that more general type of dependences can also be reduced to a similar inf--sup form; see, for instance \cite[Section 2.1]{KuoTrudinger92}, \cite{MR1312585} and \cite{MR3188552}.

\section{Elements of the theory of strongly nonlinear elliptic PDE}\label{sec:PDEs}


In order to get an idea of how to discretize strongly nonlinear partial differential equations, we must first understand their underlying structure and the main ideas that are at the basis of their theory and analysis. In this, introductory, section we collect all the relevant information that later will serve as a guide in the construction and analysis of numerical schemes. We will describe the fundamental properties that define an elliptic equation, even in the case of strong nonlinearities and, on the basis of them, define various suitable notions of solutions and their properties. We will provide existence and nonexistence results, as well as a review of the available regularity results. While it is not our intention to provide a thorough exposition of the theory, which can be found in textbooks like \cite{Evans,MR2777537,GT,MR2435520,MR1406091,MR2260015} we believe understanding this is fundamental if one wishes to provide a rigorous analysis of approximation schemes.

\subsection{Two defining consequences of ellipticity}
\label{sec:maxNenergy}

We begin our description by providing two fundamental properties that lie at the heart of much of the theory for elliptic PDEs. Namely, energy considerations and comparison principles. We will see how these give rise to various concepts of solutions and how much of the existence and regularity theory 
stems from these two simple ideas.

Let us, to make matters precise, set $\Omega \subset \Real^d$ with $d \geq 1$ to be a bounded domain with Lipschitz boundary. If further smoothness of the domain becomes necessary we will specify this at every stage. For simplicity, and because these ideas are better motivated in this case, let us consider {\it the Laplacian} which, for a function $u \in C^2(\Omega)$, is defined by
\[
  \Delta u = \sum_{i=1}^d \frac{ \partial^2 u}{\partial x_i^2}.
\]

The first fundamental property that can be observed for this operator is a {\it maximum principle}:

\begin{thm}[maximum principle]
\label{thm:maxprinc}
  Let $u \in C^2(\Omega)\cap C(\bar\Omega)$ be such that $\Delta u \geq 0$ then
  \[
    \sup_{x \in \bar \Omega} u(x) = \sup_{x \in \partial\Omega} u(x)
  \]
\end{thm}

While we will not provide a detailed proof here, we wish to provide some intuition into this fact. Namely, if we assume that a strict global maximum is attained at an interior point $z \in \Omega$, then elementary considerations from calculus will give us that:
\[
  D u(z) = 0 \qquad D^2 u(z) < 0,
\]
where $Du$ denotes the gradient and $D^2u$ the Hessian of $u$, respectively. Since it can be easily seen that $\Delta u = \tr D^2u$ a contradiction ensues.

An important consequence of this result is the following {\it comparison principle}.

\begin{col}[comparison principle]
\label{col:compprinc}
  Let $u,v \in C^2(\Omega)\cap C(\bar\Omega)$ be such that $u \leq v$ in $\partial\Omega$ and $\Delta u \geq \Delta v$ in $\Omega$. Then $u \leq v$ in $\Omega$.
\end{col}

This result easily follows by setting $w=u-v$ and observing that $w \leq 0$ on $\partial\Omega$ and $\Delta w \geq 0$ in $\Omega$. An application of the maximum principle allows us then to conclude the result.

The comparison principle is one of the fundamental properties of an elliptic operator. Namely, that an ordering on the boundary and a (reverse) ordering of the operators implies an order of the underlying functions. Throughout this survey the application of a similar principle will be a recurring feature.

Having understood comparison principles we now proceed to describe {\it energy} considerations. Consider the equation
\begin{equation}
\label{eq:secondev}
  \Delta u = f
\end{equation}
in $\Omega$ and multiply it by a sufficiently smooth function $\varphi$ that vanishes on $\partial\Omega$. An application of Green's identity reveals that
\begin{equation}
\label{eq:firstweak}
  \int_\Omega Du\cdot D \varphi = - \int_\Omega f \varphi
\end{equation}
which we immediately recognize as the Euler Lagrange equation for the minimization of the {\it energy} functional
\[
  J(v) = \int_\Omega \left( \frac12 |Dv|^2 + fv \right)
\]
subject to the condition $v=u$ on $\partial\Omega$.

It is important to notice that, as opposed to \eqref{eq:secondev}, identity \eqref{eq:firstweak} only requires the existence of square integrable first derivatives. Notice also, that the second variation of $J$ is nonnegative
\begin{equation}
\label{eq:positive}
  \partial^2 J(v)[w_1,w_2] = \int_\Omega Dw_1 \cdot D w_2, \qquad \partial^2J(v)[w,w] = \int_\Omega |Dw|^2 \geq 0.
\end{equation}
Which shows a sort of positivity.

On the basis of this observation, we now introduce our first definition of ellipticity \cite[Chapter 3]{GT}. In order to do so, in what follows we denote by $\polS^d$ the space of symmetric $d \times d$ matrices. We endow $\polS^d$ with the usual partial order
\[
  M,N \in \polS^d : \quad M \leq N\quad \Longleftrightarrow\quad \bxi \cdot M \bxi \leq \bxi\cdot N \bxi\quad \forall \bxi \in \Real^d.
\]
We denote the identity matrix by $I\in \polS^d$.

\begin{definition}[elliptic operator in divergence form]
\label{def:diveliptic}
Let $A :\Omega \to \polS^d$. We say that the operator
\begin{equation}
\label{eqn:LOperatorDef}
  L u(x) = -D\cdot (A(x) D u(x) )
\end{equation}
is {\it elliptic} at $x \in \Omega$ if $ 0 < \lambda(x) I \leq A(x) \leq \Lambda(x) I$, is {\it strictly} elliptic if $\lambda(x)\geq \lambda_0 >0 $ for all $x \in \Omega$, and {\it uniformly} elliptic if $\Lambda(x)/\lambda(x)$ is bounded in $\Omega$.
\end{definition}

\begin{ex}[lower order terms]
\label{ex:lineardiv}
The concept of elliptic operators in divergence form can be extended to operators having lower order terms. For instance,
the operator
\[
  \tilde{L} u(x) = -D\cdot\left( A(x) Du(x) + \bb(x) u(x) \right) + \bc(x)\cdot Du(x) + e(x)u(x),
\]
where $A$ is as in Definition~\ref{def:diveliptic}, and the functions $\bb,\bc:\Omega \to \Real^d$ and $e:\Omega \to \Real$ are assumed to be measurable.
\end{ex}

\begin{ex}[quasilinear operators]
\label{ex:qlineardiv}
Given a differentiable vector valued function $\Omega \times \Real \times \Real^d \ni (x,z,\bp) \mapsto \ba(x,z,\bp) \in \Real^d$ and a scalar function $\Omega \times \Real \times \Real^d \ni (x,z,\bp) \mapsto b(x,z,\bp) \in \Real$ we consider the {\it quasilinear} operator
\[
  Qu(x) = -D\cdot\ba(x,u,Du) + b(x,u,Du)
\]
defined for $u \in C^2(\Omega)$. We say that this operator is {\it variational} if it is the Euler Lagrange operator of the energy functional
\[
  \int_\Omega E(x,u,Du),
\]
that is, $\ba(x,z,\bp) = D_\bp E(x,z,\bp)$ and $b(x,z,\bp) = D_z E(x,z,\bp)$. Following Definition~\ref{def:diveliptic} we realize that the ellipticity of $Q$ is equivalent to the strict convexity of $E$ with respect to the $\bp$ variables. This immediately hints at the fact that tools from calculus of variations will be essential in the study of equations with this type of operators. Examples of quasilinear operators can be given by choosing appropriate energies $E$. For instance, setting \cite[Chapter 10]{GT}
\[
  E = E(\bp) = \left( 1 + |\bp|^2 \right)^{s/2}
\]
for $s>1$ we obtain a family of uniformly elliptic quasilinear operators.
\end{ex}

On the other hand, many problems cannot be cast into this form. The prototypical example is that given by the operator
\begin{equation}
\label{eq:nondiv}
  \calL u(x) = A(x):D^2 u(x),
\end{equation}
where, for $M,N \in \polS^d$, $M:N$ denotes the 
Fr\"obenius inner product:
\[
  M:N = \sum_{i,j=1}^d M_{i,j} N_{i,j}.
\]
The Fr\"obenius norm of a matrix $M$ will be denoted by $|M|:=\sqrt{M:M}$.

If $A$ is sufficiently smooth, the operator \eqref{eq:nondiv} can be recast in divergence form and $-\calL$ can be understood as an elliptic operator in the sense of Definition~\ref{def:diveliptic}. However, this is not always possible and, consequently, we must extend the notion of ellipticity to nondivergence form operators \cite[Chapter 3]{GT}.

\begin{definition}[nondivergence elliptic operator]
\label{def:nondiveliptic}
We say that the operator \eqref{eq:nondiv} is {\it elliptic} at $x \in \Omega$ if $ 0 < \lambda(x) I \leq A(x) \leq \Lambda(x) I$, is {\it strictly} elliptic if $\lambda(x)\geq \lambda_0 >0 $ for all $x \in \Omega$ and {\it uniformly} elliptic if $\Lambda(x)/\lambda(x)$ is bounded in $\Omega$.
\end{definition}

The reader is encouraged to verify that, for an operator that is strictly elliptic in the sense of Definition~\ref{def:nondiveliptic}, variants of Theorem~\ref{thm:maxprinc} and Corollary~\ref{col:compprinc} hold.

\begin{ex}[lower order terms]
\label{ex:lowerordernondiv}
As in the divergence form case, the notion of elliptic 
operators extend to those with lower order terms, e.g., 
\[
  \tilde{\calL} u(x) = A(x):D^2u(x) + \bb(x)\cdot Du(x) + c(x)u(x),
\]
where $A$ is as in Definition~\ref{def:nondiveliptic}, and the lower order coefficients $\bb$ and $c$ are (vector and scalar valued) functions defined on $\Omega$.
\end{ex}

\begin{ex}[linear operators in divergence form]
\label{ex:divisnondiv}
Consider the operator $L$ of Definition~\ref{def:diveliptic} and assume that the coefficient matrix $A$ is differentiable. One can then rewrite $Lu(x)$ as
\[
  L u(x) = -A(x):D^2u(x) -(D\cdot A(x))\cdot Du(x),
\]
where the divergence operator acts on $A$ column--wise.
Consequently, the operator $-L$ is of the form $\tilde{\calL}$ of Example~\ref{ex:lowerordernondiv}.
\end{ex}

\begin{ex}[quasilinear operators in nondivergence form]
\label{ex:qlinnondiv}
For a function $u \in C^2(\Omega)$ we define the {\it quasilinear} operator
\[
  \calQ u(x) = A(x,u,Du):D^2u(x) + b(x,u,Du)
\]
where $\Omega \times \Real \times \Real^d \ni (x,z,\bp) \mapsto A(x,z,\bp) \in \polS^d$ and $\Omega \times \Real \times \Real^d \ni (x,z,\bp) \mapsto b(x,z,\bp) \in \Real$. We say that $\calQ$ is elliptic at the function $u$ if $A(x,u,Du)$ satisfies the positivity conditions of Definition~\ref{def:nondiveliptic}.
\end{ex}

The previous definitions and examples entailed linear and quasilinear operators. While, by linearization, one could extend Definitions~\ref{def:diveliptic} and \ref{def:nondiveliptic} to more general nonlinear problems, we shall instead give a general definition of ellipticity, one that preserves the fundamental concept of comparison for these type of problems \cite{CIL,MR1351007}.

\begin{definition}[elliptic operator]
\label{def:FLelliptic}
Let $F \in C( \Omega \times \Real \times \Real^d \times \polS^d)$. We say that $F$ is {\it elliptic} in $\Omega$ if $F$ satisfies the following {\it monotonicity condition}: If $r,s \in \Real$ and $M,N \in \polS^d$ with $r \geq s$ and $M \leq N$ then
\[
  F(x,r,\bp,M) \leq F(x,s,\bp,N).
\]
We will say, moreover, that $F$ is {\it uniformly elliptic} if there are constants $0<\lambda \leq \Lambda$ such that for all $M \in \polS^d$ and $r \geq s$ we have
\[
  \lambda |N| \leq F(x,r,\bp,M+N) - F(x,s,\bp,M) \leq \Lambda |N|, \quad \forall N \geq 0.
\]
\end{definition}

\begin{ex}[linear and quasilinear equations]
\label{ex:liniseliptic}
Let us, as a first example, show that the linear operator in nondivergence form $\calL$ of \eqref{eq:nondiv} is elliptic in the sense of Definition~\ref{def:FLelliptic}. By doing so and following the considerations of the examples previously given, we see that all the other cases also fit into this framework. Define
\[
  F(x,r,\bp,M) = \tr(A(x)M).
\]
Since, for symmetric matrices, $A:B = \tr(AB)$ we see that $\calL u(x) = F(x,u(x),Du(x),D^2u(x))$. Moreover, the positivity of $A$ implies the monotonicity of $F$.
\end{ex}

We now present several examples of {\it fully nonlinear} equations that fit into Definition~\ref{def:FLelliptic}.

\begin{ex}[Hamilton Jacobi Bellman operator]
\label{ex:HJB}
Let $\calA$ be any compact set and assume that for every $\alpha \in \calA$ we are given a uniformly elliptic linear operator
\[
  \calL^\alpha u(x) = A^\alpha(x):D^2 u(x)
\]
and a function $f^\alpha \in C(\Omega)$. Define
\[
  F(x,r,\bp,M) = \sup_{\alpha \in \calA} \left[ \tr( A^\alpha(x) M) - f^\alpha(x) \right].
\]
Notice immediately that
\[
  F(x,u(x),Du(x),D^2u(x)) = \inf_{\alpha \in \calA} \left[ \calL^\alpha u(x) - f^\alpha (x) \right].
\]
Moreover since, for every $\alpha \in \calA$, $x \in \Omega$ and $M,N \in \polS^d$, we have that $A^\alpha(x)M \leq A^\alpha(x)N$ whenever $M \leq N$, we immediately conclude that the operator $F$ is monotone and thus elliptic in the sense of Definition~\ref{def:FLelliptic}. A similar argument shows that $F$ is uniformly elliptic whenever the family of linear operators $\{\calL^\alpha\}_{\alpha \in \calA}$ is uniformly elliptic. More importantly, we notice that the function $F$ is convex with respect to $M$.
\end{ex}

\begin{ex}[Isaacs operator]
\label{ex:Isaacs}
The previous example can be generalized as follows. Assume now that we have two index 
sets $\calA$ and $\calB$ and, for each $(\alpha,\beta) \in \calA\times \calB$, we have a uniformly elliptic linear operator
\[
  \calL^{\alpha,\beta} u(x) = A^{\alpha,\beta}(x):D^2 u(x).
\]
Define
\[
  F(x,r,\bp,M) = \inf_{\beta \in \calB} \sup_{\alpha \in \calA} \left[ \tr(A^{\alpha,\beta}(x)M) - f^{\alpha,\beta}(x) \right],
\]
and notice that
\[
  F(x,u(x),Du(x),D^2u(x)) = \inf_{\beta \in \calB} \sup_{\alpha \in \calA} \left[ \calL^{\alpha,\beta}u(x) - f^{\alpha,\beta}(x) \right].
\]
One more time, the uniform ellipticity of the operators $\calL^{\alpha,\beta}$ yields the uniform ellipticity of $F$. Notice, 
that $F$ is neither convex nor concave with respect to $M$.
\end{ex}

\begin{ex}[Monge Amp\`ere operator]
\label{ex:MongeAmpere}
As a final example, consider the operator
\[
  F(x,r,\bp,M) = \det M - f(x).
\]
Notice that, in general, this operator does not satisfy the monotonicity condition of Definition~\ref{def:FLelliptic}. However, if we restrict it to {\it positive definite} matrices, then this operator is uniformly elliptic. Consequently, for a positive $f$ and a {\it strictly convex} function $u \in C^2(\Omega)$ we define the Monge Amp\`ere operator as
\[
  F(x,u(x),Du(x),D^2u(x)) = \det D^2u(x) - f(x).
\]
\end{ex}


With Definition~\ref{def:FLelliptic} at hand, we turn our attention to boundary value problems for elliptic operators. In other words, for an elliptic operator $F$ we consider the problem: find $u : \bar\Omega \to \Real$ such that
\begin{equation}
\label{eq:BVP}
  F(x,u,Du,D^2u) = 0 \ \text{in } \Omega, \quad u = g \ \text{on } \partial\Omega.
\end{equation}
The meaning in which \eqref{eq:BVP} is satisfied will give rise to the various existing concepts of solutions.

\subsection{Classical solutions}
\label{sub:classicalsol}
The first, and obvious, notion of solution is when identity \eqref{eq:BVP} is understood in a pointwise sense. This gives rise to so-called classical solutions.

\begin{definition}[classical solution]
\label{def:clasSol}
Let $F \in C( \Omega \times \Real \times \Real^d \times \polS^d)$, then the function $u \in C^2(\Omega) \cap C(\bar\Omega)$ is said to be a {\it classical} solution of \eqref{eq:BVP} if this identity holds for every $x \in \bar\Omega$.
\end{definition}

An immediate consequence of ellipticity is that classical solutions are unique.

\begin{thm}[uniqueness]
\label{thm:classicalunique}
Let $F$ be elliptic in the sense of Definition~\ref{def:FLelliptic}. If $F$ is strictly decreasing in the $r$ variable or uniformly elliptic, then problem \eqref{eq:BVP} cannot have more than one classical solution.
\end{thm}
\begin{proof}
Let us prove this result under the assumption that the map $F$ is strictly decreasing in the $r$ variable. 
The remaining case can be found, for instance, in \cite[Corollary 17.2]{GT}. Assume that $u$ and $v$ are classical solutions to \eqref{eq:BVP} and set $w=u-v$. Notice that $w=0$ on $\partial\Omega$ and that if $w$ attains a (positive) maximum at $x_0 \in \Omega$, then $Dw(x_0) = 0$ and $D^2 w(x_0)\leq0$. Therefore, if $u(x_0)>v(x_0)$, ellipticity and the fact that the map is strictly decreasing imply
\[
  0 = F(x_0,u(x_0),Du(x_0),D^2u(x_0)) < F(x_0,v(x_0), Dv(x_0), D^2v(x_0)),
\]
which is a contradiction. Similarly, the function $w$ cannot attain a negative minimum in $\Omega$ and, consequently, $w \equiv 0$.
\end{proof}

Let us, as an example, mention that Theorem~\ref{thm:classicalunique} holds for the operator $\tilde\calL$ of Example~\ref{ex:lowerordernondiv} whenever the zero order coefficient $c \leq 0$.

In the linear case of Example~\ref{ex:liniseliptic}, the Dirichlet problem \eqref{eq:BVP} reads: find $u \in C^2(\Omega)\cap C(\bar\Omega)$ such that
\begin{equation}
\label{eq:linnondiv}
  \calL u = f \ \text{in } \Omega, \quad u=g \ \text{on } \partial\Omega.
\end{equation}
In this case, the existence of classical solutions is guaranteed by what is known as Schauder estimates which, simply put, boil down to freezing the coefficients and a continuity argument; see \cite[Theorems 6.13-6.14]{GT}.

\begin{thm}[existence]
\label{thm:Schauder}
Let $\Omega$ satisfy an exterior sphere condition at every boundary point. Assume the operator \eqref{eq:nondiv} is strictly elliptic in the sense of Definition~\ref{def:nondiveliptic}. If $g \in C(\partial\Omega)$ and, for some $\alpha \in (0,1)$, $f$ and the coefficients of $\calL$ are bounded and belong to $C^\alpha(\Omega)$, then the Dirichlet problem \eqref{eq:linnondiv}
has a unique classical solution $u \in C^{2,\alpha}(\Omega)\cap C(\bar\Omega)$. If, in addition, we assume that $\partial \Omega \in C^{2,\alpha}$, that $f$ and the coefficients of $\calL$ belong to $C^{\alpha}(\bar\Omega)$; and $g \in C^{2,\alpha}(\bar\Omega)$, then $u \in C^{2,\alpha}(\bar\Omega)$ and
\[
  \| u \|_{C^{2,\alpha}(\bar\Omega)} \leq C \left( \| f \|_{C^\alpha(\bar\Omega)} + \| g \|_{C^{2,\alpha}(\bar\Omega)} \right),
\]
where the constant $C$ is independent of $u$, $f$ and $g$.
\end{thm}

At this point, the reader may wonder if H\"older continuity is indeed necessary for these results. Example~\ref{ex:DunotC1} below will show us that this is the case.

If more regularity is assumed on the domain and problem data, it can be shown that the (unique) classical solution is also more regular \cite[Theorem 6.19]{GT}.

\begin{thm}[regularity]
\label{thm:regclassic}
In the setting of Theorem~\ref{thm:Schauder} assume additionally that, for some $k\geq0$ we have $\partial\Omega \in C^{k+2,\alpha}$, $g \in C^{k+2,\alpha}(\bar\Omega)$ and that $f$ and the coefficients of $\calL$ belong to $C^{k,\alpha}(\bar\Omega)$. Then $u \in C^{k+2,\alpha}(\bar\Omega)$.
\end{thm}

While Theorems~\ref{thm:Schauder} and \ref{thm:regclassic} provide a satisfactory and conclusive answer for a linear operator with smooth coefficients, it does not cover rough coefficients or nonlinear problems, in which a classical solution may not exist. This is why we must depart from classical solutions and consider {\it weakened} or {\it generalized} concepts of solutions.

\subsection{Weak (variational) solutions}
\label{sub:weakvarsols}

We now turn our attention to the case of divergence form operators as in Definition~\ref{def:diveliptic} and consider, for this particular operator, the Dirichlet problem \eqref{eq:BVP}. The natural solution concept in this case is called weak solution, and follows from an integration by parts argument, and an integral identity similar to \eqref{eq:firstweak}.

\begin{definition}[weak solutions of linear equations]
\label{def:weaksol}
A function $u \in H^1(\Omega)$ is said to be a {\it weak} solution of the Dirichlet problem
\begin{equation}
\label{eq:ff}
  Lu = f, \ \text{in } \Omega, \quad u=g, \ \text{on } \partial\Omega
\end{equation}
if $u-g \in H^1_0(\Omega)$ and, for every $\varphi \in H^1_0(\Omega)$, we have
\begin{equation}
\label{eq:weaksol}
  \int_\Omega \big(A D u\big)\cdot D\varphi = \int_\Omega f \varphi.
\end{equation}
\end{definition}

Existence and uniqueness follow from the classical Lax-Milgram lemma or, more generally, from so-called inf-sup conditions.

\begin{thm}[existence and uniqueness]
\label{thm:laxmilgram}
Let $\Omega$ be bounded, $L$ be uniformly elliptic in the sense of Definition~\ref{def:diveliptic} and such that its coefficients belong to $L^\infty(\Omega)$. If $f \in H^{-1}(\Omega)$ and $g \in H^1(\Omega)$, then problem \eqref{eq:ff} has a unique weak solution $u \in H^1(\Omega)$.
\end{thm}

Again, under additional smoothness assumptions on the domain and problem data, one can assert further differentiability of the solution. This is the content of the following result \cite{GT,GrisvardBook}.

\begin{thm}[regularity]
\label{thm:regweak}
Assume, in addition to the conditions of Theorem~\ref{thm:laxmilgram} that $\partial\Omega \in C^2$ or that $\Omega$ is convex. If $A \in C^{0,1}(\bar\Omega,\polS^d)$, $f \in L^2(\Omega)$ and $g \in H^2(\Omega)$ then we have $u \in H^2(\Omega)\cap H^1(\Omega)$ and 
\[
  \| u \|_{H^2(\Omega)} \leq C \left( \| u \|_{L^2(\Omega)} + \| f \|_{L^2(\Omega)} + \| g \|_{H^2(\Omega)} \right)
\]
where the constant $C$ is independent of $u$, $f$ and $g$.
\end{thm}

We wish to also mention the remarkable result by E.~De Giorgi concerning the H\"older regularity of weak solutions.

\begin{thm}[De Giorgi I]
\label{thm:DeGiorgi}
Let $u \in H^1(\Omega)$ be a weak solution to \eqref{eq:weaksol} with $g=0$, $f \in L^q(\Omega)$ with $q>d/2$, then there is $\alpha \in (0,1)$ for which $u \in C^\alpha_{loc}(\Omega)$. If, in addition, $q>d$ and $A \in C^\beta(\bar\Omega,\polS^d)$, with $\beta = 1-d/q$, then $u \in C^{1,\beta}_{loc}(\Omega)$.
\end{thm}

In light of the second part of the previous result, it is natural to ask whether $f \in L^\infty(\Omega)$ with appropriate assumptions on the boundary data $g$ and the coefficients of $L$ would yield that $Du \in C^1_{loc}(\Omega)$. The following example shows that this, in general, is false \cite[Section 3.4]{MR2777537}.

\begin{ex}[second derivatives are not continuous]
\label{ex:DunotC1}
For $R<1$ let $\Omega = \{ x \in \Real^d: |x|<R \}$ and consider
\[
  f(x) = \begin{dcases}
           0, & x = 0, \\
           \frac{x_2^2-x_1^2}{2|x|^2} \left( \frac{d+2}{\sqrt{-\ln|x|}} + \frac1{2(-\ln|x|)^{3/2}} \right), & x \neq 0.
         \end{dcases}
\]
Notice that $f \in C(\bar\Omega)$ and that the function $u(x) = (x_1^2-x_2^2)\sqrt{-\ln|x|} \in C(\bar\Omega) \cap C^\infty(\bar\Omega \setminus \{0\})$ satisfies $\Delta u = f$ with boundary conditions
\[
  g = \sqrt{-\ln R}(x_1^2 - x_2^2).
\]
However, this function cannot be a classical solution since
\[
  \lim_{|x|\to 0} \frac{\partial^2 u(x)}{\partial x_1^2} = \infty
\]
so that $u \not \in C^2(\Omega)$. In fact, although the problem has a weak solution, it does not have a classical one. This example also shows that, in the classical solution theory given in Theorem~\ref{thm:Schauder}, mere continuity of the data is not sufficient, thus justifying the need for H\"older continuity.
\end{ex}

Let us now focus our attention on the quasilinear operator $Q$ of Example~\ref{ex:qlineardiv} and consider the Dirichlet problem
\begin{equation}
\label{eq:qlin}
  Q u = f, \ \text{in } \Omega, \quad u = g, \text{on } \partial\Omega.
\end{equation}
The definition of weak solution is as follows.

\begin{definition}[weak solutions of quasilinear equations]
\label{def:weakqlin}
A function $u \in W^{1,p}(\Omega)$ ($1<p<\infty$) is called a {\it weak} solution of \eqref{eq:qlin} if $u-g \in W^{1,p}_0(\Omega)$ and
\[
  \int_\Omega \big(\ba(x,u,Du) \cdot Dv + b(x,u,Du) v\big) = \int_\Omega f v
\]
for all $v \in W^{1,p}_0(\Omega)$.
\end{definition}

Notice that the equation that defines weak solutions to \eqref{eq:qlin} are the Euler Lagrange equations of the functional
\[
  I(u) = \int_\Omega \big(E(x,u,Du) - fu\big)
\]
over the set of functions $v \in W^{1,p}(\Omega)$ such that $u-g \in W^{1,p}_0(\Omega)$. Consequently, the existence of weak solutions is tightly bound with the calculus of variations.

\begin{thm}[existence and uniqueness]
\label{thm:calcvar}
Assume that there is a $p \in (1,\infty)$ for which the function $E$ satisfies the coercivity condition: there are constants $C_1>0$, $C_2\geq0$ such that, for every $x \in \Omega$, $z \in \Real$ and $\bp \in \Real^d$ we have
\[
  E(x,z,\bp) \geq C_1 |\bp|^p - C_2.
\]
Assume, in addition, that $E$ is convex in the $\bp$ variable. Then, for $f \in L^{p'}(\Omega)$, the functional $I$ has a minimizer $u \in W^{1,p}(\Omega)$ such that $u-g \in W^{1,p}_0(\Omega)$. Finally, if $E$ does not depend on $z$ and is uniformly convex, then this minimizer is unique.
\end{thm}

With this theorem at hand it can be readily shown that, in this setting, the (unique) minimizer $u$ of $I$ is a weak solution of \eqref{eq:qlin} in the sense of Definition~\ref{def:weakqlin}.

We can also establish, under additional assumptions on $E$, further differentiability of minimizers. To shorten the exposition we confine ourselves to the case where $E$ is independent of $x$ and $z$, it is coercive with $p=2$, satisfies the {\it growth condition}
\begin{equation}
\label{eq:Egrowth}
  |D_\bp E(\bp)| \leq C (|\bp| +1), \quad \forall \bp \in \Real^d
\end{equation}
and
\begin{equation}
\label{eq:D2Ebdd}
  |D^2E(\bp)| \leq C , \quad \forall \bp \in \Real^d.
\end{equation}
With these additional assumptions we have the following regularity result \cite[Theorem 8.3.1]{Evans}.

\begin{thm}[regularity]
\label{thm:regqlin}
In the setting of Theorem~\ref{thm:calcvar} assume, in addition, that $E$ depends only on $\bp$ and satisfies \eqref{eq:Egrowth} and \eqref{eq:D2Ebdd}. If $g=0$, $f \in L^2(\Omega)$ and $\partial\Omega \in C^2$ we have that $u \in H^2(\Omega)$ with the estimate
\[
  \| u \|_{H^2(\Omega)} \leq C \| f \|_{L^2(\Omega)}.
\]
\end{thm}

What is more interesting and remarkable is that the results of De Giorgi presented in Theorem~\ref{thm:DeGiorgi} can be extended to this case as well.

\begin{thm}[De Giorgi II]
\label{thm:DeGiorgiII}
Let $u \in W^{1,p}(\Omega)$ be a minimizer of $I$. If $E$ satisfies the growth and monotonicity conditions
\[
  | D_\bp E(x,z,\bp) | \leq C_1 \left( 1 + |\bp|^{p-1} \right), \qquad
  D_\bp E(x,z,\bp).\bp \geq C_2 |\bp|^p - C_3
\]
then there is $\alpha \in (0,1)$ for which $u \in C^\alpha_{loc}(\Omega)$.
\end{thm}

Under suitable assumptions, local H\"older continuity of the gradients of the minimizers can also be established. For further regularity results for quasilinear problems the reader is referred, for instance, to \cite{MR2291779}.

While, in this setting, we have a sufficiently rich theory, it only applies to divergence form operators. Below, in Section~\ref{sub:viscosols} we will describe the right generalization of the notion of solutions for more general problems.

\subsection{Strong solutions}
\label{sub:strongsols}

We now describe a solution concept that, in a sense, lies in between classical and weak solutions, and that can also be applied to nondivergence form operators such as \eqref{eq:nondiv} and that of Example~\ref{ex:qlinnondiv}. These
solutions are called strong.

\begin{definition}[strong solutions]
\label{def:strongsol}
The function $u \in W^{2,p}(\Omega)$ ($1 < p<\infty$) is a {\it strong} solution of the boundary value problem \eqref{eq:BVP} if the equation and boundary conditions hold almost everywhere in $\Omega$ and $\partial\Omega$, respectively.
\end{definition}

We immediately remark that every classical solution is a strong solution. Moreover, an integration by parts and density argument shows that a sufficiently regular weak solution (\cf Theorems~\ref{thm:regweak} and \ref{thm:regqlin}) is also a strong solution. Therefore, strong solutions for the divergence form  equations \eqref{eq:ff} and \eqref{eq:qlin} can be obtained from regularity considerations. 

Let us now turn our attention to the nondivergence form problem \eqref{eq:linnondiv} and study the existence of strong solutions. In this case we have the following result.

\begin{thm}[existence]
\label{thm:CZ}
Let $\Omega$ be a $C^{1,1}$ domain and the coefficients of the operator $L$ belong to $C(\bar\Omega)$. If $f \in L^p(\Omega)$ and $g \in W^{2,p}(\Omega)$\ ($1<p<\infty$), then the Dirichlet problem \eqref{eq:linnondiv} has a unique strong solution $u \in W^{2,p}(\Omega)$ and, moreover
\[
  \| u \|_{W^{2,p}(\Omega)} \leq C \left( \| f \|_{L^p(\Omega)} + \| g \|_{W^{2,p}(\Omega)} \right),
\]
where the constant $C$ is independent of $u$, $f$ and $g$, but depends on $\|A\|_{C(\bar\Omega,\polS^d)}$, the dimension $d$ and the exponent $p$.
\end{thm}

We must comment on the technique of proof for this result. First, for $A =I$ and $p=2$, this follows from the regularity result of Theorem~\ref{thm:regweak}. An interpolation result, in conjunction with the celebrated Calder\'on Zygmund decomposition technique  \cite{MR0052553} yields the result for any $p$. Using the continuity of $A$ the result can be extended to 
a general $\calL$.

\begin{rem}[H\"older regularity]
\label{rem:SafonovcounterEx}
Let us briefly describe the results of Krylov and Safonov, see \cite[Section 9.8]{GT} and Theorem~\ref{thm:Calpha} below. To do so, we assume that $f \in L^d(\Omega)$, $g \in C^\beta(\bar\Omega)$ for some $\beta \in (0,1)$ and $\partial\Omega$ satisfies a uniform exterior cone condition. Then, given $\omega \Subset \Omega$, there is a constant $\alpha \in (0,1)$ such that
\[
  | u |_{C^\alpha(\omega)} \leq C.
\]
The constants $\alpha$ and $C$ depend, in particular, on the dimension $d$ and the ratio $\Lambda/\lambda$ that defines the ellipticity of $\calL$. A natural question to ask is whether a similar estimate for the gradient $Du$ (possibly under stricter smoothness assumptions) is possible. A result by Nirenberg, see Theorem~\ref{thm:Nirenberg}, showed that this is the case for $d=2$. For higher dimensions, however, this turns out to be false. Safonov \cite{MR882838} showed that in $B_1\subset \Real^3$, the unit ball, 
for every $\alpha \in (0,1]$ there are:
\begin{enumerate}[$\bullet$]
  \item A bounded function $v \in C^\infty( \Real^3\setminus\{0\})$,
  \item a constant $\nu \in (0,1)$,
  \item A family $\{A_\vare\}_{\vare>0} \subset C^\infty(\bar B_1, \polS^d)$ such that the associated nondivergence operators $\calL_\vare$ are uniformly elliptic with $\Lambda/\lambda =1/\nu^2$.
\end{enumerate}
With these objects at hand, he showed that the solution to the problem
\[
  \calL_\vare u_\vare = 0, \ \text{in } B_1, \quad u_\vare = v \ \text{on } \partial B_1,
\]
satisfies $u_\vare \in C^\infty(\bar B_1)$, $\| u_\vare \|_{L^\infty(B_1)} = 1$ but
\[
  \lim_{\vare \downarrow 0} | u_\vare |_{C^\alpha(B_{1/2})} = \infty.
\]
From this it immediately follows that H\"older estimates on the derivatives are not possible.
\end{rem}

Let us point out now that, in Theorem~\ref{thm:CZ}, the assumption that $A \in C(\bar\Omega,\polS^d)$ cannot be, in general, weakened. The following example is due to Pucci.

\begin{ex}[nonuniqueness]
\label{ex:Talentinonunique}
Let us show, following \cite[Section 1.1]{MR2260015}, that for $d\geq3$ there is a bounded measurable matrix $A$ such that problem \eqref{eq:linnondiv} with $f=0$ and $g=0$ has more than one strong solution in $H^2(\Omega) \cap H^1_0(\Omega)$.
Let $\Omega$ be the unit ball of $\Real^d$ and define
\[
  A(x) = I + b \frac{x x^\intercal}{|x|^2}, \quad b = \frac{d-2+\lambda}{1-\lambda},
  \quad \max\{2-d/2,0\} < \lambda < 1.
\]
Obviously
\[
  |\bxi|^2 \leq \bxi \cdot  A \bxi \leq (1+b) |\bxi|^2,
\]
so that $A$ is bounded and the associated operators $\calL$ are uniformly elliptic. Define $u(x) = |x|^\lambda -1$  and notice that
\[
  D^2 u(x) = \lambda (\lambda-2) |x|^{\lambda -4} x x^\intercal + \lambda |x|^{\lambda - 2} I,
\]
which, since $\lambda > 2 - d/2$, shows that $ u \in H^2(\Omega) \cap H^1_0(\Omega)$. Moreover, due to the choice of $b$, we have
\[
\calL u(x) =   A:D^2 u(x) = \lambda |x|^{\lambda -2} \left[ (1+b) \lambda + d - b -2 \right] = 0.
\]
\end{ex}

Since this will be important in subsequent developments, we now focus on conditions weaker that continuity that allow for the existence and uniqueness of a strong solution for \eqref{eq:linnondiv}.

\subsubsection{The Cordes condition}
\label{sub:Cordes}

Since, as Example~\ref{ex:Talentinonunique} shows, mere boundedness of the coefficients in the operator of \eqref{eq:linnondiv} does not suffice to ensure uniqueness of strong solutions, here we study the so-called Cordes condition for linear operators in nondivergence form. The idea behind it and the theory that follows is to reformulate the operator in a way that the result is ``close'' to a one in divergence form, in particular, the Poisson equation.  This reformulation allows us to apply classical tools in functional analysis to study the existence, uniqueness and a priori estimates for problem \eqref{eq:nondiv}.

To motivate and derive this condition consider the following problem: given $x\in \Omega$, find $\gamma(x)\in \Real$ that minimizes the quadratic function 
\[
  \tau \mapsto |\tau A(x) - I|^2.
\]
Simple  arguments show that the minimum is attained at 
\begin{equation}
\label{eq:Cordes1}
  \gamma(x) = \frac{ \tr A(x) }{ |A(x)|^2 },
  \quad\text{and}\quad
  |\gamma(x)A(x) -I|^2 = d - \frac{(\tr A(x))^2}{|A(x)|^2}.
\end{equation}
In particular, this simple calculation shows that
\begin{equation}
\begin{aligned}
  \left|\gamma(x) \calL v(x)-\Delta v(x)\right|^2 &= \left|\left(\gamma(x) A(x)-I\right):D^2 v(x) \right|^2 \\
  &\leq \left( d - \frac{ \tr A(x)^2}{|A(x)|^2}\right)|D^2 v(x)|^2.
\end{aligned}
\label{eq:Cordes1B}
\end{equation}
The Cordes condition ensures that the multiplicative constant on the right-hand side of \eqref{eq:Cordes1B}  is less than one.  

\begin{definition}[Cordes condition]
\label{def:Cordes}
A positive definite matrix $A\in L^\infty(\Omega,\polS^d)$ 
satisfies the {\it Cordes condition} provided there exists an $\epsilon\in (0,1]$ such that
\begin{equation}
\label{eq:CordesCond}
  \frac{|A|^2}{ (\tr A)^2} \leq \frac{1}{d-1+\epsilon} \quad \mae \Omega.
\end{equation}
\end{definition}

Notice that the Cordes condition ensures that there exists $\gamma>0$ such that, for all $v\in H^2(\Omega)$, we have
\begin{equation}
\label{eq:Cordes1C}
  \|\gamma \calL v -\Delta v\|_{L^2(\Omega)}\le \sqrt{1-\epsilon} \|D^2 v\|_{L^2(\Omega)}.
\end{equation}

\begin{rem}[the Cordes condition in spectral terms]
\label{rem:Cordeseigen}
Since, by assumption, for $\mae x \in \Omega$ we have that $A(x) \in \polS^d$ and that it is positive definite, it is diagonalizable and all its eigenvalues $\{\lambda_i\}_{i=1}^d = \{\lambda_i(x)\}_{i=1}^d $ are positive. Using the well known identities $|A|^2 = \sum_{i=1}^d \lambda_i^2$, $ \tr A = \sum_{i=1}^d \lambda_i$ and $(\sum_{i=1}^d \lambda_i)^2\le d \sum_{i=1}^d \lambda_i^2$ condition \eqref{eq:CordesCond} can be recast, in terms of the eigenvalues of $A$ as follows:
\[
  \frac1d \leq \frac{\sum_{i=1}^d \lambda_i^2}{\left(\sum_{i=1}^d \lambda_i \right)^2}
  \leq \frac1{d-1+\epsilon} \quad \mae \Omega.
\]
In other words, \eqref{eq:CordesCond} is an anisotropy condition on $A$ that becomes more stringent in higher dimensions.
\end{rem}

The considerations in Remark~\ref{rem:Cordeseigen} show that the Cordes condition is always satisfied in two dimensions with  $\epsilon = \inf_{x\in \Omega} 2\lambda_1\lambda_2/(\lambda_1^2+\lambda_2^2)\in (0,1]$. On the other hand, there exist symmetric positive definite matrices in three dimensions (and higher) that do not satisfy \eqref{eq:CordesCond}. 

\begin{ex}[three dimensions]
\label{ex:Cordes3d}
Consider the matrix
\begin{align*}
  A = 
  \begin{pmatrix}
  1 & 0 & c\\
  0 & 1 & b\\
  c & b & 4
  \end{pmatrix}
\end{align*}
with $b^2+c^2<4$ so that $\det(A) = 4-(b^2+c^2)>0$. Sylvester's criterion ensures that the matrix is positive definite,
and a straightforward calculation shows that
\begin{align*}
\frac{|A|^2}{(\tr A)^2} = \frac{18+2(b^2+c^2)}{6^2} \geq \frac12.
\end{align*}
Thus $A$ does not satisfy the Cordes condition.
\end{ex}

\begin{ex}[the example of Pucci]
\label{ex:TalentinotCordes}
As another example consider the matrix of Example~\ref{ex:Talentinonunique} for $d=3$. Notice, first of all that, that in this case we have
\[
  \frac12 < \lambda < 1, \qquad b = \frac{3-2+\lambda}{1-\lambda} > 3.
\]
Simple calculations then yield
\begin{align*}
  |A|^2 &=  d + 2b + b^2
  \\
  (\tr A)^2 &= \left(d+ b \right)^2 
\end{align*}
and therefore
\begin{align*}
  |A|^2 - \frac{1}{d-1} (\tr A)^2
  & = d + 2b + b^2 - \frac{(d+b)^2}{d-1} > 0
\end{align*}
Thus one concludes that $A$ does not satisfy the Cordes condition.
\end{ex}

The Cordes condition is a key assumption to establish the well-posedness of the elliptic problem \eqref{eq:nondiv} with discontinuous coefficients.  Another crucial ingredient is the Miranda-Talenti estimate which is summarized in the next lemma.

\begin{lem}[Miranda-Talenti estimate]
Let $\Omega\subset \Real^d$ be a bounded convex domain.  Then for any $v\in H^2(\Omega)\cap H^1_0(\Omega)$ there holds
\begin{equation}
\label{eq:MirandaTalenti}
  |v|_{H^2(\Omega)}\le \|\Delta v\|_{L^2(\Omega)}.
\end{equation}
\end{lem}

While this result can be understood as a regularity estimate in the spirit of Theorem~\ref{thm:regweak}, we remark that it can be obtained without appealing to this theory; we refer the reader to \cite[Lemma 1.2.2]{MR2260015} for a proof. Moreover, while the aforementioned regularity results yield that, for functions in $H^2(\Omega)\cap H^1_0(\Omega)$, the norm $v \mapsto \|\Delta v\|_{L^2(\Omega)}$ is equivalent to the $H^2(\Omega)$-norm; the important feature of estimate \eqref{eq:MirandaTalenti} is that the equivalence constant is exactly one on convex domains.

\begin{rem}[polygonal domains]
\label{rem:poly2d}
In two dimensions the Miranda-Talenti estimate \eqref{eq:MirandaTalenti} holds for a polygonal domain $\Omega$ {\it without the convexity assumption}. Indeed, assuming that $u \in C^\infty(\bar\Omega)$, we have
\[
  |D^2 u|^2 = |\Delta u|^2 + 2( |\partial_{12}u|^2 - \partial_{11}u\partial_{22}u ),
\]
where we explicitly used that we are in two dimensions. In addition, integration by parts and some algebraic manipulations show (see \cite[equation (1.2.9)]{CiarletBook}) that
\[
  \int_\Omega ( |\partial_{12}u|^2 - \partial_{11}u\partial_{22}u ) = 
  \int_{\partial\Omega} ( - \partial_{\tau\tau}u \partial_n u + \partial_{n\tau}u\partial_\tau u ),
\]
where $\btau$ is the unit tangential vector along the boundary $\partial\Omega$, $\partial_\tau$ is the derivative in its direction and $\partial_n$ denotes the normal derivative. Now, if $u=0$ on $\partial\Omega$ then we have that $\partial_\tau u = 0$ so that the second term on the right hand side of this expression vanishes. If, in addition, $\partial\Omega$ is polygonal, this also implies that $\partial_{\tau\tau}u = 0$, which allows us to obtain \eqref{eq:MirandaTalenti}. By density, the same result holds for every $u \in H^2(\Omega) \cap H^1_0(\Omega)$.
\end{rem}

Identities \eqref{eq:Cordes1} and \eqref{eq:MirandaTalenti} motivate the introduction of the bilinear form
\[
  a(\cdot,\cdot): \left(H^2(\Omega)\cap H^1_0(\Omega) \right)^2 \ni (v,w) \mapsto 
  a(v,w) =  \int_\Omega \gamma \calL v \Delta w  \in \Real.
\]
The properties of $a$ are as follows.

\begin{lem}[properties of $a$]
\label{lem:propofa}
Assume that the coefficient $A$ of the operator $\calL$ satisfies the Cordes condition \eqref{eq:CordesCond}. If $\Omega$ is convex, then the bilinear form $a$ is bounded and coercive on $H^2(\Omega)\cap H^1_0(\Omega)$.
\end{lem}
\begin{proof}
Since $\gamma$ is bounded, the continuity immediately follows.

If the Cordes condition \eqref{eq:CordesCond} is satisfied and $\Omega$ is convex, then from the Miranda-Talenti estimate \eqref{eq:MirandaTalenti} and Cauchy Schwarz inequality we obtain,
\begin{equation}
\label{eq:coerciveProof}
\begin{aligned}
a(v,v) 
&= \|\Delta v\|_{L^2(\Omega)}^2 + \int_\Omega (\gamma \calL v  -\Delta v) \Delta v \\
&\ge \|\Delta v\|_{L^2(\Omega)}^2 -\sqrt{1-\epsilon} |v|_{H^2(\Omega)}\|\Delta v\|_{L^2(\Omega)}\\
& \ge \big(1-\sqrt{1-\epsilon}\big)\|\Delta v\|_{L^2(\Omega)}^2.
\end{aligned}
\end{equation}
In conclusion, $a$ is coercive on $H^2(\Omega)\cap H^1_0(\Omega)$.
\end{proof}

The coercivity estimate of Lemma~\ref{lem:propofa} allows us to show the existence and uniqueness of strong solutions under the Cordes condition.

\begin{thm}[existence and uniqueness]
\label{thm:exuniqueCordes}
Assume that the coefficient $A$ of the operator $\calL$ satisfies the Cordes condition \eqref{eq:CordesCond}. If $\Omega$ is convex, then the Dirichlet problem \eqref{eq:linnondiv} with $f \in L^2(\Omega)$ and $g = 0$ has a unique strong solution $u \in H^2(\Omega) \cap H^1_0(\Omega)$. Moreover, we have
\[
\|u\|_{H^2(\Omega)}\le C\frac{\|\gamma\|_{L^\infty(\Omega)}}{1-\sqrt{1-\epsilon}} \|f\|_{L^2(\Omega)},
\]
where the constant $C$ is independent of $u$ and $f$.
\end{thm}
\begin{proof}
From Lemma~\ref{lem:propofa} and the Lax-Milgram Lemma, there exists a unique $u\in H^2(\Omega)\cap H^1_0(\Omega)$ satisfying
\begin{equation*}
  a(u,v) = \int_\Omega \gamma f \Delta v \quad \forall v\in H^2(\Omega)\cap H^1_0(\Omega).
\end{equation*}
Since $u\in H^2(\Omega)$ and the Laplace operator $\Delta:H^2(\Omega)\cap H^1_0(\Omega)\to L^2(\Omega)$ is surjective
on convex domains, standard arguments  show that $u$ satisfies $\calL u = f$ almost everywhere, \ie it is a strong
solution to the elliptic problem \eqref{eq:linnondiv}. The coercivity condition \eqref{eq:coerciveProof} also implies the a priori estimate
\[
  \|u\|_{H^2(\Omega)} \le C\|\Delta u\|_{L^2(\Omega)}\le C\frac{\|\gamma\|_{L^\infty(\Omega)}}{1-\sqrt{1-\epsilon}} \|f\|_{L^2(\Omega)}.
\]
\end{proof}

\begin{rem}[inf-sup conditions]
\label{rem:infsup}
Since the Laplace operator is surjective from $H^2(\Omega)\cap H^1_0(\Omega)$ to $L^2(\Omega)$ on convex domains, the above arguments show that the inf--sup condition
\[
  \sup_{w\in L^2(\Omega) \setminus \{0\}} \frac{\int_\Omega \gamma \calL v w }{\|w\|_{L^2(\Omega)}} \geq \left(1-\sqrt{1-\epsilon}\right)\|\Delta v\|_{L^2(\Omega)}
  \quad \forall v\in H^2(\Omega)\cap H^1_0(\Omega)
\]
is satisfied. One can then appeal to the Babu{\v s}ka-Brezzi theorem to deduce the existence of strong solutions to \eqref{eq:nondiv}.  To our knowledge the use of this inf-sup condition for the numerical approximation has yet to be investigated.
\end{rem}

\begin{rem}[the case $p \neq 2$]
\label{rem:pnot2}
It is possible to show \cite[Theorem 1.2.3]{MR2260015} that, if $\Omega$ is convex and $A$ satisfies the Cordes condition, there are $1<p_l<2<p_r<\infty$ such that if $p \in (p_l,p_r)$, $f\in L^p(\Omega)$ and $g=0$, then problem \eqref{eq:linnondiv} has a unique strong solution $u \in W^{2,p}(\Omega) \cap W^{1,p}_0(\Omega)$. We also have an a priori estimate in which the constant now depends on $p$.
\end{rem}

\begin{rem}[strong solutions under other conditions]
\label{rem:W1nandVMO}
It is possible to obtain the existence and uniqueness of strong solutions for problem \eqref{eq:linnondiv} under other assumptions. Let us discuss two of them:
\begin{enumerate}[$\bullet$]
  \item Assuming that $A \in W^{1,d}(\Omega,\polS^d)$ one can rewrite the operator in nondivergence form and extend the theory of weak solutions, described in Section~\ref{sub:weakvarsols}, to coefficients in this class. What is remarkable is that, in Example~\ref{ex:Talentinonunique}, for every $\varepsilon>0$ the parameter $\lambda$ can be chosen so that $A \in W^{1,d-\varepsilon}(\Omega,\polS^d)$, thus showing that $A \in W^{1,d-\varepsilon}(\Omega,\polS^d)$ is not sufficient for uniqueness. On the other hand, if $A \in W^{1,d+\varepsilon}(\Omega,\polS^d)$, for some $\varepsilon >0$, then $A \in C^{0,\alpha}(\bar\Omega,\polS^d)$ and, consequently, the classical Schauder theory applies (cf. Theorem~\ref{thm:Schauder}).

  \item In essence, the case of 
  uniformly continuous coefficients boils down to realizing that, locally, their oscillation in the $L^\infty(\Omega)$-norm is small, and so they can be considered a constant. These ideas have been extended, see \cite[Chapter 2]{MR2260015}, to the case of a coefficient $A \in VMO(\Omega,\polS^d)$, thus showing that this is a sufficient condition to obtain strong solutions. Since $W^{1,d}(\Omega)$ is a proper subset of $VMO(\Omega)$ this result truly extends the case of Sobolev coefficients detailed above.
\end{enumerate}
\end{rem}

\subsection{Viscosity solutions}
\label{sub:viscosols}

At this point we wish to introduce one final notion of solution, the one that will be suited for the study of fully nonlinear equations. This is that of a viscosity solution. The reader may recall that the notion of weak solutions, introduced in Section~\ref{sub:weakvarsols}, was based on an integration by parts argument \eqref{eq:firstweak} and the positivity \eqref{eq:positive} of the resulting operators. While this proved sufficient for linear and quasilinear operators in divergence form, different arguments are necessary for fully nonlinear operators as those of Examples~\ref{ex:HJB}--\ref{ex:MongeAmpere}. The fundamental property that will be used to define solutions in this case will be, as in Corollary~\ref{col:compprinc}, a comparison principle.

\subsubsection{Definition and first properties}
\label{subsub:defvisco}
Let us begin by motivating the definition following \cite{MR3289084}. Let $F$ be an elliptic operator in the sense of Definition~\ref{def:FLelliptic} and $u \in C^2(\Omega)$ a classical solution to 
\begin{equation}
\label{eq:PDEnobcs}
  F(x,u,Du,D^2 u) = 0, \ \text{in } \Omega.
\end{equation}
Let $x_0 \in \Omega$ and assume that there is a smooth function $\varphi \in C^2(\Omega)$ that can {\it touch from above} the graph of $u$ at $x_0$. More precisely, we assume that
\[
  u(x) \leq \varphi(x) \ \forall x \in \Omega, \qquad u(x_0) = \varphi(x_0).
\]
These conditions imply that the function $u-\varphi$ has a local maximum at $x_0$ and, consequently,
\[
  D (u - \varphi)(x_0) = 0, \qquad D^2(u-\varphi)(x_0) \leq 0.
\]
Since the operator $F$ is assumed to be elliptic we obtain
\begin{align*}
  0 &= F(x_0,u(x_0),Du(x_0),D^2 u(x_0)) = F(x_0,\varphi(x_0),D \varphi(x_0),D^2 u(x_0)) \\
    &\leq F(x_0,\varphi(x_0),D \varphi(x_0),D^2 \varphi(x_0)).
\end{align*}
Similar considerations will give us that if $\psi \in C^2(\Omega)$ {\it touches from below} the graph of $u$ at $x_0$ we would obtain
\[
  F(x_0,\psi(x_0),D \psi(x_0),D^2 \psi(x_0)) \leq 0.
\]
Finally we notice that it is possible to reach the same conclusions if we replace the equality in \eqref{eq:PDEnobcs} by a corresponding inequality. These considerations motivate the following definition.

\begin{definition}[viscosity solution]
\label{def:viscosol}
Let $F$ be elliptic in the sense of Definition~\ref{def:FLelliptic}. We say that the function $u \in C(\Omega)$ is:
\begin{enumerate}[(a)]
  \item A {\it viscosity subsolution} of \eqref{eq:PDEnobcs} if whenever $x_0 \in \Omega$, $\varphi \in C^2(\Omega)$ and $u-\varphi$ has a local maximum at $x_0$ we have that
  \[
    F(x_0,\varphi(x_0),D \varphi(x_0),D^2 \varphi(x_0)) \geq 0.
  \]
  
  \item A {\it viscosity supersolution} of \eqref{eq:PDEnobcs} if whenever $x_0 \in \Omega$, $\varphi \in C^2(\Omega)$ and $u-\varphi$ has a local minimum at $x_0$ we have that
  \[
    F(x_0,\varphi(x_0),D \varphi(x_0),D^2 \varphi(x_0)) \leq 0.
  \]
  
  \item A {\it viscosity solution} if it is a sub- and supersolution.
\end{enumerate}
\end{definition}

\begin{rem}[viscosity solutions]
\label{rem:viscos}
Several remarks must be immediately made about Definition~\ref{def:viscosol}.
\begin{enumerate}[$\bullet$]
  \item While the motivation provided assumed that the function $u$ is smooth, the definition only requires its continuity.
  
  \item By approximation and continuity of $F$, it is sufficient to verify the condition for quadratic polynomials $\varphi \in \polP_2$, see \cite[Proposition 2.4]{MR1351007}.
  
  \item If $u \in C^2(\Omega)$ is a classical solution then it is a viscosity solution. This follows from the ellipticity of $F$. Moreover, sufficiently smooth viscosity solutions are also classical \cite[Lemma 2.5]{MR1351007} and \cite[Theorem 2.11]{MR3289084}.
  
  \item This definition talks only about solutions to equation \eqref{eq:PDEnobcs} not the boundary value problem \eqref{eq:BVP}. More details on this issue will be provided below; see Definition~\ref{def:viscobvp} and Section~\ref{subsub:weirdBCs}.
  
  \item The definition assumes that the candidate solution {\it can} be touched from above (below). At points where this is not possible there is nothing to verify and the function automatically satisfies the equation at these points.
  
  \item For the divergence form operators $L$ and $\tilde L$ it is known \cite{IshiiEquiv} that the concepts of weak solution, in the sense of Definition~\ref{def:weaksol}, and viscosity solutions coincide.
  
  \item We will not provide a historical account of the origin and development of this definition. The interested reader can consult the classical reference \cite{CIL}.
\end{enumerate}
\end{rem}

A remarkable property of viscosity solutions is its {\it stability}, which is detailed in the following two results. For a proof of the first one we refer to \cite[Proposition 2.8]{MR1351007} \cite[Theorem 3.2]{MR3289084} or \cite[Section 6]{CIL}. For the second one, we refer to \cite[Proposition 2.7]{MR1351007} or \cite[Theorem 3.12]{MR3289084}.

\begin{thm}[limits and viscosity solutions]
\label{thm:visconlimits}
Let $\{F_k\}_{k \in \polN}$ be a sequence of uniformly elliptic operators in the sense of Definition~\ref{def:FLelliptic} and let $\{u_k\}_{k\in \polN} \subset C(\Omega)$ be, for each $k$, viscosity subsolutions to the equations
\[
  F_k(x,u_k,Du_k,D^2u_k)=0.
\]
If, as $k \to \infty$, $F_k \to F$ uniformly on compact subsets of $\Omega \times \Real \times \Real^d \times \polS^d$and $u_k \to u$ uniformly in compact subsets of $\Omega$, then $u$ is a viscosity subsolution of 
\[
  F(x,u,Du,D^2u)=0.
\]
\end{thm}

\begin{thm}[suprema of subsolutions]
\label{thm:viscosup}
Let $\calU \subset C(\Omega)$ be a set of viscosity subsolutions of \eqref{eq:PDEnobcs}. For $x \in \Omega$ define
\[
 \bar u (x) = \sup\{ u(x) : u \in \calU \}.
\]
Suppose that $\bar u \in C(\Omega)$ and is bounded. Then $\bar u$ is a viscosity subsolution of \eqref{eq:PDEnobcs}.
\end{thm}

\subsubsection{Existence and uniqueness}
\label{subsub:existenceuniquenessvisco}

Let us now turn our attention to the existence and uniqueness of viscosity solutions to \eqref{eq:BVP}. To do so we must specify in which sense the boundary conditions are being understood. We begin by introducing the notion of semicontinuity.

\begin{definition}[semicontinuity]
\label{def:usclsc}
We say that the function $u \in LSC(\Omega)$ (is {\it lower semicontinuous}) if, for all $x \in \Omega$,
\[
  u(x) \leq \liminf_{y \to x } u(y).
\]
On the other hand, we say that $u \in USC(\Omega)$ (is {\it upper semicontinuous}) if $-u \in LSC(\Omega)$.
\end{definition}

Notice that, in Definition~\ref{def:viscosol} and the discussion that followed, nothing would have changed if we had only required that subsolutions and supersolutions are upper and lower semicontinuous, respectively. With this definition at hand, we may define viscosity solutions to the Dirichlet problem.

\begin{definition}[solution to the Dirichlet problem]
\label{def:viscobvp}
Let $F$ be elliptic in the sense of Definition~\ref{def:FLelliptic} and $g \in C(\partial\Omega)$. We say that:
\begin{enumerate}[(a)]
  \item The function $u_\star \in USC(\bar\Omega)$ is a viscosity subsolution to \eqref{eq:BVP} if it is a viscosity subsolution to the equation (\eg \eqref{eq:PDEnobcs}) and $u_\star(x) \leq g(x)$ for all $x \in \partial\Omega$.
  
  \item The function $u^\star \in LSC(\bar\Omega)$ is a viscosity supersolution to \eqref{eq:BVP} if it is a viscosity supersolution to \eqref{eq:PDEnobcs} and $u^\star(x) \geq g(x)$ for all $x \in \partial\Omega$.
  
  \item The function $u \in C(\bar\Omega)$ is a viscosity solution to \eqref{eq:BVP} if it is a sub- and supersolution.
\end{enumerate}
\end{definition}

Notice that this definition  requires the boundary values to be attained in the {\it classical sense}. Different boundary conditions might require a different interpretation, and we will briefly comment on this below.

We now turn our attention to the existence of solutions and the so-called Perron's method. Simply put, this method provides existence under the assumption that the problem cannot have more than one solution. While, as shown in Theorem~\ref{thm:classicalunique}, uniqueness of classical solutions is immediate; in this more general setting we need one additional condition.

\begin{definition}[comparison]
\label{def:comparison}
We say that the Dirichlet problem \eqref{eq:BVP} satisfies a comparison principle if, whenever $u_\star \in USC(\Omega)$ and $u^\star \in LSC(\Omega)$ are sub- and supersolutions, respectively, we have
\[
  u_\star \leq u^\star \ \text{in } \Omega.
\]
\end{definition}

Notice that from Definition~\ref{def:comparison}, it immediately follows that \eqref{eq:BVP} cannot have more than one solution. Indeed, if $u$ and $v$ are two viscosity solutions then, in particular, $u$ is a subsolution and $v$ a supersolution; consequently, $u \leq v$. An analogous reasoning yields the reverse inequality.

With these two conditions 
at hand, we proceed to show existence of solutions. For a proof, we refer the reader, for instance, to \cite[Theorem 4.1]{CIL} and \cite[Theorem 5.3]{MR3289084}.

\begin{thm}[Perron's method]
\label{thm:Perronexistence}
Let $F$ be elliptic in the sense of Definition~\ref{def:FLelliptic} and $g \in C(\partial\Omega)$. Assume that the Dirichlet problem \eqref{eq:BVP} satisfies a comparison principle in the sense of Definition~\ref{def:comparison}. If there exist a subsolution $u_\star$ and a supersolution $u^\star$ to \eqref{eq:BVP} that satisfy the boundary condition, then
\[
  u(x) = \sup \left\{ v(x): u_\star \leq v \leq u^\star \ \text{and $v$ is a subsolution} \right\}
\]
defines a viscosity solution to \eqref{eq:BVP}.
\end{thm}

Notice that, while Theorem~\ref{thm:Perronexistence} provides a somewhat explicit construction of the unique solution to \eqref{eq:BVP}, one still needs to verify the existence of sub- and supersolutions that satisfy the boundary condition in a {\it classical sense}. This must be done on a case by case basis and we refer the reader to \cite[Example 4.6]{CIL} and \cite[Application 5.9]{MR3289084} for two examples where these are constructed.

It remains to understand which operators satisfy the comparison principle of Definition~\ref{def:comparison}. Loosely speaking, similar ideas to those presented in Theorem~\ref{thm:classicalunique} should yield uniqueness of viscosity solutions. However, the arguments presented there cannot be applied directly since we are dealing with functions that are merely continuous and additional structural conditions must be imposed. This is due to the subtle fact, which may have escaped the reader, that Definition~\ref{def:FLelliptic} is {\it too general}. By this we mean that, for instance, first order and parabolic equations fit into this definition. For this reason, many authors say that an operator is {\it degenerate elliptic} if it only satisfies the monotonicity condition with respect to the $M$ variable. This is in contrast with uniform ellipticity, which precludes these two degenerate cases.

Let us then, for the sake of illustration, concentrate our efforts in finding a comparison principle for uniformly elliptic equations. We begin by showing, following \cite[Example 1]{MR2354491} that uniform ellipticity is not enough to ensure a comparison principle.

\begin{ex}[lack of comparison]
\label{ex:kawohl}
Consider the Dirichlet problem
\[
  u'' + 18x (u')^4 = 0 \ \text{ in } (-1,1), \quad u(-1)=-b, \ u(1)=b,
\]
with $b>1$. Clearly, the equation is uniformly elliptic. It is easy to check that the functions
\[
  u_\star(x) = \begin{dcases}
                 \sqrt[3]{x}-1+b, & x \in [0,1], \\
                 \sqrt[3]{x}+1-b, & x \in [-1,0),
               \end{dcases}
  \ 
  u^\star(x) = \begin{dcases}
                 \sqrt[3]{x}-1+b, & x \in (0,1], \\
                 \sqrt[3]{x}+1-b, & x \in [-1,0],
               \end{dcases}
\]
which differ only at the origin, are viscosity sub and supersolutions, respectively, and that they satisfy the boundary values. Notice, however, that
\[
  \max_{x \in (-1,1)} \left\{ u_\star(x) - u^\star(x) \right\} = u_\star(0) - u^\star(0) = 2b-2>0.
\]
\end{ex}

While, to our knowledge, necessary and sufficient conditions for the existence of a comparison principle for a general elliptic operator are not known, there are several sufficient conditions. We collect these in the following result.

\begin{thm}[existence of comparison principle]
\label{thm:comparison}
If the Dirichlet problem \eqref{eq:BVP} satisfies any of the structural conditions given below, then it satisfies a comparison principle in the sense of Definition~\ref{def:comparison}.
\begin{enumerate}[(a)]
  \item \cite[Theorem 6.1]{MR3289084} The dependence with respect to $x$ is decoupled, \ie the equation reads
  \[
    F(u,Du,D^2u) = f
  \]
  with $F \in C(\Real,\Real^d,\polS^d)$, $f \in C(\bar\Omega)$. The operator $F$ is elliptic and satisfies, for some $\gamma>0$
  \[
    F(r, \bp, M) \geq F(s, \bp, M) + \gamma(s-r), \quad \forall r \leq s.
  \]
  
  \item \label{item:compclq0} The operator $F$ is elliptic, independent of the $r$ and $\bp$ variables and satisfies, for some $\lambda >0$,
  \[
    F(x,M+tI) \geq F(x,M) + \lambda t, \quad \forall t \geq 0.
  \]
  
  \item \cite{MR1048584} The operator $F$ is uniformly elliptic, Lipschitz continuous in $\bp$ and the following continuity assumption holds:
  \[
    \left| F(x,r,\bp,M) - F(y,r,\bp,M) \right| \leq \mu_2 |x-y|^{1/2} |M| + \omega( |x-y|),
  \]
  for all $x,y \in \Omega$, $r$ and $\bp$ in a suitable ball and $\omega(a) \to 0$ as $a \downarrow 0$. Additionally, one must assume that sub and supersolutions belong to $C^{0,1}(\Omega)$.
  
  \item \cite{Koike} The dependence with respect the $r$ variable is decoupled, \ie the equation reads
  \[
    \nu u + F(x,Du,D^2u) = 0
  \]
  with, either $\nu <0$ and $F$ elliptic, or $\nu \leq 0$, $F$ uniformly elliptic and Lipschitz in the $\bp$ variable, for $\bp \in \Real^d$.
  
  \item \cite{Silvestrenotes} The operator $F$ is independent of $x$ and $\bp$ and is strictly decresasing in $r$, \ie whenever $r > s$
  \[
    F(r,M) < F(s,M), \ \forall M \in \polS^d.
  \]

\end{enumerate}
\end{thm}

Other conditions can be found in the literature. 

\subsubsection{Other boundary conditions}
\label{subsub:weirdBCs}

So far, for all notions of solutions, we have only discussed the Dirichlet problem (see second equation in \eqref{eq:BVP}). Moreover, for viscosity solutions we have assumed that the boundary conditions are attained in a classical sense. Let us here consider other types of boundary conditions as well as generalized notions for them. Consider
\begin{equation}
\label{eq:newBVP}
  F(x,u,Du,D^2u) = 0, \ \text{in }\Omega, \quad B(x,u,Du) = 0, \ \text{on } \partial \Omega,
\end{equation}
where the map $F$ is, as before, elliptic but its domain of definition on the $x$ variable is now $\bar\Omega$. The function $B : \partial \Omega \times \Real \times \Real^d$ is assumed to be nonincreasing in its second argument, \ie
\[
  r \geq s \Rightarrow B(x,r, \bp) \leq B(x,s,\bp), \quad \forall x \in \partial\Omega, \ \bp \in \Real^d.
\]
The Dirichlet problem, obviously, falls into this description with $B(x,r,\bp) = g(x)-r$, but others are also admissible. For instance, let $\bn(x)$ denote the outer normal to $\p\Omega$ at $x$ and $\bnu : \partial\Omega \to  \Real^d$ be such that, for all $x \in \p \Omega$, we have $\bnu(x)\cdot\bn(x) >0$. The boundary condition
\[
  B(x,\bp) = \bnu(x)\cdot\bp - g(x)
\]
gives rise to the so-called {\it oblique derivative} problem; if $\bnu = \bn$, this is the Neumann problem. A nonlinear example is the capillarity condition
\[
  B(x,r,\bp) = \bn\cdot \bp - g(x,r)\sqrt{1 + |\bp|^2}.
\]

At the beginning of Section~\ref{subsub:defvisco} the introduction of viscosity solution was motivated by the assumption that the function $u-\varphi$ had a local maximum (minimum) at $x_0 \in \Omega$. When dealing with boundary conditions, we must now allow for $x \in \partial\Omega$. At these points the relations that led to the definition of viscosity solution do not hold anymore and a modification is necessary. It turns out that the correct notion is as follows.

\begin{definition}[viscosity subsolution]
\label{def:viscoBVP}
With the functions $F$ and $B$ as above, we say that $u \in USC(\bar\Omega)$ is a {\it viscosity subsolution} to \eqref{eq:newBVP} if it is a viscosity subsolution to \eqref{eq:PDEnobcs} and, whenever there is a $\varphi \in C^2(\Real^d)$ that touches the graph of $u$ from above at $x_0 \in \partial\Omega$, then either
\[
  B(x_0, \varphi(x_0), D\varphi(x_0)) \geq 0 \quad \text{or} \quad F(x_0, \varphi(x_0), D\varphi(x_0), D^2 \varphi(x_0) ) \geq 0.
\]
\end{definition}

In an analogous manner we can consider supersolutions and, as before, a solution to \eqref{eq:newBVP} is a function $u \in C(\bar\Omega)$ that is both a sub and supersolution.

It is important to realize that boundary conditions in the viscosity sense, in general, are {\it not} equivalent to those in the classical sense. The reason behind this, once more, is that Definition~\ref{def:FLelliptic} is rather general and allows, for instance, to consider first order equations for which 
Dirichlet conditions cannot be imposed on the whole boundary. It is natural to ask then when a boundary condition in the viscosity sense is attained classically. Let us briefly elaborate on this issue for the Dirichlet problem \eqref{eq:BVP}. We begin by the definition of a barrier.

\begin{definition}[barrier]
\label{def:defofbarrier}
We say that \eqref{eq:BVP} has {\it barriers} at $x_0 \in \partial\Omega$ if there exists two continuous functions $\bar u$, $\underline u$ that are super- and subsolutions to \eqref{eq:BVP}, respectively, and that satisfy $\bar u(x_0) = \underline u (x_0) = g(x_0)$.
\end{definition}

\begin{prop}[viscosity vs. classical]
\label{prop:viscoisclassic}
Let $u$ be a viscosity solution to \eqref{eq:BVP} in the sense of Definition~\ref{def:viscoBVP}. If barriers exist at $x_0 \in \partial\Omega$, then $u(x_0) = g(x_0)$.
\end{prop}

In other words, classical and viscosity conditions coincide at points where it is possible to construct a barrier. We conclude this discussion by providing a sufficient condition for the existence of barriers.

\begin{prop}[existence of barriers]
\label{prop:barriersexists}
Let $\Omega$ be such that it has a tangent ball from outside at every point of $\partial\Omega$. If $F$ is uniformly elliptic, Lipschitz with respect to all its variables and, for every $x \in \bar \Omega$, we have $F(x,0,\boldsymbol{0},0) = 0$, then barriers exist at every point $x_0 \in \partial\Omega$.
\end{prop}

\subsubsection{Regularity}
\label{subsub:regularity}

To finalize the presentation on viscosity solutions, we elaborate on their regularity. This is important not only because these results will serve as a guide to establish rates of convergence for numerical schemes, but also many of the ideas and techniques that we present here have a discrete analogue that will be detailed in subsequent sections.

We begin with a result by Nirenberg \cite{MR0064986} that shows that in two dimensions, essentially, all solutions to elliptic equations are locally $C^{2,\alpha}$.

\begin{thm}[regularity in two dimensions]
\label{thm:Nirenberg}
Let $d=2$. Assume that $F$ is uniformly elliptic in the sense of Definition~\ref{def:FLelliptic} and that it has bounded first derivatives with respect to all its arguments. If $u$ is a solution to \eqref{eq:PDEnobcs}, then for every $\omega \Subset \Omega$ there are  $C>0$, $\alpha \in (0,1)$ that depend only on the ellipticity constants of $F$, the bounds on its first derivatives and the distance between $\omega$ and $\partial\Omega$ for which
\[
  \| u \|_{C^{2,\alpha}(\omega)} \leq C \| u \|_{L^\infty(\Omega)}.
\]
\end{thm}

It is remarkable that this result was obtained long before the development of the theory of viscosity solutions. 

To obtain global regularity or results in more dimensions we begin by introducing several notions of a more or less geometrical nature. Recall that a function $u: \Omega \to \Real$ is convex if
\[
  u(\alpha x + (1-\alpha)y) \leq \alpha u(x) + (1-\alpha)u(y), \quad \forall x, y\in \Omega, \ \alpha \in [0,1].
\]
For a convex function we define its subdifferential as follows.

\begin{definition}[subdifferential]
\label{def:subdiff}
Let $u \in C(\Omega)$. The subdifferential of $u$ at the point $x \in \Omega$ is
\[
  \partial u(x) = \left\{ \bp \in \Real^d: u(y) - u(x) \geq \bp\cdot(y-x) \ \forall y \in \Omega \right\}.
\]
\end{definition}

It is well known that \cite{MR1727362}, if $u$ is convex, then $\partial u(x) \neq \emptyset$ and that if $u$ is differentiable at $x$ then $\partial u(x) = \{ Du(x) \}$. Given a function $u$, we can always construct the largest convex function lying below $u$, this gives rise to the {\it convex envelope}. In what follows we will only need this concept for the negative part of a function, so we define the convex envelope in this restricted setting.

\begin{definition}[convex envelope and contact set]
\label{def:convexenvelope}
Let $B_r$ be a ball such that $\Omega \subset B_r$ and let $v \in C(\Omega)$ with $v \geq 0$ on $\partial\Omega$. Extend $v^-$ by zero to $B_r\setminus \Omega$. The {\it convex envelope} of $v$ is defined, for $x\in B_r$, by
\[
  \Gamma(v)(x) = \sup \left\{ L(x):  \ L(z) \leq -v^-(z) \ \forall z \in B_r, \ L \in \polP_1 \right\}.
\]
The points at which these two functions coincide are called {\it contact points}
\[
  \calC^-(v) = \left\{ x \in B_r: v(x) = \Gamma(v)(x) \right\}.
\]
\end{definition}

\begin{figure}[h]
  \begin{center}
    \includegraphics[scale=0.4]{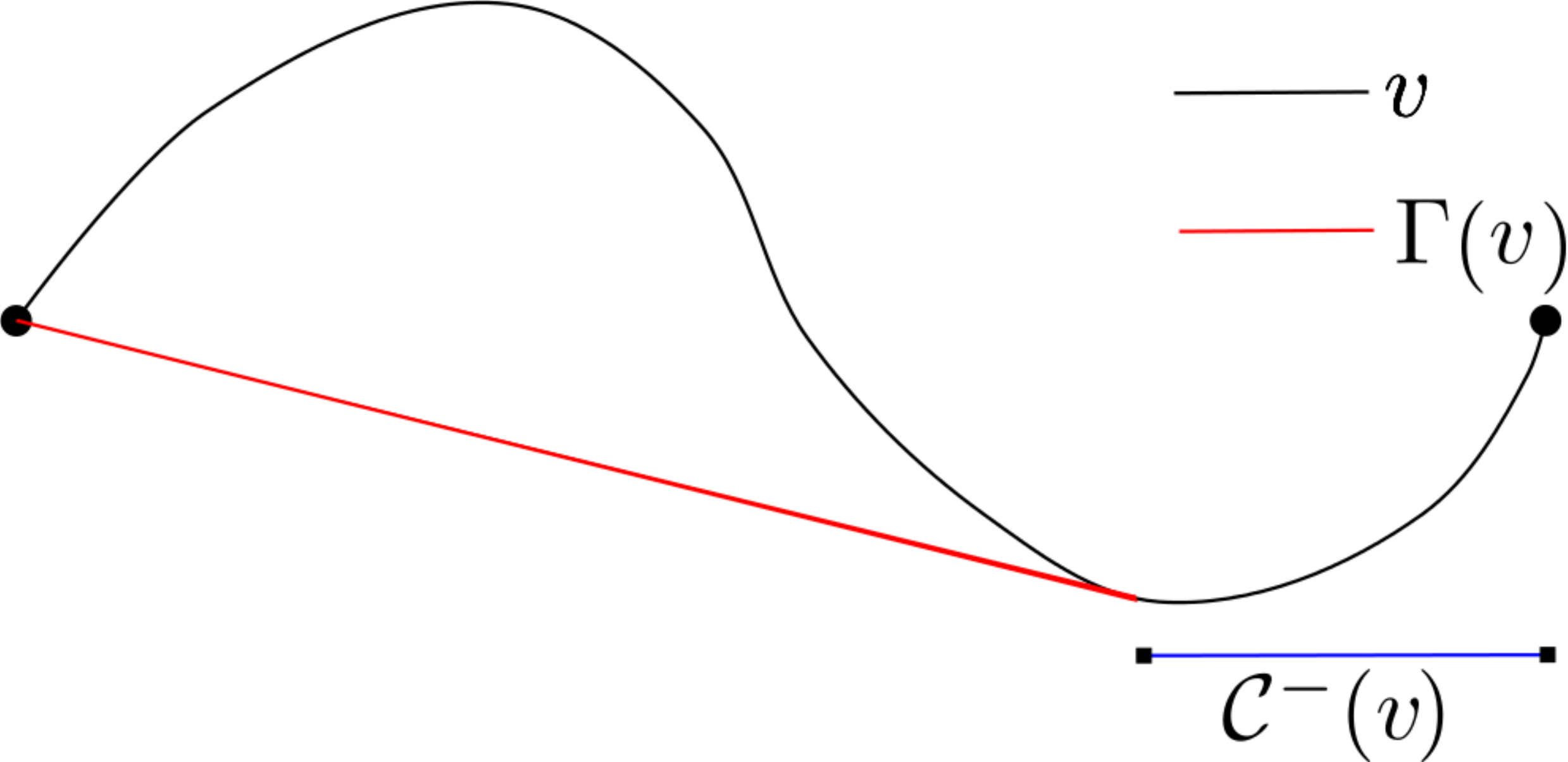}
  \end{center}
\caption{Convex envelope and contact set.}
\label{fig:ABP}
\end{figure}

An illustration of the convex envelope of a function and its contact set is given in Figure~\ref{fig:ABP}. From the figure it is intuitively clear that, for fixed values of $v$ on the boundary, how deep the graph of $v$ can go depends only on the values of $v$ at $\calC^-(v)$. The formalization of this observation is the so-called Alexandrov estimate.

\begin{thm}[Alexandrov estimate]
\label{thm:Alexandrov}
Let $v \in C(\bar B_r)$ with $v \geq 0$ on $\partial B_r$. If $\Gamma(v) \in C^{1,1}(B_r)$, then
\[
  \sup_{B_r} v^- \leq C r | \partial \Gamma(v)(\calC^-(v)) |^{1/d}.
\]
In other words, there is a set $A \subset B_r$ that satisfies $|B_r \setminus A |=0 $ and for which we have
\[
  \sup_{B_r} v^- \leq C r \left( \int_{A \cap \calC^-(v)} \det D^2 \Gamma(v) \right)^{1/d},
\]
where the constant $C$ depends only on $d$.
\end{thm}
\begin{proof}
Let us, for the sake of completeness, sketch the proof for $v \in C^2(B_r)$, since in this case $\partial \Gamma(v)$ is single valued on $\calC^-(v)$.

Let $M = \sup_{B_r} v^-/2r$ and assume that $B_M \subset \partial \Gamma(v)(\calC^-(v))$. If that is the case,
\[
  M^d \leq C | \partial \Gamma(v)(\calC^-(v)) |,
\]
for a constant that depends only on the dimension $d$. This shows the first estimate. On the other hand, a simple change of variables yields
\[
  | \partial \Gamma(v)(\calC^-(v)) | = \int_{ \partial \Gamma(v)(\calC^-(v)) } =
  \int_{\calC^-(v)} \det D^2 \Gamma(v),
\]
so that the second statement follows from the first one.

We now show the inclusion $B_M \subset \partial \Gamma(v)(\calC^-(v))$. Let $z \in B_r$ be a point where $\sup_{B_r} v^-$ is attained. For $\ba \in B_M$, define the affine function
\[
  L(x) = -\sup_{B_r} v^- + \ba \cdot (x-z)
\]
and notice that $L(z)=-\sup_{B_r} v^-$ and, for all $x \in B_r$,
\[
  L(x) \leq -\sup_{B_r} v^- + |\ba||x-z| < - \sup_{B_r} v^- + 2Mr = 0.
\]
Since $Dv(z) =0$ there is a $x_1 \in B_r$ such that $v(x_1)<L(x_1)<0$. In addition, we have that $v(x) \geq 0 > L(x)$ for $x \in \partial B_r$. This shows that, if $\bar x \in B_r$ is a point where $v-L$ attains its minimum, then $v(\bar x) < L(\bar x) \leq 0$ and $Dv(\bar x) = DL( \bar x) = \ba$.

Define $\tilde L(x) = L(x) + v(\bar x) - L(\bar x)$ and notice that $v(\bar x) = \tilde L(\bar x)$, $Dv(\bar x) = D \tilde L(\bar x) = \ba$ and, for every $x \in B_r$ $v(x) \geq \tilde L(x)$. In other words, $\tilde L$ is a supporting hyperplane for $v$. This shows that $\bar x \in \calC^-(v)$ and that $\ba \in \partial \Gamma(v)(\bar x)$, \ie $B_M \subset \partial\Gamma(v)(\calC^-(v))$.
\end{proof}

Notice that in Theorem~\ref{thm:Alexandrov} only the contact set is relevant. This is due to the fact that if for $x_0 \in B_r$ we have $\Gamma(v)(x_0) < v(x_0)$, then locally $\Gamma(v)$ is affine, and thus $D^2 \Gamma(v) = 0$.

With this estimate at hand we can proceed to obtain the fundamental a priori estimate for viscosity solutions, the so-called Alexandrov-Bakelman-Pucci estimate. We begin by providing some motivation for this result. To do so, assume that $u \in C^2(\bar\Omega)$ with $u\geq0$ on $\partial B_r$ satisfies $\calL u \leq f$ in $B_r$. In this setting we have that, for $x \in \calC^-(u)$, $D^2 u(x) \geq 0$ and, consequently, $A(x): D^2u(x) \geq 0$ as well. This, in particular implies that $f(x) \geq 0$. Denote $D = \inf_{x \in \bar B_r} \det A(x) >0$ and observe that
\[
  \det A(x) D^2 u(x) = \det A(x) \det D^2 u(x) \geq D \det D^2 u(x) \geq 0.
\]
Let $\sigma(A(x) D^2u(x)) = \{\mu_i(x) \}_{i=1}^d$, an application of the arithmetic-geometric inequality reveals that
\begin{align*}
  \det A(x) D^2 u(x) &= \prod_{i=1}^d \mu_i(x) = \left( \prod_{i=1}^d \mu_i(x)^{1/d} \right)^d \leq
  \left( \frac1d \sum_{i=1}^d \mu_i(x) \right)^d \\
  &= \left( \frac1d \tr A(x)D^2u(x) \right)^d = \left( \frac1d A(x):D^2u(x) \right)^d 
  \leq \left( \frac1d f(x) \right)^d,
\end{align*}
where in the last step we used that $\calL u \leq f$. Theorem~\ref{thm:Alexandrov} then yields that
\[
  \sup_{B_r} u^- \leq C r\left( \int_{\calC^-(u)} (f^+)^d \right)^{1/d},
\]
for a constant that depends only on $d$ and $D$.

While the considerations presented assumed that we were working with a linear equation, we essentially used that the matrix was uniformly positive definite and bounded, \ie that the operator $\calL$ is elliptic. A similar conclusion can be drawn from the fact that an operator is elliptic in the sense of Definition~\ref{def:FLelliptic}. We begin by observing \cite[Lemma 2.2]{MR1351007} that $F$ is uniformly elliptic if and only if
\[
  F(x,r,\bp,M +N) - F(x,r,\bp,M) \leq \Lambda |N^+| - \lambda |N^-|,
\]
where $N = N^+ - N^-$ with $N^+,N^- \geq 0$ and $N^+ N^- = 0$. Now, if $u$ is a sufficiently smooth subsolution of \eqref{eq:PDEnobcs}, from the observation above we have
\begin{align*}
  &\leq   \Lambda \sum_{\lambda_i(u) >0 } \lambda_i(u) + \lambda \sum_{\lambda_i(u) <0} \lambda_i(u),
\end{align*}
where $\sigma(D^2u(x)) = \{ \lambda_i(u) \}_{i=1}^d$. Similarly, for a supersolution we have
\[
  f(x)  \geq \Lambda \sum_{\lambda_i(u) <0 } \lambda_i(u) + \lambda \sum_{\lambda_i(u) >0} \lambda_i(u).
\]
This motivates the following definitions which, in a sense, describe the class of all possible viscosity solutions to uniformly elliptic equations.

\begin{definition}[class $\calS$]
Let the operator $F$ be uniformly elliptic in the sense of Definition~\ref{def:FLelliptic} and denote $f=- F(\cdot,0, \boldsymbol0,0)$. We say that $u \in \underline{\calS}(\lambda,\Lambda,f)$ if $u \in C(\Omega)$ and the inequality
\[
  f(x) \leq \Lambda \sum_{\lambda_i(u) >0 } \lambda_i(u) + \lambda \sum_{\lambda_i(u) <0} \lambda_i(u),
\]
holds in the viscosity sense. Similarly, we say that $u \in \overline{\calS}(\lambda,\Lambda,f)$ if $u \in C(\Omega)$ and
\[
  f(x)  \geq \Lambda \sum_{\lambda_i(u) <0 } \lambda_i(u) + \lambda \sum_{\lambda_i(u) >0} \lambda_i(u)
\]
in the viscosity sense. Finally $\calS(\lambda,\Lambda,f) = \underline{\calS}(\lambda,\Lambda,f) \cap \overline{\calS}(\lambda,\Lambda,f)$.
\end{definition}

With this notation at hand we present the Alexandrov-Bakelman-Pucci (ABP) estimate

\begin{thm}[ABP estimate]
\label{thm:ABP}
Let $u \in \overline{\calS}(\lambda,\Lambda,f)$ in $\Omega$ with $u\geq 0$ on $\partial \Omega$ and assume that $f$ is continuous and bounded in $\Omega$. Then
\[
  \sup_{\Omega} u^- \leq C r \left( \int_{\calC^-(u)} (f^+)^d \right)^{1/d},
\]
where the constant $C$ depends only on $d$, $\lambda $ and $\Lambda$ and $r$ is such that $\Omega \subset B_{r/2}$ and we have extended $u$ by zero outside $\Omega$.
\end{thm}

Notice that, as in the Alexandrov estimate, only the contact set $\calC^-(u)$ is relevant in this estimate. Note also that we obtain control of the $L^\infty$-norm of $u$ in terms of the $L^d$-norm of the data $f$. While Theorem~\ref{thm:ABP} is a sort of stability estimate, it is also useful in establishing regularity of solutions. To do so, we begin with the Harnack inequality of Krylov and Safonov; see \cite{MR579490}. In what follows, by $Q_l$ we denote a cube with sides parallel to the coordinate axes and of length $l$.

\begin{thm}[Harnack inequality]
\label{thm:Harnack}
Let $u \in \calS(\lambda,\Lambda,f)$ in $Q_1$ with $f \in C(Q_1)\cap L^\infty(Q_1)$. If $u \geq 0$ in $Q_1$, then 
\[
  \sup_{Q_{1/2}} u \leq C \left( \inf_{Q_1} u + \| f \|_{L^d(Q_1)} \right),
\]
where the constant $C$ depends only on $d$, $\lambda $ and $\Lambda$.
\end{thm}

Since this will be useful in the sequel, let us now show how from a Harnack inequality one can obtain interior H\"older continuity of functions in $\calS(\lambda,\Lambda,f)$. We begin with a technical result, commonly referred as an iteration lemma; see \cite[Lemma 3.4]{MR2777537} and \cite[Lemma 8.23]{GT}.

\begin{lem}[iteration]
\label{lem:iteration}
Let $\varphi : (0,R] \to \Real$ be nondecreasing. Assume that for some $A>0$, $B\geq 0$ and $\alpha > \beta$ we have
\[
  \varphi(\rho) \leq A \left[ \left( \frac\rho{r}\right)^\alpha + \varepsilon \right] \varphi(r) + Br^\beta
\]
whenever $0<\rho\leq r \leq R$. Then, for every $\gamma \in (\beta,\alpha)$ there is $\varepsilon_0$ such that if $\varepsilon  < \varepsilon_0$, then
\[
  \varphi(r) \leq C \left[ \frac{\varphi(R_0)}{R_0} r^\gamma + B r^\beta \right], \quad \forall r \in [0,R_0),
\]
for some fixed constant $C$.
\end{lem}

With this result at hand we obtain local H\"older contiuity.

\begin{thm}[local H\"older regularity]
\label{thm:Calpha}
Let $u \in \calS(\lambda,\Lambda,f)$ in $Q_1$. then there is $\alpha \in (0,1)$ for which $u \in C^\alpha(\bar{Q}_{1/2})$ and
\[
  \| u \|_{C^\alpha(\bar Q_{1/2})} \leq C \left( \| u \|_{L^\infty(Q_1)} + \| f \|_{L^d(Q_1)} \right),
\]
where the constant $C$ is independent of $u$ and $f$.
\end{thm}
\begin{proof}
The proof is rather standard, so we merely sketch it. Let $m_r = \inf_{Q_r} u$, $M_r = \sup_{Q_r} u$ and $\varpi_r = M_r - m_r$. Applying the Harnack inequality of Theorem~\ref{thm:Harnack} to the nonnegative function $u-m_1$ yields
\[
  M_{1/2} - m_1 \leq C( m_{1/2} - m_1 + \| f \|_{L^d(Q_1)} ).
\]
Since the function $M_1 - u \geq 0$ we can, once more, apply the Harnack inequality to obtain
\[
  M_1 - m_{1/2} \leq C ( M_1 - M_{1/2} + \| f \|_{L^d(Q_1)} ).
\]
Adding these two inequalities yields,
\[
  \varpi_{1/2} \leq \mu \varpi_1 + 2 \| f \|_{L^d(Q_1)},
\]
where $\mu = (C-1)/(C+1) \in (0,1)$.

A similar argument for the functions $u_r(y) = u(ry)/r^2$ and $f_r(y) = f(ry)$ with $y \in Q_1$ reveals that
\[
  \varpi_r \leq \mu \varpi_{r/2} + 2r \| f \|_{L^d(Q_1)}.
\]
An application of the iteration Lemma~\ref{lem:iteration} with $\varphi(r) = \varpi_r$ immediately yields the H\"older continuity and the estimate.
\end{proof}

In a similar fashion, we can establish smoothness up to the boundary; see \cite[Proposition 4.14]{MR1351007}.

\begin{thm}[global regularity]\label{thm:VTglobalReg}
Let $\Omega$ be sufficiently smooth and $u \in \calS(\lambda,\Lambda,f) \cap C(\bar\Omega)$ with $f \in C(\Omega)$. Let $g = u_{|\partial\Omega}$ and $\varrho$ be a modulus of continuity of $g$. Then there is a modulus of continuity $\varrho^*$ of $u$ in $\bar\Omega$ which depends only on $\lambda, \Lambda, \varrho, \|f\|_{L^d(\Omega)}$ and $\| g \|_{L^\infty(\Omega)}$.
\end{thm}

We now focus on H\"older estimates for first and second derivatives. To simplify the presentation, in problem \eqref{eq:BVP}, the PDE takes the form
\begin{equation}
\label{eq:newPDE}
  F(x,D^2u) = f, \ \text{in } \Omega.
\end{equation}
To quantify the smoothness of $F$ with respect to the $x$ variable we introduce the function
\[
  \beta(x) = \sup_{M \in \polS^d} \frac{ | F(M,x) - F(M,0) |}{ |M| +1 }.
\]
The local regularity is as follows.

\begin{thm}[local $C^{2,\alpha}$ regularity]
\label{thm:locC2a}
Assume that $F$ is uniformly elliptic in the sense of Definition~\ref{def:FLelliptic}, $\beta, f \in C^\alpha(B_1)$ and that there is a constant $\bar\alpha \in (0,1)$ such that for any $M \in \polS^d$ with $F(0,M)=0$ and $w_0 \in C(\partial B_1)$ there is 
$w \in C^2(B_1) \cap C(\bar B_1) \cap C^{2,\bar\alpha}(B_{1/2})$ which satisfies
\begin{equation}
\label{eq:concavecondition}
  F(0,D^2w +M ) =0, \ \text{in } B_1, \quad w=w_0, \ \text{on } \partial B_1,
\end{equation}
with an a priori estimate. If $u$ is a viscosity solution of \eqref{eq:newPDE} then $u \in C^{2,\alpha}(\bar B_{1/2})$ for some $\alpha \in (0,1)$ with an a priori estimate.
\end{thm}

To apply this theorem, one must verify that solutions to \eqref{eq:concavecondition} have $C^{2,\bar\alpha}$ estimates. For an $F$ that is convex, the Evans-Krylov theorem \cite{MR2550191} provides such an estimate.

\begin{thm}[Evans-Krylov]
\label{thm:EvansKrylov}
Let $F$ be convex and depend only on $M$. If $u$ is a viscosity solution of $F(D^2 u)=0$ in $B_1$, then
\[
  \| u \|_{C^{2,\bar\alpha}(\bar B_{1/2})} \leq C \left( \| u \|_{L^\infty(B_1)} + |F(0)| \right),
\]
for some constants $\bar\alpha \in (0,1)$ and $C$ that depend only on the dimension $d$ and the ellipticity of $F$.
\end{thm}

We mention that Theorem~\ref{thm:locC2a} can be applied to the Hamilton-Jacobi-Bellman operators of Example~\ref{ex:HJB}. With the aid of the so-called method of continuity, this allows us to show that solutions to the Dirichlet problem \eqref{eq:BVP} for this class of operators are classical.

On the other hand, it is natural to ask if the convexity of $F$ is essential for this result. To understand this, we begin by providing a $C^{1,\alpha}$ estimate without convexity assumptions.

\begin{thm}[$C^{1,\alpha}$ regularity]
\label{eq:C1alphaloc}
Let $u$ be a viscosity solution of
\[
  F(D^2u)=0 \ \text{in } B_1.
\]
Then
\[
  \| u \|_{C^{1,\bar\alpha}(\bar B_{1/2})} \leq C \left( \| u \|_{L^\infty(B_1)} + |F(0)| \right),
\]
where the constants $\bar\alpha \in (0,1)$ and $C$ depend only on the dimension and ellipticity of $F$.
\end{thm}

A similar argument to Theorem~\ref{thm:locC2a} allows us to conclude then that, in this setting and under similar assumptions, solutions to \eqref{eq:BVP} are locally $C^{1,\alpha}$ with an a priori estimate. However, there exists a series of counterexamples \cite{MR2373018,MR2391642,MR2964616} showing that, in general, the convexity assumption on $F$ cannot be removed.

\begin{ex}[nonclassical viscosity solution]
\label{ex:Nadirashvili}
Let $d=5$. Define
\[
  P_5(x) = x_1^3 + \frac32 x_1 \left( x_3^2 + x_4^2 - 2x_5^2 - 2 x_2^2 \right) 
  + \frac{3\sqrt3}2\left( x_2x_3^2 - x_2x_4^2 + 2x_3x_4x_5 \right),
\]
and, for $\delta \in [0,1)$, $w(x) = P_5(x)/|x|^{1+\delta} \in C^{1,1-\delta}(\bar B_1) \setminus C^2(B_1)$. There exists an Isaacs operator $F$ that depends only on $M$ and is Lipschitz, such that $w$ is a viscosity solution of
\[
  F(D^2w)=0, \ \text{in } B_1, \quad w = P_5 \ \text{on } \partial B_1.
\]
\end{ex}

The existence of nonclassical solutions in dimensions $3 \leq d<5$ is an open problem.

We conclude by providing global regularity results in the general case; see \cite{CaffarelliSoug08,CafSil10,Tura15}.

\begin{thm}[global regularity]
\label{thm:ggtthhmmreg}
Let $\Omega$ be sufficiently regular and $u$ be a viscosity solution to \eqref{eq:BVP} with $F$ of the form \eqref{eq:newPDE} being Lipschitz and uniformly elliptic. If $f \in C^{0,1}(\Omega)$ and, for some $\gamma \in (0,1]$, $g \in C^{1,\gamma}(\partial\Omega)$, then $u \in C^{1,\alpha}(\Omega) \cap C^{0,1}(\bar\Omega)$ with
\[
  \| u \|_{C^{1,\alpha}(\Omega)} \leq C \left( \| f \|_{L^\infty(\Omega)} + \| g \|_{C^{1,\gamma}(\partial\Omega)} \right),
\]
where $\alpha \in (0,1)$ and $C$ depend only on $d$, $\lambda$, $\Lambda$ and the smoothness of $F$.
\end{thm}

\section{Monotonicity in numerical methods}
\label{sec:numericmono}

In this section we review
some basic properties of numerical methods
and state sufficient conditions to ensure that discrete approximations
converge to the solutions of the underlying PDE.  
The underlying theme of this section is that as the notion
of solution to the PDE becomes weaker, additional
conditions of the numerical approximation are required
to guarantee convergence.  For example, 
for linear differential equations, the well-known
Lax-{Richtmyer} equivalence theorem
shows that any consistent scheme
is convergent if and only if it is stable; these results
extend to mildly nonlinear problems as well.
However, in the fully nonlinear regime, 
consistency and stability 
are no longer sufficient in general.
Rather, additional monotonicity conditions, 
which essentially mimic the comparison principles discussed in the previous section, are required.

We note that while the content of this section deals {with}
finite difference and finite element methods,
the main ideas extend to other discretization techniques
as well.

\subsection{Stability, consistency and monotonicity implies convergence}
\label{sub:LaxEquiv}

As before, let $\Omega$ be an open subset of $\bbR^d$
with Lipschitz boundary $\p \Omega$.
We consider numerical approximations
of elliptic problems of the most general form {\eqref{eq:newBVP}, which for convenience we recall below}
\begin{align}\label{NonlinearProblem}
  F(x,{u},Du,D^2 u) {=0} \quad \text{in }{\Omega},\qquad
B(x,u,Du)=0\quad \text{on }\p\Omega.
\end{align}
Here
$F\in C({\Omega}\times \bbR\times \bbR^d\times \bbS^d)$
is {locally bounded and elliptic in the sense of Definition~\ref{def:FLelliptic}}.

We further assume that {$B \in C(\Omega\times \bbR\times \bbR^d)$
is nonincreasing} in its second argument.
For the moment, we consider 
viscosity solutions that satisfy
the boundary conditions only in 
a viscosity sense; see Definition~\ref{def:viscoBVP}.
In this case, and to simplify the presentation, we define the operator
\begin{align*}
\sF(x,r,\bp,M)
= 
\left\{
\begin{array}{ll}
F(x,r,\bp,M) & \text{if }x\in \Omega\\
B(x,r,\bp) & \text{if } x\in \p\Omega.
\end{array}
\right.
\end{align*}
so that
\eqref{NonlinearProblem} becomes 
\begin{align}
\label{eqn:NonlinearProblem2}
\sF[u]:=\sF(x,u,Du,D^2 u)=0\qquad \text{in }\bar\Omega.
\end{align}
Note that, since $B$ is nonincreasing in its second
argument and $F$ is elliptic, the operator $\sF$ is elliptic 
in the sense of Definition \ref{def:FLelliptic}.
We further see that, {according to Definition~\ref{def:viscoBVP}}, $u\in C(\bar\Omega)$ 
is a viscosity solution to \eqref{NonlinearProblem}
(equivalently, \eqref{eqn:NonlinearProblem2})
if it is a viscosity solution to
\begin{align*}
F(x,u,Du,D^2 u)=0\qquad\text{in }\Omega,
\end{align*}
and 
\begin{align*}
\max\{F(x,u,Du,D^2 u),B(x,u,Du)\}\ge 0\quad \text{on }\p\Omega,\\
\min\{F(x,u,Du,D^2 u),B(x,u,Du)\}\le 0\quad \text{on }\p\Omega
\end{align*}
{in the viscosity sense}.

We now consider approximation schemes:
Find $u_h\in X_h$ satisfying
\begin{align}
\label{eqn:NonlinearApproximation}
\sF_h[u_h]({z})=0\qquad \text{in }\bar{\Omega}_h,
\end{align}
where $\sF_h$
is a locally bounded operator which we may think
as an approximation to $\sF$,
and $X_h$
is some finite dimensional space.
The operator is parameterized by
$h>0$, which we may view as a discretization parameter, or in some cases,
a regularization parameter.  The discrete domain
$\overline{\Omega}_h$ is an approximation
to $\overline{\Omega}$ with the property $\lim_{h\to 0^+} \bar{\Omega}_h = \bar\Omega$;
namely, for all ${z_0}\in \bar\Omega$, there exists
a sequence $\{{z}_h\}_{h>0}\subset \bar\Omega_h$
such that $\lim_{h\to 0^+} {z}_h = {z}_0$.

We now address the
well-posedness of \eqref{eqn:NonlinearApproximation},
and the sufficient structure conditions on $\sF_h$
to ensure that the discrete solutions
to \eqref{eqn:NonlinearApproximation} (if they exist)
converge.  As a first step 
we state the fundamental notions of 
consistency and stability.

\begin{definition}[{consistency}]
\label{def:ConsistentOperator}
The discrete problem \eqref{eqn:NonlinearApproximation}
is said to be {\it consistent} with \eqref{eqn:NonlinearProblem2} if there exists
an operator $I_h:{C(\bar\Omega)}\to X_h$ such that $I_h$ 
converges uniformly to the identity operator as $h\to 0^+$, and
for all sequences $\{{z}_h\}_{h>0}$ with ${z}_h\in \bar\Omega_h$ and
 ${z}_h\to {z}_0\in \bar\Omega$
 and $\phi\in C^2(\bar{\Omega})$,
\begin{align*}
\lim_{h\to 0^+} \sF_h[I_h \phi]({z}_h) = \sF[\phi]({z}_0).
\end{align*}
\end{definition}

\begin{rem}[{envelopes}]
\label{rem:envelopeConsistent}
If the operators
are not continuous, then the notion of consistency
is changed to
\begin{align*}
\limsup_{h\to 0^+}\sF_h[I_h \phi]({z}_h)\le \sF^*[\phi]({z}_0),\\
\liminf_{h\to 0^+} \sF_h[I_h \phi]({z}_h)\ge \sF_*[\phi]({z}_0),
\end{align*}
where $\sF^*$ (resp., $\sF_*$) denote
the upper (resp., lower) semi-continuous envelope of $\sF$; see Definition~\ref{def:usclsc}.
\end{rem}
%

\begin{definition}[{stability}]
\label{def:NumericalStability}
We say that problem \eqref{eqn:NonlinearApproximation}
is {\it stable} if, for all $h>0$, there exists
a solution $u_h\in X_h$ to \eqref{eqn:NonlinearApproximation}, and 
moreover, if ${w_h}\in X_h$ satisfies 
$\sF_h[{w_h}] = \epsilon_h$, then
 $\|u_h-{w_h}\|_{L^\infty(\bar\Omega_h)} \le C\|\epsilon_h\|_{L^\infty(\bar\Omega_h)}$, with $C>0$ independent of $h$.
\end{definition}

The following theorem
states the well-known result that, for linear problems
with classical solutions,
consistent and stable schemes
converge.
\begin{thm}[Lax-Richtmyer]
\label{thm:LaxTheorem}
Suppose 
that $\sF_h$ is an affine operator,
and that there exists a classical
solution $u\in C^2(\bar\Omega)$
satisfying \eqref{eqn:NonlinearProblem2}.
Suppose further that problem 
\eqref{eqn:NonlinearApproximation} 
is stable and that the operator
in Definition \ref{def:ConsistentOperator}
satisfies the stronger condition
$\lim_{h\to 0^+} \|\sF_h[I_h u]\|_{L^\infty(\bar\Omega_h)}=0$.
Then $u_h$ converges locally uniformly
to $u$.
\end{thm}
\begin{proof}
By the given assumptions, there exists
a linear operator $\sG_h$
and a {function} $\sell_h$
such that
\begin{align*}
\sF_h[v_h]({z}) = \sG_h[v_h]({z})-\sell_h({z})\quad {z}\in \bar\Omega_h.
\end{align*}
The stability of the scheme shows
that $\sG_h[\cdot]$ is an isomorphism
{whose inverse is bounded independent of $h$.}
Since
\begin{align*}
\sG_h[u_h-I_h u] = \sell_h - \sG_h[I_h u] = -\sF_h[I_h u],
\end{align*}
the {consistency} of the scheme
implies that $\lim_{h\to 0^+} \|u_h-I_h u\|_{L^\infty(\bar\Omega_h)}=0$,
and thus, since $I_h$ converges to the identity 
operator, $u_h$ converges locally uniformly to $u$.
\end{proof}

While Theorem~\ref{thm:LaxTheorem}
is a useful result for a large
class of problems, 
it is not applicable to fully nonlinear
problems nor to weaker notions of solutions, in particular, viscosity solutions.
The issue is that, if $u$ is not a classical solution,
then the approximation $I_h u$ may not be well--defined,
and the consistency $\sF_h[I_h u]\to 0$
used in the proof of Theorem~\ref{thm:LaxTheorem}
is no longer valid.
Rather, to prove convergence to viscosity solutions, an additional structure
condition is required.
This requirement is summarized in the following definition.
\begin{definition}[{monotone operator}]
\label{def:discreteMonotone}
The discrete operator $\sF_h$ is said to be
monotone if {whenever $u_h-v_h$ has a global nonnegative
maximum at ${z}\in \bar\Omega_h$ we have}
\begin{equation}
\label{equn:discreteMonotone}
  \sF_h[u_h]({z})\le \sF_h[v_h]({z}).
\end{equation}
\end{definition}

\begin{rem}[{monotonicity}]
\label{rem:monotoneEQelliptic}
The notion of monotonicity is essentially
a discrete version of ellipticity.
Indeed, following the proof of Theorem 
\ref{thm:classicalunique}, if $\sF$ is an elliptic operator in the sense
of Definition \ref{def:FLelliptic},
and if $u-v$, with $u,v\in C^2(\bar\Omega)$,
has a global nonnegative maximum
at $x\in \bar\Omega$, then $u(x)\ge v(x)$,
$D u(x) = Dv(x)$ and $D^2 u(x)\le D^2 v(x)$.  
Since $\sF$ is nonincreasing in its second argument,
and nondecreasing in its fourth, we have
\begin{align*}
\sF[u](x) = \sF(x,u(x),Du(x),D^2 u(x)) 
\le \sF(x,v(x),Dv(x),D^2 v(x)) = \sF[v](x).
\end{align*}
Conversely, it can be shown that the operator $\sF$
is elliptic if, whenever $u-v$ has a global nonnegative maximum
at $x\in \bar\Omega$, then $\sF[u](x)\le \sF[v](x)$.
\end{rem}

Consistency, stability and monotonicity are the three sufficient ingredients to guarantee 
convergence to viscosity solutions. The following result closely follows \cite[Theorem 2.1]{BarlesSoug91}.

\begin{thm}[{Barles-Souganidis}]
\label{thm:BSTHM1}
Suppose
that problem \eqref{eqn:NonlinearApproximation}
is  consistent, stable and monotone in the sense
of Definitions \ref{def:ConsistentOperator}, \ref{def:NumericalStability},
and \ref{def:discreteMonotone}, respectively.  Suppose
further that $\sF$ satisfies the comparison
principle given in Definition \ref{def:comparison}.
Then $u_h$ converges locally
uniformly to the unique continuous viscosity
solution of \eqref{eqn:NonlinearProblem2}.
\end{thm}
\begin{proof}
Define $\bar{u}\in USC(\Omega)$
and $\underline{u}\in LSC(\Omega)$ by
\begin{align}\label{eqn:baruubar}
\bar{u}(x) := \mathop{\limsup_{y\to x}}_{h\to 0^+} u_h(y),
\quad
\underline{u}(x) := \mathop{\liminf_{y\to x}}_{h\to 0^+} u_h(y),\quad x\in \bar\Omega.
\end{align}
Note that the stability of the scheme implies that
both $\bar{u}$ and $\underline{u}$ are well-defined.
The proof proceeds by showing
that $\bar{u}$ and $\underline{u}$
are (viscosity) subsolutions
and supersolutions to \eqref{eqn:NonlinearProblem2}, respectively,
and then appealing to the comparison principle.

To this end, suppose that
$z_0\in {\Omega}$
is a strict local maximum of $\underline{u}-\phi$
for some $\phi\in C^2({\Omega})$.
Then standard arguments
show that there exists
sequences $\{h_n\}_{n=1}^\infty$
and $\{z_{h_n}\}_{n=1}^\infty$ such that
\begin{align*}
z_n\to 0,\quad z_{h_n}\to x_0,\quad \text{ as }n\to \infty,
\end{align*}
and $u_{h_n}-I_{h_n}\phi$ obtains
a strict local maximum at $z_{h_n}$.
Since $\sF_{h_n}[u_{h_n}](z_{h_n})=0$, 
the monotonicity of the discrete operator
implies that
\begin{align*}
\sF_{h_n}[I_{h_n}\phi](z_{h_n})\ge 0.
\end{align*}
Passing to the limit, together with consistency of the scheme, yields
\begin{align*}
0 &\le \lim_{n\to \infty} \sF_{h_n}[I_{h_n}\phi](z_{h_n}) = \sF[\phi](z_0).
\end{align*}
Similar arguments show that if $\bar{u}-\phi$ obtains
a strict minimum at $z_0\in {\Omega}$, there holds
$0\ge \sF[\phi](z_0)$.
Thus, $\bar{u}$ and $\underline{u}$ are subsolutions
and supersolutions to \eqref{eqn:NonlinearProblem2}, respectively.
Since $\underline{u}\le \bar{u}$, the comparison
principle of $\sF$ implies that $\underline{u} = \bar{u} = u$,
and $u$ is the viscosity solution to \eqref{eqn:NonlinearProblem2}.
\end{proof}

We once again mention
that the problems and discretizations
considered so far take
into account the boundary conditions
in a viscosity sense.  While this setup
simplifies the proof of convergence,
it may have practical limitations
since the framework requires
a consistent
and monotone discretization of both the boundary
conditions and the differential operator 
for all $x\in \p\Omega$ and smooth functions $\phi$.
Also recall that, in general, viscosity boundary conditions
are not equivalent to those
imposed pointwise unless other
conditions are assumed; see Proposition
\ref{prop:viscoisclassic}.  Here we turn
our attention to the elliptic boundary value problem \eqref{eq:BVP},
where the Dirichlet boundary condition
is understood in the classical sense; see Definition \ref{def:viscobvp}.

To this end, we consider approximations of the form:
\begin{align}\label{eqn:NonlinearApproximation2}
F_h[u_h] = 0\quad \text{in }\Omega^I_h,\qquad u_h = g_h\quad \text{on } \Omega_h^B,
\end{align}
where $g_h = I_h g\in X_h$ is a discrete approximation to the Dirichlet data $g\in C(\bar\Omega)$,
$\Omega^I_h$ and $\Omega^B_h$ are disjoint sets
with $\Omega_h^I\to \Omega$ and $\Omega_h^B\to \p \Omega$ as $h\to 0^+$.
Note that, with minor notational changes, the notions
of consistency, stability, and monotonicity are applicable 
to the operator $F_h$.  A natural question then, is whether
the results of Theorem \ref{thm:BSTHM1} carry over
to the discrete problem \eqref{eqn:NonlinearApproximation2}.
This issue is addressed in the next theorem.

\begin{thm}[convergence]
\label{thm:BSTHM1Alt}
Suppose that 
$F_h$ is a consistent, stable, 
and monotone operator.  Suppose further that
either
\begin{enumerate}[(i)]
\item $\underline{u}(x)\ge g(x)$ and $\bar u(x)\le g(x)$ for all $x\in \p\Omega$,
where $\underline{u},\bar u$ are given by \eqref{eqn:baruubar}; or
\item The sequence of solutions $\{u_h\}_{h>0}$ is equicontinuous.
\end{enumerate}
Then $u_h$ converges locally uniformly to the unique continuous
viscosity solution of \eqref{NonlinearProblem} with $B(x,u,Du) = {g-u}$.
\end{thm}
\begin{proof}
The proof of the first case (i)  follows directly
from the arguments given in Theorem \ref{thm:BSTHM1}.
The proof of the second case (ii) follows
from the Arzel\`a-Ascoli theorem,
and again appealing to the proof of Theorem 
\ref{thm:BSTHM1}.
\end{proof}

\subsection{Monotonicity in finite difference schemes}\label{subsec:monoFD}

Here we discuss basic monotonicity results of 
finite difference schemes.
The main message given in this section is that for any
uniformly elliptic operator, one can construct a consistent 
and monotone finite difference scheme.  
The drawback however is that monotonicity
requires a wide-stencil, which
may severely impact its practical use.
Much of the material in this section is found in \cite{KuoTrudinger90,KuoTrudinger92,Kocan95,MotzkinWasow53,Oberman06}.

For simplicity we assume that the domain $\Omega$
is discretized on an equally spaced cartesian grid and that each coordinate
direction is discretized uniformly; in particular, 
by a possible change of coordinates, we assume that the grid
is given by
\begin{align*}
\bar\Omega_h=\bbZ^d_h\cap \overline{\Omega},\quad \text{with }\bbZ_h^d:=\{ h e:\ e\in \bbZ^d\},
\end{align*}
where $h>0$ is the grid scale
and $\bbZ^d$ is the set of $d$--tuples of integers.
A finite subset $\St\subset \bbZ^d\backslash \{0\}$
is called a {\it stencil}, and the space of {\it nodal functions}, denoted
by $\fd$, consist of  real-valued functions with domain $\bar\Omega_h$.
The canonical interpolant $I_h^{fd}:C^0(\bar\Omega)\to {\fd}$
is the operator satisfying $I_h^{fd} v({z}) = v({z})$ for all ${z}\in \bar\Omega_h$.
We assume the existence of a positive integer $\sm$ such that
$\St$ is of the form
\begin{align}\label{def:stencil}
\St = \{y:\ y\in \bbZ^d\backslash \{0\}:\ |y|_{\ell^\infty}\le \sm\}.
\end{align}
The value $\sm$ satisfying \eqref{def:stencil}
is called the {\it stencil size} of $\St$. The cardinality of 
$\St$ is $|\St|:=(2\sm+1)^d-1$.

We consider finite difference schemes 
with stencil $\St$ acting on grid functions.
These discrete operators are thus of the (implicit) form
\begin{align}\label{eqn:implicitFDForm}
F_h[v_h]({z}) = F_h({z},v_h(z),Tv_h({z})),
\end{align}
where $Tv_h(x) = \{v_h(x+h y):\ y\in \St\}$
is the set of translates of $v_h(x)$ with respect to the stencil.
The method \eqref{eqn:implicitFDForm}
is called a {\it one-step} scheme 
if $m=1$, i.e., the value $F_h[v_h]({z})$
only depends on $z$, $v_h({z})$, and the 
values of $v_h$ at neighboring points of ${z}$.
Otherwise, we call the scheme 
a {\it wide-stencil} scheme
if $\sm\ge 2$.

To construct monotone schemes,
 we first reformulate
this property so that it is easier to work with.
\begin{definition}[{nonegative operator}]
\label{def:positiveType}
The operator $F_h$ is of {\it nonnegative type} (or simply, nonnegative)
if
\begin{align}\label{eqn:DefNonNegativeType}
F_h({z},{r},q+\tau)\ge F_h({z},{r},q)\ge F_h({z},{r}+t,q+\tau)
\end{align}
for all ${z}\in \bbR^d $, ${r},t\in \bbR$,
and $q,\tau\in \bbR^{|\St|}$ satisfying
\begin{align}\label{eqn:nonnegTypeHyp}
0\le \tau_i\le t\qquad i=1,2,\ldots |\St|.
\end{align}
\end{definition}
We see
that if $F_h$ is of nonnegative
type then $F_h$ is nonincreasing
in its second argument and 
nondecreasing in its third argument.
If $F_h$ is differentiable, 
then it is of nonnegative type provided that
\begin{align*}
\frac{\p F_h}{\p q_i} \ge 0\ (i=1,2,\ldots,|\St|),\qquad \frac{\p F_h}{\p {r}} + \sum_{i=1}^{|\St|} \frac{\p F_h}{\p q_i}\le 0.
\end{align*}


\begin{rem}[{reformulation}]
\label{rem:reformulation}
Alternatively, {as in \cite{Oberman06}}, one can
consider finite difference schemes
of the form 
\[F_h[u]({z}) = G_h({z},u({z}),u({z})-Tu({z})).\]
Using the correspondence $F_h({z},{r},q) = G_h({z},{r},{r}{\boldsymbol{1}}-q)$,
one sees that $F_h$ is of nonnegative type if and only
if $G_h$ is nonincreasing in its second and third arguments.
\end{rem}

Let us now show that nonegativity is nothing but a reformulation of monotonicity.

\begin{lem}[{equivalence}]
\label{lem:MonoPosoEquivo}
A finite difference scheme of the form
\eqref{eqn:implicitFDForm} is monotone
if and only if it is of nonnegative type.
\end{lem}
\begin{proof}
Suppose that $F_h$ is of nonnegative type.
Let $u_h$ and $v_h$ be two grid 
functions
such that $u_h-v_h$ has a global
nonnegative maximum at some grid point ${z}$.
Set ${r} = v_h({z})$, $t = u_h({z})-v_h({z})\ge 0$,
{$q_i = v_h({z}+h y_i)$, and $\tau_i = \max\{0,u_h({z}+hy_i)-v_h({z}+hy_i)\}$},
so that $t\ge \tau_i$.
Noting that {$q_i + \tau_i  = v_h({z}+h y_i) + \max\{0,u_h({z}+h y_i)-v_h({z}+hy_i)\}
\ge u_h({z}+h y_i)$}, and $F_h$ is non--decreasing
in its third argument, we find that
$F_h[u_h]({z}) \le F_h({z},{r}+t,q+\tau).$
Therefore by the second inequality in \eqref{eqn:DefNonNegativeType} we have
\begin{align*}
F_h[u_h]({z})\le F_h({z},{r}+t,q+\tau)\le F_h({z},{r},q) = F_h[v_h]({z}).
\end{align*}
Thus, $F_h$ is monotone.

Now suppose that $F_h$ is monotone.
Let ${z}\in \bbR^d$ be fixed,
and let ${r},t\in \bbR$ and $q,\tau\in \bbR^{|S|}$
satisfy \eqref{eqn:DefNonNegativeType}. Then define 
the  grid functions $u_h,v_h$ (locally) 
as 
\begin{align*}
v_h({z}) = {r},\quad v_h({z}+h y_i) = q_i,\quad u_h({z}) = {r}+t,\quad u_h({z}+y_i) = q_i+\tau_i.
\end{align*}
Then 
\begin{align*}
u_h({z}) - v_h({z}) = t \ge \tau_i = u_h({z}+h y_i)-v_h({z}+h y_i),
\end{align*}
i.e., $u_h-v_h$ has a nonnegative maximum at ${z}$.
The monotonicity of $F_h$ yields $F_h[u_h]({z})\le F_h[v_h]({z})$; thus
\begin{align}\label{eqn:PositveProofLine1}
F_h({z},{r}+t,q+\tau)\le F_h({z},{r},q).
\end{align}
On the other hand, with $v_h$ as before, we 
consider the grid function $w_h$ with
\begin{align*}
w_h({z}) = {r},\quad w_h({z}+h y_i) = q_i+\tau_i.
\end{align*}
Then $v_h-w_h$ has a global maximum
at ${z}$ and thus
\begin{align}\label{eqn:PositveProofLine2}
F_h({z},{r},q) = F_h[v_h]({z})\le F_h[w_h]({z}) = F_h({z},{r},q+\tau).
\end{align}
We conclude from \eqref{eqn:PositveProofLine1}--\eqref{eqn:PositveProofLine2}
that $F_h$ is of nonnegative type.
\end{proof}

Following the framework given in 
\cite{KuoTrudinger90,KuoTrudinger92,Kocan95}
we consider discrete operators constructed
from the first and second order difference operators
\begin{align*}
\delta_{y,h}^+ u({z}):&= \frac{1}{h} \big(u({z}+h y)-u({z})\big),\\
\delta_{y,h}^- u({z}):&=\frac{1}{h} \big(u({z})-u({z}-h y)\big),\\
\delta_{y,h} u({z}):&=\frac12 \big(\delta_y^++\delta_y^-\big)u({z}) = \frac{1}{2h} \big(u({z}+h y)-u({z}-h y)\big),\\
\delta_{y,h}^2 u({z}):&=\frac{1}{ h^2} \big( u({z}+h y)-2u(x)+u({z}-h y)\big),
\end{align*}
with $y\in \St$.  
Taylor's Theorem 
shows that the  differences $\delta_{y,h}^\pm u(z)$
are first order approximations to 
$ \frac{\p u({z})}{\p y}:=D u({z})\cdot y$,
{whereas} $\delta_{y,h}u({z})$ and $\delta_{y,h}^2 u({z})$
are second-order approximations to 
$\frac{\p u({z})}{\p y}$ and 
$\frac{\p^2 u({z})}{\p y^2}:=y \cdot D^2 u({z}) y$, respectively; 
by this, we mean that $| \frac{\p^\pm u({z})}{\p y} - \delta_{y,h}^\pm u({z})| = \mathcal{O}(h |y|^2 )$,
$| \frac{\p u({z})}{\p y} - \delta_{y,h} u({z})| = \mathcal{O}(h^2 |y|^3)$,
and $| \frac{\p^2 u({z})}{\p y^2} - \delta^2_{y,h} u({z})| = \mathcal{O}(h^2 |y|^4)$
for sufficiently smooth $u$. 

{Let}
$\delta_h u_h({z}) = \{\delta_{y,h} u_h({z}) :\ y\in \St\}$
and $\delta^2_h u_h({z}) = \{\delta_{y,h}^2 u_h({z}):\ y\in \St\}$.
Then a consistent and monotone finite difference scheme
can be constructed in the form \cite{KuoTrudinger90,KuoTrudinger92,MotzkinWasow53}
\begin{align}\label{eqn:FhDifferencesForm}
F_h[u_h]({z}) = \mathcal{F}_h({z},u_h({z}),\delta_h u_h({z}),\delta^2_h u_h({z})),
\end{align}
 where $\mathcal{F}_h:\Omega_h\times \bbR\times \bbR^{|\St|}\times \bbR^{|\St|}\to \bbR$.
 {Denote} points in the domain of $\mathcal{F}_h$
 by $({z},{r},q,s)$ and assume that $\mathcal{F}_h$ is symmetric
 with respect to $\pm q_{\pm i}$ and $ s_{\pm i}$.
 Then from Definition \ref{lem:MonoPosoEquivo} and Lemma \ref{lem:MonoPosoEquivo}, we see 
 that $F_h$ of the form \eqref{eqn:FhDifferencesForm}
 is monotone provided {that}
 \begin{align}\label{eqn:calFNNeg}
 \frac{ h }2 \Big| \frac{\p \mathcal{F}}{\p q_i}\Big|\le \frac{\p \mathcal{F}}{\p s_i}\quad i=1,2,\ldots,|\St|,\quad\text{and}\quad
 \frac{\p \mathcal{F}}{\p {r}}\le 0.
 \end{align} 
  
 In what follows, we require
 slightly stronger conditions on the operator $\mathcal{F}_h$.

\begin{definition}[{positive operator}]
\label{def:OperationPositiveType}
An operator of the form \eqref{eqn:FhDifferencesForm}
is of {\it positive type} (or simply, positive) if \eqref{eqn:calFNNeg}
is satisfied and there exists $\lambda_{0,h}>0$ 
and an orthogonal set of vectors $\{y_i\}_{i=1}^d\subset \St$
such that
\begin{align*}
\lambda_{0,h}+ \frac{ h }2 \Big| \frac{\p \mathcal{F}}{\p q_i}\Big|\le \frac{\p \mathcal{F}}{\p s_i} .
\end{align*}
\end{definition}
\begin{rem}[discrete ellipticity]
The discrete ellipticity constant $\lambda_{0,h}$ may depend
on the stencil size $\sm$; see Theorem \ref{KOCTHM1}.
\end{rem}

\subsection{Finite difference stability estimates: Alexandrov estimates and Alexandrov-Bakelman-Pucci maximum principle}\label{sub:FDMaxPrince}

We now turn our attention 
to maximum principles of discrete operators, 
and correspondingly, stability estimates.
As a starting point, we discuss 
monotone finite difference schemes for {the} linear nondivergence form PDEs {of Example~\ref{ex:liniseliptic}}
\begin{align}
\label{eqn:FisLinear}
F[u] = \mathcal{L}u-f = A:D^2 u-f=0.
\end{align}
While this setting may seem overly
simplistic, as we shall see, the construction
and theoretical results for the linear problem
form all of the necessary tools to approximate
viscosity solutions of nonlinear elliptic equations.

We assume that $f\in C(\bar\Omega)$
and that the coefficient matrix $A$ 
 is bounded and uniformly symmetric positive definite.  It is then
reasonable to assume
that $F_h$ is linear and thus has the form
\begin{align}\label{def:FhisLinear}
F_h[u_h]({z}) ={\mathcal{L}_h u_h(z)-f(z)} :=  \sum_{y\in \St} a_y({z})\delta_{y,h}^2 u_h({z})-f({z})
\end{align}
for {nodal} functions (or coefficients) $a_y$.
Applying Definition \ref{def:positiveType} to \eqref{def:FhisLinear}, we see
that $F_h$ is nonnegative (and hence monotone)
 provided   
\begin{align}\label{eqn:MmatrixCond1}
a_y({z}) \ge 0, 
\end{align}
and of positive type if 
\begin{align*}
a_{y_i}({z})\ge \lambda_{0,h} 
 \qquad 
 i=1,2,\ldots,d.
\end{align*}
for some orthogonal basis $\{y_i\}_{i=1}^d\subset \St$.
These inequalities suggest
that the negation of the {ensuing system}
is an $M$--matrix, and hence solutions
to the discrete problem
satisfy certain maximum principles, analogous to
the continuous setting.
This issue is discussed in the next section.

\subsubsection{Finite difference Alexandrov estimates}
In this section we state
and prove ABP maximum principles
for grid functions.
To get started, we first
specify the fundamental notion of interior and boundary nodes
used in this section.
\begin{definition}[interior and boundary nodes]\label{def:IBNodes}
For a discrete operator $F_h$, we define
 the {\it set of interior nodes} $\Omega^I_h$
as the set of grid points ${z}\in \bar\Omega_h$
such that for any mesh function $v_h$, 
$F_h[v_h]({z})$ depends only on the translates of $v_h$ at points
in $\Omega_h$.  The {\it set of boundary nodes}
are given by $\Omega_h^B :=\bar\Omega_h\backslash \Omega_h^I$.
\end{definition}

\begin{rem}[discrete domain]
If $F_h$ is a one-step method (i.e., $\sm=1$), then
$\Omega_h^B = \p \Omega \cap \bbZ_h^d$ and
$\Omega_h^I = \Omega \cap \bbZ_h^d$.
\end{rem}

As a next step we introduce 
and discuss several basic properties of 
convexity and the subdifferential for discrete (nodal) functions. 

\begin{definition}[convex nodal function]\label{def:convexnodalfcn}
We say that a nodal function $v_h\in \fd$
is a {\it convex nodal function}
if there is a supporting hyperplane of $v_h$ at all interior nodes
$z\in \Omega_h^I.$
\end{definition}
Note that if $v_h$ is the nodal
interpolant of a convex function,
then $v_h$ is a convex nodal functions.

\begin{definition}[discrete convex envelope]\label{def:ConvexEnvGF}
Let $R>0$ be sufficiently large such
that $\bar\Omega$ (and hence $\bar\Omega_h$)
is compactly contained in a ball $B_R$.  
For a nodal function (or continuous function) $v_h$ with $v_h\ge 0$
on $\Omega_h^B$, we extend $v^-_h$
to $B_{R,h}\backslash \bar\Omega_h$ by zero,
where $B_{R,h} = {B_R}\cap \bbZ_h^d$.  We define the {\it discrete convex
envelope} of $-v^-_h$ as
\begin{align}\label{eqn:ConvexEnvGF}
\Gamma_h(v_h)(x):=\sup \{L(x):\ L(z)\le -v^-_h(z)\ \forall z\in B_{R,h},\ L\in \mathbb{P}_1\}
\end{align}
for all $x\in \bar{B}_R$.
\end{definition}

\begin{rem}[discrete convexity]
\label{rem:discreteconvex}
There are some subtle issues
in the above definitions that
require some elaboration.
\begin{enumerate}[$\bullet$]
\item
If $v_h\in \fd$ is convex and $v_h\le 0$, then we have
\begin{align}\label{extension1}
  v_h(z) = \Gamma_h (v_h)(z) \quad \text{for all $z \in\Omega_h^I$.}  
\end{align}
Thus, $\Gamma_h (v_h)$ is a natural convex extension of $v_h$. 
With an abuse of notation, we still use $v_h$ to denote the convex envelope of {this} nodal function. 

\item 
Since
for every $x\in \p B_R$,
there exists an affine
function $L$ with $L(z)\le -v_h^-(z)$
for all $z\in \bar\Omega_h$
and $L(x)=0$, we conclude
that $\Gamma_h(v_h)=0$
on $\p B_R$.

\item 
Definition \ref{def:ConvexEnvGF}
implies that $\Gamma_h(v_h)$
is a convex, piecewise linear function
with respect to a simplicial triangulation.
The vertices of the triangulation 
are a subset of the gridpoints $B_{R,h}$,
and its configuration depends on $v_h$;
see Examples \ref{ex:FEMCE1}--\ref{ex:FEMCE2}.

\item For $v_h\in \fd$,
 denote by $\tilde{v}_h\in C(\bar\Omega)$
the canonical multi-linear function.  Then, since
the inequality constraints in \eqref{eqn:ConvexEnvGF}
are only posed on a discrete set of points, and
since $\tilde{v}_h$ is not necessarily piecewise affine, we have
$\Gamma(\tilde{v}_h)\neq \Gamma_h(v_h)$
in general \cite{MR1273696}. 
\end{enumerate}
\end{rem}

Next, we require the notion
of a subdifferential acting on nodal 
functions.
Recall from Definition \ref{def:subdiff} 
that the subdifferential requires
function values at all points
of the domain $\Omega$, and thus,
this notion is not directly applicable 
to the discrete case.
Instead, with a slight abuse of notation,
we define its natural extension
to nodal functions as follows:
\begin{align}\label{eqn:NodalSubDiff}
\p v_h(z) = \big\{\bp\in \bbR^d:\ v_h(x)-v_h(z)\ge \bp\cdot (x-z),\ \forall x\in \bar\Omega_h\big\}
\end{align}
for all $z\in \bar\Omega_h$ and $v_h\in \fd$.

\begin{lem}[{discrete subdifferential}]
If $v_h\in \fd$ is a convex nodal function, then $\partial v_h(z) = \partial \Gamma_h(v_h)(z)$ for all $z \in \Omega_h^I$. 
\end{lem}
\begin{proof}
  Thanks to {\eqref{extension1}}, if $\bp \in  \partial \Gamma_h(v_h)(z)$, that is, 
\[
  \Gamma_h(v_h)(x) \geq  \Gamma_h(v_h)(z) + \bp\cdot (x - z)  \quad \forall x \in \dm,
\]
then $\bp \in \partial v_h(z)$. 

Conversely, let $\mathcal{T}_z$ be a local mesh induced by $\Gamma(u_h)(z)$, 
$K\in \mathcal{T}_z$ a $d$-dimensional simplex,
and $\{z_j\}_{j=1}^{d+1}$ be the vertices of $K$. 
If $\bp \in \partial v_h(z)$, then we clearly have
$
  v_h(x) \ge v_h(z)+ \bp \cdot (x-z)
$
for all vertices $x\in \bar\Omega_h$.  Again, thanks to \eqref{extension1}, we have
$
   \Gamma(v_h)(z_i) \ge \Gamma(v_h)(z)+ \bp \cdot (z_i-z)
$
for all vertices.
Since $\Gamma(v_h)$ is linear on element $K$, we have
$
   \Gamma(v_h)(x) \ge \Gamma(v_h)(z)+ \bp \cdot (x-z)
$
for any $x \in K$. This shows that $\bp \in {\partial}\Gamma (u_h)(z)$ as well. 
\end{proof}

Let us state two properties of subdifferentials. The proof of the first one follows directly from its definition.

\begin{lem}[monotonicity of subdifferential] \label{monotonicity}
Let $w_h,v_h\in \fd$ be two convex nodal functions
such that, for a fixed $z_*\in \bar\Omega_h$, 
$w_h(z_*) = v_h(z_*)$ and $w_h(z) \le v_h(z)$ for
all $z\neq z_*$.  Then,
\begin{align*}
\partial w_h(z_*) \subset \partial v_h(z_*).
\end{align*}
\end{lem}

\begin{lem}[addition inequality]
\label{Msum}
Let $w_h$ and $v_h$ be two convex nodal functions. Then 
\begin{align*}
\partial w_h(z) + \partial v_h(z) \subset \partial (w_h + v_h)(z) 
\quad \forall z \in \Omega_h^I,
\end{align*}
where 
$\partial w_h(z) + \partial v_h(z)$ is the Minkowski sum:
\[
  \partial w_h(z) + \partial v_h(z)= \{ \bp + \bq \in \mathbb R^d: \bp \in \partial w_h(z),    \bq \in \partial v_h(z)\}
\]
\end{lem}

\begin{proof}
We note that if $\bp \in \partial w_h(z)$ and $\bq \in \partial v_h(z)$, then
\[
  w_h(x) \ge w_h(z)+ w \cdot (x-z)
 \quad \text{and} \quad
  v_h(x) \ge v_h(z)+ v \cdot (x-z) 
\]
for all $x \in \bar\Omega_h$. Adding both inequalities yields
\[
 w_h(x) + v_h(x) \ge w_h(z) + v_h(z) + (\bp+\bq) \cdot (x-z)
\]
which implies that $(\bp + \bq) \in \partial (w_h + v_h)$.
\end{proof}

Given a convex nodal function $u_h$, computing its discrete subdifferetial set is not a trivial task. The 
following lemma shows that it involves computing the convex envelope of $u_h$. 

\begin{lem}[characterization of subdifferential]\label{char_subdifferential}
Let $u_h$ be a convex nodal function, and let $\Th$ be the {simplicial} mesh induced by its convex envelope.
{The} subdifferential of $u_h$ at $z$ is the convex hull of the piecewise gradient, {that is},
\[
  { \conv \left\{D u_h|_K, K \in \Th, \ z \in {\bar{K}} \right\} }.
\]
{Here, $Du_h$ is the gradient of the piecewise linear polynomial induced by $u_h$ and $\mct$.}
\end{lem}

{As final preparation to state the finite difference version of the Alexandrov estimate, we define the nodal contact set}.

\begin{definition}[{nodal contact set}]
\label{def:NodalContactSet}
Let $v_h$ be either a nodal function or a continuous function
with $v_h\ge 0$ on $\Omega_h^B$.
The {\it (lower) nodal 
contact set} of $v_h$ is
given by
\begin{align*}
\mathcal{C}_h^-(v_h) = \{z\in \Omega_h^I:\ \Gamma_h(v_h)(z) = v_h(z)\}.
\end{align*}
\end{definition}

{Note that, for $x\in \mathcal{C}_h^-(v_h)$, we have}
\begin{align*}
v_h(z) \ge \Gamma_h(v_h)(z) \ge v_h(x)+\bp \cdot (z-x)\quad \forall z\in B_{R,h},\ \forall \bp\in \p \Gamma_h(v_h)(x).
\end{align*}

We are now ready to state and prove the finite difference Alexandrov estimate. 
Recall that we assume $\Omega$ is compactly contained in a ball $B_R$ of radius $R$,
and that we set $B_{R,h} = B_R\cap \bbZ_h^d$.

\begin{lem}[finite difference Alexandrov estimate]\label{lem:FDAE}
Let $v_h\in \fd$ with $v_h\ge 0$
on $\Omega_h^B$.  Then
\begin{align}
\label{alexFD}
\sup_{\bar\Omega_h} v_h^- 
\le C{R} \Big(\sum_{{z}\in \mathcal{C}_h^-(v_h)} |\p \Gamma_h(v_h)({z})|\Big)^{1/d},
\end{align}
where the constant $C>0$ depends only on $d$.
\end{lem}
\begin{proof}
We follow the arguments given in \cite[Proposition 5.1]{NochettoZhang16}; {also see}
\cite{KuoTrudinger90}.

Let ${z}_*\in B_{R,h}$ satisfy $\sup_{B_{R,h}} v_h^- = v_h^-({z}_*)$,
and let $L$ be a horizontal plane touching $v_h$ from below
at ${z}_*$.  By Definition \ref{def:ConvexEnvGF} we have
\begin{align*}
\Gamma_h(v_h)({z})\ge L({z}) = L({z}_*) = v_h({z}_*)\qquad \forall {z}\in B_{R,h}.
\end{align*}
Thus, $\sup_{B_{R,h}} \Gamma_h(v_h)^- \le v_h^-({z}_*)$.
Since $\Gamma_h(v_h)\le v_h$ on $B_{R,h}$ implies
\[
  \sup_{B_{R,h}} v_h^- \le \sup_{B_{R,h}} \Gamma_h(v_h)^-,
\]
we conclude that 
\begin{align*}
\sup_{\bar\Omega_h} v_h^- = \sup_{B_{R,h}} v_h^- = \sup_{B_{R,h}} \Gamma_h(v_h)^- = \sup_{B_R} \Gamma_h(v_h)^-.
\end{align*}
Therefore to conclude the proof, it suffices to show that
\begin{align*}
\max_{\bar\Omega_h} \Gamma_h(v_h)^- 
\le C{R} \Big(\sum_{{z}\in \mathcal{C}_h^-(v_h)} |\p \Gamma_h(v_h)({z})|\Big)^{1/d}.
\end{align*}
This is done in three steps.

{\it Step 1.} 
Let $K(x)$ be the cone with vertex ${z}_*$ satisfying
\begin{align*}
  K({z}_*) = - \sup_{B_{R}} \Gamma_h (v_h)^- =: - M
  \quad \text{ and } \quad
  K(x) = 0 \text{ on $\partial B_{R}$},
\end{align*}
and assume that $M>0$ for otherwise \eqref{alexFD} is trivial.
We note that for any vector $\bp \in B_{\frac{M}{2R}}(0)$, the affine function 
$
L(x) = -M +  \bp \cdot {x - {z}_*}
$
is a supporting plane of $K(x)$ at point ${z}_*$, namely
$L(x) \le K(x)$ for all $x\in B_R$ and $L({z}_*) = K({z}_*)$.  This implies
that
$
  \partial K({z}_*) \supset  B_{\frac{M}{2R}}(0),
$
and therefore
\[
|\partial K({z}_*)| \geq C\left( \frac{M}{R} \right)^d.
\]

{\it Step 2.}
We claim that
\begin{equation}\label{cone-subgrad}
  \partial K({z}_*) \subset \bigcup_{{z}\in \mathcal{C}^-_h(v_h) }
  \partial \Gamma_h (v_h)({z}) .
\end{equation}

This is equivalent to showing that for any supporting plane $L$ of
$K$ at ${z}_*$, there is a parallel supporting plane
$\tilde{L}$ for $\Gamma_h (v_h)$ at some contact node $y\in \mathcal{C}_h^-(v_h)$.

Consider the {(nodal)} function $v_h - L$, and observe that $ v_h \geq 0$ 
on $ \Omega_h^B$ and $v_h({z}_*) = K({z}_*) = L({z}_*)$, whence
\begin{align*}
    v_h({z}) - L({z}) &\; \geq \; K({z}) - L({z}) 
    \geq 0 \qquad \text{on $  \Omega_h^B$},
  \\
    v_h({z}_*) - L({z}_*) &\; = \; K({z}_*) - L({z}_*) = 0.
\end{align*}
We infer that 
$
v_h - L
$
attains a non-positive minimum for some $x\in \Omega_h^I$. Hence, 
$\widetilde{L}({z}) = L({z}) + v_h(x) - L(x)$
satisfies $\tilde{L}({z})\le v_h({z})$ for all ${z\in}B_{R,h}$
and $\tilde{L}(x) = v_h(x)$.  Applying
Definition \ref{def:ConvexEnvGF} we conclude
that $\tilde{L}\le \Gamma_h(v_h)\le v_h$
and therefore $\Gamma_h(v_h)(x) = v_h(x)$;
thus $x\in \mathcal{C}_h^-(v_h)$.


{\it Step 3.}
Computing Lebesgue measures in \eqref{cone-subgrad} yields
\[
 C\left( \frac{M}{R} \right)^d  \leq 
 | \partial K(z_*) | \leq \sum_{{z} \in \mathcal{C}^-_h(v_h) } | \partial \Gamma_h (v_h)({z}) |.
\]
Finally, \eqref{alex} follows from {this last inequality and some} simple algebraic manipulation.
\end{proof}

The following theorem states that
positive finite difference operators
satisfy a discrete Alexandrov Bakelman Pucci estimate
(cf. Theorem~\ref{thm:ABP} and \cite{KuoTrudinger90}).

%
\begin{thm}[finite difference ABP estimate]\label{thm:KTTHM1}
Suppose that $\mathcal{L}_h$ is of positive type
and of the form \eqref{def:FhisLinear}.
Suppose that $u_h\in \fd$ satisfies
\begin{equation}
\label{eqn:KTLinearDirichlet}
  \begin{dcases}
    \mathcal{L}_h u_h \le f & \text{in }\Omega^I_h,\\
    u_h  = g_h & \text{on }\Omega^B_h
  \end{dcases}
\end{equation}
for some $g_h\in \fd$.
Then there holds
\begin{align*}
\sup_{\bar\Omega_h} u^-_h \le \sup_{\Omega_h^B} g^-_h+ C \frac{R}{\lambda_{0,h}} \Big(\sum_{{z}\in \calC^-_{h}(u_h)} h^d (f^+({z}))^d\Big)^{1/d},
\end{align*}
where 
$\lambda_{0,h}$ is 
given in Definition \ref{def:OperationPositiveType}, and
the constant $C>0$ only depends on $d$.
\end{thm}
\begin{proof}
We follow the arguments given in \cite[Theorem 2.1]{KuoTrudinger90}.

Note  that we can assume, by replacing $u_h$
with $u_h +\max_{\Omega^B_h} g_h^-$ 
that $u_h\ge 0$ on $ \Omega_h^B$.  Moreover we can assume
that $\max_{\bar\Omega_h} u^-_h>0$, since otherwise
the proof is trivial.

Let ${z}\in \calC^-_{h}(u_h)$ and $y\in \St$.
Since $\Gamma_h(u_h)$ is convex,
we have $\delta_{y,h}^2\Gamma_h(u_h)({z})\ge 0$.
Thus, since the coefficients {of $F_h$}
are positive, $\Gamma_h(u_h)({z}) = u_h({z})$ and $\Gamma_h(u_h)({z}\pm h y)\le u_h({z}\pm h y)$, {we have}
\begin{align*}
0 &\le a_{y}({z})\delta^2_{y,h} \Gamma_h(u_h)({z}) \le  a_y({z})\delta_{y,h}^2 u_h({z}).
\end{align*}
Summing {over} $y\in \St$ yields
\begin{align*}
0 \le  a_{y}({z})\delta^2_{y,h} \Gamma_h(u_h)({z})
&\le \sum_{y^\prime \in \St} a_{y^\prime} ({z})\delta_{y^\prime,h}^2 u_h({z}) \le f({z}) = f^+({z}).
\end{align*}
{Take} $y$ to be the orthogonal 
set $\{y_i\}_{i=1}^d$ given in Definition \ref{def:OperationPositiveType}.
{Expand the left hand side of the previous inequality to get}
\begin{align}\label{eqn:KTinequalityOne}
\delta_{y_i,h}^+\Gamma_h(u_h)({z})-\delta_{y_i,h}^- \Gamma_h(u_h)({z}) = {h } \delta_{y_i,h}^2 \Gamma_h(u_h)({z})\le \frac{h }{\lambda_{0,h}}f^+({z}).
\end{align}
Now, let $\bp \in \p \Gamma_h(u_h)({z})$ so that $\Gamma_h(u_h)({z}\pm h y_i){\ge} \Gamma_h(u_h)({z})\pm h \bp\cdot y_i$.
By manipulating terms
and applying inequality \eqref{eqn:KTinequalityOne} {we obtain}
\begin{align*}
 \delta_{y_i,h}^- \Gamma_h(u_h)({z})\le   \bp \cdot y_i
\le \delta_{y_i,h}^+ \Gamma_h(u_h)({z})\le \delta_{y_i,h}^- \Gamma_h(u_h)({z}) + \frac{h}{\lambda_{0,h}} f^+({z}).
\end{align*}

Since $\{y_i/|y_i|\}_{i=1}^d$ is an orthonormal basis of $\bbR^d$,
these two inequalities show 
that the Lebesgue measure of $\p \Gamma_h(u_h) ({z})$ is bounded by
\begin{align*}
|\p \Gamma_h(u_h)({z})| \le \frac{h^d}{\lambda_{0,h}^d} |{f^+(z)}|^d,
\end{align*}
and therefore
\begin{align}\label{eqn:KTinequalityTwo}
\sum_{{z}\in \mathcal{C}_h^-(u_h)} |\p \Gamma_h(u_h)({z})| \le \sum_{{z}\in \calC^-_{h}(u_h)} \frac{(h f^+({z}))^d}{\lambda_{0,h}^d}.
\end{align}
Combining \eqref{eqn:KTinequalityTwo} and Lemma \ref{lem:FDAE}
yields the desired result.
\end{proof}

\begin{rem}[{extensions}]
The finite difference  ABP estimate 
given in Theorem \ref{thm:KTTHM1}
has been extended to operators 
with lower-order terms and to general meshes in 
\cite{KuoTrud96,KuoTrud00}.
\end{rem}

Theorem~\ref{thm:KTTHM1} implies 
that if $u_h$ solves
\begin{equation}
\label{eqn:KTLinearDirichlet2}
  \begin{dcases}
    \mathcal{L}_h u_h = f & \text{in }\Omega^I_h,\\
    u_h  = g_h & \text{on }\Omega^B_h
  \end{dcases}
\end{equation}
then
\begin{align}\label{eqn:FDApriori}
\max_{\bar\Omega_h} |u_h|\le \max_{\Omega_h^B} |g_h| + \frac{C {R }}{\lambda_{0,h}} \Big(\sum_{{z}\in \Omega^I_h} h^d |f({z})|^d\Big)^{1/d}.
\end{align}
Since problem \eqref{eqn:KTLinearDirichlet2} is linear
this estimate shows that there exists a unique solution to \eqref{eqn:KTLinearDirichlet2}.

Similar to the continuous case (cf. Corollary \ref{col:compprinc}), Theorem~\ref{thm:KTTHM1} implies a comparison principle.

\begin{col}[{discrete comparison}]
Suppose that ${\mathcal{L}_h}$ is of positive type
and of the form \eqref{def:FhisLinear}.
Let $u_h$ and $v_h$
be two {nodal} functions
with $u_h\le v_h$ on $\Omega_h^B$
and {$\mathcal{L}_hu_h \ge \mathcal{L}_h v_h$} in $\Omega_h^I$.
Then $u_h\le v_h$ in {$\bar\Omega_h$}.
\end{col}

Finally, since problems \eqref{def:FhisLinear}
and \eqref{eqn:KTLinearDirichlet2} 
are linear, the Lax--Richtmyer theorem
immediately gives us error estimates.

\begin{col}[{rate of convergence}]
\label{col:FDLinearErrorEstimates}
Let $I^{fd}_h:C(\bar\Omega)\to {\fd}$ denote the canonical 
interpolant onto {nodal} functions.
Let $u$ be {the} solution to
\eqref{eqn:FisLinear},
 and let
 $u_h\in {\fd}$ be the unique solution to \eqref{eqn:KTLinearDirichlet2}.
Then there holds
\begin{align*}
\|u_h -I^{fd}_hu\|_{L^\infty(\bar{\Omega}_h)}\le \|g_h-I^{fd}_h u\|_{L^\infty(\Omega_h^B)}+ \frac{C {R}}{\lambda_{0,h}} \Big(\sum_{{z}\in \Omega_h^I} 
h^d |{\mathcal{L}_h I^{fd}_h u(z)})|^d\Big)^{1/d}.
\end{align*}
\end{col}

Based on Corollary \ref{col:FDLinearErrorEstimates}
and the consistency of the approximation scheme,
one can derive error estimates with explicit 
dependence on the discretization parameter $h$.  
For example, if we can show that $\|g_h-I^{fd}_h u\|_{L^\infty(\Omega_h^B)} = \mathcal{O}(h^k)$
and $\|{\mathcal{L}_h I^{fd}_h u}\|_{L^\infty(\Omega_h^I)}= \mathcal{O}(h^k)$ 
for some positive integer $k\in \bbN$, then Corollary
\ref{col:FDLinearErrorEstimates}
shows that the error satisfies $\|u-I^{fd}_h u\|_{L^\infty(\bar\Omega_h)}\le C h^k$.
The value of $k$  is determined by the consistency
of the scheme, which typically follows from Taylor's Theorem
and the regularity of the exact solution.  Unfortunately, 
the monotonicity of a scheme restricts the size of
the order of convergence {as shown, for instance, in}  \cite{Oberman06,MR3416386}.
\begin{thm}[accuracy of monotone schemes]
A monotone finite difference scheme of the form \eqref{eqn:implicitFDForm}
is at most second order accurate 
for second order equations.
\end{thm}

%
Finally, solutions to the discrete problem
are H\"older continuous \cite[Corollary 4.6,Theorem 5.1]{KuoTrudinger90}.

\begin{thm}[{H\"older continuity}]\label{thm:KTLinearHolder}
{Let $\mathcal{L}_h$ be of positive type, and} let $u_h$ satisfy ${\mathcal{L}_h u_h=f}$
in $\Omega_h^I$.  Assume that $\Omega$ satisfies 
a uniform exterior cone codition.
Then there exist $\eta\in (0,1)$ and $C>0$, independent of $h$, 
such that
\begin{align*}
|u_h(z)-u_h(y)|\le C |z-y|^\eta,
\end{align*}
for all $z,y\in \bar\Omega_h$.
\end{thm}

Similar to the continuous setting, the development
of H\"older estimates depends on discrete Harnack inequalities.

\subsection{Construction of monotone finite difference schemes}\label{sub:ConstructingFD}
Theorem \ref{thm:KTTHM1}
shows that linear, positive finite difference
schemes are uniquely solvable
with solutions uniformly bounded
with respect to the data.  We will also
see that many of these results carry over
to the fully nonlinear case, and thus,
applying the Barles-Souganidis framework,
such schemes converge to the viscosity 
solution of the nonlinear PDE.
However, the theorem
does not indicate how to construct such schemes.  We now
discuss this issue.
First, we have 
the following classical results
\cite[Theorems 1 and 2]{MotzkinWasow53}.
\begin{thm}[{impossibility}]\label{thm:MWTHM1}
For a given (fixed) stencil width $\sm\in \bbN$,
there exists an linear, elliptic operator
${\mathcal{L}}$ such that any linear and consistent 
finite difference scheme of the form
\eqref{def:FhisLinear} is not of positive type.
\end{thm}

\begin{thm}[{existence}]\label{thm:MWTHM2}
Let ${\mathcal{L}}$ be a linear and uniformly
elliptic operator satisfying \eqref{eqn:FisLinear}.  
Then there exists, for sufficiently small $h$, 
a consistent finite difference scheme of the form \eqref{def:FhisLinear}
that is of positive type.
\end{thm}

The main punchline of these theorems
is that wide-stencils are a necessary 
feature 
of consistent
and positive type finite difference
discretizations, even for linear problems.
The proof of Theorem \ref{thm:MWTHM2}, as 
presented in \cite[Theorem 2]{MotzkinWasow53},
is not constructive.
On the other hand, the arguments given
in \cite{KuoTrudinger90,Kocan95}
explicitly give an algorithm
to construct consistent and positive
schemes, and as a result,
provide an estimate of the stencil width.
We end this section by summarizing these results.
As a first step, we state the following trivial observation.
\begin{lem}[positivity criterion]
\label{lem:ADecompLemmaFD}
Suppose that
${\mathcal{L}u} = A:D^2 u$
where the coefficient matrix 
is positive definite and has the form
\begin{align}\label{eqn:AEasyForm}
A(x) = \mathop{\sum_{y\in \bbZ^d}}_{|y|_{\infty}\le \sm} a_y(x) y\otimes y
\end{align}
for some $M\in \bbN$ and with coefficients $a_y(x)\ge 0$ for all $x\in \Omega$.
Assume further that there exists an orthogonal set $\{y_i\}_{i=1}^d\subset \bbZ^d$
with $|y_i|_{\infty}\le M$ such that $a_{y_i}(x)\ge c$
for some $c>0$.
Then the finite difference operator
\begin{align}\label{eqn:Fheasymonotone}
{\mathcal{L}_h}u_h({z}) = \mathop{\sum_{y\in \bbZ^d}}_{|y|_{\infty}\le \sm} a_y({z}) \delta_{y,h}^2 u_h({z})
\end{align}
is of positive type and
 a consistent approximation to $\mathcal{L}$
 with $\lambda_{0,h} = c$.
\end{lem}
%

Of course, not all positive definite matrices 
are of the form \eqref{eqn:AEasyForm}.
However, quite surprisingly, 
Lemma \ref{lem:ADecompLemmaFD} 
provides the essential tools to construct
consistent and positive finite difference schemes
for linear (and nonlinear) elliptic equations.
{Let us consider an example}.
%
\begin{lem}[{diagonally dominant matrix}]
\label{lem:AIsDiagonalDom}
Suppose that
${\mathcal{L}u} = A:D^2 u$
with
\begin{align}\label{eqn:NoReallyAISofTheForm}
A{(x)} = \sum_{i,j=1}^d a_{i,j}{(x)} y_i\otimes y_j,
\end{align}
where  $\{y_i\}_{i=1}^d\subset \bbZ^d$ is an orthogonal
basis of $\bbR^d$, {\ie
\[
  a_{i,j} = y_i\cdot A y_j/(|y_i| |y_j|).
\]
Suppose, in addition, that for some $c>0$ we have}
\begin{align}\label{eqn:AofTheForm}
\mathop{\sum_{i,j=1}^d}_{j\neq i} |a_{i,j}(x)|\le a_{i,i}(x)-c\qquad i=1,2,\ldots,\quad x\in \Omega,
\end{align}
Then there exists {a positive finite difference method ${\mathcal{L}_h}$ that is consistent with {$\mathcal{L}$}}
and with $\lambda_{0,h}=c$.
\end{lem}
\begin{proof}
By manipulating
terms, we may write
\begin{align*}
A 
& = \sum_{i=1}^d \Big(a_{ii} - \mathop{\sum_{j=1}^d}_{j\neq i} |a_{i,j}|\Big) y_i\otimes y_i \\
&+ \frac14 \mathop{\sum_{i,j=1}^d}_{i\neq j} \big(|a_{i,j}|+a_{i,j}\big) \big(y_i+y_j\big)\otimes \big(y_i+y_j\big)\\
&+\frac14 \mathop{\sum_{i,j=1}^d}_{i\neq j}\big(|y_{i,j}|-y_{i,j}\big)\big(y_i-y_j\big)\otimes \big(y_i-y_j\big).
\end{align*}
Thus, $A$ satisfies the conditions in Lemma \ref{lem:ADecompLemmaFD}
with $\sm\le 2 \max_i |y_i|_{\infty}$.
The result now follows from Lemma \ref{lem:ADecompLemmaFD}.
\end{proof}

\begin{rem}[{stencil size}]
By taking $\{y_i\}_{i=1}^d\subset \bbZ^d$ 
as the standard basis of $\bbR^d$ in Lemma \ref{lem:AIsDiagonalDom},
we deduce that if $A$ is strictly diagonally dominant,
then there exists a consistent and positive one-step finite difference method.
\end{rem}

{We can now, following \cite[Theorem 5]{Kocan95}, estimate the stencil size for a general positive definite matrix.} 

\begin{thm}[{existence}]\label{KOCTHM1}
Consider the elliptic operator {$\mathcal{L}u = A:D^2 u$},
where $A$ is uniformly positive definite in $\Omega$
with
\begin{align*}
\lambda {I}\le A\le \Lambda {I}
\end{align*}
for positive constants $\lambda\le \Lambda$.
Then there exists
a consistent and positive operator {$\mathcal{L}_h$}.
The stencil size $\sm$ satisfies, $\sm\le 2 M_d(8d^3\mathcal{E})$,
where $\mathcal{E}:=\Lambda/\lambda$
and
\begin{align*}
M_d(s)
= C
\begin{dcases}
  s &  d=2,\\
  s^{5/2} & d=3,\\
  s^{2d-4} &  d\ge 4.
\end{dcases}
\end{align*}
Here the constant $C>0$ only depends on the dimension $d$.
Moreover, one can take the discrete ellipticity constant
to be $\lambda_{0,h} = \lambda/(2d m^2)$.
\end{thm}
\begin{proof}
Denote by
$\{\lambda_i\}_{i=1}^d\subset [\lambda,\Lambda]$
 the eigenvalues of $A$
and by $\{\varphi_i\}_{i=1}^d$ an orthonormal
set of eigenvectors of $A$, labeled such that
\begin{align*}
A = \sum_{i=1}^d \lambda_i \varphi_i\otimes \varphi_i.
\end{align*}
By \cite[Theorem 2]{Kocan95},
for  $s>0$ to be determined, there exists ${y_i}\in \bbZ^d$ such that
\begin{align}\label{eqn:Approximatingvarphi}
\left|\varphi_i - \frac{y_i}{|y_i|}\right|_{\infty}\le \frac1s,\qquad \frac{M_d(s)}{2} \le |y_i|_{\infty} \le M_d(s).
\end{align}
Since $\{y_i\otimes y_j\}_{i,j=1}^d$ spans
$\bbS^d$, we may write
\begin{align*}
A 
&=  \sum_{i=1}^d \frac{\lambda_i}{|y_i|^2} y_i\otimes y_i
+\sum_{i=1}^d {\lambda_i}\left(\varphi_i\otimes \varphi_i- \frac{y_i\otimes y_i}{|y_i|^2}\right)\\
&=  \sum_{i=1}^d \frac{\lambda_i}{|y_i|^2} y_i\otimes y_i
+\sum_{i,j=1}^d \frac{B_{i,j}}{|y_i| |y_j|}  y_i\otimes y_j
\end{align*}
for some  $B\in \bbR^{d\times d}$.  Thus, $A$ is of the form 
 \eqref{eqn:NoReallyAISofTheForm} with $a_{i,i} = (\lambda_i+B_{i,i})/|y_i|^2$
 and $a_{i,j} = B_{i,j}/(|y_i| |y_j|)$.

Applying \eqref{eqn:Approximatingvarphi}
and the inequalities $\lambda\le \lambda_i\le \Lambda$ we obtain
$|B_{i,j}|\le  \frac{2\Lambda d^{3/2}}{s}$, and hence, 
\begin{align*}
|a_{i,i}| \ge \frac{1}{|y_i|^2} \Big(\lambda - \frac{2\Lambda d^{3/2}}{s}\Big) \ge \frac{\lambda}{|y_i|^2},\qquad |a_{i,j}|\le \frac{2 d^{3/2}\Lambda}{s|y_i| |y_j|}.
\end{align*}
Thus \eqref{eqn:AofTheForm} will be satisfied if
\begin{align*}
{\lambda} \ge \max_i \Big(\frac{2 d^{3/2} \Lambda}{s} \sum_{j=1}^d \frac{|y_i|}{|y_j|}+c|y_i|^2\Big)
\end{align*}
for some constant $c>0$.  
Taking $c = \lambda/(2 d \sm^2)$,
and noting that $|y_i|/|y_j|\le 2 d^{1/2}$, we conclude
that if 
%
\begin{align*}
s \ge 8 d^{3} \mathcal{E}, 
\end{align*}
then \eqref{eqn:AofTheForm} is satisfied.  The desired result
now follows from Lemma \ref{lem:AIsDiagonalDom}.
\end{proof}

Finally we end this section
with a result which shows that
the uniform ellipticity of {$\mathcal{L}$}
condition in Theorem~\ref{KOCTHM1} cannot
be relaxed, \cf \cite[Theorem 1]{Kocan95}.

\begin{ex}[{uniform ellipticity is necessary}]
Suppose that $d=2$ {and} ${L}u = A:D^2 u$
with
\begin{align*}
A = \begin{pmatrix}
\alpha & \beta\\
\beta & \gamma
\end{pmatrix}
\end{align*}
$\alpha,\gamma>0$ and $\alpha \gamma = \beta^2$ (so that $\det(A) =0$).
Then if $\sqrt{\alpha/\gamma}$ is irrational,
there does not exist a consistent and nonnegative scheme
for the problem {$\mathcal{L}u = f$}.
\end{ex}

\subsection{Monotonicity in finite element methods}
\label{sub:MonoFEM}

In this section we discuss the construction
of monotone finite element methods
for second order elliptic problems.
Similar to the previous section, we
focus on the linear case, where the elliptic
problem is given by \eqref{eqn:FisLinear}
and extend these results
to nonlinear problems in subsequent sections.
We further simplify the presentation
and analysis by assuming that the coefficient
matrix is the identity matrix, $A = I$ and assume Dirichlet
boundary conditions; 
thus, we focus on the Poisson problem
\begin{equation}
\label{eqn:PoissonProblemD}
  \Delta u = f \ \text{in }\Omega, \qquad u  = g \ \text{on }\p\Omega.
\end{equation}
Quite surprisingly, the results
given here extend to fully nonlinear problems.

Let $\mct$ be a simplicial, conforming, and quasi-uniform
triangulation of $\Omega$ \cite{CiarletBook}.
For simplicity, and to communicate the essential
ideas, we shall ignore the approximation
of $\Omega$ by the triangulation induced polytope
and simply assume throughout the paper that $\bar\Omega = \cup_{T\in \mct} \bar T$. 
Let
 $\lc$ 
be the linear Lagrange finite element space, i.e.,
\begin{align}\label{eqn:LagrangeSpace}
\lc = \{v_h\in C(\bar\Omega):\ v_h|_T\in \mathbb{P}_1\ \forall T\in \mct\}.
\end{align}
We extend the definitions given in the previous sections to unstructured
meshes by denoting  $\Omega_h^I$ 
and $\Omega_h^B$ 
the  sets of vertices (or nodes) of the triangulation $\mct$
that belong to $\Omega$ and $\p\Omega$
respectively.
Then any function $w_h\in \lc$ is uniquely
determined by the values $w_h(z)$ for all $z\in \bar\Omega_h:=\Omega_h^I\cup \Omega_h^B$.

To describe the finite element method and to facilitate
further developments, we assume that
the vertices are labeled such that $\bar\Omega_h = \{z_i\}_{i=1}^{N+M}$
for positive integers $N,M$, with $\Omega_h^I = \{z_i\}_{i=1}^N$
and $\Omega_h^B = \{z_i\}_{i=N+1}^M$.
The fact that continuous piecewise linear polynomials
are uniquely determined by their values at the vertices
induce a basis of {\it hat functions}
$\{\tilde\phi_i\}\subset \lc$,
with the unique property
$\tilde\phi_i(z_j) = \delta_{i,j}$.  
We define
the {\it normalized hat functions}
as ${\phi}_i = c^{-1}_i \tilde\phi_i$ 
with $c_i = \big(\int_\Omega \tilde\phi_i\big)>0$,
and note that
\begin{align*}
v_h = \sum_{i=1}^{N+M} c_i v_h(z_i){\phi}_i\qquad \forall v_h\in \lc,
\end{align*}
and
\begin{align}\label{eqn:homoLagrange}
v_h = \sum_{i=1}^{N} c_i v_h(z_i){\phi}_i\qquad \forall v_h\in {\lco:=}\lc\cap H^1_0(\Omega).
\end{align}
{We set $\omega_{z_i} = {\rm supp}(\phi_i)$, which
is the union of elements in $\mct$ that have $z_i$ as a vertex.}

A finite element method for the Poisson problem
simply restricts the variational formulation \eqref{eq:weaksol} (with $A = I$) 
onto the piecewise polynomial space $\lc$.  Thus, we consider the problem:
Find $u_h\in \lc$ with $u_h(z_i) = g(z_i)$ ($N+1\le i\le N+M$)
and
\begin{align}\label{eqn:standardPoissonFEM}
-\int_\Omega D u_h\cdot D v_h = \int_\Omega f v_h \qquad \forall v_h\in {\lco}.
\end{align}
As in the continuous setting, the existence and uniqueness of $u_h$ readily follows
from the Lax--Milgram Theorem.  An application of Cea's Lemma 
and interpolation results also
show that the error satisfies $\|\nab (u-u_h)\|_{L^2(\Omega)} = \mathcal{O}(h)$
provided $u\in H^2(\Omega)$; we refer the reader to, e.g.,
\cite{CiarletBook,BrennerBook,ErnBook} for proofs of these basic results.

To pose this problem in the operator framework
of the previous sections, we first note that
\eqref{eqn:standardPoissonFEM} is equivalent to the conditions
\begin{align}\label{eqn:standardPoissonFEM2}
-\int_\Omega D u_h\cdot D {\phi}_i = \int_\Omega f {\phi}_i \qquad i=1,2,\ldots,N.
\end{align}
Define {$\mathcal{L}_h$} such that for all interior vertices $z_i\in \Omega_h^I$,
\begin{align}
\label{eqn:FhPoissonFEM}
{\mathcal{L}_hu_h (z_i) =  \Delta_h u_h(z_i)},
\end{align}
where the {\it finite element Laplacian} is defined by
\begin{align}\label{eqn:FEMLaplacian}
\Delta_h u_h(z_i):=-\int_\Omega D u_h\cdot D \phi_i.
\end{align}
{Set
\begin{align*}
f_h(z_i) = \int_\Omega f \phi_i.
\end{align*}
}
We further define the piecewise linear function $g_h$
on $\p\Omega$ with the property
\begin{align}
\label{eqn:ghFEMdef}
g_h(z_i) = g(z_i)\qquad i=N+1,\ldots,N+M,
\end{align}
i.e., $g_h = I_h^{fe} g$ is the nodal interpolant of $g$.  
With this notation, we see that the finite element method \eqref{eqn:standardPoissonFEM}
is equivalent to  problem \eqref{eqn:NonlinearApproximation2} {with $F_h = \mathcal{L}_h-f_h$}.

Before discussing the monotonicity of the scheme \eqref{eqn:standardPoissonFEM},
let us first point out that, unlike the finite difference scheme,
one cannot take $I_h = I_h^{fe}$, the nodal interpolant,
in Definition \ref{def:ConsistentOperator}
to deduce the (operator) consistency of the finite element approximation.
The next lemma exemplifies this point. {For further details the reader is referred to \cite{JensenSmears13,NochettoZhang16}}.
\begin{lem}[finite element inconsistency]\label{lem:FEMInconsistent}
Let $\Delta_h$ be the finite element Laplacian, defined in \eqref{eqn:FEMLaplacian},
and denote by $I^{fe}_h:C(\bar\Omega)\to \lc$, the nodal interpolant {onto
the linear Lagrange finite element space}.
Then, in general, we have
\begin{align*}
\Delta_h(I^{fe}_h u)(z) \not\to \Delta u(z)\quad\forall z\in \bar\Omega_h\ \ \text{as } h\to 0
\end{align*}
for all $u\in C^2(\Omega)$.
\end{lem}
\begin{proof}
Let $d=2$,
and consider the triangulation
$\mct$ with four triangles
and vertices $z_1 = (0,0)$,
$z_2 = (h,0)$, $z_3 = (0,h)$,
$z_4 = (-h,0)$ and $z_5 = (0,-h)$.
Let $u$ be a $C^2$ function
that vanishes at the origin.
Then a calculation shows that
\begin{align*}
{\Delta_h I^{fe}_h u}(z_1) = \frac32 h^{-2} \sum_{i=2}^5 u(z_i).
\end{align*}
Taking, for example, $u(x_1,x_2) = x_1^2$ then yields
\begin{align*}
{\Delta_h I^{fe}_h u}(z_1)  = 3 \neq 2 = \Delta u(z_1)\qquad \forall h>0.
\end{align*}
In other words, $\Delta_h$  and $I_h^{fe}$ are not a consistent approximation scheme.
\end{proof}

The inconsistency in Lemma~\ref{lem:FEMInconsistent} is caused by the wrong choice of interpolation operator. The so-called elliptic projection gives a correct one.

{
\begin{definition}[elliptic projection]
The elliptic projection 
\[
  I_h^{ep}:H^1(\Omega)\cap C(\bar\Omega)\to \lc
\]
is defined by
\begin{align}\label{eqn:EllipticProjectionDef}
\Delta_h I_h^{ep}u(z_i)= -\int_\Omega D u\cdot D \phi_i\quad i=1,2,\ldots,N,
\end{align}
and $I_h^{ep} u = u$ on $\Omega_h^B$.
\end{definition}

The elliptic projection is (almost) quasi-optimal in the $L^\infty$ norm \cite{SchatzWahlbin82}.

\begin{prop}[properties of $I_h^{ep}$]
\label{prop:SchatzWahlbin}
Let $I_h^{ep} u\in \lco$ be the elliptic projection 
of $u\in C(\bar\Omega)\cap H^1_0(\Omega)$ defined by \eqref{eqn:EllipticProjectionDef}.
Then there holds
\begin{align*}
  \|u-I_h^{ep} u\|_{L^\infty(\Omega)}\le C |\log h| \inf_{v_h\in \lco} \|u-v_h\|_{L^\infty(\Omega)}.
\end{align*}
If $u\in W^{2,\infty}(\Omega)$, then
\begin{align*}
\|u-I_h^{ep} u\|_{L^\infty(\Omega)}\le C |\log h| h^2 \|u\|_{W^{2,\infty}(\Omega)}.
\end{align*}
\end{prop}

More importantly, the finite element method is consistent when one uses the elliptic projection $I_h^{fe}$.}

\begin{lem}[finite element consistency]
\label{lem:EllipticProjectionConsistency}
Let $\{z_h\}_{h>0}$ with $z_h\in \bar\Omega^I_h$ and $z_h\to z_0 \in \Omega$.
Then, for all $u\in C^2(\bar\Omega)$,
\begin{align*}
\mathcal{L}_h I^{ep}_h u(z_h) = \Delta_h(I^{ep}_h u)(z_h) \to \Delta u(z_0)\quad\forall z\in \bar\Omega_h,
\end{align*}
as $h \to 0+$. Moreover, $I_h^{ep} u\to u$ on $\p\Omega$.
\end{lem}
\begin{proof}
The convergence $I_h^{ep} u\to u$ on $\p\Omega$ follows from the definition of $I_h^{eps}$ and standard interpolation theory.

Owing to the regularity of $u$, we have $\Delta u(z_h)\to \Delta u(z_0)$ as $h\to 0^+$.
Now, denote by $\{\phi_h\}_{h>0}\subset \lco$ the normalized
hat functions.
Then, since
$\|\phi_h\|_{L^1(\Omega)} = 1$, and $\phi_h\ge 0$, we have
\[
\int_\Omega (\Delta u)\phi_h = \int_{\omega_{z_h}} (\Delta u)\phi_h\to \Delta u(z_0).
\]
Therefore by integration by parts
\begin{align*}
\Delta_h I_h^{ep} u(z_h)
= \int_\Omega (\Delta u) \phi_h \to \Delta u(z_0)
\end{align*}
as $h \to 0+$ and, consequently, $\Delta_h$ is consistent.
\end{proof}


%


\begin{lem}[{finite element monotonicity}]
\label{lem:FEMMmatrix}
Suppose that the bases satisfy 
\begin{align}\label{eqn:FEMMMatrix}
\int_\Omega D \phi_i \cdot D\phi_j \le 0.
\end{align}
for $i,j=1,2,\ldots,N+M$ and $i\neq j$.
Then {$\mathcal{L}_h$}, given by \eqref{eqn:FhPoissonFEM},  is monotone.
\end{lem}
\begin{proof}
Suppose that $v_h,w_h\in \lc$
and $w_h-v_h$ has a nonnegative maximum
at an interior vertex $z_i\in \Omega_h^B$.  Without loss of generality 
we may assume that $w_h\le v_h$ and $w_h(z_i) = v_h(z_i)$.
Then
we find that
\begin{align*}
{\mathcal{L}_h w_h(z_i)-\mathcal{L}_hv_h(z_i)}
& = - \int_\Omega D (w_h-v_h)\cdot D {\phi}_i\\
& = - \sum_{j=1}^{N+M} c_j \big(w_h(z_j)-v_h(z_j)\big) \int_\Omega D {\phi}_j \cdot D {\phi}_i\le 0.
%
\end{align*}
Thus, {$\mathcal{L}_hw_h(z_i)\le \mathcal{L}_hv_h(z_i)$}, and therefore {$\mathcal{L}_h$} is monotone.
\end{proof}

Lemma~\ref{lem:FEMMmatrix}
indicates that the finite element method
is monotone provided that a certain
mesh condition is satisfied.  
Indeed, for an edge $E\subset \partial T$, we denote
by $\theta_E^T$ the angle
between the faces not containing 
$E$, and by $\kappa_E^T$ the $(d-2)$
dimensional simplex opposite to $E$, 
then there holds \cite{XuZik99,StrangFix73}
\begin{align*}
\int_\Omega D\tilde \phi_i \cdot D\tilde\phi_j = -\frac{1}{d(d-1)}\sum_{T\supset E} |\kappa_E^T| \cot \theta^T_E,
\end{align*}
where $E$ is the edge with vertices $z_i$ and $z_j$.
The condition \eqref{eqn:FEMMMatrix} is satisfied
if the mesh is weakly acute.
For example, in two dimensions, this condition
means that the sum of the angles
opposite to any edge is less than or equal 
to $\pi$.

\begin{col}[maximum principle]
\label{cor:MaxFEM}
Let $u_h\in \lc$
solve {\eqref{eqn:standardPoissonFEM2}}.
Suppose that \eqref{eqn:FEMMMatrix} is satisfied and
$f\ge 0$ and $g\le 0$.
Then $u_h\le 0$.
\end{col}
\begin{proof}
Define the stiffness matrix $S\in \bbR^{(N+M)\times (N+M)}$
by
\begin{align}\label{eqn:StiffnessMatrix}
S_{i,j} = -\int_\Omega D \phi_i\cdot D \phi_j,
\end{align}
and note that condition \eqref{eqn:FEMMMatrix} 
is equivalent to $S_{i,j}\ge 0$ {for $i\neq j$}.  Moreover, 
since the hat {functions} form a partition of unity,
there holds
\begin{align*}
\sum_{j=1}^{N+M} c_j S_{i,j} = 0.
\end{align*}

Now, suppose that $u_h$ attains a strict positive maximum
at an interior node $z_i$.  We then find
\begin{align*}
0 &= {\mathcal{L}_hu_h(z_i)}  = -\int_\Omega D u_h\cdot D \phi_i - \int_\Omega f\phi_i\\
 &\le  -\int_\Omega D u_h\cdot D \phi_i\\
& = \sum_{j=1}^{N+M} c_j u_h(z_j) S_{i,j} = \sum_{j=1}^{N+M} c_j \big(u_h(z_j)-u_h(z_i)\big) S_{i,j}<0,
\end{align*}
a contradiction.
\end{proof}

\subsection{Finite element stability estimates: Alexandrov estimates and Alexandrov-Bakelman-Pucci maximum principle}

Similar to the finite difference
schemes discussed in the previous
section, 
we develop 
some discrete Alexandrov estimates
for finite element functions
and analogous ABP maximum principles.
Before stating and proving these results, 
it is useful to discuss some properties
of the convex envelope of piecewise
linear polynomials. 

For $v_h\in \lc$ with $v_h\ge 0$
on $\p\Omega$, let $\Gamma(v_h)$ and
$\Gamma_h(v_h)$ denote
the convex envelope
and discrete convex envelope of $v_h$
given in Definitions \ref{def:convexenvelope}
and \ref{def:ConvexEnvGF}, respectively.
Then, since $v_h$ is piecewise affine,
we find that $\Gamma_h(v_h) = \Gamma(v_h)$,
and furthermore, $\Gamma_h(v_h)$ is also
piecewise affine.  However, perhaps unexpectedly,
$\Gamma(v_h)$ is not necessarily piecewise linear
subordinate to $\mct$!  The following
examples illustrate this feature. 

\begin{ex}[{convex envelope}]\label{ex:FEMCE1}
Consider a triangulation with vertices  $z_1 = (1,0),\ z_2 = (0,1),\ z_3 = (-1,0),\
z_4 = (0,-1)$ and $z_5 = (0,0)$.
Consider the piecewise linear functions satisfying
\begin{alignat*}{2}
&v_1(z_1) = v_1(z_3)=1,\quad v_2(z_2)=v_2(z_4) = 1,\\
&v_3(z_1) = v_3(z_2)=v_3(z_3)=v_3(v_4)=1,
\end{alignat*}
and $v_j(z_i)=0$ otherwise.
The convex envelopes are $\CE {v_1} = |x_1|$, $\CE {v_2} = |x_2|$,
 and $\CE {v_3} = |x_1|+|x_2|$.
The convex envelopes are subordinate to the meshes  depicted in Figure \ref{meshFig}.
\end{ex}

As shown in the example above, since $\Gamma (v_h)$ is a piecewise linear function, 
it induces a mesh $\tilde{\mathcal{T}}_h$ which depends on $v_h$. 
The following example shows that 
if $v_h$ is the nodal interpolant of a function $v$, 
and if the Hessian $D^2 v$ is degenerate (or nearly degenerate), the induced mesh may be anisotropic.
\begin{ex}[{anisotropy}]\label{ex:FEMCE2}
\label{ex:anisotropy}
  Let $\dm= \mathbb{R}^2$ and $\bar\Omega_h = \{ (k, m) \}$. Let 
  $v(x) = (x \cdot e)^2$ where $e = (1, A)$ for some integer $0< A$ and $v_h(z) = v(z)$ for all $z \in \bar\Omega_h$. 
  Then the convex envelope induces an anisotropic mesh depicted in Figure \ref{fig:anisotropy}. 
  The convex envelope in the star of the origin is $|x \cdot e|$. 
\end{ex}

\begin{figure}
\begin{center}
\begin{tikzpicture}[scale = 1.25]
\draw[-,thick](1,0)--(0,1)--(-1,0)--(0,-1)--(1,0);
\draw[-,thick](0,-1)--(0,1);
\node[inner sep = 0pt,minimum size=6pt,fill=black!100,circle] (n2) at (1,0)  {};
\node[inner sep = 0pt,minimum size=6pt,fill=black!100,circle] (n2) at (0,1)  {};
\node[inner sep = 0pt,minimum size=6pt,fill=black!100,circle] (n2) at (-1,0)  {};
\node[inner sep = 0pt,minimum size=6pt,fill=black!100,circle] (n2) at (0,-1)  {};
\node[inner sep = 0pt,minimum size=6pt,fill=black!100,circle] (n2) at (0,0)  {};
\node at (1.2,0) {1};
\node at (0,1.2) {0};
\node at (-1.2,0) {1};
\node at (0,-1.2) {0};
\node at (0.2,0) {0};
\end{tikzpicture}\quad
\begin{tikzpicture}[scale = 1.25]
\draw[-,thick](1,0)--(0,1)--(-1,0)--(0,-1)--(1,0);
\draw[-,thick](-1,0)--(1,0);
\node[inner sep = 0pt,minimum size=6pt,fill=black!100,circle] (n2) at (1,0)  {};
\node[inner sep = 0pt,minimum size=6pt,fill=black!100,circle] (n2) at (0,1)  {};
\node[inner sep = 0pt,minimum size=6pt,fill=black!100,circle] (n2) at (-1,0)  {};
\node[inner sep = 0pt,minimum size=6pt,fill=black!100,circle] (n2) at (0,-1)  {};
\node[inner sep = 0pt,minimum size=6pt,fill=black!100,circle] (n2) at (0,0)  {};
\node at (1.2,0) {0};
\node at (0,1.2) {1};
\node at (-1.2,0) {0};
\node at (0,-1.2) {1};
\node at (0,0.2) {0};
\end{tikzpicture}\quad
\begin{tikzpicture}[scale = 1.25]
\draw[-,thick](1,0)--(0,1)--(-1,0)--(0,-1)--(1,0);
\draw[-,thick](-1,0)--(1,0);
\draw[-,thick](0,-1)--(0,1);
\node[inner sep = 0pt,minimum size=6pt,fill=black!100,circle] (n2) at (1,0)  {};
\node[inner sep = 0pt,minimum size=6pt,fill=black!100,circle] (n2) at (0,1)  {};
\node[inner sep = 0pt,minimum size=6pt,fill=black!100,circle] (n2) at (-1,0)  {};
\node[inner sep = 0pt,minimum size=6pt,fill=black!100,circle] (n2) at (0,-1)  {};
\node[inner sep = 0pt,minimum size=6pt,fill=black!100,circle] (n2) at (0,0)  {};
\node at (1.2,0) {1};
\node at (0,1.2) {1};
\node at (-1.2,0) {1};
\node at (0,-1.2) {1};
\node at (0.2,0.2) {0};
\end{tikzpicture}
\end{center}
\caption{\label{meshFig} Meshes corresponding
to convex envelopes $\Gamma(v_1) = |x_1|$ (left)
and $\Gamma(v_2)= |x_2|$ (middle),
and $\Gamma(v_3) = |x_1|+|x_2|$ (right).}
\end{figure}
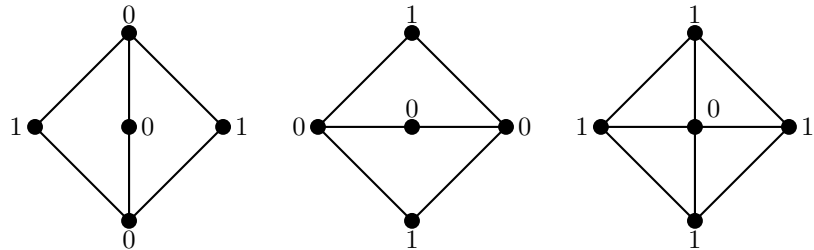

\begin{figure}
\begin{center}
\begin{tikzpicture}[scale = 1.25]
\draw[-,thick](1,0)--(-1,1)--(-2,1)--(-1,0)--(1,-1)--(2,-1)--(1,0);
\draw[-,thick](1,0)--(0,0)--(-1, 0);
\draw[-,thick](-1,1)--(0,0)--(1, -1);
\draw[-,thick](-2,1)--(0,0)--(2, -1);
\node[inner sep = 0pt,minimum size=6pt,fill=black!100,circle] (n2) at (1,0)  {};
\node[inner sep = 0pt,minimum size=6pt,fill=black!100,circle] (n2) at (-1,1)  {};
\node[inner sep = 0pt,minimum size=6pt,fill=black!100,circle] (n2) at (-2,1)  {};
\node[inner sep = 0pt,minimum size=6pt,fill=black!100,circle] (n2) at (-1,0)  {};
\node[inner sep = 0pt,minimum size=6pt,fill=black!100,circle] (n2) at (1,-1)  {};
\node[inner sep = 0pt,minimum size=6pt,fill=black!100,circle] (n2) at (2,-1)  {};
\node at (1,0) [above right] {1};
\node at (-1,0) [below left] {1};
\node at (-1,1) [above right] {1};
\node at (1,-1) [below left] {1};
\node at (0,0) [above] {0};
\node at (-2,1) [above left] {0};
\node at (2,-1) [below right] {0};
\end{tikzpicture}
\label{fig:anisotropy}
\caption{ Mesh induced by the nodal interpolant  of $v(x) = (x \cdot e)^2$ where $e = (1,2)$. Its convex envelope 
equals $|x \cdot e|$ in the star of $(0,0)$.}
\end{center}
\end{figure}
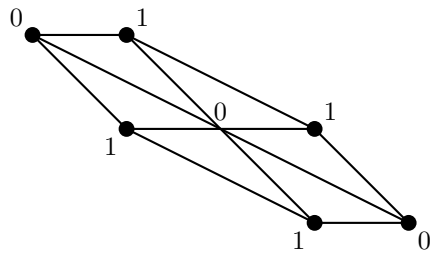


Let us now state the Alexandrov estimate for finite element functions.
Recall that $\Omega$
is compactly contained in a ball $B_R$, and the nodal
contact set $\mathcal{C}_h^-(v_h)$ is given in Definition \ref{def:NodalContactSet}.

\begin{lem}[finite element Alexandrov estimate]
\label{Alexandroff}
For every $v_h \in \lc$ such that  $v_h \geq 0$ on $\partial \Omega$, we have
\begin{align}\label{alex}
  \sup_{\bar \Omega} v_h^- \leq C R \left( \sum_{z \in \mathcal{C}^-_h(v_h) } |\partial  \Gamma (v_h) (z)| \right)^{1/d},
\end{align}
where the constant $C$ depends only on the dimension $d$ and the domain $\Omega$.
\end{lem}
\begin{proof}
The result directly follows from the proof of Lemma~\ref{lem:FDAE}. Indeed it suffices to realize that, for every $v_h \in \lc$,
$\sup_{\bar\Omega} v_h^- = \sup_{\bar\Omega_h} v_h^-$
and $\Gamma_h(v_h) = \Gamma(v_h)$; see \cite[Proposition 5.1]{NochettoZhang16} for details.
\end{proof}

Let us point out that, with Lemma~\ref{Alexandroff} {in} hand,
it may be possible to extend the arguments given
in Theorem \ref{thm:KTTHM1}
to develop ABP estimates for piecewise linear polynomials.  Instead, following 
\cite[Section 5]{NochettoZhang16}, 
we outline a proof which is more geometric
and is based on the characterization of the subdifferential
of piecewise linear functions.  

As a first step
we define the local convex envelope and local subdifferential at a point $z\in \mathcal{C}_h^-(v_h)$
(\cf Definition~\ref{def:NodalContactSet}).

\begin{definition}[{local convex envelope}]
\label{def:LocalCE}
For $v_h\in \lc$ and contact node $z\in \mathcal{C}_h^-(v_h)$,
let $\omega_z$ denote the union of elements in $\mct$
that have $z$ as a vertex.
We then define the {\it local convex envelope} by
\begin{align*}
\Gamma_z(v_h)(x) = \sup \{L(x):\ L\le v_h\text{ in }\omega_z,\ L\in \mathbb{P}_1,\ L(z) = v_h(z)\}
\end{align*}
for all $x\in \omega_z$.
Its {\it local sub-differential} is
\begin{align}\label{eqn:localsub}
\p \Gamma_z(v_h)(z) = \{\bp\in \bbR^d:\ \Gamma_z(v_h)(x)\ge \Gamma_z(v_h)(z)+\bp\cdot (x-z),\ \forall x\in \omega_z\}.
\end{align}
\end{definition}

Comparing \eqref{eqn:localsub}
with Definition \ref{def:convexenvelope}, we
easily deduce that
\begin{align}
\p \Gamma(v_h)(z)\subset \p \Gamma_z(v_h)(z)\qquad \forall z\in \mathcal{C}_h^-(v_h),
\end{align}
and therefore, by Lemma \ref{Alexandroff},
\begin{align}\label{eqn:vhrelationW2}
  \sup_{\bar \Omega} v_h^- \leq C R \left( \sum_{z \in \mathcal{C}^-_h(v_h) } |\partial  \Gamma_z (v_h) (z)| \right)^{1/d}.
\end{align}
Less obvious is the following result.
\begin{prop}[{subordination}]
\label{prop:subordinate}
Suppose that $d=2$ and, for $v_h\in \lc$ and contact point $z\in \mathcal{C}_h^-(v_h)$,
let $\Gamma_z(v_h)$ be given by {Definition~\ref{def:LocalCE}}.
Then $\Gamma_z(v_h)$ is subordinate to $\mct$.
\end{prop}

{We refer the reader to \cite[Lemma 5.1]{NochettoZhang16} for a proof of this result. Let us here, instead, show that if $d \geq3 $ the assertion is no longer true. Set
\begin{align*}
z_0 = (0,0,-1),\quad z_1 = (-1,0,0),\quad z_2 = (0,1,0),\quad
z_3 = (1,0,0),
\end{align*}
and let $T_1,T_2$ be the convex hulls
of $z_0,z_1,z_2,z_3$ and $z_0,z_1,-z_2,z_3$.
Consider the piecewise linear function $v_h$
with values $v_h(z_0) = -1,\ v_h(z_1) = v_h(z_3)=0$
and $v_h(\pm z_2) = -1$.  Then $\Gamma_{z_0}(v_h)(x) = |x_1|-1$
is not affine on $T_i$ for each $i=1,2$}.

%
%

Next, to derive a ABP maximum principle, 
we state the relation between the subdifferential 
of a convex, piecewise linear 
polynomial with its finite element Laplacian. 
As a first step, we first integrate by parts
in \eqref{eqn:FhPoissonFEM} to get the identity 
\begin{align*}
\Delta_h v_h(z_i) = -\sum_{F\in \mathcal{F}_{z_i}} \int_F \jump{D v_h}\phi_i \qquad \forall v_h\in \lc.
\end{align*}
Here, $\mathcal{F}_{z_i}$ is the set of (interior) $(d-1)$-dimensional
simplices that have $z_i$ as a vertex, and, for a vector-valued function {$\bw$}, the {\it jump of $\bw$} across
the face $F$ is given by

\begin{align}\label{eqn:JumpDef}
\jump{{\bw}}\big|_F:=
\left\{
\begin{array}{ll}
\bn_F^+ \cdot {\bw}^+\big|_F + \bn_F^{-}\cdot {\bw}^-\big|_{F} & \text{if }F = \p K^+\cap \p K^-,\\
\bn_F^+\cdot \bw^+\big|_F & \text{if }F = \p K^+\cap \p \Omega,
\end{array}
\right.
\end{align}
with  {$\bw^\pm = \bw|_{K^\pm}$}, and $\bn_F^\pm$ denoting the outward unit normal vectors of $K^\pm$
on $F$.  We also define the jump of a scalar function $v$ across $F$ as 
\begin{align}\label{eqn:JumpDefScalar}
\jump{v}\big|_F :=
\left\{
\begin{array}{ll}
\bn_F^+  v^+\big|_F+ \bn_F^- {v}^-\big|_F &  \text{if }F = \p K^+\cap \p K^-,\\
\bn_F^+ v^+\big|_F & \text{if }F = \p K^+\cap \p \Omega.
\end{array}
\right.
\end{align}

Now, since $\jump{Dv_h}$ is constant on $F$,
and 
\begin{align*}
\int_F \phi_i = \frac{d+1}{d}\frac{|F|}{|\omega_{z_i}|},
\end{align*}
we can obtain an expression on the finite element Laplacian
with explicit dependence on the jumps:
\begin{align}\label{eqn:ExplicitDep}
\Delta_h v_h(z_i) = \frac{-(d+1)}{d} \sum_{F\in \mathcal{F}_{z_i}}  \frac{|F|}{|\omega_{z_i}|} \jump{D v_h}\big|_F.
\end{align}
A relationship
between the jumps of the gradients (and hence the discrete Laplacian)
and the subdifferential of a convex, piecewise
affine function is now given.
\begin{prop}[{subdifferential vs. jumps}]
\label{prop:convexSubJump}
Let $\gamma$ be a piecewise affine convex function
on a patch $\omega_z$ for some node $z {\in \Omega_h^I}$,
and denote by $\mathcal{F}_z$ the set 
of $(d-1)$--dimensional simplices that touch $z$.
Then, {for any $F \in \calF_z$, the jump $\jump{D \gamma}|_F$} is nonpositive and 
\begin{align}\label{eqn:convexSubJump}
|\p \gamma(z)| \le C \Big(\sum_{F\in \mathcal{F}_z} -\jump{D \gamma}\big|_F\Big)^d.
\end{align}
\end{prop}
We will not give a complete proof of Proposition
\ref{prop:convexSubJump}, but rather give
a rough idea {of} how such a result is obtained
in two dimensions.  Further details can be found 
in \cite[Section 5.2]{NochettoZhang16}.

Without loss of generality, assume that $z=0$
and $\gamma(0) = 0$.  We further denote
by $\{z_j\}_{j=1}^m$ the set of nodes
in $\omega_z$.
Now, since $\gamma$
is piecewise affine function,
a vector  $\bp\in \p \gamma(0)$ is characterized
by the inequalities
\begin{align*}
\bp \cdot z_j \le \gamma(z_j)\qquad 1\le j\le m.
\end{align*}
Therefore, we conclude that
the subdifferential of $\gamma(0)$ is 
a convex polygon determined by the intersection
of the half-spaces
\begin{align}\label{eqn:SjDef}
S_j:=\{\bp\in \bbR^2:\ \bp \cdot z_j\le \gamma(z_j)\},
\end{align}
and that a vector $\bp$ is in the interior of $\p \gamma(0)$ if and only {if}
\begin{align*}
\bp \cdot z_j< \gamma(z_j),\qquad 1\le j\le m,
\end{align*}
and is on the boundary of $\p \gamma(0)$ if 
\begin{align*}
\bp \cdot z_j = \gamma(z_j)
\end{align*}
for some $j$.
This characterization of the boundary motivates
the introduction of a $\p \gamma(0)$ induced
{\it dual mesh}, which we now explain.

Let $T$ be an $n$--dimensional simplex in $\omega_z$
with $0\le n\le 2$ such that $0\in T$.  
We then define the $(2-n)$-dimensional dual set $T^*$ as follows (see Figure \ref{fig:subdifferential})
\begin{enumerate}[$\bullet$]
\item If $n=0$, so that $T=\{0\}$, then we define $T^*$ as the sub-differential $\p \gamma(0)$.

\item If $n=2$, so that $T=K$ is an element of $\omega_z$, 
then $T^*$ is the vector $\p \gamma\big|_K$.

\item If $n=1$, so that $T=F = \p K^+\cap K^-$ is an (interior) edge in $\omega_z$,
then $T^*$ is the line segment jointing the two vectors $\p \gamma|_{K^\pm} = D \gamma_{K^\pm}$.
\end{enumerate}
Note that $T^*$ is a convex polytope contained in the $(2-n)$-dimensional plane
\begin{align*}
P_T = \{\bp\in \bbR^2:\ \bp\cdot z = \gamma(z)\ \forall z\in T\},
\end{align*}
and therefore, for arbitrary $\bp_1,\bp_2\in P_{{T}}$, $(\bp_1-\bp_2)\cdot z=0$
for all $z\in T$, i.e., $P_T$ is orthogonal to $T$.

Now, the proceeding discussion implies that 
the boundary of $\p \gamma(0)$ is given by
\begin{align*}
\bigcup \big\{F^*:\ \text{edges $F$ in $\omega_z$ with $0\in F$}\big\},
\end{align*}
{in other words}, the boundary of $\p\gamma$ is made of line segments
that join $\p \gamma(0)$ on neighboring triangles (see \cite[Proposition 5.6]{NochettoZhang16}
for further details).  If $F = \p K^+\cap \p K^-$, then it follows
that the length of $F^*$ is given by
\begin{align*}
|F^*| = \big|D \gamma|_{K^+}-D \gamma|_{K^-}\big| = -\jump{D \gamma}\big|_F,
\end{align*}
where we have used the fact that
$D\gamma|_{K^+}- D\gamma_{K^-}$ 
is perpindicular to $F$ and the nonpositivity of $\jump{D \gamma}$
in the second equality.  Putting everything together,
we conclude that the boundary of $\p \gamma(0)$ is bounded
by $\sum_{F\in \mathcal{F}_z} - \jump{D\gamma}\big|_F$; thus, by the {isoperimetric}
inequality,
\begin{align*}
|\p \gamma(0)|\le C \Big(\sum_{F\in \mathcal{F}_z} -\jump{D \gamma}\big|_F\Big)^2.
\end{align*}
This last statement is \eqref{eqn:convexSubJump} for $d=2$.

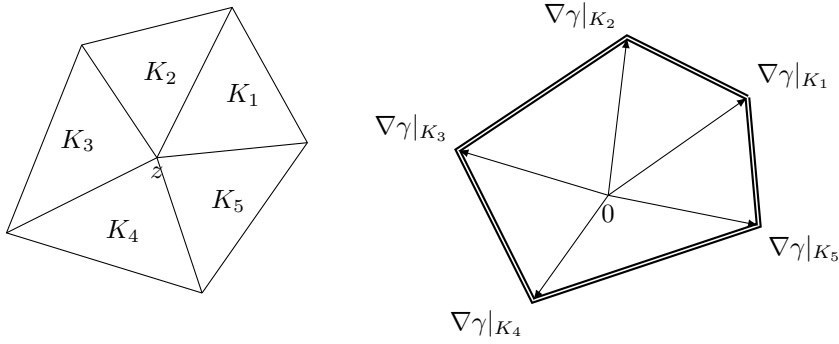
\begin{figure}
\label{fig:subdifferential}
\begin{tikzpicture}
  \coordinate [label=below:$z$] (z1) at (1,1);
  \coordinate  (z2) at (3,1.2);
  \coordinate  (z3) at (2,3);
  \coordinate  (z4) at (0,2.5);
  \coordinate  (z5) at (-1,0);
  \coordinate  (z6) at (1.6,-0.8);
  \coordinate [label=below left:$K_1$] (k1) at (2.5, 2.1);
  \coordinate [label=below left:$K_2$] (k3) at (1.4,2.4);
  \coordinate [label=below left:$K_3$] (k4) at (0.3,1.5);
  \coordinate [label=below left:$K_4$] (k5) at (0.9,0.3);
  \coordinate [label=below left:$K_5$] (k2) at (2.3, 0.7);

  \coordinate [label=below :$0$] (o) at (7,0.5);
  \coordinate [label=above right:$\gradv \gamma|_{K_1}$] (g1) at (8.85,1.8);
  \coordinate [label=above left :$\gradv \gamma|_{K_2}$] (g2) at (7.25,2.6);
  \coordinate [label=above left:$\gradv \gamma|_{K_3}$] (g3) at  (5,1.1);
  \coordinate [label=below left:$\gradv \gamma|_{K_4}$] (g4) at  (6,-0.9);
  \coordinate [label=below right:$\gradv \gamma|_{K_5}$] (g5) at  (9,0.1);  
  \draw  (z1) -- (z2)  ;
  \draw  (z1) -- (z3)  ;
  \draw  (z1) -- (z4)  ;
  \draw  (z1) -- (z5)  ;
  \draw  (z1) -- (z6)  ;
  \draw  (z6)- - (z5) -- (z4) -- (z3) -- (z2) -- (z6);
  \draw [double, thick](g1) -- (g2) -- (g3) -- (g4) -- (g5) -- (g1);
  \draw [-latex] (o) -- (g1);
  \draw [-latex] (o) -- (g2);
  \draw [-latex] (o) -- (g3);
  \draw [-latex] (o) -- (g4);
  \draw [-latex] (o) -- (g5);
\end{tikzpicture}
\caption{A pictorial description of the dual set of a patch.}
\end{figure}

\begin{thm}[finite element ABP estimate]\label{thm:FEABP}
Suppose that the simplicial mesh $\mct$
satisfies \eqref{eqn:FEMMMatrix} and that $u_h\in \lc$ satisfies
\[
  \begin{dcases}
    \mathcal{L}_hu_h \le f_h  & \text{in }\Omega_h^I,\\
    u_h  = g_h  & \text{on }\Omega_h^B,
  \end{dcases}
\]
where $\mathcal{L}_h$ and $g_h$ are given by \eqref{eqn:FhPoissonFEM} and \eqref{eqn:ghFEMdef}, respectively,
and $\lc$ is the linear Lagrange space defined by \eqref{eqn:LagrangeSpace}.
Then there holds
\begin{align}\label{eqn:FEABP}
\sup_{\bar\Omega} u_h^-\le \sup_{\Omega_h^B} g_h^- + C R \Big(\sum_{z\in \mathcal{C}_h^-(u_h)} |\omega_z| ({f^+_h(z)})^d\Big)^{1/d}.
\end{align}
\end{thm}
\begin{proof}
We give the proof under the assumption that 
$\Gamma_z(u_h)$ is subordinate to $\mct$
(which is the case in two dimensions).
For the proof of the general case, 
we refer the interested readers to \cite{NochettoZhang16}.

As in the proof of Theorem \ref{thm:KTTHM1}, we may assume $u_h\ge 0$
on $\Omega_h^B$ and $\sup_{\bar\Omega} u_h^->0$.

Let $z\in \mathcal{C}_h^-(u_h)$,
and note that $\Gamma_z(u_h)(x)\le u_h(x)$
for all $x\in \omega_z$ with equality 
at $z$.  Then by \eqref{eqn:FEMMMatrix}
and Lemma \ref{lem:FEMMmatrix},
and since $\Gamma_z(u_h)$ is subordinate
to $\mct$, we find that
\begin{align*}
\Delta_h \Gamma_z(u_h)(z)\le \Delta_h u_h(z).
\end{align*}
Moreover, by \eqref{eqn:ExplicitDep},
we have
\begin{align*}
 -\sum_{F\in \mathcal{F}_z}  \frac{|F|}{|\omega_{z}|} \jump{D \Gamma_z(u_h)}\big|_F = \frac{d}{d+1} \Delta_h \Gamma_z(u_h) (z) \le \frac{d}{d+1} \Delta_h u_h(z).
\end{align*}
Applying Proposition 
 \ref{prop:convexSubJump}, and using
  $|F| \approx |\omega_z|^{1-1/d}$,
we conclude that
\begin{align*}
|\p \Gamma_z(u_h)(z)|\le C \Big(- \sum_{F\in \mathcal{F}_z} \frac{|F|}{|\omega_z|} \jump{D\Gamma_z}\big|_F\Big)^d |\omega_z|
\le C \frac{  |\omega_z|}{d+1} \Delta_h u_h(z).
\end{align*}
Combining this last inequality with \eqref{eqn:vhrelationW2} yields the result.
\end{proof}

%

\begin{rem}[contact set]
\label{rem:INeedThisLater}
The proof of Theorem \ref{thm:FEABP}
shows that if $\mathcal{L}_h u_h\le 0$ only
on $\mathcal{C}_h^-(u_h)$,
then  \eqref{eqn:FEABP}
is still satisfied. {This will be important in subsequent developments}.
\end{rem}

\section{Finite element methods for elliptic problems in non--divergence form}
\label{sec:FEMnondiv}

In this section we summarize
recent advancements of finite element
methods for elliptic problems in 
nondivergence form with nonsmooth coefficients. 
For simplicity, and to illustrate the main ideas,
we  consider problems of the form \eqref{eq:nondiv}
with no lower--order terms
and with homogeneous Dirichlet boundary conditions.  
In this setting the problem
reads
\begin{equation}
\label{nonDiv}
\mathcal{L}u=A:D^2 u  =f \ \text{in } \Omega, \qquad
u = 0 \ \text{on } \p\Omega,
\end{equation}
where $A$ is a symmetric positive definite in $\bar{\Omega}\subset \bbR^d$
with either $A\in C(\bar{\Omega}, \polS^d)$ or $A\in L^\infty(\Omega, \polS^d)$.
Further assumptions of the domain $\Omega$, its boundary $\p\Omega$,
the coefficient matrix $A$, and the source function $f$ will be made as they become necessary.

Recall from Section \ref{sec:maxNenergy}
that if the coefficient matrix 
is sufficiently smooth,
then we can write the PDE in divergence form
with $A$ as the diffusion coefficient and $D\cdot A$ (taken column--wise)
as the convective coefficient.   In this setting, 
(weak) solutions are defined by an integration by parts argument (\cf Section~\ref{sub:weakvarsols}),
and as such, finite element methods are easily constructed.
However, in the case that $A$ is not differentiable
 the clear--cut methodology of Galerkin methods is no longer
 valid. 

This section summarizes three classes of 
finite element  methods for problem \eqref{nonDiv}, each
motivated by the different solution concepts
presented in Sections \ref{sub:strongsols}--\ref{sub:viscosols}.
The first class considers problem \eqref{nonDiv} on convex domains 
with $A\in L^\infty(\Omega, \polS^d)$ satisfying
the Cordes condition.
  The second class of methods
is motivated by the notion of strong solutions (\cf Section~\ref{sub:strongsols})
under the assumption that $A\in C(\bar{\Omega}, \polS^d)$.
Finally, the third method is motivated by the notion
of viscosity solutions (\cf Section~\ref{sub:viscosols}), where comparison principles
and monotonicity of the scheme are the central themes.


\subsection{Discretization of nondivergence form PDEs satisfying the Cordes condition}\label{sec:DiscreteCordes}
Here we discuss recent numerical methods
for second--order elliptic PDEs in non--divergence 
form satisfying the Cordes condition \eqref{eq:CordesCond}.
Recall from Section \ref{sub:Cordes} that, if the domain is convex, this condition
on the coefficient matrix ensures that the bilinear mapping
\begin{align}\label{CordesaDef}
a(\cdot,\cdot):\big(H^2(\Omega)\cap H^1_0(\Omega)\big)^2\ni (v,w)\to a(v,w)= \int_\Omega \gamma \mathcal{L}v \Delta w
\end{align}
is coercive,
thus allowing one to define strong solutions
via variational principles.  Here, the function $\gamma$
is given by \eqref{eq:Cordes1}.
In this setting the existence of 
a strong solution $u\in H^2(\Omega)\cap H^1_0(\Omega)$ to \eqref{nonDiv}
is deduced from the variational formulation 
\begin{align}\label{CordesVariational}
a(u,v) = \int_\Omega \gamma f\Delta v\, dx\qquad\forall  v\in H^2(\Omega)\cap H^1_0(\Omega)
\end{align}
and appealing to the Lax--Milgram Theorem.
Since the mapping
 $\Delta:H^2(\Omega)\cap H^1_0(\Omega)\to L^2(\Omega)$
is surjective on convex domains, one concludes
that a function $u\in H^2(\Omega)\cap H^1_0(\Omega)$ satisfying \eqref{CordesVariational}
satisfies \eqref{nonDiv} almost everywhere, i.e.,
$u\in H^2(\Omega)\cap H^1_0(\Omega)$ is a strong solution to \eqref{nonDiv}.
We refer the reader to Theorem \ref{thm:exuniqueCordes} for details.

One immediately sees that the solution concept
lends itself to a finite element approximation which would
pose \eqref{CordesVariational}
over a finite dimensional subspace of $H^2(\Omega)\cap H^1_0(\Omega)$
consisting of piecewise polynomials.
To describe this procedure, we denote
by $\mct$ a simplicial, conforming, and shape regular
triangulation of $\Omega$ parameterized by $h>0$,
and let  $\co\subset H^2(\Omega)\cap H^1_0(\Omega)$ 
be a finite dimensional
subspace consisting of piecewise polynomials
with respect to $\mct$.
A conforming finite element approximation to
\eqref{CordesVariational} seeks a function
$u_h\in \co$ satisfying the discrete variational formulation
\begin{align}
\label{CordesConformingFEM}
a(u_h,v_h) = \int_\Omega \gamma f \Delta v_h \qquad \forall v_h\in \co.
\end{align}
Problem \eqref{CordesConformingFEM} represents
a square linear system of equations.
The coercivity of the bilinear form $a(\cdot,\cdot)$
over $H^2(\Omega)\cap H^1_0(\Omega)$ implies that this system
is invertible, and thus, there exists
a unique solution $u_h\in \co$ to problem
\eqref{CordesConformingFEM}.
Moreover, the continuity of $a(\cdot,\cdot)$ and 
Cea's Lemma show that such approximations
are quasi--optimal in the sense that
\begin{align*}
\|u-u_h\|_{H^2(\Omega)}\le \frac{C}{\alpha} \inf_{v_h\in \co} \|u-v_h\|_{H^2(\Omega)},
\end{align*}
where $C>0$ and $\alpha = 1-\sqrt{1-\epsilon}$ are respectively 
the continuity and coercivity constants of $a(\cdot,\cdot)$.

While method \eqref{CordesConformingFEM}
is a stable and convergent numerical
scheme to compute solutions to \eqref{CordesVariational},
there are some potential practical drawbacks of the method.
Piecewise polynomial subspaces of $H^2(\Omega)$ are difficult
to construct and implement, and are not a practical
option to solve second--order PDEs.  These properties 
are further exacerbated in  three dimensions, 
where, \eg  polynomials of degree of at least nine
are required to construct $H^2$ conforming finite element spaces
on general simplicial partitions; see \cite[Remark 1]{LaiSchu07} and \cite{Zenny73}.
{Nevertheless, this path is explored in \cite{Gallistl16}, where a mixed formulation is also presented}.

\subsubsection{$C^0$ finite element approximations}
We now discuss finite element methods 
for problem \eqref{nonDiv} that use continuous
basis functions, commonly used for second--order problems
in divergence form.   To this end, with $\mct$ given in the previous section,
we define 
the Lagrange finite element space 
\begin{align}\label{eqn:LagrangeSpaceK}
\cg=\{v_h\in H^1_0(\Omega):\ v_h\big|_T\in \mathbb{P}_k(T)\ \forall T\in \mct\}
\end{align}
with $k\in \bbN$.
Note that functions in $\cg$ are locally smooth, yet
not globally $H^2(\Omega)$, and therefore second-order derivatives
are only defined piecewise with respect to $\mct$.
To simplify the presentation, we shall write
\begin{align*}
\|v\|_{L^2(\mct)}:=\Big(\sum_{K\in \mct} \|v\|_{L^2(K)}^2\Big)^{1/2}
\end{align*}
for piecewise $L^2$ functions.

Since $\cg\not \subset H^2(\Omega)\cap H^1_0(\Omega)$, the bilinear form $a(\cdot,\cdot)$
defined by \eqref{CordesaDef} is not well--defined on $\cg\times \cg$.
Moreover, a piecewise version of the bilinear form
is not generally coercive on $\cg$, since, e.g.,
\begin{align*}
\sum_{K\in \mct} \int_K  \gamma \mathcal{L} v_h\Delta v_h=0
\end{align*} 
for all piecewise linear $v_h\in \cg$.
This stems from the fact that a piecewise version
of the Miranda--Talenti estimate $\|D^2 v\|_{L^2(\Omega)}\le \|\Delta v\|_{L^2(\Omega)}$
is not satisfied on $\cg$.  As such, the coercivity proof
of $a(\cdot,\cdot)$ found at the continuous level does not directly
carry over to the discrete setting.

To overcome this we develop a discrete Miranda--Talenti estimate
suitable for piecewise polynomials.  To do so, we introduce
some notation. Denote by $\mcf^I$
and $\mcf^B$ the set of interior 
and boundary edges/faces, respectively,
and set $\mcf:=\mcf^I\cup \mcf^B$
and $h_F:={\rm diam}(F)$ for $F\in \mcf$.
We recall that 
the jump
of a vector-valued function ${\bw}$ is given by \eqref{eqn:JumpDef}.

A discrete Miranda--Talenti estimate is based on the
following result, whose proof can be found in \cite{copyBrenner}.
\begin{lem}[enrichment operator]\label{lem:EhProp}
Suppose that $d=2$.
Then there exists a finite dimensional
space $\co\subset H^2(\Omega)\cap H^1_0(\Omega)$
and an (enrichment) operator $E_h:\cg\to \co$ 
satisfying
\begin{align}
\label{EhProp}
\|D^2(v_h-E_h v_h)\|_{L^2(\mct)}\le C \Big(\sum_{F\in \mcf^I} h_F^{-1}\big\|\jump{D v_h}\big\|_{L^2(F)}^2\Big)^{1/2}
\end{align}
for all $v_h\in \cg$.
Here, the constant $C>0$ depends on 
the polynomial degree $k$ and the shape--regularity of $\mct$, but is independent of $h$.
\end{lem}

\begin{thm}[discrete Miranda-Talenti estimate]\label{thm:discreteMT}
Suppose that $\Omega\subset \bbR^2$ is a convex polygon.
Then there holds for all $v_h\in \cg$,
\begin{align}
\label{eqn:discreteMT}
\|D^2 v_h\|_{L^2(\mct)}\le \|\Delta v_h\|_{L^2(\mct)}+ C\Big(\sum_{F\in \mcf^I} h_F^{-1}\big\|\jump{D v_h}\big\|_{L^2(F)}^2\Big)^{1/2}.
\end{align}
\end{thm}
\begin{proof}
Fix $v_h\in \cg$ and let $E_h:\cg\to \co$ be the enrichment operator
of Lemma~\ref{lem:EhProp}.  Since $E_h v_h\in H^2(\Omega)\cap H^1_0(\Omega)$
and $\Omega$ is convex
the Miranda--Talenti estimate $\|D^2 E_h v_h\|_{L^2(\Omega)}\le \|\Delta E_h v_h\|_{L^2(\Omega)} = \|\Delta E_h v_h\|_{L^2(\mct)}$
is satisfied.
Applying the triangle inequality, the inequality $\|\Delta E_h v_h\|_{L^2(\Omega)}\le \sqrt{2}\|D^2 E_h v_h\|_{L^2(\Omega)}$,
and estimate \eqref{EhProp} yields
\begin{align*}
\|D^2 v_h\|_{L^2(\mct)}
&\le \|D^2 E_h v_h\|_{L^2(\Omega)}+\|D^2(v_h-E_h v_h)\|_{L^2(\mct)}\\
&\le \|\Delta E_h v_h\|_{L^2(\Omega)} + \|D^2(v_h-E_h v_h)\|_{L^2(\mct)}\\
&\le \|\Delta v_h\|_{L^2(\mct)} +C\Big(\sum_{F\in \mcf^I} h_F^{-1} \big\|\jump{D v_h}\big\|_{L^2(F)}^2\Big)^{1/2},
\end{align*}
where, in the last step, we used that 
\begin{align*}
  \|\Delta E_h v_h\|_{L^2(\Omega)} &\le \|\Delta v_h\|_{L^2(\mct)} +\|\Delta (E_h v_h-v_h)\|_{L^2(\mct)} \\
    &\le \|\Delta v_h\|_{L^2(\mct)} + \sqrt{2}\|D^2(E_h v_h-v_h)\|_{L^2(\mct)}.
\end{align*}
This concludes the proof.
\end{proof}

Motivated by Theorem \ref{thm:discreteMT} we define the bilinear form
\begin{align}\label{eqn:ahDef}
a_h(v,w):=\sum_{K\in \mct} \int_K \gamma \mathcal{L} v\Delta w + \sum_{F\in \mcf^I} \mu h_F^{-1} \int_F \jump{D v} \jump{D w},
\end{align}
where $\mu$ is a positive penalty parameter.  We also define the discrete $H^2$--type norm
\begin{align}\label{eqn:DiscreteH2Norm}
\|v\|_{H^2_h(\Omega)}^2:=\|\Delta v\|_{L^2(\mct)}^2 + \sum_{F\in \mcf^I}  h_F^{-1} \big\|\jump{D v}\big\|_{L^2(F)}^2.
\end{align}
We consider the finite element method: Find $u_h\in \cg$
satisfying
\begin{align}\label{eqn:C0CordesMethod}
a_h(u_h,v_h) = \sum_{K\in \mct} \int_K \gamma f \Delta v_h\qquad \forall v_h\in \cg.
\end{align}
Note that the additional
penalization term does
not affect the consistency
of the scheme; i.e., 
there holds $a_h(u,v_h) = \sum_{K\in \mct} \int_K \gamma f \Delta v_h$
for all $v_h\in \cg$.
The role of this term is to weakly enforce $H^2$--regularity
and to ensure   that the bilinear form
$a_h(\cdot,\cdot)$ is coercive provided
$\mu$ is sufficiently large. 
\begin{lem}[coercivity]\label{lem:C0CordesCoercive}
Suppose that $\Omega\subset \bbR^2$
is convex and that $A\in L^\infty(\Omega,\polS^2)$
satisfies the Cordes condition \eqref{eq:CordesCond}
with parameter $\eps$.
There exists $\mu_*>0$ depending
on the shape--regularity of the mesh,
 polynomial degree $k$, and the parameter $\eps$  such that
for $\mu\ge \mu_*$, there holds
\begin{align*}
 \frac{\alpha}2 \|v_h\|_{H^2_h(\Omega)}^2\le a_h(v_h,v_h)\qquad \forall v_h\in \cg,
\end{align*}
where $\alpha = 1-\sqrt{1-\epsilon}$ is the coercivity constant of $a(\cdot,\cdot)$.
\end{lem}
\begin{proof}
Adding a subtracting $\Delta v_h$
and applying the Cordes condition yields
\begin{align*}
a_h(v_h,v_h) 
&= \sum_{K\in \mct} \int_K \gamma \mathcal{L} v_h \Delta v_h + \mu \sum_{F\in \mcf^I} h_F^{-1} \big\|\jump{D v_h}\big\|_{L^2(F)}^2\\
& = \|\Delta v_h \|_{L^2(\mct)}^2 + \sum_{K\in \mct}  \int_K \big(\gamma \mathcal{L} v_h -\Delta v_h\big)\Delta v_h\\
&+ \mu \sum_{F\in \mcf^I} h_F^{-1} \big\|\jump{D v_h}\big\|_{L^2(F)}^2\\
&\ge \|\Delta v_h\|_{L^2(\mct)}^2 - \sqrt{1-\epsilon} \|\Delta v_h\|_{L^2(\mct)} \|D^2 v_h\|_{L^2(\mct)} \\
 &+ \mu \sum_{F\in \mcf^I} h_F^{-1} \big\|\jump{D v_h}\big\|_{L^2(F)}^2.
\end{align*}
Applying the discrete Miranda-Talenti estimate and the Cauchy-Schwarz inequality then gets, 
for any $\tau>0$,
\begin{align*}
a_h(v_h,v_h) &\ge \big(\alpha- \frac{\tau}2\sqrt{1-\eps}\big) \|\Delta v_h\|_{L^2(\mct)}^2  \\
&+\big(\mu - \frac{1}{2\tau} \sqrt{1-\epsilon}\big) \sum_{F\in \mcf^I} h_F^{-1} \big\|\jump{D v_h}\big\|_{L^2(F)}^2.
\end{align*}
Taking $\tau  = \alpha/\sqrt{1-\epsilon} = 1/\sqrt{1-\epsilon} -1$
and $\mu_* = \alpha/2 + \sqrt{1-\epsilon}(1-1/\alpha)$ yields
$\frac{\alpha}2 \|v_h\|_{H^2_h(\Omega)}^2 \le a_h(v_h,v_h).$
\end{proof}
The coercivity stated in Lemma \ref{lem:C0CordesCoercive}
shows that there exists a unique
solution to the finite element method \eqref{eqn:C0CordesMethod}.
Combined with the consistency
of the scheme, we immediately
obtain quasi-optimal error estimates
in the discrete $H^2$--norm.
\begin{thm}[existence and error estimates]\label{thm:CordesC0}
Suppose that $\Omega\subset \bbR^2$
is convex and that $A\in L^\infty(\Omega,\polS^2)$
satisfies the Cordes condition \eqref{eq:CordesCond}
with parameter $\eps$.
Suppose that
$\mu\ge \mu_*$, and let
$u_h\in \cg$ be the unique solution
to \eqref{eqn:C0CordesMethod}
with $\cg$ given by \eqref{eqn:LagrangeSpaceK}.
Then if the solution to \eqref{nonDiv}
satisfies $u\in H^s(\Omega)$ with $2\le s\le k+1$, 
there holds
\begin{align*}
\alpha \|u-u_h\|_{H^2_h(\Omega)}\le C \inf_{v_h\in \cg} \|u-v_h\|_{H^2_h(\Omega)}\le C h^{s-2}\|u\|_{H^s(\Omega)},
\end{align*}
where $\alpha = 1-\sqrt{1-\eps}$.
\end{thm}

\begin{rem}[linear case]
Note that in the piecewise linear case $(k=1)$,
Theorem \ref{thm:CordesC0} does not give
a convergence result.   In fact, it is 
easy to see that in this case the solution
to \eqref{eqn:C0CordesMethod} is the trivial one $u_h\equiv 0$.
\end{rem}

\begin{rem}[three dimensions]
The results  in  Lemma \ref{lem:C0CordesCoercive}
and Theorem \ref{thm:CordesC0} 
are restricted to the two-dimensional case
due to Lemma \ref{lem:EhProp}.
If there exists an enrichment operator
satisfying \eqref{EhProp} with $d=3$,
then these results carry over to the three
dimensional case.
\end{rem}

\subsubsection{Discontinuous Galerkin approximations}

In this section we summarize the discretization developed and analyzed
in in \cite{SmearsSuli13}, where a consistent 
discontinuous Galerkin (DG) method is constructed.
Instead of developing a discrete Miranda-Talenti via penalization,
the key idea of this approach
is to add auxiliary terms in the bilinear form to 
bypass the Miranda-Talenti estimate found at the continuous level.

Define, for $k \in \polN$, the piecewise polynomial space without
continuity
\begin{align}\label{eqn:DGSpace}
\dg=\{v\in L^2(\Omega):\ v|_T\in \mathbb{P}_k(T)\ \forall T\in \mct\}.
\end{align}
 We note
that the method in \cite{SmearsSuli13}
considers discretizations in an $hp$-framework
where the polynomial degree is element-dependent.
For simplicity, and to ease the presentation,
 we consider here only the $h$-version of the method,
where the polynomial degree is globally fixed and 
we do not trace the dependence of the constants on the polynomial degree $k$.

To motivate the method,
we again emphasize
that a Miranda-Talenti estimate fails to hold
for piecewise polynomials, and as such,
the coercivity proof found at the continuous
levels fails in the discrete setting.
Indeed, mimicking  the calculations  in
Lemma \ref{lem:propofa} element-wise over $\dg$ leads to
\begin{equation}
\label{coerciveProofFail}
  \begin{aligned}
    \sum_{K\in \mct}  \int_K \gamma \mathcal{L} v_h\Delta v_h
    & = \sum_{K\in \mct} \int_K \Delta v_h \Delta v_h \\
    &+ \sum_{K\in \mct} \int_K \big(\gamma \mathcal{L} v_h -\Delta v_h \big) \Delta v_h.
  \end{aligned}
\end{equation}
Applying the Cordes condition and the Cauchy-Schwarz inequality yields
the inequality
\begin{align*}
\sum_{K\in \mct} \int_K \gamma \mathcal{L} v_h\Delta_h v
&\nonumber \ge \|\Delta v_h\|_{L^2(\mct)}^2 - \sqrt{1-\epsilon} \|D^2 v_h\|_{L^2(\mct)}\|\Delta v_h\|_{L^2(\mct)}.
\end{align*}
Since the piecewise Hessian matrix of $v_h$
cannot be controlled by its piecewise Laplacian (e.g., if $v_h$ is piecewise harmonic),
one concludes that the bilinear mapping $(v_h,w_h)\to \sum_{K\in \mct} \int_K \gamma \mathcal{L} v_h \Delta w_h$
is not coercive over $\dg\times \dg$ in general.

The essential idea presented in \cite{SmearsSuli13}
is to replace the bilinear form $(v_h,w_h)\to \sum_{K\in \mct} \int_K \Delta v_h\Delta w_h$
implicit in the right-hand side of \eqref{coerciveProofFail}
with a consistent bilinear form that is coercive with
a discrete $H^2$-type norm.   

For a face $F\in \mathcal{F}_h$, let $\{t_i\}_{i=1}^{d-1}$
be an orthonormal coordinate system, and define the tangental
gradient, tangental divergence, and tangental Laplacian, respectively, as
\begin{align*}
D_T v = \sum_{i=1}^{d-1} t_i \frac{\p v}{\p t_i},\quad D_T\cdot \bw = 
\sum_{i=1}^{d-1} \frac{\p w_i}{\p t_i},\quad
\Delta_T v = D_T \cdot D_T v.
\end{align*}
We also define the average of a scalar or vector--valued function as
\begin{align}\label{eqn:AverageDef}
\avg{v}|_F:=
\left\{
\begin{array}{ll}
\frac12 \big(v_++v_-\big) & \text{if }F = \p K^+\cap \p K^-\in \mcf^I,\\
v_+ & \text{if }F = \p K^+\cap \p \Omega\in \mcf^B.
\end{array}
\right.
\end{align}
Then define the bilinear form 
\begin{align}\label{eqn:BhDefinition}
&B_h(u_h,v_h)
=
\sum_{K\in \mct}  \int_K \Big(D^2 u_h:D^2 v_h+\Delta u_h\Delta v_h\Big)\\
&\nonumber \qquad + \sum_{F\in \mcf^I} \int_F \Big( 
 \avg{\Delta_T u_h}
\jump{D v_h}+ \avg{\Delta_T v_h}
\jump{D u_h}\Big)\\
&\nonumber\qquad \qquad - \sum_{F\in \mcf} \int_F \Big(D_T  \avg{D u_h\cdot n}\cdot \jump{D_T v_h}_T
+D_T  \avg{D v_h\cdot n} \cdot\jump{D_T u_h}_T\Big)\\
&\nonumber \qquad\qquad\qquad 
+\sum_{F\in \mcf^I} \mu  h_F^{-1} \int_F  \jump{D u_h}  \jump{D v_h}
+\sum_{F\in \mcf} \mu h_F^{-3}\int_F \jump{u_h}\cdot \jump{v_h},
%
\end{align}
where $\jump{\bv}_T|_F:=\bv_+-\bv_-$ on $\mcf^I$
and $\jump{\bv}_T|_F:=\bv_+$ on $\mcf^B$,
and $\mu>0$ is a penalization parameter.

Let us define, for $\theta \in [0,1]$, the discrete $DG$-norm
\begin{equation}
\label{eqn:DGTNorm}
  \begin{aligned}
    \|v\|_{DG(\theta)}^2 &= (1-\theta)\|\Delta v\|_{L^2(\mct)}^2 + \theta \|D^2 v\|_{L^2(\mct)}^2\\
    & +c_*\Big(\sum_{F\in \mcf^I} h_F^{-1} \big\|\jump{Dv}\big\|_{L^2(F)}^2 +\sum_{F\in \mcf} h_F^{-3} \big\|\jump{v}\big\|_{L^2(F)}^2\Big).
  \end{aligned}
\end{equation}
The seemingly abstruse bilinear form $B_h(\cdot,\cdot)$ is carefully 
defined to satisfy the following  properties  \cite[Lemma 5 and Lemma 7]{SmearsSuli13}.

\begin{lem}[properties of $B_h$]
\label{lem:propBh}
The bilinear form $B_h : \dg \times \dg \to \Real$,  defined in \eqref{eqn:BhDefinition}, satisfies the following properties:
\begin{enumerate}[$\bullet$]
  \item Consistency. If $u\in H^s(\Omega)\cap H^1_0(\Omega)$ for some $s> 5/2$, then
  \[
    B_h(u,v_h) = 2  \sum_{K\in \mct} \int_K \Delta u\Delta v_h
  \]
  for all $v_h\in \dg$.
  
  \item Coercivity. For any $\kappa>1$, there exists a $\mu_* = C\kappa/(\kappa-1)$ with $C>0$ depending only on the shape regularity of $\mct$ and $k$ such that for $\mu\ge \mu_*$, 
  \begin{align}
  \label{Bhcoercive}
    2 \|v_h\|_{DG(1/2)}^2 \le {\kappa} B_h(v_h,v_h)\qquad \forall v_h\in \dg,
  \end{align}
  for some constant $c_*>0$ independent of the discretization parameter $h$ and polynomial degree $k$.
\end{enumerate}
\end{lem}

The previously shown properties of $B_h(\cdot,\cdot)$ allow us to define 
the following DG method: Find $u_h\in \dg$ such that, for all $v_h\in \dg$,
\begin{equation}
\label{CordesMethod}
  \begin{aligned}
    a^{DG}_h(u_h,v_h) &:=\sum_{K\in \mct} \int_K \gamma \big(\calL u_h-\Delta u_h\big)\Delta v_h +\frac12 B_h(u_h,v_h) \\
    &=  \sum_{K\in \mct} \int_K \gamma f \Delta v_h.
  \end{aligned}
\end{equation}

Due to the consistency of $B_h(\cdot,\cdot)$ we see that the scheme is consistent provided the exact solution is sufficiently smooth:  If $u$ is the 
solution to \eqref{nonDiv} and satisfies $u\in H^s(\Omega)$ for some $s> 5/2$, then $a^{DG}_h(u,v_h) = \int_\Omega \gamma f \Delta_h v_h$
for all $v_h\in \dg$.
In addition, the coercivity of $B_h(\cdot,\cdot)$ implies the coercivity of $a^{DG}_h(\cdot,\cdot)$; see \cite[Theorem 8]{SmearsSuli13}.
\begin{lem}[coercivity]\label{lem:SSCordesCoercive}
Suppose that $\Omega\subset \bbR^d$
is convex and that $A$ satisfies 
the Cordes condition \eqref{eq:CordesCond}.  Then there exists $\mu_* = \mathcal{O}(\eps^{-1})>0$
such that
\begin{align*}
C\|v_h\|_{DG(1)}^2\le a_h^{DG}(v_h,v_h)\qquad \forall v_h\in \dg.
\end{align*}
Consequently, there exists a unique solution $u_h\in \dg$ to \eqref{CordesMethod}.
\end{lem}
Combined with consistency 
of $a^{DG}_h(\cdot,\cdot)$, Lemma~\ref{lem:SSCordesCoercive}
implies  quasi-optimal error estimates in the discrete $H^2$-type norm
\cite[Theorem 9]{SmearsSuli13}.

\begin{thm}[existence and error estimates]
Suppose that the hypotheses of Lemma \ref{lem:SSCordesCoercive}
hold.  In addition suppose that the solution to \eqref{nonDiv}
satisfies $u\in H^s(\Omega)$ for some $5/2<s\le k+1$.  Then
there exists an $h$-independent constant $C>0$ such that
\begin{align*}
\|u-u_h\|_{DG(1)}\le C h^{s-2} \|u\|_{H^s(\Omega)}.
\end{align*}
\end{thm}
\begin{rem}[regularity]
The regularity assumption $u\in H^s(\Omega)$ with $s>5/2$
ensures that $a_h^{DG}(u,v_h)$ is well-defined.
\end{rem}

\begin{rem}[extensions]
A primal dual Discontinuous Galerkin method
for second order elliptic equations in nondivergence
form has recently been proposed and analyzed
in \cite{WangWang16}.
\end{rem}

\subsection{Discrete finite element Calder\'on-Zygmund estimates}

In this section we describe
finite element discretizations to nondivergence form
elliptic operators based on the notion and theory 
of strong solutions. Recall from Definition \ref{def:strongsol}
that a function $u\in W^{2,p}(\Omega)$ is a strong
solution to \eqref{nonDiv} if the equation
and boundary conditions hold almost everywhere
in $\bar\Omega$. Such solutions exist
provided the data is sufficiently regular
and $A\in C(\bar\Omega,\polS^d)$; 
see Theorem \ref{thm:CZ}.  This 
result is obtained by using the
Calder\'on-Zygmund decomposition technique
in the case $A = I$, and then extended to general
$\mathcal{L}$ using the continuity of the coefficient matrix.
In this section, we develop a discrete version of this theory
to develop a priori estimates and convergence results
of finite element solutions.  Let us 
first present the derivation of the method.

Assume for the moment 
that the coefficient matrix $A$ in \eqref{nonDiv}
is sufficiently smooth.  Then, as explained in 
Example \ref{ex:divisnondiv}, 
we can write problem \eqref{nonDiv}
in divergence form:
\begin{align*}
D\cdot (A D u) -(D\cdot A)\cdot Du =f\quad \text{in }\Omega.
\end{align*}
A standard finite element method
for this problem (without stabilization)
reads: Find $u_h\in \cg$ such that
\begin{align}
\label{eqn:BadFEM}
-\int_\Omega \big(AD u_h\big)\cdot D v_h -\int_\Omega \big((D\cdot A)\cdot Du_h\big)v_h = \int_\Omega f v_h\quad \forall v_h\in \cg,
\end{align}
where $\cg\subset H^1_0(\Omega)$ is the Lagrange
finite element space of degree $k\ge 1$ defined 
by \eqref{eqn:LagrangeSpace}.  It is well-known that,
for $h$ sufficiently small, there exists
a unique solution to \eqref{eqn:BadFEM}.

If $A$ is not sufficiently smooth and/or if the locations 
of the singularities are complex/unknown, then the classical finite 
element method \eqref{eqn:BadFEM}
is not viable due to the differential operators acting on $A$.
However, we easily circumvent this issue by 
using the integration by parts identity
\begin{align*}
\int_\Omega \boldsymbol{\tau} \cdot D v_h  = -\sum_{K\in \mct} \int_K (D \cdot \boldsymbol{\tau})v_h
+ \sum_{F\in \mcf^I} \int_F\jump{\boldsymbol{\tau}}{v_h}, 
\end{align*}
which holds for all piecewise smooth  $\boldsymbol{\tau}$ and  $v_h\in \cg$.  
Taking $\boldsymbol{\tau} = A D u_h$ and applying the product rule yields
\begin{align*}
-\int_\Omega \big(A D u_h)\cdot Dv_h &= \sum_{K\in \mct} \int_K (A:D^2 u_h)v_h\\
&\qquad +\int_\Omega \big((D\cdot A)\cdot D u_h\big) v_h + \sum_{F\in \mcf^I} \int_F \jump{AD u_h}v_h.
\end{align*}
Substituting this identity into the (ill-posed) formulation leads to the finite element method
\begin{align}\label{C0NonDivMethod}
b_h(u_h,v_h):=\sum_{K\in \mct} \int_K (A:D^2 u_h)v_h - \sum_{F\in \mcf^I} \int_F \jump{A D u_h}v_h = \int_\Omega fv_h
\end{align}
for all $v_h\in \cg$.
In contrast to \eqref{eqn:BadFEM}, the formulation \eqref{C0NonDivMethod}
is well--defined for non--differentiable $A$.  Furthermore, by
reversing the arguments,
we see that \eqref{C0NonDivMethod} is equivalent to \eqref{eqn:BadFEM}
if $A$ is sufficiently smooth.  In particular, in the case that $A(x)\equiv \bar{A}$ is a constant 
SPD matrix,
the method \eqref{C0NonDivMethod} reduces to the (well--posed) problem
\begin{align}\label{eqn:ConstantBilinear}
b_{h,0}(u_h,v_h) := -\int_\Omega \big(\bar{A} D u_h)\cdot D v_h = \int_\Omega fv_h\quad \forall v_h\in \cg.
\end{align}
We point out that method \eqref{C0NonDivMethod} is consistent  
and meaningful in the piecewise linear case $(k=1)$.

While the derivation of the finite element method \eqref{C0NonDivMethod}
is relatively simple, a stability and convergence of the method is
less obvious.  The key difficulty is that integration by parts
is not at our disposal, and it is unclear whether a clever choice
of test function will render a coercivity or inf--sup condition.
Rather, the stability analysis of \eqref{C0NonDivMethod}
mimics the techniques found in the PDE theory,
where Calder\'on-Zygmund estimates are the essential tools.

To describe the stability and convergence
theory, we first define a discrete $W^{2,p}$-type norm:
\begin{align*}
\|v\|_{W^{2,p}_h(\Omega)}^p :=\sum_{K\in \mct} \|D^2 v \|^p_{L^p(K)} + \sum_{F\in \mcf^I} h_F^{1-p} \big\|\jump{D v}\big\|^p_{L^p(F)},\quad (1<p<\infty).
\end{align*}

A discrete Calder\'on-Zygmund-type estimate
with respect to this norm
in the case of constant coefficients
is now given.

\begin{lem}[discrete Calder\'on-Zygmund estimate]\label{lem:DCZE}
Let $L$ be
the elliptic, divergence form operator \eqref{eqn:LOperatorDef},
where the coefficient matrix $A$ is 
constant and SPD.  Suppose
that the a priori estimate
$C\|w\|_{W^{2,p}(\Omega)}\le \|L w\|_{L^p(\Omega)}$
is satisfied for all $w\in W^{2,p}(\Omega)\cap W^{1,p}_0(\Omega)$.
Let $b_{h,0}(\cdot,\cdot)$ be defined by
\eqref{eqn:ConstantBilinear}.  Then, 
for $h$ sufficiently small, there holds%
\begin{align*}
C \|v_h\|_{W^{2,p}_h(\Omega)}\le \sup_{w_h\in \cg\backslash \{0\}} \frac{b_{h,0}(v_h,w_h)}{\|w_h\|_{L^{p^\prime}(\Omega)}}\qquad \forall v_h\in \cg,
\end{align*}
where $1/p+1/p^\prime=1$.
\end{lem}
\begin{proof}
We give a proof of the simpler case $p=2$ and refer the reader to \cite[Lemma 2.6]{FengNeilan16} and \cite[Lemma 4.1]{Neilan13} for  general $p\in (1,\infty)$.

First, let $\mathcal{B}_h(v_h)\in \cg$ be the unique solution to the
problem
\begin{align*}
\int_\Omega \mathcal{B}_h(v_h) w_h\, dx = b_{h,0}(v_h,w_h)\qquad \forall w_h\in \cg,
\end{align*}
and  let $\varphi\in H^1_0(\Omega)$
be the unique (weak) solution to $L \varphi = -\mathcal{B}_h(v_h)$ in $\Omega$.
Then, for $w_h\in \cg$, we find
\begin{align*}
b_{h,0}(v_h,w_h) = \int_\Omega \mathcal{B}_h(v_h)w_h  = -\int_\Omega \big(A D\varphi\big)\cdot D w_h = b_{h,0}(\varphi,w_h).
\end{align*}
Thus, $v_h$ is the elliptic projection of $\varphi$ with respect to 
$b_{h,0}(\cdot,\cdot)$.  Therefore, Cea's Lemma
and the hypothesis $\varphi\in H^2(\Omega)$ with $\|\varphi\|_{H^2(\Omega)}\le C \|\mathcal{B}_h(v_h)\|_{L^2(\Omega)}$
yield
\begin{align}\label{eqn:varphivhDiff}
\|\varphi-v_h\|_{H^1(\Omega)}\le C h\|\varphi\|_{H^2(\Omega)}\le Ch \|\mathcal{B}_h(v_h)\|_{L^2(\Omega)}.
\end{align}
Next, for any $\varphi_h\in \cg$, the triangle inequality and a scaling argument show that
\begin{align*}
&\|v_h\|_{H^2_h(\Omega)}
\le \|v_h-\varphi_h\|_{H^2_h(\Omega)}+\|\varphi_h-\varphi\|_{H^2_h(\Omega)}+\|\varphi\|_{H^2_h(\Omega)}\\
&\le Ch^{-1} \big(\|v_h-\varphi\|_{H^1(\Omega)}+\|\varphi-\varphi_h\|_{H^1(\Omega)}\big)+\|\varphi_h-\varphi\|_{H^2_h(\Omega)}+\|\varphi\|_{H^2(\Omega)}.
\end{align*}
By taking $\varphi_h$ to be the nodal interpolant
of $\varphi$ and applying \eqref{eqn:varphivhDiff} and
the definition of $\mathcal{B}_h(v_h)$, we obtain
\begin{align*}
\|v_h\|_{H^2_h(\Omega)}
&\le C\big[ h^{-1} \|v_h-\varphi\|_{H^1(\Omega)}+ \|\varphi\|_{H^2(\Omega)}\big]\le C  \|\mathcal{B}_h(v_h)\|_{L^2(\Omega)}\\
& = C \sup_{w_h\in \cg\backslash \{0\}} \frac{\int_\Omega \mathcal{B}_h(v_h) w_h}{\|w_h\|_{L^2(\Omega)}}
= C\sup_{w_h\in \cg\backslash \{0\}}  \frac{b_{h,0}(v_h,w_h)}{\|w_h\|_{L^2(\Omega)}}.
\end{align*}
\end{proof}

\begin{rem}[$\boldsymbol{p=2}$]
Notice that, in the case that $\Omega$ is convex, the assumptions of Lemma~\ref{lem:DCZE} hold for $p=2$.
\end{rem}

A corollary of this result 
is a local stability estimate of the discrete
adjoint problem.

\begin{col}[local stability]\label{col:LocalCZStability}
For a domain $D\subset \Omega$, 
let $\rho_D$ denote the radius
of the largest ball inscribed in $D$,
and let $\cg(D)$ denote the 
set of functions in $\cg$
that vanish outside $D$.
Suppose that the hypothesis of Lemma \ref{lem:DCZE}
are satisfied.
Then, if $h$ and $\rho_D$ are sufficiently small,
\begin{align*}
C \|w_h\|_{L^{p'}(D)}\le \sup_{v_h\in \cg(D_h)\backslash \{0\}} \frac{b_h(v_h,w_h)}{\|v_h\|_{W^{2,p}_h(D_h)}}\qquad \forall w_h\in \cg(D).
\end{align*}
with $D_h = \{x\in \Omega:\ {\rm dist}(x,D)\le h\}$.
\end{col}
\begin{proof}
Again, we prove the case $p=p' =2$
and refer the reader to \cite[Appendix B]{FengNeilan16}
for the general result.

Let $w_h\in \cg(D)$, and let $\varphi_h\in \cg$
satisfy 
\[
b_{h,0}(\varphi_h,v_h) = \int_\Omega w_h v_h,\quad \forall v_h\in \cg,
\]
where $b_{h,0}(\cdot,\cdot)$ is defined
by \eqref{eqn:ConstantBilinear} with
\begin{align*}
\bar{A} = \frac{1}{|D|} \int_D A.
\end{align*}

Taking $v_h = w_h$ in the method yields
\begin{align*}
\|w_h\|_{L^2(D)}^2
&= b_{h,0}(\varphi_h,w_h) = b_{h}(\varphi_h,w_h)+\big(b_{h,0}(\varphi_h,w_h)-b_h(\varphi_h,w_h)\big).
\end{align*}
Using the (uniform) continuity of $A$,
for any $\tau>0$, we 
have 
\begin{align*}
\big|b_{h,0}(\varphi_h,w_h)-b_h(\varphi_h,w_h)\big|\le \tau \|\varphi_h\|_{H^2_h(\Omega)}\|w_h\|_{L^2(D)}
\end{align*}
provided $\rho_D$ is sufficiently small.  

Applying the estimate $\|\varphi_h\|_{H^2_h(\Omega)}\le C\|w_h\|_{L^2(\Omega)} = C \|w_h\|_{L^2(D)}$
established in Theorem \ref{lem:DCZE} we obtain
\begin{align*}
(1-C \tau)\|w_h\|_{L^2(D)}^2
&\le b_h(\varphi_h,w_h) =\Big( \frac{b_h(\varphi_h,w_h)}{\|\varphi_h\|_{H^2_h(D_h)}}\Big) \|\varphi\|_{H^2_h(\Omega)}\\
&\le C \Big(\sup_{v_h\in \cg(D_h)\backslash \{0\}} \frac{b_h(v_h,w_h)}{\|v_h\|_{H^2_h(D_h)}}\Big)\|w_h\|_{L^2(\Omega)}.
\end{align*}
Taking $\tau$ sufficiently small and manipulating terms in the last inequality yields the result.
\end{proof}

The local stability result for the discrete adjoint problem
given in Corollary \ref{col:LocalCZStability} 
leads to a stability
estimate for method \eqref{C0NonDivMethod}.
\begin{thm}[global stability]\label{thm:C0CZStability}
Suppose that $A\in C(\bar\Omega,\polS^d)$
and that elliptic
and  divergence form operators
with constant coefficients inherit $W^{2,p}$-regularity
($1<p<\infty)$.
Then there exists $h_*>0$ depending on the modulus of continuity
of $A$ and $p$ such that for $h\le h_*$, there holds
\begin{align}\label{eqn:C0CZStab}
C \|v_h\|_{W^{2,p}_h(\Omega)}\le \sup_{w_h\in \cg\backslash \{0\}} \frac{b_{h}(v_h,w_h)}{\|w_h\|_{L^{p^\prime}(\Omega)}}\quad \forall v_h\in \cg.
\end{align}
\end{thm}
\begin{proof}
We outline the main steps of the proof 
and refer to \cite{FengNeilan16} for details.

Combining Corollary \ref{col:LocalCZStability} with
cut-off functions techniques
and a covering argument leads to 
the G\"arding-type inequality
\begin{align*}
C \|w_h\|_{L^{p'}(\Omega)}\le \sup_{v_h\in \cg\backslash \{0\}} \frac{b_h(v_h,w_h)}{\|v_h\|_{W^{2,p}_h(\Omega)}}+ \|w_h\|_{W^{-1,p'}(\Omega)}\qquad \forall w_h\in \cg.
\end{align*}
A standard duality argument then shows that, for $h$ sufficiently small,
\begin{align*}
C \|w_h\|_{L^{p'}(\Omega)}\le \sup_{v_h\in \cg\backslash \{0\}} \frac{b_h(v_h,w_h)}{\|v_h\|_{W^{2,p}_h(\Omega)}}\quad \forall w_h\in \cg.
\end{align*}
This estimate shows that, for fixed $v_h\in \cg$,
there exists a unique $w_h\in \cg$ satisfying
\begin{align}\label{eqn:MJNDualityArg}
b_h(z_h,w_h) &= \sum_{K\in \mct} \int_K |D^2 v_h|^{p-2} D^2 v_h :D^2 z_h\\
&\nonumber\qquad + \sum_{F\in \mcf} h_F^{-1} \int_F |\jump{Dv_h}|^{p-2} \jump{Dv_h} \jump{Dz_h}\quad \forall z_h\in \cg.
\end{align}
Applying the global stability
estimate for the adjoint problem and H\"older's inequality we obtain
\begin{align*}
C\|w_h\|_{L^{p'}(\Omega)}\le 
 \sup_{z_h\in \cg\backslash \{0\}} \frac{b_h(z_h,w_h)}{\|z_h\|_{W^{2,p}_h(\Omega)}}\le C \|v_h\|_{W^{2,p}_h(\Omega)}^{p-1}.
  \end{align*}
  On the other hand, setting $z_h = v_h$ in \eqref{eqn:MJNDualityArg} 
  yields
  \begin{align*}
  \|v_h\|_{W^{2,p}_h(\Omega)}^p 
  &= b_h(v_h,w_h)\le \Big(\sup_{z_h\in \cg\backslash \{0\}} \frac{b_h(v_h,z_h)}{\|z_h\|_{L^{p'}(\Omega)}}\Big)\|w_h\|_{L^{p'}(\Omega)}\\
  &\le C \Big(\sup_{z_h\in \cg\backslash \{0\}} \frac{b_h(v_h,z_h)}{\|z_h\|_{L^{p'}(\Omega)}}\Big)\|v_h\|_{W^{2,p}_h(\Omega)}^{p-1}.
  \end{align*}
Dividing by $\|v_h\|_{W^{2,p}_h(\Omega)}$ we obtain \eqref{eqn:C0CZStab}.
\end{proof}

\begin{thm}[existence and error estimates]
Suppose that the hypotheses
of Theorem \ref{thm:C0CZStability}
are satisfied.  Then there
exists a unique solution
$u_h\in \cg$ to \eqref{C0NonDivMethod}.
If the solution to \eqref{nonDiv}
satisfies $u\in W^{s,p}(\Omega)$
for $2\le s\le k+1$, then
\begin{align*}
\|u-u_h\|_{W^{2,p}_h(\Omega)}\le C h^{s-2}\|u\|_{W^{s,p}(\Omega)}.
\end{align*}
\end{thm}
\begin{proof}
The existence and uniqueness of
a solution to \eqref{C0NonDivMethod}
follows from the stability estimate 
\eqref{eqn:C0CZStab}.  The error
estimate follows 
from the stability and continuity
of the bilinear form $b_h(\cdot,\cdot)$,
the consistency of the scheme,
and approximation properties
of $\cg$ with respect to the
discrete $W^{2,p}$-norm.
\end{proof}

\begin{rem}[extensions]
The ideas and analysis presented
in this section has been extended
to discontinuous Galerkin methods \cite{FengNeilanSchnake16}
and mixed finite element methods \cite{LakkisPryer11,Neilan16}.
\end{rem}

\subsection{Finite element method based on integro-differential approximation}
\label{sub:Wujunnondiv}

In this section
we consider a two-scale finite element discretization
for problem \eqref{nonDiv} developed in 
\cite{NochettoZhang16} which 
is based on a regularized, integro-differential approximation
proposed in  \cite{CaffarelliSilvestre10}. 
As in the previous sections we assume that 
the PDE operator $\mathcal{L}$ is 
uniformly elliptic, i.e., 
there exists 
strictly positive constants $\lambda,\Lambda$
satisfying 
$\lambda I \le A(x) \le \Lambda(x) I,\ \forall x\in \bar\Omega$.
We further make the simplifying
assumption that $A\in C(\bar\Omega,\polS^d)$,
and make remarks when this regularity can be relaxed.

To explain and motivate the method, 
we first perform some algebraic manipulations and rewrite the PDE as
\begin{align}\label{eqn:addsubLap}
A:D^2 u = \frac{\lambda}{2} \Delta u + A_\lambda^2 :D^2 u,\quad A_\lambda:=\big(A - \frac{\lambda}2 I\big)^{1/2}.
\end{align}
Let $\varphi$ be a
radially symmetric function
with compact support in the unit ball satisfying $\int_{\bbR^d} |z|^2 \varphi(z) = d$.  
We then find
that
$\int_{\bbR^d} z_i z_j \varphi(z)=0$ for $i\neq j$,
and $\int_{\bbR^d} z_i^2 \varphi(z) = 1$.  Consequently, 
we have 
\begin{align*}
\int_{\bbR^d} z\otimes z \varphi(z) = I,
\end{align*}
and therefore
\begin{align*}
\big(A_\lambda(x)\big)^2:D^2 u(x) = A_\lambda(x) \Big(\int_{\bbR^d} z\otimes z \varphi(z)\Big) A_\lambda(x):D^2 u(x).
\end{align*}
For a regularization parameter $\epsilon>0$, we make the change of variables $y =\epsilon A_\lambda(x) z$
in the integral to obtain
\begin{align*}
\big(A_\lambda(x)\big)^2:D^2 u(x) = \int_{\bbR^d} \frac{(y\otimes y):D^2 u(x)}{\epsilon^{d+2} \det(A_\lambda(x))}  \varphi\Big( \frac{A_\lambda^{-1}(x) y}{\epsilon}\Big).
\end{align*}
Set 
\begin{align}\label{eqn:WJNDomains}
Q = \big(\Lambda - \frac{\lambda}2\big)^{1/2},\quad
\Omega_\eps = \{x\in \Omega:\ {\rm dist}(x,\p\Omega)>Q \eps\},\quad
\omega_\eps = \Omega\backslash \Omega_\eps,
\end{align}
and note that $\varphi(A_\lambda^{-1}(x)y/\epsilon)$ has support
in the ball $B_{Q\epsilon}(0)$.  
For $x\in \Omega$, let $\theta = \theta(x)\in (0,1]$ 
be the largest number such that $x\pm \theta y\in \Omega$
for all $y\in B_{Q\epsilon}(0)$.  Recall that the second difference operator is given by
\begin{align*}
\delta^2_{\theta y, \theta} u(x) = \frac{u(x+\theta y) - 2u(x) +u(x+-\theta y)}{\theta^2},
\end{align*}
and note that $\delta^2_{\theta y, \theta} u(x) = (y\otimes y):D^2 u(x)$ if $u$ is a quadratic polynomial,
and that $\theta = 1$ for $x \in \Omega_\epsilon$.

Combining these calculations and identities, we are led to the approximation
\begin{align}\label{eqn:IeDef}
\big(A_\lambda(x)\big)^2:D^2 u(x)
\approx \int_{\bbR^d} \frac{|y|^2 \delta^2_{\theta y, \theta u(x)}}{\epsilon^{d+2} \det(A_\lambda(x))}  \varphi\Big( \frac{A_\lambda^{-1}(x) y}{\epsilon}\Big)=:I_\epsilon u(x).
\end{align}
The approximation
is quantified in the next lemma \cite[Lemma 2.1]{NochettoZhang16}.
\begin{lem}[rate of convergence of integral transform]\label{lem:RateConvIT}
Let $I_\eps $ be the integral operator defined by \eqref{eqn:IeDef},
and 
let $U_{Q\eps}(x):=\bar{B}_{Q \eps}(x)\cap \bar\Omega$.
\begin{enumerate}[$\bullet$]
\item If $u\in C^2(\bar\Omega)$, then $I_\eps u(x) \to \big(A(x)-\frac{\lambda}{2} I\big):D^2 u(x)$
as $\eps \to 0^+$ for all $x\in \Omega$.
%
%
\item If 
$u\in C^{2,\alpha}(U_{Q\eps}(x))$ for some  $\alpha\in (0,1]$,
then
\begin{align*}
\Big|I_\eps u(x) - \big(A(x)-\frac{\lambda}{2} I\big):D^2 u(x)\Big|\le C \|u\|_{C^{2,\alpha}(U_{Q\eps})} \theta^{\alpha} \eps^{\alpha},
\end{align*}
for all $x\in \Omega$.
\end{enumerate}
\end{lem}

We approximate the equation \eqref{nonDiv} by the integro-differential equation
\begin{align}\label{eqn:IDEReg}
\mathcal{L}^\eps u^\eps := \frac{\lambda}2 \Delta u^\eps +I_\epsilon u^\eps = f\quad \text{in }\Omega.
\end{align}
We refer the reader to \cite{CaffarelliSilvestre10} 
for details about the existence, uniqueness, and regularity estimates
of solution $u^\eps$.

We now describe a 
convergent finite element scheme
for the nondivergence form problem \eqref{nonDiv}
based on the regularized
problem \eqref{eqn:IDEReg}.
To this end, we let $\lco= \lc\cap H^1_0(\Omega)$ be the linear, Lagrange
finite element space with vanishing trace.
Let $\phi_i\in \lco$ denote
the normalized hat function
with respect to the interior node $z_i\in \Omega_h^I$,
and let
$\Delta_h$ be the finite element Laplacian
defined by \eqref{eqn:FEMLaplacian}.
We consider the finite element method: Find $u_h\in \lco$
such that
\begin{align}
\label{eqn:NZMethod}
\mathcal{L}_h^\eps u^\eps_h(z_i):=\frac{\lambda}{2}\Delta_h u^\eps_h(z_i) +I_{\epsilon} u^\eps_h(z_i) = f_i:=\int_\Omega f \phi_i \qquad \forall z_i\in \Omega_h^I.
\end{align}
Note that the formulation \eqref{eqn:NZMethod}
is {\it not} obtained by testing \eqref{eqn:IDEReg} with $\phi_i$ 
(which would introduce the term $\int_\Omega I_\eps u_h^\eps \phi_i$).
Instead, mass lumping is used to preserve the monotonicity
of the scheme.

\begin{lem}[monotonicity]
\label{lem:IEMonotone}
Suppose that $v_h,w_h\in \lco$
satisfy $v_h\le w_h$ with
equality at $z\in \Omega_h^I$.
Then $I_\eps v_h(z)\le I_\eps w_h(z)$.
Consequently, 
if $\mct$ satisfies  \eqref{eqn:FEMMMatrix},
then $\mathcal{L}_h^\eps$ is monotone.
\end{lem}
\begin{proof}
From the hypotheses
and the definition of $\delta^2_{\theta y, \theta}$, we have
 $ \delta^2_{\theta y, \theta} v(z) \leq \delta^2_{\theta y, \theta} w(z)$,
and therefore $I_\eps v_h(z_i)\le I_\eps w_h(z_i)$.
The monotonicity
of $\mathcal{L}_h^\eps$ then follows
from Lemma \ref{lem:FEMMmatrix}.
\end{proof}

The monotonicity, along with the Alexandrov-Bakelman-Pucci
estimate for the finite element Laplacian, 
yields the following maximum principle.
\begin{thm}[discrete ABP estimate for $\mathcal{L}_h^\eps$]\label{thm:FEMABP123}
Suppose that $\mct$ satisfies \eqref{eqn:FEMMMatrix}.
Then for $v_h\in \lco$ satisfying
\begin{align*}
\mathcal{L}_h^\eps v_h(z_i)\le f_i\quad \forall z_i\in \Omega_h^I,
\end{align*}
there holds
\begin{align*}
\sup_\Omega v_h^- \le \frac{C}{\lambda} \Big(\sum_{z_i\in \mathcal{C}_h^-(v_h)} |f_i^+|^d|\omega_{z_i}|\Big)^{1/d},
\end{align*}
where 
$\mathcal{C}_h^-(v_h)$ is the nodal contact set 
given in Definition \ref{def:NodalContactSet}.
\end{thm}
\begin{proof}
Let $\Gamma_h(v_h)= \Gamma(v_h)$
be the convex envelope of $v_h$.
Then, for a contact point $z_i\in \mathcal{C}_h^-(v_h)$, 
there holds
\begin{align*}
0\le I_\eps \Gamma(v_h)(z_i)\le I_\eps v_h(z_i),
\end{align*}
where the first inequality follows
from the convexity of $\Gamma(v_h)$
and the second one from the monotonicity
of $I_\eps$ in Lemma \ref{lem:IEMonotone}.
Consequently,
\begin{align*}
\frac{\lambda}2 \Delta_h v_h(z_i)\le \mathcal{L}_h^\eps v_h(z_i)\le  f_i^+ 
\end{align*}
since $f_i\ge 0$ for $z_i\in \mathcal{C}_h^-(v_h)$.  
The result now follows from this inequality and Theorem \ref{thm:FEABP}
(cf. Remark \ref{rem:INeedThisLater}).
\end{proof}

Since the method \eqref{eqn:NZMethod}
is linear, a corollary of the ABP estimate
is the existence and uniqueness
of a solution $u_h$.
\begin{col}[existence and uniqueness]
Suppose that $\mct$
satisfies \eqref{eqn:FEMMMatrix}.  Then there exists
a unique $u_h\in \lco$ satisfying
\eqref{eqn:NZMethod}.
\end{col}

We now turn our attention
to error estimates of 
the finite element method \eqref{eqn:NZMethod}
and derive a rate of convergence in the $L^\infty$
norm.  To do so we assume
that the solution to \eqref{nonDiv}
satisfies $u\in C^{2,\alpha}(\Omega)$.
Recall (cf. Theorem \ref{thm:Schauder})
that this regularity is guaranteed provided
that $A$ is H\"older continuous 
and $\p\Omega$ is sufficiently smooth.

Now, since the method is linear and 
the problem is stable, such estimates
reduce to the consistency of the method.
However, as shown in Lemma~\ref{lem:FEMInconsistent}, the finite element method is not consistent, in the sense of Definition~\ref{def:ConsistentOperator}, when supplemented with the canonical interpolant $I_h^{fe}$. Instead, we make use of the elliptic projection $I_h^{ep}$ defined in \eqref{eqn:EllipticProjectionDef} and its properties (\cf Lemma~\ref{lem:EllipticProjectionConsistency} and Proposition~\ref{prop:SchatzWahlbin}).
%


In conclusion, 
in order to derive error estimates,
it suffices to derive upper bounds
for the difference $u_h^\eps-I_h^{ep} u$.
To this end, we apply the definition
of the method \eqref{eqn:NZMethod}
to obtain the {\it error equation}:
\begin{align}\label{eqn:NZErrorEquation}
\mathcal{L}_h^\eps[I_h^{ep} u-u_h^\eps](z_i) = \int_{\omega_{z_i}} \big(T_1^{(i)}+T_2^{(i)}+T_3^{(i)}\big)\phi_i
\end{align}
with
\begin{align*}
T_1^{(i)}
&= I_\eps \left[ I_h^{ep} u \right] (z_i)-I_\eps u(z_i),\\
T_2^{(i)} & = I_\eps u(z_i) - \big({A}(z_i)- \frac{\lambda}2 I\big):D^2 u(z_i),\\
T_3^{(i)} & = \big(({A}(z_i)- \frac{\lambda}2 I\big):\big(D^2u(z_i)-D^2 u(x)\big).
\end{align*}
Note that, with the finite element ABP estimate
given in Theorem \ref{thm:FEMABP123}
and the approximation results 
of the elliptic projection stated
in Proposition \ref{prop:SchatzWahlbin},
upper bound estimates of $T_j^{(i)}$ yield
error estimates of $u-u_h$.  

With the assumed regularity $u\in C^{2,\alpha}(\Omega)$,
we immediately find that $T_3^{(i)}$ can be bounded
by
\begin{align}\label{eqn:T3Estimate}
|T_3^{(i)}|\le C h^\alpha \|u\|_{C^{2,\alpha}(\Omega)}.
\end{align}
For $T^{(i)}_2$, we apply
Lemma \ref{lem:RateConvIT}
to obtain
\begin{align}\label{eqn:T2Estimate}
|T^{(i)}_2|\le C \theta^{\alpha} \eps^{\alpha} \|u\|_{C^{2,\alpha}(\Omega)} \le C\eps^{\alpha} \|u\|_{C^{2,\alpha}(\Omega)}.
\end{align}
%
Finally, we apply the approximation results
of Proposition \ref{prop:SchatzWahlbin}
and the definition of the second-order difference
operator $\delta^2_{\theta y, \theta}$ to obtain
 \begin{align*}
 \big|\delta^2_{\theta y, \theta} \big(I_h^{ep} u(z_i)- u(z_i) \big)\big|\le C \frac{h^2}{\theta^2} |\log h| \|u\|_{W^{2,\infty}(\Omega)}.
 \end{align*}
This leads to the estimate
\begin{align}
\label{eqn:T1Estimate}
|T_1^{(i)}|\le  C  \|u\|_{W^{2,\infty}(\Omega)}\Big( \frac{h^2}{\eps^2}|\log h| + \frac{h^2}{\theta^2 \eps^2} |\log h| \chi_{\omega_\eps}(z_i)\Big),
\end{align}
where $\chi_{\omega_\eps}$ is the indicator function of $\omega_\eps$,
which is defined in \eqref{eqn:WJNDomains}.
Combining \eqref{eqn:T3Estimate}--\eqref{eqn:T1Estimate},
we obtain
\begin{align}\label{eqn:T123}
&|T_1^{(i)}+T_2^{(i)}+T_3^{(i)}|\\
&\quad\nonumber\le C\Big(h^\alpha + \eps^{\alpha} 
+ \frac{h^2}{\eps^2}|\log h| + \frac{h^2}{\theta^2\eps^2} |\log h| |\chi_{\omega_\eps}(z_i)\Big),
\end{align}
where we have absorbed the factor $\|u\|_{C^{2,\alpha}(\Omega)}$ into the constant $C$.

Note that, 
owing to the last term on the right hand side, estimate \eqref{eqn:T123}
reduces to order $1$ in the boundary layer $\omega_\eps$.
In order  to derive meaningful estimates in this region,
we introduce the barrier layer function
\begin{align}\label{eqn:BarrierFunction}
b(x) := \xi({\rm dist}(x,\p\Omega)),\ \text{with}\ 
\xi(s):= 
\left\{
\begin{array}{ll}
Q^{-2} (s-Q\eps)^2 -\eps^2 & \text{if }s\le Q\eps,\\
-\eps^2 & \text{if }s>Q_\eps,
\end{array}
\right.
\end{align}
The discrete boundary layer function is defined as
$b_h:=I_h^{fe} b$.  The next result
summarizes key properties of $b_h$; see
\cite[Lemma 6.1]{NochettoZhang16}

\begin{lem}[properties of barrier layer function]
Let $b_h = I_h^{fe} b$,
where $b$ is given by \eqref{eqn:BarrierFunction}.
Then there holds
\begin{align*}
\mathcal{L}_h^\eps b_h(z_i) \ge C \chi_{\omega_\eps(z_i)},\qquad |b_h(z_i)|\le C\eps^2.
\end{align*}
\end{lem}

Notice that since every $z_i\in \omega_\eps$ is at most 
$\mathcal{O}(h)$ from $\p\Omega$, there holds $\eps \theta \ge C h$ on $\omega_\eps$;
consequently, the last term in \eqref{eqn:T123} can
be bounded by the discrete barrier function as follows:
\begin{align*}
\frac{h^2}{\theta^2 \eps^2} |\log h| \chi_{\omega_\eps(z_i)}\|u\|_{C^{2,\alpha}(\Omega)}\le C |\log h| \chi_{\omega_\eps}(z_i)\le C |\log h| \mathcal{L}_h^\eps b_h(z_i).
\end{align*}
Thus, combining this estimate with \eqref{eqn:T123} and \eqref{eqn:NZErrorEquation}
leads to
\begin{align}\label{eqn:NZErrorEquation2}
\mathcal{L}_h^\eps\big[I_h^{ep} u - u_h^\eps - C |\log h| b_h\big](z_i)\le C\big(h^\alpha+ \eps^{\alpha} + \frac{h^2}{\eps^2} |\log h|
\big).
\end{align}
From this expression, we easily obtain
estimates of $u-u_h$.
\begin{thm}[rate of convergence]
\label{thm:NZErrorEstimateC2}
Suppose that $\mct$ satisfies \eqref{eqn:FEMMMatrix},
and that the solution to \eqref{nonDiv} satisfies $u\in C^{2,\alpha}(\Omega)$.
Let $u_h\in \lco$ be the solution to \eqref{eqn:NZMethod} with $\eps = C(h^2 |\log h|)^{1/(2+\alpha)}$.  
Then there holds
\begin{align}\label{eqn:NZErrorEstimateC2}
\|u-u_h\|_{L^\infty(\Omega)}\le C\big(h^2 |\log h|\big)^{\alpha/(2+\alpha)}.
\end{align}
\end{thm}
\begin{proof}
Applying the finite element ABP 
for $\mathcal{L}_h^\eps$  (cf. Theorem \ref{thm:FEMABP123})
to \eqref{eqn:NZErrorEquation2} yields
\begin{align*}
\sup_\Omega (I_h^{ep} u-u_h^\eps - C |\log h| b_h)^-\le C \big(h^\alpha + \eps^{\alpha}+ \frac{h^2}{\eps^2} |\log h|\big).
\end{align*}
Therefore, since $|b_h|\le \eps^2$, we get
\begin{align*}
\sup (I_h^{ep} u-u_h^\eps)^-\le C \Big(h^\alpha +\eps^{\alpha}+ \big(\eps^2 +\frac{h^2}{\eps^2}\big) |\log h|\Big).
\end{align*}
Similar estimates are obtained for $\sup_\Omega (I_h^{ep} u-u_h^\eps)^+$, thus leading to
\begin{align*}
\|I_h^{ep} u-u_h^\eps\|_{L^\infty(\Omega)}&\le C \Big(h^\alpha +\eps^{\alpha}+\big(\eps^2 +\frac{h^2}{\eps^2}\big) |\log h|\Big).
\end{align*}
If $\eps = C (h^2 |\log h|)^{1/(2+\alpha)}$, then
$(h^2/\eps^2)|\log h|\le C \eps^{\alpha}\le C (h^2 |\log h|)^{(\alpha)/(2+\alpha)}$,
and therefore
\begin{align*}
\|I_h^{ep} u-u_h^\eps\|_{L^\infty(\Omega)}\le C  (h^2 |\log h|)^{(\alpha)/(2+\alpha)}.
\end{align*}
Finally, applying Proposition \ref{prop:SchatzWahlbin}
and the triangle inequality yields \eqref{eqn:NZErrorEstimateC2}.
The proof is complete.
\end{proof}

Theorem \ref{thm:NZErrorEstimateC2} shows
that if $\alpha=1$, then
the error satisfies $\|u-u_h\|_{L^\infty(\Omega)} = \mathcal{O}(h^{2/3-\tau})$
for arbitrary $\tau>0$.  If more regularity
is assumed then an almost linear rate 
is obtained \cite[Corollary 6.8]{NochettoZhang16}

\begin{thm}[improved rate]
Let $h$ and $\eps$ satisfy $\eps = C h^{2/(3+\alpha)}$.
If the solution to \eqref{nonDiv} has the regularity $u\in C^{3,\alpha}(\Omega)$,
and if $\mct$ satisfies \eqref{eqn:FEMMMatrix}, then
\begin{align*}
\|u-u^\eps_h\|_{L^\infty(\Omega)}\le C h^{2(1+\alpha)/(3+\alpha)}|\log h|.
\end{align*}
\end{thm}

In the opposite direction, 
if we assume less regularity
of the solution and data, then
convergence is still obtained, although
without rates
\cite[Corollary 6.5]{NochettoZhang16}.

\begin{thm}[convergence]
Assume that $u\in C^2(\bar\Omega)$, that the coefficient matrix satisfies
$A\in VMO(\Omega,\polS^d)$,
and that the two scales $\eps$ and $h$
satisfy $\eps = C h |\log h|$.
Let $u_h^\eps\in \lco$ satisfy \eqref{eqn:NZMethod}
with $A(z_i)$ replaced by its average
%
\begin{align*}
\bar{A}(z_i):&=\frac{1}{|\omega_{z_i}|} \int_{\omega_{z_i}} A(y).
\end{align*}
If the mesh condition  \eqref{eqn:FEMMMatrix}
is satisfied, then $\lim_{h\to 0^+} \|u-u_h^\eps\|_{L^\infty(\Omega)}=0$.
\end{thm}

%

\section{Discretizations of convex second--order elliptic equations}
\label{sec:convex}

Up to this point we have discussed general issues regarding stability and convergence of numerical methods 
for general fully nonlinear equations, as well as the construction of suitable schemes for linear problems in nondivergence form. 
For the rest of this overview we will merge the ideas presented in the previous sections.
Since convergence of these schemes has already been discussed in Theorems~\ref{thm:BSTHM1} and \ref{thm:BSTHM1Alt}, we will pay special attention to obtaining rates of convergence for them. In this section we will focus on {\it convex} equations. Moreover, as explained Section~\ref{subsec:Equivalence}, there is no loss of generality in assuming that we are dealing with the Hamilton Jacobi Bellman equation of Example~\ref{ex:HJB}, which we recall 
reads
\begin{equation}
\label{eq:HJBbvp}
  \begin{dcases}
    F(x,u,Du,D^2u) := \sup_{\alpha \in \calA} \left[ \tilde{\calL}^\alpha u(x) - f^\alpha(x) \right] = 0, & \text{in } \Omega, \\
    u = g, & \text{on } \partial \Omega.
  \end{dcases}
\end{equation}
We assume that the linear elliptic operators are such that, for every $\alpha \in \calA$,
\[
  A^\alpha \geq \lambda I , \quad \forall \alpha \in \calA,
\]
so that the operator $F$ is uniformly elliptic.

The numerical schemes for problem \eqref{eq:HJBbvp} can be roughly classified as finite difference, 
finite element and semi-Lagrangian methods. In this section we will focus on the first two. 
Semi-Lagrangian schemes will be illustrated in a particular case, the Monge Amp\`ere equation, in Section \ref{subsec:FengJensen}.

\subsection{Finite difference methods}
\label{sub:FDHJB}

The analysis of the convergence properties of finite difference schemes for 
\eqref{eq:HJBbvp} dates back to \cite{Krylov97,MR1759507} 
where the problem is considered for $\Omega = \Real^d$ and constant ``coefficients'', \ie when the operators $\tilde{\calL}^\alpha$ are $x$-independent. 
These results were later extended to Lipschitz coefficients in \cite{MR1916291}. Before embarking into the technical details of their results, let us give some intuition into them.
Recall that, in general, a finite difference scheme is written in the form \eqref{eqn:FhDifferencesForm}. 
Ideally, to approximate $u$, the solution of \eqref{eq:HJBbvp}, one would first construct a smooth function 
$u_\vare$, parameterized by $\varepsilon$ which, for a constant $C$ independent of $\vare$, satisfies
\[
  \| u - u_\vare \|_{L^\infty(\Omega)} \leq C \vare^{\kappa_1}.
\]
A strengthened notion of consistency, \cf Definition~\ref{def:ConsistentOperator}, would then imply that $F_h[I^{fd}_hu_\vare] \approx h^{\kappa_2}$,
where the hidden constants in this expression may depend on $\varepsilon$.
By stability, Definition~\ref{def:NumericalStability}, and monotonicity, Definition~\ref{def:discreteMonotone}, we obtain 
\[
  \| I^{fd}_h u_\vare - u_h \|_{L^\infty(\bar\Omega_h)} \leq C h^{\kappa_2}.
\]
An application of the triangle inequality and relating the smoothing parameter $\vare$ with the discretization $h$ would 
yield a rate of convergence.

Unfortunately, the construction of such a smooth approximation $u_\vare$ is 
not immediate in practice. 
The groundbreaking idea of Krylov was to ``shake the coefficients''. He introduced $u_\vare$ as the solution of 
\begin{equation}
\label{eq:perturbedHJB}
  \inf_{|e| \leq \vare}  F(x+e,u_\vare, Du_\vare, D^2 u_\vare)  = 0,
\end{equation}
and, from $u_\vare$, he was able to construct a smooth subsolution to \eqref{eq:HJBbvp} so that, by comparison we can obtain an upper bound for $u_h - u$. To obtain a lower bound, the original work of Krylov invoked arguments that some authors have characterized as probabilistic. However, \cite[page 3]{Krylov15} disagrees with this statement. Another line of reasoning was given by Barles and Jakobsen who, instead, proposed
that the problem and scheme should play a symmetric role. In other words, they introduce $u_h^\vare$ which solves
\begin{equation}
\label{eq:perturbedschemeHJB}
  \inf_{|e| \leq \vare} F_h[u_h^\vare](z+e) = 0.
\end{equation}
Under suitable assumptions they show that this family of operators possesses unique smooth solutions, where the smoothness is independent of $\vare$. This then allows us to compare $u$ and $u_h^\vare$ and obtain a rate of convergence. The assumptions that they rely on, however, must be checked for every particular instance.

Let us now proceed with the details. Recall that
we are operating in the whole space $\Real^d$ and that the mesh is given by $\bar \Omega_h = \bbZ_h^d$. We assume that all the coefficients of $\tilde{\calL}^\alpha$ and $f^\alpha$ are Lipschitz continuous uniformly in $\alpha$ and that $\sup_{\alpha \in \calA} c^\alpha(x) \leq c<0$
for every $x \in \Real^d$. This, as indicated in Theorem~\ref{thm:comparison}, part \ref{item:compclq0},  implies that the operator has a comparison principle. The structural requirements on the discretization scheme are summarized below.

\begin{ass}[finite differences]
\label{ass:FDHJB}
The finite difference scheme 
\[
  F_h= F_h(z,r,q)
\]
satisfies:
\begin{enumerate}[1.]
  \item The scheme is consistent in a sense stronger than Definition~\ref{def:ConsistentOperator}, namely, for some $\kappa > 0$
  \[
    \| F_h[I^{fd}_h \phi] - F[\phi] \|_{L^\infty(\bar\Omega_h)} \leq C h^\kappa,
  \]
  for all sufficiently smooth $\phi$.

  \item\label{enum:2.2} The scheme is of nonnegative type, \cf Definition~\ref{def:positiveType}. 
  In addition, there exists $\bar{c}>0$ such that  $F_h(z,r+t,q+{\bm 1}t)\le F_h(z,r,q)-\bar{c} t$ for all $t\ge 0$.

  \item The scheme is convex in the $r$ and $q$ variables.
  
  \item\label{enum:2.4} The scheme is uniformly solvable and smooth under perturbations. In other words, for $h>0$ small enough and $\vare \in [0,1]$ problem \eqref{eq:perturbedschemeHJB} has a unique solution $u_h^\vare$ and, moreover, for some $\delta>0$ we have
  \[
    \| u_h - u_h^\vare \|_{L^\infty(
    \bar\Omega_h 
    )} \leq C \vare^\delta.
  \]
\end{enumerate}
\end{ass}

Notice that the first three conditions of Assumption~\ref{ass:FDHJB} are relatively easy to enforce. However, the last one must be verified in each case. As a first step, we show that schemes satisfying the nonnegativity condition satisfy a comparison principle.

\begin{lem}[discrete comparison principle]
\label{lem:HJBComparisonP}
Let $F_h$ satisfy Assumption \ref{ass:FDHJB}, item \ref{enum:2.2}. Let $v_h,w_h\in \fd$ be two bounded nodal functions that are sub- and supersolutions to the discrete problem $F_h[u_h]=0$, respectively. Then $v_h\le w_h$.
\end{lem}
\begin{proof}
We assume $r = \sup_{\bar\Omega_h}
 (v_h-w_h)>0$ and derive a contradiction. For simplicity we assume there is $z \in \bar\Omega_h$ such
that $v_h(z)-w_h(z) = r$, for otherwise we can apply a standard limiting argument.

As in Section~\ref{subsec:monoFD} we write
$F_h[v_h](z) = F_h(z,v_h(z),Tv_h(z))$, where the set of translates of $v_h(z)$ is $Tv_h(z) = \{v_h(z+h y):\ y\in \St\}$, and $\St$ is the stencil.
Note that $v_h(z) = w_h(z)+r$
and $v_h(z+h y)\le w_h(z+h y)+r$ for $y\in \St$.  
Since $F_h$ is increasing in its third argument
(\cf Definition \ref{def:positiveType}), we have
\begin{align*}
0 
&\le F_h[v_h](z)-F_h[w_h](z)\\
& = F_h(z,w_h(z)+r,Tv_h(z)) - F_h(z,w_h(z),Tw_h(z))\\
&\le F_h(z,w_h(z)+r,Tw_h(z)+{\bm 1}r) - F_h(z,w_h(z),Tw_h(z)).
\end{align*}
Applying Assumption~\ref{ass:FDHJB}, item \ref{enum:2.2}, then yields
\begin{align*}
0\le \big(F_h(z,w_h(z),Tw_h(z)) - \bar{c} r\big) -F_h(z,w_h(z),Tw_h(z)) = - \bar{c} r,
\end{align*}
which is a contradiction.
\end{proof}

Having shown the comparison principle, we now state the convergence rate of the scheme.

\begin{thm}[rate of convergence]
\label{thm:rateFDHJB}
Assume that the finite difference scheme satisfies Assumption~\ref{ass:FDHJB}. Then there are constants $C_1,C_2>0$ 
such that for every $z \in \bar\Omega_h$
\[
  I_h^{fd} u(z) - u_h(z)  \leq C_1 h^{\kappa_1}, \qquad u_h (z)- I_h^{fd} u(z) \leq C_2 h^{\kappa_2},
\]
where the rates $\kappa_1,\kappa_2 >0$ are not necessarily the same.
\end{thm}
\begin{proof}
Let us sketch the proof of each one of these bounds.

{\it Proof of $I_h^{fd} u(z) - u_h(z)  \leq C_1 h^{\kappa_1}$.}
Let $u_\vare$ be the solution of \eqref{eq:perturbedHJB}. By definition, after the change of variables $y=x+e$, we realize that $u_\vare(\cdot - e)$ is, for $|e|\leq \vare$, a subsolution to the equation, \ie it satisfies, in the viscosity sense,
\[
  F\left(y,u_\vare(\cdot - e), Du_\vare(\cdot-e),D^2 u_\vare(\cdot-e) \right) \geq 0.
\]
We regularize this function by 
the 
mollification $u^\vare = u_\vare \star \rho_\vare$, where $\rho_\vare$ are the standard mollifiers. 
From the convexity of the operator $F$ and Jensen's inequality, it follows that $u^\vare$ is also a subsolution. Since $u^\vare$ is now a smooth function, we can invoke consistency to obtain
\[
  F_h[I^{fd}_hu^\vare](z) \geq F(z,u^\vare(z),Du^\vare(z),D^2u^\vare(z))  -  Ch^\kappa \geq -Ch^\kappa,
\]
for some constant $C$ that depends on the smoothness of $u^\vare$ which, in turn, scales like negative powers of $\vare$, say $C \leq \vare^{-\delta_1}$. In conclusion, we have obtained that
\[
  F_h[ I^{fd}_hu^\vare ](z) \geq -C  h^\kappa \vare^{-\delta_1}.
\]
Assumption \ref{ass:FDHJB}, item \ref{enum:2.2}, shows
that  the function $I^{fd}_h(u^\vare - Ch^\kappa\vare^{-\delta_1})$ is, for a suitably chosen $C$, a subsolution of the scheme, i.e.,
$F_h[I_h^{fd}(u^\vare - C h^\kappa \vare^{-\delta_1})] \ge 0$.
Therefore, by the comparison principle given in Lemma \ref{lem:HJBComparisonP},
\[
    I^{fd}_hu^\vare(z) - u_h (z) \leq Ch^\kappa\vare^{-\delta_1}.
\]

In conclusion, using the continuity properties of the equation \eqref{eq:HJBbvp} and properties of mollifiers we obtain
\begin{align*}
  I_h^{fd}u(z) - u_h(z)  &= I^{fd}_h (u - u_\vare)(z) + I^{fd}_h( u_\vare - u^\vare)(z) + I_h^{fd} (u^\vare - u_h)(z) \\
  &\leq \frac{C_1}2( \vare^{\delta_2} + h^\kappa\vare^{-\delta_1}),
\end{align*}
where $\delta_2>0$ depends on the smoothness of $u$ (\cf Theorem~\ref{thm:locC2a}). Optimizing with respect to $\vare$ we get the result.


{\it Proof of $u_h (z)- I_h^{fd} u(z) \leq C_2 h^{\kappa_2}$.}
We follow a similar reasoning but this time, we interchange the roles that the equation and the scheme have played in the previous step. Indeed, by item \ref{enum:2.4} of Assumption~\ref{ass:FDHJB} we know that there is $u_h^\vare \in \fd$ that solves \eqref{eq:perturbedschemeHJB} and, again with the change of variables $\tilde{z}=z+e$, this function satisfies
\[
  F_h[u_h^\vare](z) \geq 0,
\]
so that it is a subsolution of the scheme. Now, convexity of $F_h$ implies that $I_h^{fd}(u_h^\vare \star \rho_\vare)$ is also a subsolution of the scheme. Moreover, we have that $u_h^\vare \star \rho_\vare$ is a smooth function and, therefore, consistency implies that, for some $\delta_3>0$ we have
\[
  F[u_h^\vare \star \rho_\vare](z) \geq -C  h^\kappa \vare^{-\delta_3}.
\]
Monotonicity of the equation shows that, for a suitably chosen $C$, the function $u_h^\vare \star \rho_\vare -Ch^\kappa \vare^{-\delta_3}$ is a subsolution which, by comparison, readily implies that
\[
  u_h^\vare \star \rho_\vare (z) - u(z) \leq C h^\kappa \vare^{-\delta_3}.
\]
Properties of convolutions together with item \ref{enum:2.4} of Assumption~\ref{ass:FDHJB} then imply
\begin{align*}
  u_h(z) - I_h^{fd} u(z) &= (u_h - u_h^\vare)(z) + (u_h^\vare - I_h^{fd} (u_h^\vare \star \rho_\vare))(z) \\
      &+ I_h^{fd}(u_h^\vare \star \rho_\vare - u)(z) \leq \frac{C_2}3( \vare^\delta + \vare^{\delta_2} + h^\kappa \vare^{-\delta_3}).
\end{align*}
An optimization in $\vare$ once again yields the result.
\end{proof}

Let us now give an example of finite difference schemes for which the assumptions of Theorem~\ref{thm:rateFDHJB} can be verified.

\begin{ex}[monotone finite differences]
\label{ex:BarlesJakobsen}
Let $A^\alpha$ be independent of $x$ for every $\alpha$ and, possibly after a renormalization, verify
\[
  \sum_{i=1}^d \left[ a_{i,i}^\alpha - \sum_{j \neq i} |a_{i,j}^\alpha| \right] \leq 1.
\]
More importantly, we assume that these matrices are diagonally dominant, \ie \eqref{eqn:AofTheForm} holds. For simplicity assume also that $\bb^\alpha \equiv 0$ for all $\alpha \in \calA$. As shown in Lemma~\ref{lem:AIsDiagonalDom}, there exists a monotone finite difference $\calL_h^\alpha + c^\alpha$ that is consistent with $\tilde\calL^\alpha$. We then define
\[
  F_h[u_h](z) = \sup_{\alpha \in \calA} \left[ \calL_h^\alpha u_h(z) + c^\alpha(z) u_h(z) - f^\alpha (z) \right] = 0.
\]
In this case, \cite[Section 4]{MR1916291} shows that all the assumptions are verified and, moreover, that $\kappa_1 = \kappa_2 = 1/3$.
\end{ex}

\begin{rem}[examples and improvements]
\label{rem:betterratesFDHJB}
Let us comment on the various improvements and refinements of Theorem~\ref{thm:rateFDHJB} as well as on some extensions of Example~\ref{ex:BarlesJakobsen}.
\begin{enumerate}[$\bullet$]
  \item In Example~\ref{ex:BarlesJakobsen} it is assumed that the leading coefficients $A^\alpha$ do not depend on the spatial variable $x$. In \cite{BarlesJakobsen05} the authors studied the relation between \eqref{eq:HJBbvp} and a certain system of quasivariational inequalities of compliance obstacle type (see \eqref{eq:QVI} below). They used this system instead of item \ref{enum:2.4} from Assumption~\ref{ass:FDHJB} to obtain an upper bound for $u_h - I_h^{fd}u$ in the case when $A^\alpha$ is Lipschitz continuous in the space variable. Their results show that $\|I_h^{fd}u - u_h \|_{L^\infty(
  \bar\Omega_h 
  )} \leq C h^{1/5}$.
  
  \item Let $\{\be_i\}_{i=1}^d$ be an orthonormal basis of $\Real^d$ and assume that the matrices have the form
  \[
    A^\alpha(x) = \sum_{i=1}^d a^\alpha_i(x) \be_i \otimes \be_i
  \]
  for some $a^\alpha_i$ that is Lipschitz in $x$ uniformly in $\alpha$. Recall that, with this assumption, Lemma~\ref{lem:ADecompLemmaFD} guarantees the existence of a monotone finite difference scheme. In this setting \cite{Krylov05} shows that $\kappa_1 = \kappa_2 = 1/2$. Moreover, in the same setting but assuming that the coefficients are $C^{1,1}$, \cite{MR3197305} shows a that $\kappa_1= \kappa_2 = 2/3$.
  
  \item All the aforementioned results consider the case $\Omega = \Real^d$,
  so that boundary conditions are not an issue, \cf~Theorem \ref{thm:BSTHM1Alt}.
  In \cite{MR2334605} the authors consider the boundary value problem \eqref{eq:HJBbvp} under the assumption that the boundary conditions are attained classically; see Definitions~\ref{def:viscobvp} and \ref{def:viscoBVP}. Under the assumption that a {\it barrier} function can be constructed, that is a smooth $b$ such that $b > 0$ in $\Omega$, $b = 0$ on $\partial\Omega$ and
  \[
    F[b](x) \leq -1,
  \]
  the authors were able to extend the results presented here and show that the rate of convergence is $\| I_h^{fd} u - u_h \|_{L^\infty(\bar\Omega_h)} \leq C h^{1/2}$.
\end{enumerate}
Other results, which invoke the probabilistic interpretation of \eqref{eq:HJBbvp} can be found in the literature; see, for instance, \cite{FlemingSoner06} and \cite{Menaldi89}. The reader is also referred to the introduction of \cite{Krylov15} for a detailed account of the development of error estimates for \eqref{eq:HJBbvp}.
\end{rem}

\subsection{Finite element methods}
\label{sub:FEMHJB}

We now focus on the construction and analysis of finite element schemes for \eqref{eq:HJBbvp}. We will divide the exposition in two cases. First we will discuss the discretization for a variant of the problem when the operators $\tilde\calL^\alpha $ are replaced by operators in divergence form $\tilde L^\alpha$ with smooth coefficients. Then we will discuss the case of \eqref{eq:HJBbvp} where the coefficients for $\tilde\calL^\alpha$ satisfy the Cordes condition of Definition~\ref{def:Cordes}, which is based on the discretization of nondivergence form operators of Section~\ref{sec:DiscreteCordes}. We must remark that it is possible to construct discrete schemes using the integrodifferential approximation of Section~\ref{sub:Wujunnondiv}. However, to avoid repetition, its discussion will be illustrated in Section~\ref{sub:FEMIsaacs} for a nonconvex operator of Isaacs type (\cf Example~\ref{ex:Isaacs}). Setting $\#\calB =1$ there we reduce the scheme and its analysis to the case we are concerned with here.

\subsubsection{Discretization for divergence form operators}
Let us consider \eqref{eq:HJBbvp} but where the operators $\tilde\calL^\alpha$ are replaced 
by divergence form elliptic operators $\tilde L^\alpha$ with $C^2(\bar\Omega)$ 
coefficients; see Definition~\ref{def:diveliptic}. In addition, we assume that $\partial\Omega$ is sufficiently smooth, $g=0$ and that, for every $\alpha \in \calA$, we have $0 \leq f^\alpha \in L^\infty(\Omega)$. Finally, we assume that 
$\calA = \{1, \ldots, M\}$ for some $M\in \polN$.


\begin{rem}[smooth coefficients]
\label{rem:smooth}
Notice that, if the coefficients of $\tilde\calL^\alpha$ in \eqref{eq:HJBbvp} are sufficiently smooth, one can rewrite this operator in divergence form. Therefore, this reformulation is sufficiently general. 
\end{rem}

Recall that, by integration by parts, to every operator $\tilde{L}^\alpha$ we can associate the bilinear form $a^\alpha : H^1(\Omega) \times H^1(\Omega) \to \Real$, defined by
\[
  a^\alpha(v,w) = \int_\Omega \left(A^\alpha Dv \cdot Dw + v^\alpha \bb^\alpha\cdot  Dw + \bc^\alpha\cdot Dv w + d^\alpha vw\right).
\]
To simplify the discussion, we will assume that these forms are coercive uniformly in $\alpha$, that is, there is a constant $\lambda_0$ such that
\[
  \inf_{\alpha \in \calA} a^\alpha(v,v) \geq \lambda_0 \| D v\|_{L^2(\Omega)}^2, \quad \forall v \in H^1_0(\Omega).
\]

Given $k>0$ and $w \in H^1_0(\Omega)$ define
\[
  \calK(k,w) = \left\{ v \in H^1_0(\Omega): \ v \leq k + w \right\},
\]
which is a closed and convex subset of $H^1_0(\Omega)$.
For $k>0$ we introduce a system of quasivariational inequalities of compliance obstacle type as follows: Find $u^\alpha_k \in \calK(k,u^{\alpha+1})$ such that
\begin{equation}
\label{eq:QVI}
  a^\alpha(u^\alpha_k,u^\alpha_k - v) \leq ( f^\alpha, u^\alpha_k - v) \quad \forall v \in \calK(k,u^{\alpha+1}),
\end{equation}
with $u^{M+1}_k = u^1_k$. Standard results on quasivariational inequalities, like those presented in  \cite[Chapter 4]{MR756234}, imply the existence and uniqueness of $\bu_k := \{ u^\alpha_k \}_{\alpha \in \calA} \subset H^1_0(\Omega)$ that solves \eqref{eq:QVI}. In addition, adaptions of the results of Section~\ref{sub:weakvarsols} to the case of (quasi)variational inequalities (\ie by penalization and a limiting argument) allow us to conclude that $\bu_k \subset W^{1,\infty}(\Omega) \cap W^{2,p}_{loc}(\Omega)$ ($p < \infty$). The purpose of this system lies in the fact that, as $k \to 0$, the solutions to \eqref{eq:QVI} converge to the solution to \eqref{eq:HJBbvp}. For a proof of the following result see \cite[Theorem 7.2]{MR536953} and \cite[Section 4.6]{MR756234}.

\begin{thm}[convergence as $k\to0$]
\label{thm:EvansFriedman}
In this the setting described above
there exists a unique strong solution $u$ to \eqref{eq:HJBbvp}.
Moreover, defining
$\bu = \{ u\}_{\alpha \in \calA}$, there holds
\[
  \lim_{k\to 0} \| \bu - \bu_k \|_{\ell^\infty(\calA,L^\infty(\Omega))} =0.
\]
\end{thm}

It is remarkable that this result was shown before the development 
of viscosity solutions of Section~\ref{sec:PDEs}, and it was originally used to show the well-posedness of the Hamilton Jacobi Bellman equations. 
In addition, it can be used to propose finite element discretizations of \eqref{eq:HJBbvp} by instead discretizing \eqref{eq:QVI}. Then, if one is able to extend the standard $L^\infty(\Omega)$-norm estimates for variational inequalities (see \cite[Theorem 2.9]{MR3393323}) and give a rate for the limit in Theorem~\ref{thm:EvansFriedman} we obtain a convergent (with rates) finite element method. This program has been, to a certain degree of success, carried out by \cite{MR897266,Boulbrachene01,MR2543877}.

With the notation of Section~\ref{sub:MonoFEM} we begin by defining, for $k>0$ and $w_h \in \lco$, the set
\[
  \calK_h(k,w_h) = \left\{ v_h  \in \lco: v_h \leq k + w_h \right\},
\]
which is a closed and convex subset of $\lco$. Moreover, we remark that it is sufficient to impose the inequality at the nodes $z \in \Omega_h^I$. We approximate the solution to \eqref{eq:QVI} by the following set of discrete quasivariational inequalities: find $u_h^{\alpha,k} \in \calK_h(k,u_h^{\alpha+1,k})$ such that
\begin{equation}
\label{eq:QVIh}
a^\alpha(u^{\alpha,k}_h,u^{\alpha,k}_h  - v_h) \leq ( f^\alpha, u^{\alpha,k}_h - v_h) \quad \forall v_h \in \calK_h(k,u_h^{\alpha+1,k}),
\end{equation}
with $u^{M+1,k}_h = u^{1,k}_h$.
Once again, it can be shown that this problem always has a unique solution $\bu_{h,k} = \{ u_h^{\alpha,k} \}_{\alpha \in \calA} \subset \lco$.

To establish the $L^\infty(\Omega)$-norm convergence of the solutions to \eqref{eq:QVIh} to the solution to \eqref{eq:QVI} we must assume that the ensuing stiffness matrices are $M$--matrices. Examining the proof of Lemma~\ref{lem:FEMMmatrix} we realize that for this to hold, it is sufficient to require that
\begin{equation}
\label{eq:BScondition}
  a^\alpha(\phi_i,\phi_j) \leq 0, \quad \forall \alpha \in \calA, \ i \neq j,
\end{equation}
which we assume below.

\begin{rem}[lack of generality]
\label{rem:BScondition}
The discussion of Section~\ref{sub:MonoFEM} shows that condition \eqref{eq:BScondition} is satisfied if the mesh is weakly acute {\it in the metric induced by} $A^\alpha$ {\it for all} $\alpha \in \calA$. This is a severe restriction in practice, as it is not clear how to impose such a condition for one matrix, let alone for a family of them.
\end{rem}

Although a rate of convergence for the limit in Theorem~\ref{thm:EvansFriedman} does not seem possible, there is a rate for the approximation of \eqref{eq:QVI} by \eqref{eq:QVIh}.

\begin{thm}[rate of convergence]
\label{thm:CorteyDumont}
Assume that, for all $h>0$, the family of triangulations $\mct$ satisfies condition \eqref{eq:BScondition}, then we have
\[
  \| \bu_k - \bu_{h,k} \|_{\ell^\infty(\calA,L^\infty(\Omega))} \leq C h^2 |\log h|^3,
\]
where the constant $C>0$ depends on $M=\#\calA$ and $k$.
\end{thm}
\begin{proof}
We will follow the ideas of \cite{Boulbrachene01} which, in turn, borrow from the iterative schemes used to prove existence of elliptic quasivariational inequalities considered in \cite[Chapter 4]{MR756234}.

{\it Step 1.}
We introduce the following iterative scheme: Define $\hat{\bu}^0 = \{\hat{u}^{\alpha,0}\}_{\alpha\in \mathcal{A}}$ as the solutions to the unconstrained problems, \ie
\[
  a^\alpha(\hat{u}^{\alpha,0},v) = (f^\alpha, v), \quad \forall v \in H^1_0(\Omega).
\]
Assuming that, for $n\geq 0$, $\hat{\bu}^n$ has been defined we look for $\hat{u}^{\alpha,n+1} \in \calK(k,\hat{u}^{\alpha+1,n})$ such that
\[
  a^\alpha(\hat{u}^{\alpha,n+1},\hat{u}^{\alpha,n+1}-v) \leq (f^\alpha, \hat{u}^{\alpha,n+1}-v), \quad \forall v \in \calK(k,\hat{u}^{\alpha+1,n}).
\]
Using the positivity and order preserving properties of the associated map, it is possible then to show that, for $\lambda < \min\{ 1, k/\|\hat{\bu}^0\|_{\ell^\infty(\calA,L^\infty(\Omega))} \}$, we have
\begin{equation}
\label{eq:rateiterBLcont}
  \| \bu_k - \hat{\bu}^n \|_{\ell^\infty(\calA,L^\infty(\Omega))} \leq (1-\lambda)^n \|\hat{\bu}^0\|_{\ell^\infty(\calA,L^\infty(\Omega))}. 
\end{equation}

{\it Step 2.}
Define $\tilde\bu_h^0 = \{ \tilde{u}_h^{\alpha,0}\}_{\alpha \in \calA}$ as the finite element approximation of the unconstrained problems, \ie
\[
  a^\alpha(\tilde{u}^{\alpha,0}_h,v_h) = (f^\alpha, v_h), \quad \forall v_h \in \lco.
\]
Using \eqref{eq:BScondition} we can invoke standard finite element error estimates for linear problems to obtain
\begin{equation}
\label{eq:Linfu0}
  \| \hat{\bu}^0 - \tilde{\bu}_h^0\|_{\ell^\infty(\calA,L^\infty(\Omega))} \leq C h^2 |\log h|^2.
\end{equation}
For each $n \geq 0$ we define $\tilde{u}^{\alpha,n+1}_h \in \calK_h(k,\hat{u}^{\alpha+1,n})$ as the solution of
\[
  a^\alpha(\tilde{u}_h^{\alpha,n+1},\tilde{u}_h^{\alpha,n+1}-v_h) \leq (f^\alpha, \tilde{u}_h^{\alpha,n+1} - v_h), \quad \forall v_h \in \calK_h(k,I_h^{fe} \hat{u}^{\alpha+1,n}).
\]
Notice that this is nothing but the finite element approximation of $\hat{u}^{\alpha+1,n}$ as the solution to an obstacle problem. Using, once again, \eqref{eq:BScondition} we can invoke pointwise estimates for obstacle problems \cite{MR0488847} and \cite[Theorem 2.9]{MR3393323} to conclude that
\begin{equation}
\label{eq:Linfun}
  \| \hat{\bu}^n - \tilde{\bu}_h^n\|_{\ell^\infty(\calA,L^\infty(\Omega))} \leq C h^2 |\log h|^2,
\end{equation}
where the constant is independent of $n$.

{\it Step 3.}
Introduce a discrete iterative scheme analogous to the one given in Step 1. In other words, set $\hat{\bu}_h^0 = \tilde{\bu}_h^0$ and, for $n\geq 0$, find $\hat{u}_h^{\alpha,n+1} \in \calK_h(k,\hat{u}_h^{\alpha+1,n})$ as the solution of
\[
  a^\alpha(\hat{u}_h^{\alpha,n+1},\hat{u}_h^{\alpha,n+1}-v_h) \leq (f^\alpha, \hat{u}_h^{\alpha,n+1}-v_h), \quad \forall v \in \calK_h(k,\hat{u}_h^{\alpha+1,n}).
\]
Similar techniques to the ones that led to \eqref{eq:rateiterBLcont} yield
\begin{equation}
\label{eq:rateiterBLdisc}
  \| \bu_{h,k} - \hat{\bu}^n_h \|_{\ell^\infty(\calA,L^\infty(\Omega))} \leq (1-\lambda)^n \|\hat{\bu}^0_h\|_{\ell^\infty(\calA,L^\infty(\Omega))}. 
\end{equation}

{\it Step 4.}
By induction, it can be shown that
\[
  \|\hat{\bu}^n - \hat{\bu}_h^n \|_{\ell^\infty(\calA,L^\infty(\Omega))} \leq \sum_{k=0}^n \| \hat{\bu}^k - \tilde{\bu}_h^k \|_{\ell^\infty(\calA,L^\infty(\Omega))}.
\]
With this at hand the triangle inequality yields
\begin{align*}
  \| \bu_k - \bu_{h,k} \|_{\ell^\infty(\calA,L^\infty(\Omega))} &\leq 
    \| \bu_k - \hat{\bu}^n \|_{\ell^\infty(\calA,L^\infty(\Omega))} 
    + \| \hat{\bu}^n_h - \bu_{h,k} \|_{\ell^\infty(\calA,L^\infty(\Omega))} \\
    &+ \sum_{k=0}^n \| \hat{\bu}^k - \tilde{\bu}_h^k \|_{\ell^\infty(\calA,L^\infty(\Omega))},
\end{align*}
so that by using \eqref{eq:rateiterBLcont}--\eqref{eq:rateiterBLdisc} we obtain
\[
  \| \bu_k - \bu_{h,k} \|_{\ell^\infty(\calA,L^\infty(\Omega))} \leq C \left[ (1-\lambda)^n + n h^2 |\log h|^2\right].
\]
Now choose $n$ so that $(1-\lambda)^n \approx h^2$ to obtain the result.
\end{proof}

\begin{rem}[$k=0$]
\label{rem:LionsMercier}
A similar algorithm to the one used in the proof of Theorem~\ref{thm:CorteyDumont} is studied in \cite[Algorithme I]{MR596541} for $k=0$. It is shown 
there that the ensuing iterates converge monotonically to $\bu$, but no rate is given.
\end{rem}

\subsubsection{Discretization for HJB satisfying the Cordes Condition}

Let us now discuss the case of \eqref{eq:HJBbvp} with nondivergence form operators without lower order terms, \ie $\calL^\alpha$. For simplicity, we set $g=0$. More importantly, we will assume that the coefficient matrices $A^\alpha$ satisfy the Cordes condition of Definition~\ref{def:Cordes}. Furthermore, we will assume that the domain $\Omega$ is convex and, finally, that $\calA$ is a compact metric space.

The stated assumptions imply, invoking Theorem~\ref{thm:exuniqueCordes}, that each one of the operators $\calL^\alpha$ is an isomorphism between $H^2(\Omega) \cap H^1_0(\Omega)$ and $L^2(\Omega)$. Consequently, to each one of them we can apply the techniques of Section~\ref{sec:DiscreteCordes}. Let us now, following the arguments of \cite{SmearsSuli14}, show that these assumptions also imply that problem \eqref{eq:HJBbvp} is also well-posed and that its solution is strong.

To do so we must assume that the Cordes condition holds uniformly in $\calA$ which, from \eqref{eq:Cordes1C} yields the existence of $\eps>0$ for which
\begin{align}\label{eqn:CordesPMT}
  \sup_{\alpha \in \calA} \|\gamma^\alpha \calL^\alpha v-\Delta v\|_{L^2(\Omega)} \leq \sqrt{1-\eps} \|D^2 v \|_{L^2(\Omega)} \quad \forall v \in H^2(\Omega).
\end{align}
This inequality motivates the definition of the (elliptic) nonlinear operator
\begin{align}\label{eqn:Fgamma}
F_\gamma[v]:=\sup_{\alpha\in \mathcal{A}} \big[\gamma^\alpha(\mathcal{L}^\alpha v-f^\alpha)\big].
\end{align}
The equivalence between problem $F_\gamma[u]=0$ and \eqref{eq:HJBbvp} essentially follows from the continuity of the data and the positivity of $\gamma^\alpha$.
This result is summarized in the next lemma.

\begin{lem}[equivalence]
\label{lem:equivFgamma}
Suppose that $f^\alpha$ and $A^\alpha$ are uniformly continuous for each $\alpha\in \mathcal{A}$, $\mathcal{A}$ is compact, and $\Omega$ is convex. The function $u\in H^2(\Omega)\cap H^1_0(\Omega)$ satisfies $F_\gamma[u]=0\ a.e.\ \Omega$ (\ie $u$ is a strong solution to this problem) if and only if it is a strong solution to \eqref{eq:HJBbvp}.
\end{lem}

Identity \eqref{eqn:CordesPMT} and the algebraic identity 
$|\sup_{\alpha\in \mathcal{A}} x^\alpha - \sup_{\alpha\in \mathcal{A}} y^\alpha|\le \sup_{\alpha\in \mathcal{A}}|x^\alpha - y^\alpha|$
for bounded sequences $\{x^\alpha\},\{y^\alpha\}\subset \bbR$
then yields the following result.

\begin{lem}[continuity]
\label{lem:CordesLinearizationIsh}
For all $u,v\in H^2(\Omega)$ we have
\begin{align*}
|F_\gamma[v]-F_\gamma[w]-\Delta (v-w)|\le \sqrt{1-\eps} |D^2(v-w)|.
\end{align*}
\end{lem}



Define
\begin{align*}
G(v,w):=\int_\Omega F_\gamma[v]\Delta w\qquad \forall v,w\in H^2(\Omega)\cap H^1_0(\Omega).
\end{align*}
The estimate of Lemma~\ref{lem:CordesLinearizationIsh} leads to the following (strong) monotonicity property of $G$.

\begin{thm}[properties of $G$]
\label{thm:Cproperties}
Under the given assumptions, there is a positive constant $C$ for which
\begin{align}\label{eqn:CMonotone}
G(v,v-w)-G(w,v-w)\ge C\big(1-\sqrt{1-\eps}\big)\|v-w\|_{H^2(\Omega)}^2
\end{align}
for all $v,w\in H^2(\Omega)\cap H^1_0(\Omega)$. In addition, if $A^\alpha\in C(\bar\Omega,\polS^d)$ and $f^\alpha\in C(\bar\Omega)$, then
\begin{align}\label{eqn:CLipschitz}
\big|G(v,u)-G(w,u)\big|\le C\|v-w\|_{H^2(\Omega)} \|u\|_{H^2(\Omega)}.
\end{align}
\end{thm}
\begin{proof}
Applying Lemma \ref{lem:CordesLinearizationIsh} yields
\begin{align*}
G(v,v-w)&-G(w,v-w)
 = \int_\Omega \big(F_\gamma[v]-F_\gamma[w]\big)\Delta (v-w)\\
& = \|\Delta (v-w)\|_{L^2(\Omega)}^2 +  \int_\Omega \big(F_\gamma[v]-F_\gamma[w]-\Delta(v-w)\big)\Delta (v-w)\\
&\ge \|\Delta (v-w)\|_{L^2(\Omega)}^2 - \sqrt{1-\eps}\|D^2 (v-w)\|_{L^2(\Omega)} \|\Delta (v-w)\|_{L^2(\Omega)}.
\end{align*}
Since $\Omega$ is convex, we can apply the Miranda-Talenti estimate \eqref{eq:MirandaTalenti}
to get
\begin{align*}
G(v,v-w)-G(w,v-w)&\ge \big(1-\sqrt{1-\eps}\big)\|\Delta (v-w)\|_{L^2(\Omega)}^2.
\end{align*}
The inequality \eqref{eqn:CMonotone} now follows 
from the equivalence 
of $\|\Delta \cdot\|_{L^2(\Omega)}$ and $\|\cdot\|_{H^2(\Omega)}$
on convex domains.

Finally the Lipschitz property \eqref{eqn:CLipschitz}
follows from the continuity of the data and the Cauchy-Schwarz inequality.
\end{proof}

Along with the Browder-Minty Theorem
and the fact that $\Delta:H^2(\Omega)\cap H^1_0(\Omega)\to L^2(\Omega)$
is surjective on convex domains,
 Theorem \ref{thm:Cproperties}
and Lemma \ref{lem:equivFgamma}
yield the following existence and uniqueness
result for the HJB problem \eqref{eq:HJBbvp};
see \cite[Theorem 3]{SmearsSuli14} for details.

\begin{thm}[existence and uniqueness]
Suppose that $\mathcal{A}$ is a compact
metric space, $f^\alpha\in C(\bar\Omega)$, and $A^\alpha\in C(\bar\Omega,\polS^d)$ satisfies
the Cordes condition for each $\alpha\in \mathcal{A}$. 
Then there exists a unique strong solution $u\in H^2(\Omega)\cap H^1_0(\Omega)$
to $F_\gamma[u] = 0$.  Moreover, $u$ is also the unique strong solution 
to \eqref{eq:HJBbvp}.
\end{thm}

Let us now discuss how the finite element methods for linear problems given in Section~\ref{sec:DiscreteCordes} are  extended to the nonlinear problem \eqref{eq:HJBbvp}.  While any of the methods given in that section can adopted for the nonlinear problem, here we focus on the DG approximation \cite{SmearsSuli13,SmearsSuli14,SmearsSuli16}.

Recall that $\dg$, defined by \eqref{eqn:DGSpace}, is the piecewise polynomial
space of degree $k$ with respect to a conforming and shape-regular, simplicial
partition of $\Omega$.
As in Section \ref{sec:DiscreteCordes}
we only consider here the $h$-version of the method,
where the polynomial degree is globally fixed,
and we do not follow the dependence on the polynomial degree $k$ of all the ensuing constants.  
We refer the reader to \cite{SmearsSuli14}
where $hp$-DG methods are considered.

The DG method for the linear problem \eqref{CordesMethod}
extends to the nonlinear problem \eqref{eq:HJBbvp}
by essentially taking the supremum over $\mathcal{A}$;
this yields the following scheme: Find $u_h\in \dg$
such that 
\begin{align}\label{eqn:SSHJBMethod}
G_h^{dg}(u_h, v_h):=\sum_{K\in \mct} \int_K \big[F_\gamma[u_h]-\Delta u_h\big]\Delta v_h + \frac12 B_h(u_h,v_h)=0,
\end{align}
where $B_h(\cdot,\cdot)$ is the bilinear form 
defined by \eqref{eqn:BhDefinition} with penalty parameter $\mu>0$.
Recall, from Section \ref{sec:DiscreteCordes},
that if $u\in H^s(\Omega)\cap H^1_0(\Omega)$ with $s>5/2$, then
$B_h(u,v_h)= 2\sum_{K\in \mct} \int_K \Delta u\Delta v_h$ for all $v_h\in \dg$.
Therefore \eqref{eqn:SSHJBMethod} is a consistent method, i.e.,
$G_h^{dg}(u,v_h)=0$ for all $v_h\in \dg$ provided $u$ is sufficiently
smooth.  In addition, as in the linear case, 
the coercivity properties of $B_h$ ensure
that the structural properties found at the continuous level 
carry over; namely, we have the following result; see \cite[Theorem 7]{SmearsSuli14}.
%
\begin{thm}[properties of $G_h^{dg}$]
\label{thm:DCordesHJBMono}
Suppose
that the hypotheses in Theorem \ref{thm:Cproperties}
are satisfied.  Then there exists $\mu_* = \mathcal{O}(\eps^{-1})$
such that if $\mu\ge \mu_*$, there holds
\begin{align*}
C\|v_h-w_h\|_{DG(1)}^2 \le G_h^{dg}(v_h,v_h-w_h)-G_h^{dg}(w_h,v_h-w_h)\quad \forall v_h,w_h\in \dg,
\end{align*}
where the discrete $H^2(\Omega)$ norm $\|\cdot\|_{DG(1)}$
is defined by \eqref{eqn:DGTNorm}.
Moreover, there exists a constant $C>0$ such that
for all $v_h,w_h,u_h\in \cg$,
\begin{align*}
\big|G_h^{dg}(v_h,u_h)-G^{dg}_h(w_h,u_h)\big|\le C\|v_h-w_h\|_{DG(1)}\|u_h\|_{DG(1)}.
\end{align*}
\end{thm}
Similar to the continuous setting, 
Theorem~\ref{thm:DCordesHJBMono} 
yields existence and uniqueness for the DG method.

\begin{col}[existence and uniqueness]
If that the hypotheses of Theorem
\ref{thm:DCordesHJBMono} hold,
then there exists a unique $u_h\in \dg$
satisfying \eqref{eqn:SSHJBMethod}.
\end{col}

Finally, the stability/monotonicity result and the consistency of the method
lead to the following error estimates.  We refer to \cite[Theorem 8]{SmearsSuli14}
for a proof.

\begin{thm}[rates of convergence]
Suppose that the hypotheses of Theorem
\ref{thm:DCordesHJBMono} hold.  Let
$u_h\in \dg$ be the solution to \eqref{eqn:SSHJBMethod},
and suppose
 that the solution to
 \eqref{eq:HJBbvp}
has regularity $u\in H^s(\Omega)$ with $5/2<s\le k+1$.
Then the error satisfies 
\begin{align*}
\|u-u_h\|_{DG(1)}\le  C h^{s-2}\|u\|_{H^s(\Omega)}.
\end{align*}
\end{thm}

\subsubsection{Finite element methods for parabolic isotropic HJB problems}

In this section we depart from the elliptic framework
and  discuss the finite element method for {\it parabolic} (time-dependent)
Hamilton-Jacobi-Bellman problems developed in \cite{JensenSmears13}.  Assuming
homogeneous Dirichlet boundary conditions, 
we consider numerical approximations of the problem:

\begin{equation}
\label{eqn:paraHJB}
  \begin{dcases}
    \frac{\p u}{\p t} - \sup_{\alpha\in \mathcal{A}} \left[ \tilde{\mathcal{L}}^\alpha u- f^\alpha \right]=0,  & \text{in }\Omega\times (0,T),\\
    u = 0, & \text{on }\p \Omega \times (0,T),\\
    u = u_0, & \text{on }\bar\Omega\times \{0\},
  \end{dcases}
\end{equation}
where $u_0\in C(\bar\Omega)$ is the initial time condition and $T$ is the (finite) end-time.
As before we assume that $\mathcal{A}$ is compact
and that the coefficients are uniformly continuous in $\bar\Omega$.  In addition,
we assume that the coefficients 
are time-independent, $f^\alpha\ge0$, $c^\alpha\le 0$,
$u_0\ge 0$, and, most importantly, the elliptic operator $\mathcal{L}^\alpha$ is isotropic; that is,
\[
\tilde{\mathcal{L}}^\alpha u(x) = a^\alpha(x)\Delta u(x) + \bb^\alpha(x) \cdot D u(x) +c^\alpha(x) u(x),
\]
with $a^\alpha \ge 0$.  Note that, even in the isotropic case, 
each elliptic operator is in nondivergence form.

Before discussing the finite element method
let us first extend the definition of viscosity solutions
to the parabolic setting (compare to Definition \ref{def:viscobvp}).
\begin{definition}[viscosity solution]
\label{def:paraViscositySoln}
We say that:
\begin{enumerate}[(a)]
\item A function $u_\star\in USC(\bar\Omega\times [0,T])$
is a viscosity subsolution to \eqref{eqn:paraHJB} 
if $u_\star\le 0$ on  $\p\Omega\times (0,T)$,
$u_\star\le u_0$ on $\bar\Omega\times \{0\}$,
and if whenever $(x_0,t_0)\in \Omega\times (0,T)$, $\varphi\in C^2((0,T)\times \Omega)$
and $u_\star-\varphi$ has a local maximum at $(x_0,t_0)$ we have
that
\[
\frac{\p \varphi}{\p t}(x_0,t_0) - \sup_{\alpha\in \mathcal{A}} \big(\tilde{\mathcal{L}}^\alpha(x_0) \varphi(x_0,t_0) - f^\alpha(x_0)\big)\le 0.
\]
\item A function $u^\star\in LSC(\bar\Omega\times [0,T])$
is a viscosity supersolution to \eqref{eqn:paraHJB} 
if $u^\star\ge 0$ on  $\p\Omega\times (0,T)$,
$u^\star\ge u_0$ on $\bar\Omega\times \{0\}$,
and if whenever $(x_0,t_0)\in \Omega\times (0,T)$, $\varphi\in C^2((0,T)\times \Omega)$
and $u_\star-\varphi$ has a local minimum at $(x_0,t_0)$ we have
that
\[
\frac{\p \varphi}{\p t}(x_0,t_0) - \sup_{\alpha\in \mathcal{A}} \big(\tilde{\mathcal{L}}^\alpha(x_0) \varphi(x_0,t_0) - f^\alpha(x_0)\big)\ge 0.
\]

\item A function $u\in C(\bar\Omega\times [0,T])$ is a viscosity solution to \eqref{eqn:paraHJB}
if it is both a sub- and supersolution.
\end{enumerate}
\end{definition}

The discretization
of the nondivergence part of $\tilde{\mathcal{L}}^\alpha$
proposed in \cite{JensenSmears13} 
is based on a ``freezing the coefficients''  
strategy.  Let $\{\phi_i\} \subset \lco$
be the normalized hat functions
defined in Section \ref{sub:MonoFEM},
where we recall that $\lco$ is the space
of piecewise linear polynomials with vanishing trace (cf.~\eqref{eqn:homoLagrange}).
Then, since $\phi_i$ is essentially
a regularized Dirac delta distribution,
if a point $x\in\Omega $ is close to a node $z_i\in \Omega_h^I$, then
we have, at least formally,
\[
a^\alpha (x)\Delta u(x)  \approx  -a^\alpha(z_i) \int_\Omega  D u\cdot  D\phi_i = a^\alpha(z_i) \Delta_h I_h^{ep} u(z_i),
\]
where the finite element Laplacian $\Delta_h$ is given by \eqref{eqn:FEMLaplacian},
and $I_h^{ep}$ is the elliptic projection.
This heuristic approximation motivates the semi-discrete finite element method: 
Find $u_h:[0,T]\to \lco$ such that $u_h(0) = I_h^{fe}u_0$ and
\begin{align}\label{eqn:HJBSemiDiscrete}
 \frac{\p u_h}{\p t} - \sup_{\alpha\in \mathcal{A}} \big(\tilde{\mathcal{L}}^\alpha_h u_h -f_h^\alpha)=0\quad \text{in }\Omega_h^I\times (0,T),
\end{align}
where $I_h^{fe}:C^0(\bar\Omega)\to \lc$ is the 
nodal interpolant onto $\lc$, and
 the discrete operator and discrete source function are given respectively by 
\begin{align*}
\tilde{\mathcal{L}}^\alpha_h u_h(z_i) 
&= a^\alpha(z_i)\Delta_h u_h(z_i)
+\int_\Omega \big(\bb^\alpha \cdot Du_h+ c^\alpha u_h\big)\phi_i,\\
f^\alpha_h(z_i) &= \int_\Omega f^\alpha \phi_i.
\end{align*}

The fully discrete method proposed in \cite{JensenSmears13}
applies a one-step explicit-implicit ODE solver
to a regularized version of \eqref{eqn:HJBSemiDiscrete}.
Let $\tau\in (0,1)$ be the (uniform) time step size
and assume that $T/\tau=:M\in \bbN$.  For a sequence
of discrete functions $\{v_h^k\}_{k=0}^M\subset \lco$, we
define the backward
difference quotient  as
\begin{align*}
d_\tau v_h^{k+1}:= \frac1\tau\big(v_h^{k+1}-v_h^k\big)\in \lco.
\end{align*}

For each $\alpha\in \mathcal{A}$, the discrete operator
$\tilde{\mathcal{L}}^\alpha_h$ is approximately split
into an explicit and implicit part:
\begin{equation}
\label{eqn:LhBroken}
  \tilde{\mathcal{L}}^\alpha_h \approx \mathcal{E}_h^\alpha + \mathcal{I}_h^\alpha,
\end{equation}
with
\begin{align*}
\mathcal{E}^\alpha_h u_h(z_i) &= a^\alpha_e(z_i)\Delta_h u_h(z_i) + \int_\Omega \big(\bb^\alpha_e \cdot D u_h +c_e u_h\big) \phi_i,\\
\mathcal{I}^\alpha_h u_h(z_i) &= a^\alpha_i(z_i) \Delta_h u_h(z_i) + \int_\Omega \big(\bb^\alpha_i \cdot D u_h +c_i u_h\big) \phi_i.
\end{align*}
We now consider the fully discrete method:
Find the sequence $\{u_h^k\}_{k=0}^M \subset \lco$
with $u_h^0 = I_h^{fe} u_0$ and
\begin{align}\label{eqn:SmearsHJBMethod}
d_{\tau} u_h^{k+1} -\sup_{\alpha\in \mathcal{A}} \big(\mathcal{E}^\alpha_h u_h^k
+\mathcal{I}^\alpha_h u_h^{k+1} - f^\alpha_h\big)\quad \text{in }\Omega_h^I.
\end{align}

To show the well-posedness 
and to analyze method \eqref{eqn:SmearsHJBMethod}
we make several assumptions.
First we quantify the approximation in \eqref{eqn:LhBroken}
and make assumptions of the coefficients in the explicit
and implicit operators.

\begin{ass}[coefficients]\label{ass:coefficientsGood}
We assume that the coefficients 
satisfy
\begin{align*}
\lim_{h\to 0} \sup_{\alpha\in \mathcal{A}} &\big(\sup_{z\in \Omega_h^I} \|a^\alpha - (a_e^\alpha(z)+a_i^\alpha(z))\|_{L^\infty(\omega_{z})}\\
&\qquad + \|\bb^\alpha - (\bb_e^\alpha + \bb_i^\alpha)\|_{L^\infty(\Omega)}
+ \|c^\alpha- (c^\alpha_e+c^\alpha_i)\|_{L^\infty(\Omega)} \big)=0.
\end{align*}
In addition, in the case when $\mathcal{A}$ is not finite, we also assume that the mappings
 $\alpha \mapsto (a^\alpha_e,\bb^\alpha_e,c^\alpha_i)$
and $\alpha \mapsto (a^\alpha_i,\bb^\alpha_i,c^\alpha_i)$ are continuous,
$a_e,a_i\ge 0$, and $c_e,c_i\le 0$.
\end{ass}

Next we make assumptions regarding 
the monotonicity properties of the operators.
\begin{ass}[monotonicity]
\label{ass:EIMonotone}
The splitting \eqref{eqn:LhBroken} is such that:
\begin{enumerate}[(a)]
\item The time-step explicit operators satisfy
\begin{align}\label{eqn:EhAssumption}
\delta_{i,j}+\tau \mathcal{E}_h^{\alpha}(\tilde{\phi}_j)(z_i)\ge 0\quad \forall \alpha\in \mathcal{A},
\end{align}
where $\{\tilde{\phi}_j\}\subset \lco$ are the (unnormalized) hat functions 
satisfying $\tilde{\phi}_j(z_i) = \delta_{i,j}$.

\item The operator $\mathcal{I}^\alpha_h$ is monotone (in the sense
of Definition \ref{def:discreteMonotone}) for each $\alpha\in \mathcal{A}$.
\end{enumerate}
\end{ass}

Condition \eqref{eqn:EhAssumption}
is essentially a time-step restriction and a monotonicity condition
(if $\mathcal{E}_h^\alpha\not\equiv 0$ for all $\alpha\in\mathcal{A}$).
For example, if $c^\alpha_e\equiv 0$,  $\bb^\alpha_e\equiv 0$, and $a_e^\alpha>0$
for all $\alpha\in \mathcal{A}$, then the condition reads
\begin{align*}
\delta_{i,j}-\tau a^\alpha_e(z_i) \int_\Omega D \tilde{\phi}_j\cdot D\phi_i\ge 0.
\end{align*}
Therefore, in this setting, \eqref{eqn:EhAssumption} is satisfied
if and only if \eqref{eqn:FEMMMatrix} holds (so that $\Delta_h$ is monotone)  and if 
\begin{align*}
\tau \le \Big(a^\alpha_e(z_i)\int_\Omega D \phi_i \cdot D \tilde{\phi}_i\Big)^{-1}=  \mathcal{O}(h^{2}).
\end{align*}

The monotonicity of $\mathcal{I}_h^\alpha$ implies that
the inverse of $(I-\tau \mathcal{I}_h^\alpha)$ is sign preserving for any $\tau>0$, i.e., 
if $(I-\tau \mathcal{I}_h^\alpha)v_h\le 0\ (\ge 0)$, then $v_h\le 0\ (\ge 0)$.  To see this,
suppose that $v_h\in \lco$ satisfies $(I-\tau \mathcal{I}_h^\alpha)v_h\le 0$
and $v_h$ has a positive maximum at $z_i$.  Then
the monotonicity of $\mathcal{I}_h^\alpha$ implies (cf.~Definition \ref{def:discreteMonotone})
$\mathcal{I}_h^\alpha v_h(z_i)\le 
0$, leading to $(I-\tau \mathcal{I}_h^\alpha)v_h(z_i) \ge v_h(z_i)>0$, a contradiction.
This argument also shows 
that $(I-\tau\mathcal{I}_h^\alpha)$ is invertible, and therefore the following problem
is well-posed:
For fixed $\alpha\in \mathcal{A}$, find $\{u_{h,\alpha}^k\}_{k=0}^M\subset \lco$
such that $u_{h,\alpha}^0 = I_h^{fe} u_0$ and $(k=0,1,\ldots,M-1)$
\begin{align}\label{eqn:FixedAlphaResult}
d_{\tau} u_{h,\alpha}^{k+1} -\big(\mathcal{E}^\alpha_h u_{h,\alpha}^k
+\mathcal{I}^\alpha_h u_{h,\alpha}^{k+1} - f_h\big)\quad \text{in }\Omega_h^I,\quad 
u_{h,\alpha}^0 = I_h^{fe}u_0.
\end{align}

The construction of 
operators $\mathcal{E}_h^\alpha$, $\mathcal{I}^\alpha_h$ 
satisfying  Assumptions \ref{ass:coefficientsGood}--\ref{ass:EIMonotone} (on weakly
acute triangulations)
using the method of artificial diffusion can be found in 
\cite[Section 8]{JensenSmears13} and \cite{JensenSmears13B}.

\begin{thm}[existence]\label{thm:HJBSmearsExistence}
Suppose that Assumptions \ref{ass:coefficientsGood}
and \ref{ass:EIMonotone} are satisfied.  Then
for each $k\in \{0,1,\ldots,M-1\}$ there
exists a unique solution $u_h^{k+1}\in \lco$ to \eqref{eqn:SmearsHJBMethod}.
Moreover, if we denote by $\{u_{h,\alpha}^k\}\subset \lco$ 
the solution to \eqref{eqn:FixedAlphaResult}, then
$0\le u_h^k\le u_{h,\alpha}^k$. 
\end{thm}

The  proof of existence and uniqueness 
follows from the convergence results of Howard's method
which is discussed in the next section.  The monotonicity
and continuity properties of $\mathcal{E}_h^\alpha$, $\mathcal{I}_h^\alpha$
ensure that the hypotheses of Theorem~\ref{thm:convHoward} below
are satisfied (\cf~\cite[Theorem 3.1]{JensenSmears13}).

The next result establishes the stability of the method.
Its proof essentially follows from Assumption \ref{ass:EIMonotone}
and the nonpositivity of $c_e^\alpha$ and $c_i^\alpha$
(see \cite[Lemma 3.2 and Corollary 3.3]{JensenSmears13} for details).

\begin{lem}[stability]\label{lem:SmearsStability}
Suppose that the assumptions of Theorem \ref{thm:HJBSmearsExistence}
hold, and let $\{u_{h,\alpha}^k\}\subset \lco$ denote
the solution to \eqref{eqn:FixedAlphaResult} for a fixed $\alpha\in \mathcal{A}$.  
Then we have
\begin{align*}
   \| \{u_{h,\alpha}^k\}_{k=0}^M \|_{\ell^\infty(\bbN\cap [0,M], L^\infty(\Omega))}  \le \|I_h^{fe}u_0\|_{L^\infty(\Omega)} +T \|f^\alpha\|_{L^\infty(\Omega)}.
\end{align*}
Therefore, by Theorem \ref{thm:HJBSmearsExistence}, 
\begin{align*}
   \| \{u_{h}^k\}_{k=0}^M \|_{\ell^\infty(\bbN\cap [0,M], L^\infty(\Omega))}  \le \|I_h^{fe}u_0\|_{L^\infty(\Omega)} +T\inf_{\alpha\in \mathcal{A}} \|f^\alpha\|_{L^\infty(\Omega)}.
\end{align*}
\end{lem}

Next, to apply the Barles-Souganidis theory,
we look at the consistency of the method.
Essentially this result follows from Lemma \ref{lem:EllipticProjectionConsistency}
and the stability of the elliptic projection.
\begin{lem}[consistency]\label{lem:SmearsConsistency}
Let  $I_h^{ep}:H^1(\Omega)\to \lc$ be
the elliptic projection onto $\lc$.  Let
$\{(z_h,t_{\tau})\}_{h>0,\tau>0}$ with
$(z_h,t_\tau)\in \bar\Omega_h\times \bbN\cap [0,M]$
and $(z_h,t_\tau)\to (z_0,t_0)\in \bar\Omega\times [0,T]$
as $h,\tau\to 0^+$.
Then for any $\phi\in C^2(\bar\Omega\times [0,T])$,
\begin{align*}
&\lim_{h,\tau\to 0^+} \Big( d_\tau I_h^{ep} \phi(t_\tau+\tau)- 
\big(\mathcal{E}_h^\alpha I_h^{ep} (\phi(t_\tau))+\mathcal{I}_h^\alpha I_h^{ep}(\phi(t_\tau+\tau))-f_h^\alpha\big)\Big)(z_h)\\
&\qquad =\frac{\p \phi}{\p t}(z_0,t_0)- \big(\tilde{\mathcal{L}}^\alpha \phi(z_0,t_0)-f^\alpha(z_0)\big)\quad \forall \alpha\in \mathcal{A},
\end{align*}
where $\phi(t):=\phi(\cdot,t)$, and the convergence is
uniform with respect to $\alpha\in \mathcal{A}$.
\end{lem}

\begin{rem}[anisotropy]
Note that the consistency result given in Lemma \ref{lem:SmearsConsistency}
does not extend to operators with anisotropic diffusion.
\end{rem}

Finally to apply Theorem \ref{thm:BSTHM1Alt},
the following assumption is made
which ensures that the limiting solution
satisfies the boundary conditions in a classical sense.
\begin{ass}[limiting boundary values]\label{ass:bndryValues}
Define $\bar{u}_\alpha\in USC(\bar\Omega\times [0,T])$
\begin{align*}
\bar{u}_\alpha(x,t) := \mathop{\limsup_{(z,k\tau)\to (x,t)}}_{h,\tau\to 0^+} u^k_{h,\alpha}(z),
\end{align*}
where $u^k_{h,\alpha}$ is the solution to \eqref{eqn:FixedAlphaResult}.
Then
\begin{align*}
\inf_{\alpha\in \mathcal{A}} \bar{u}_\alpha(x,t)=0\qquad \forall (x,t)\in \p\Omega\times [0,T].
\end{align*}
\end{ass}

Finally combining 
Lemmas \ref{lem:SmearsConsistency} and
\ref{lem:SmearsStability} with Theorem
\ref{thm:BSTHM1Alt} yields the following convergence
result; see \cite[Theorem 6.2]{JensenSmears13}.

\begin{thm}[convergence]
Suppose that Assumptions  \ref{ass:coefficientsGood},
 \ref{ass:EIMonotone}, and \ref{ass:bndryValues}
 are satisfied.  Suppose further that the HJB
 problem \eqref{eqn:paraHJB} satisfies
 a comparison principle (cf.~Definition \ref{def:comparison}).
 Then there holds
 \begin{align*} 
 \lim_{h,\tau\to 0^+} \max_k \|u(\cdot,k\tau)-u_h^k\|_{L^\infty(\Omega)}=0.
 \end{align*}
\end{thm}

%
%
%
%
%
%

\subsection{Solution of the discrete problems}
\label{sub:solschemes}

To finalize the discussion concerning the approximation of convex equations we must, at least briefly, 
discuss how to solve the system of nonlinear equations that arises after discretization, 
be it by finite differences or finite elements. We will present one of the most popular methods --- known as policy iterations or Howard's algorithm --- and discuss its convergence properties. 

Once \eqref{eq:HJBbvp} is discretized by any of the methods described in the previous sections, we end up with the system of nonlinear equations: Find $\bx \in \Real^N$ that satisfies
\begin{equation}
\label{eq:HJBdiscrete}
  \bF(\bx) = \sup_{\alpha \in \calA} \left[ \bK^\alpha \bx - \bef^\alpha \right] = \boldsymbol0,
\end{equation}
where the supremum is taken component-wise, $N$ is the number of degrees of freedom in the discretization, $\{\bK^\alpha\}_{\alpha \in \calA}$ and $\{\bef^\alpha\}_{\alpha \in \calA}$ are discretizations of $\{ \tilde \calL^\alpha\}_{\alpha \in \calA}$ and $\{f^\alpha \}_{\alpha \in \calA}$, respectively, or as in \eqref{eqn:SmearsHJBMethod}, come from the implicit-explicit splitting of the operators \eqref{eqn:LhBroken}.

In order to define Howard's algorithm we must introduce the following operations on $\calA$, the matrices $\bK^\alpha$ and vectors $\bef^\alpha$. Given $\by \in \Real^N$ and $i \in \{1, \ldots, N\}$ we define the element $\alpha(\by,i) \in \calA$ as the element that realizes the supremum in \eqref{eq:HJBdiscrete} when applied to $\by$ at component $i$, that is
\[
  \left[ \bK^{\alpha(\by,i)} \by - \bef^{\alpha(\by,i)} \right]_i = \sup_{\alpha \in \calA} \left[ \bK^\alpha \by - \bef^\alpha \right]_i.
\]
We define $\balpha(\by) \in \calA^N$ with components $\balpha(\by)_i = \balpha(\by,i)$. Given $\balpha(\by)$ we finally define the matrix $\bK^{\balpha(\by)}$ and vector $\bef^{\balpha(\by)}$ as follows:
\begin{equation}
\label{eq:defofweirdmatrixvecs}
  \bK^{\balpha(\by)}_{i,j} = \bK^{\balpha(\by,i)}_{i,j}, \qquad \bef^{\balpha(\by)}_i = \bef^{\balpha(\by,i)}_i.
\end{equation}
An illustration of this selection procedure in the case $\#\calA = 2$ is shown in Figure~\ref{fig:Howardpic}. With these operations at hand Howard's algorithm, as presented in \cite{BMSZ09}, is described in Algorithm~\ref{alg:Howard}.

\begin{figure}
\label{fig:Howardpic}
  \begin{center}
     \includegraphics[scale=0.25]{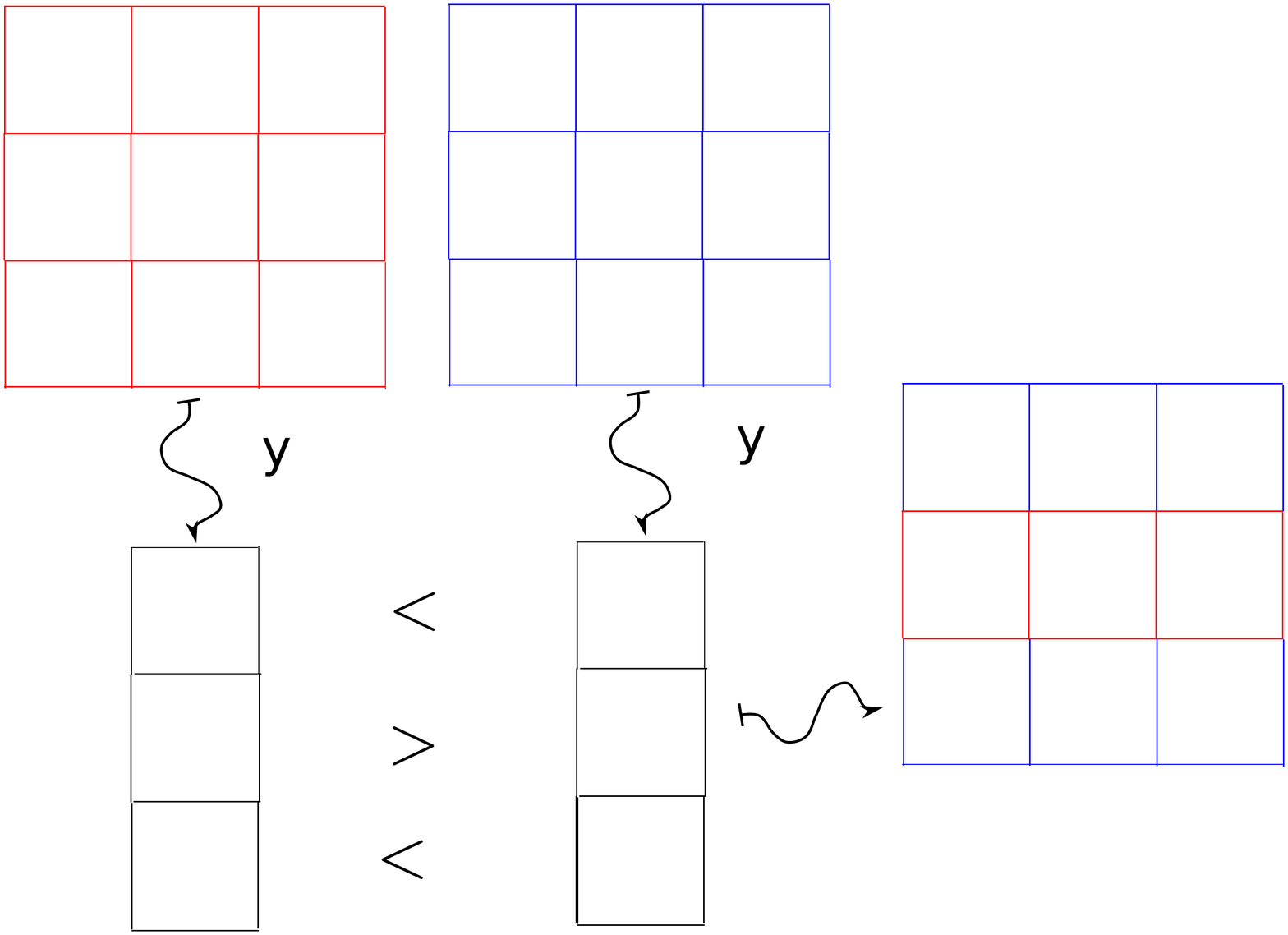}
  \end{center}
  \caption{An illustration of the selection procedure defined in \eqref{eq:defofweirdmatrixvecs}, for the simplified case of $\bef^\alpha = \boldsymbol0$ for all $\alpha \in \calA$ and $\# \calA = 2$. We multiply each matrix (\textcolor{red}{\texttt{red}} and \textcolor{blue}{\texttt{blue}} in the figue) by the vector $\by$ and compare the components of the results. The matrix $\bK^{\balpha(\by)}$ is constructed by choosing the row of the matrix that gives the largest result. }
\end{figure}

\begin{algorithm}
  \SetKwInOut{Input}{input}
  \SetKwInOut{Output}{output}
  \SetKw{Return}{return}
  
  \Input{Set $\calA$. \\
          Matrices $\{\bK^\alpha\}_{\alpha \in \calA} \subset \Real^{N\times N}$. \\
          Right hand sides $\{\bef^\alpha\}_{\alpha \in \calA} \subset \Real^N$.}
  
  \Output{Vector $\bx \in \Real^N$, solution of \eqref{eq:HJBdiscrete}.}
  \BlankLine
  Initialization: Choose $\bx_{-1} \in \Real^N$ \;
  \BlankLine
  \For{$k\geq0$}{
    Set $\balpha_k=\balpha(\bx_{k-1})$ \;
    Compute $\bK^{\balpha_k} \in \Real^{N \times N}$  and $\bef^{\balpha_k} \in \Real^N$ \;
    Find: $\bx_k \in \Real^N$ that solves 
    \begin{equation}
    \label{eq:HowardSolvingStep}
      \bK^{\balpha_k} \bx_k = \bef^{\balpha_k}\text{\,\;}
    \end{equation}
    \If{ $k\geq 1$ {\rm\textbf{and}} $\bx_k = \bx_{k-1}$ }{
       \Return{$\bx_k$} \;
    }
  }

  \caption{Howard's algorithm.}
  \label{alg:Howard}
\end{algorithm}

Notice that, at each step, Algorithm~\ref{alg:Howard} requires the solution to the linear system \eqref{eq:HowardSolvingStep}. Because of the way the matrices are constructed it is necessary to ensure that they are invertible. This is guaranteed by the following condition.

\begin{ass}[monotonicity]
\label{ass:howardmonotone}
For every $\by \in \Real^N$ the matrix $\bK^{\balpha(\by)}$ is monotone in the sense of Definition~\ref{def:discreteMonotone}.
\end{ass}
It is natural to wonder when such a condition is satisfied. As shown in \cite{MR2914277} if, for every $\alpha \in \calA$, the matrix $\bK^\alpha$ is strictly diagonally dominant and monotone and its off diagonal entries are nonpositive, then this condition is satisfied. We comment finally that a monotone matrix is nonsingular.

Having guaranteed that Howard's algorithm can proceed, we focus our attention on its convergence. While many references have studied it, we will follow \cite{BMSZ09,MR1972219} and draw a connection between this method and active set strategies which, in turn, can be analyzed as semismooth Newton methods. We begin with the definition of slant derivative.

\begin{definition}[slant derivative]
Let $X$, and $Z$ be Banach spaces and $D \subset X$ be open. The mapping $F : D \to Z$ is called {\it slantly differentiable} in the open subset $U \subset D$ if there is a family of mappings $G: U \to \mathfrak{L}(X,Z)$ such that, for every $x \in U$,
\[
  \lim_{h \to 0} \frac1{\| h \|_X} \left \| F(x+h) - F(x) - G(x+h)h \right\|_Z = 0.
\]
We call $G$ a {\it slanting function} for $F$.
\end{definition}

To approximate $x^\star \in X$, the solution of $F(x)=0$, one can then apply the semismooth Newton method which, starting from $x_0 \in D$, computes iterates via:
\begin{equation}
\label{eq:semismoothNewton}
  x_{k+1} = x_k - G(x_k)^{-1} F(x_k).
\end{equation}

The convergence of \eqref{eq:semismoothNewton} is given in the following result, whose proof can be found in \cite[Theorem 1.1]{MR1972219}.

\begin{thm}[convergence]
\label{thm:convNewton}
Assume that $F$ is slantly differentiable in a neighborhood of $x^\star$. If $G(x)$ is nonsingular for all $x \in U$ and $\{ \| G(x)^{-1} \|: x \in U \}$ is bounded, then the semismooth Newton method \eqref{eq:semismoothNewton} converges superlinearly provided $x_0$ is sufficiently close to $x^\star$.
\end{thm}

The connection between Howard's algorithm and semismooth Newton methods is given by the following result; see \cite[Lemma 3.1]{MR1972219}.

\begin{lem}[slant derivative of \boldsymbol{$\max$}] 
\label{lem:maxslantD}
The mapping 
\[
  \bM: \Real^N \ni \by \mapsto \max\{\boldsymbol0,\by\} \in \Real^N,
\]
is slantly differentiable on $\Real^M$ and a slanting function is
\[
  \bG_\bM(\by)_{i,j} = g(\by_j) \delta_{i,j}, \quad 
  g(z) = \begin{dcases}
           0, & z \leq 0, \\
           1, & z > 0.
         \end{dcases}
\]
\end{lem}

Let us explain how this result relates to Howard's algorithm in the (perhaps overly) simplistic case that 
$N=1$, 
$\calA = \{1,2\}$ and $\bef^1 = \bef^2 = \bef$. In this setting, we are looking for $\bx \in \Real$ that solves
\[
  \boldsymbol0 = \bF(\bx) = \max\{ \bK^1 \bx, \bK^2 \bx \} - \bef = \bK^1 \bx + \max\{ \boldsymbol0, (\bK^2-\bK^1) \bx \} - \bef.
\]
By Lemma~\ref{lem:maxslantD} a slanting function for $\bF$ at $\by$ is
\[
  \bG(\by) = \bK^1 + \bG_\bM\left( (\bK^2-\bK^1) \by \right)(\bK^2-\bK^1).
\]
If $(\bK^2-\bK^1) \by \geq 0$, then
\[
  \bG(\by) = \bK^1 + (\bK^2-\bK^1) = \bK^2,
\]
and, similarly, if $(\bK^2-\bK^1) \by \leq 0$ we get $\bG(\by) = \bK^1$. In other words, $\bG(\by)$ 
always coincides with the coefficient that gives the largest result. Having made this observation, we can now establish convergence for Howard's algorithm \cite{BMSZ09}.

\begin{thm}[convergence]
\label{thm:convHoward}
If Assumption~\ref{ass:howardmonotone} is satisfied, then the sequence $\{\bx_k\}_{k \geq 0}$, generated by Algorithm~\ref{alg:Howard}, satisfies $\bx_k \geq \bx_{k+1}$ for every $k\geq 0$ and converges to $\bx$, the solution of \eqref{eq:HJBdiscrete}. Moreover,
\begin{enumerate}[1.]
  \item If $\#\calA$ is finite, the sequence converges in at most $(\#\calA)^N$ steps.
  \item If $\calA$ is infinite, the convergence is asymptotically superlinear.
\end{enumerate}
\end{thm}
\begin{proof}
To draw a connection between Algorithm~\ref{alg:Howard} and the semismooth Newton method \eqref{eq:semismoothNewton} let us sketch the proof. Recall that we are looking for a solution of \eqref{eq:HJBdiscrete}, which can be understood as looking for a zero of the map $\bF$. Let us now, following \cite[Theorem 3.8]{BMSZ09}, show that for $\by \in \Real^N$ the matrix $\bK^{\balpha(\by)}$ is a slanting function for $\bF$ at the point $\by$.
For if that is the case, then \eqref{eq:semismoothNewton} can be rewritten as
\[
  \bK^{\balpha(\bx_k)} \left( \bx_{k+1} - \bx_k \right) = -\bF(\bx_k) = -\bK^{\balpha(\bx_k)} \bx_k + \bef^{\balpha(\bx_k)},
\]
which is clearly \eqref{eq:HowardSolvingStep}. By invoking Theorem~\ref{thm:convNewton} this will yield the superlinear convergence of $\{\bx_k\}_{k\geq 0}$.

Now, to show the slant differentiability of $\bF$, let $\by, \bh \in \Real^N$ and notice that
\begin{align*}
  \bF(\by) + \bK^{\balpha(\by + \bh)} \bh &\geq \bK^{\balpha(\by+\bh)} \by - \bef^{\balpha(\by+\bh)} + \bK^{\balpha(\by + \bh)} \bh = \bF(\by + \bh) \\
  &\geq \bK^{\balpha(\by)}( \by + \bh) - \bef^{\balpha(\by)} = \bF(\by) + \bK^{\balpha(\by)}\bh.
\end{align*}
Consequently,
\begin{align*}
  \boldsymbol0 &\geq \bF(\by+\bh) - \bF(\by) - \bK^{\balpha(\by+\bh)} \bh \geq \left[ \bK^{\balpha(\by)} - \bK^{\balpha(\by+\bh)} \right] \bh \\
  &\geq - | \bK^{\balpha(\by)} - \bK^{\balpha(\by+\bh)} |_{\infty} | \bh |_{\infty}.
\end{align*}
Since Lemma 3.2 of \cite{BMSZ09} shows that $\lim_{ | \bh |_{\infty} \to 0} | \bK^{\balpha(\by)} - \bK^{\balpha(\by+\bh)} |_{\infty} = 0$ the result follows.

In the case $\#\calA$ is finite we remark that the convergence in a finite number of steps follows from the monotonicity of the iterates. To show the monotonicity, we exploit the construction of the parameters $\balpha_k$. Indeed, since
\[
  \bK^{\balpha_{k+1}} \bx_k - \bef^{\balpha_{k+1}} = \bF(\bx_k ) \geq \bK^{\balpha_k} \bx_k - \bef^{\balpha_k} = \boldsymbol0
  = \bK^{\balpha_{k+1}} \bx_{k+1} - \bef^{\balpha_{k+1}},
\]
we conclude that $\bK^{\balpha_{k+1}} \left( \bx_k - \bx_{k+1} \right) \geq 0$. From Assumption~\ref{ass:howardmonotone} we have that $\bK^{\balpha_{k+1}}$ is monotone and, therefore, $\bx_k \geq \bx_{k+1}$. This concludes the proof.
\end{proof}

We conclude by commenting that, in Algorithm~\ref{alg:Howard}, at every step it is necessary to solve the linear system of equations \eqref{eq:HowardSolvingStep} which, for large $N$, can be very time consuming. We refer to \cite{MR848524} where a multilevel technique to solve this problem is described and its global convergence is shown. Variations of this multigrid strategy are explored in \cite{MR3301311} and \cite{MR3120775}. Numerical methods of penalty type are studied in \cite{MR2914277} and \cite[Algorithme III]{MR596541}.

\section{The Monge-Amp\`ere equation}\label{sec:MongeAmpere}
{The focus of this section is the study of discretizations for a particular convex equation}: the \MA equation 
with Dirichlet boundary conditions:
\begin{equation}
\label{MA}
  \begin{dcases}
    \det(D^2 u)  = f, & \text{in }\Omega,\\
    u = g, & \text{on }\p\Omega.
  \end{dcases}
\end{equation}
The \MA equation finds relevance 
 in several areas of mathematics and its applications, 
including differential geometry, 
calculus of variations, economics, meteorology, and the optimal mass transportation problem
\cite{Gutierrez01,TrudyWang05,Villani03,BenyB98,BenyB00,Stoj05}.
As we have already mentioned, see also Remark~\ref{rmk:determinantconvex} below, the \MA problem fits into the framework of the previous section.
However, due to its many applications, and correspondingly its many discretizations and convergence results, we find it appropriate to devote an entire section to a sample of detailed results. Before discussing these methods, let us here give a brief synopsis of some discretizations for the \MA problem
that we do not cover in detail.

\begin{enumerate}[$\bullet$]

\item {\it Monotone, wide-stencil finite difference schemes:} \cite{Oberman08,FroeseOberman11,Froese16,BFA14}
The basis of these methods is to use determinant identities (such as Hadamard's inequality)
to construct consistent and monotone finite difference schemes for the \MA problem.
Such schemes conform to the Barles-Souganidis theory under certain assumptions
and convergence to the viscosity solution is immediate.  We refer the reader
to \cite{NochettoZhang16A} where rates of convergence for these methods 
 have recently been derived.

\item {\it Regularization:}  In \cite{FengNeilan09,FengNeilan09b,Neilan10,FengNeilan11}
the \MA equation is first approximated at the continuous level
by a fourth-order quasi-linear PDE. In this framework, solutions are defined
via variational principles, and therefore Galerkin methods are readily constructed.
Error estimates between the finite element approximation
and the regularized solution are derived, but a general convergence 
theory for viscosity solutions is an open problem.

\item {\it Variational methods:}  Classical Galerkin methodologies are used in references like
\cite{Bohmer08,BGNS11,BrennerNeilan12,Neilan13,Awanou15,Awanou15b,Awanou15c}
to construct finite element approximations for strong or classical solutions 
of the \MA equation. Assuming that the solution to \eqref{MA} is sufficiently
regular error estimates are typically derived in $H^1$ or $H^2$-type norms.

\item {\it Optimal control:}
In \cite{GlowinskiDean06B,GlowinskiDean06A}, the \MA equation is treated as a constraint 
of a minimization  problem, leading to a saddle-point reformulation.   Computationally
efficient algorithms are proposed, but convergence to the viscosity solution 
is not shown. 

\end{enumerate}

We refer
to the review paper \cite{FGN13} for a summary 
of these discretization types.

\subsection{Solution concepts}

The \MA equation has a very unique structure and, because of this, the solution to problem \eqref{MA} 
can be understood not only with the concepts described in Section~\ref{sec:PDEs}, but with some additional ones. To discuss these issues, we first make the (sometimes necessary) assumption that the domain $\Omega$ 
is convex.  With this assumption, and with an abuse of notation,
we (re)define the convex envelope of a function $v$ with $v=0$ on $\p\Omega$ as
\[
  \Gamma (v)(x) := \sup \left\{ L(x):  \ L(z) \leq v(z) \ \forall z \in \dm, \ L\in \mathbb{P}_1 \right\}.
\]

As we previously described, in contrast to the problems discussed
thus far, 
the \MA operator $F(x,r,\bp,M):=\det(M)-f(x)$
is {\it not} elliptic on $\Omega\times \bbR\times \bbR^d\times \bbS^d$
(cf.~Definition \ref{def:FLelliptic} and Example \ref{ex:MongeAmpere}).  Indeed, there holds
\begin{align*}
\det(M)\le \det(N)\qquad \forall N\ge M
\end{align*}
if and only if $M\ge 0$, and consequently 
$F$ is only elliptic restricted to the class of semi-positive definite matrices.
In particular, if a classical solution $u$ to \eqref{MA} exists,
then $F$ is elliptic at $u$ if and only if $D^2 u\ge 0$ in $\Omega$,
implying that $u$ is a convex function.  
Further note that, besides ensuring ellipticity, the convexity
assumption is required to have any hope of uniqueness, since, when $d$ is even
\begin{align*}
\det(D^2 (-u)) = \det(D^2 u).
\end{align*}

Let us now discuss the existence and uniqueness of different types of solutions of the \MA problem \eqref{MA}
and the required conditions on the domain $\Omega$
and data, $f$ and $g$, to guarantee that such solutions exist.
First we have the following 
result regarding classical solutions.
\cite{Wang96,TrudWang08}.
\begin{thm}[classical solutions]
\label{thm:MAClassical}
Suppose that $\Omega$ is an open, bounded,
and strictly convex domain with boundary 
$\p\Omega \in C^3$.  Suppose that $g\in C^3(\bar\Omega)$,
 $f\in C^\alpha(\Omega)$ for some $\alpha\in (0,1)$,
 and that $\min_{\bar\Omega} f>0$.  
 Then there exists a unique 
 (classical) convex solution $u\in C^{2,\alpha}(\Omega)$
 satisfying \eqref{MA}.
\end{thm}

\begin{rem}[optimality]
The conditions on $\p\Omega$, $g$, and $f$ stated in 
Theorem \ref{thm:MAClassical} are optimal \cite{TrudWang08}.
\end{rem}

Next we turn our attention to 
viscosity solutions to \eqref{MA}.  
Since the operator $F$ is only elliptic 
on the class of convex functions, we must
modify the definition of viscosity solutions to reflect
this fact.

\begin{definition}[viscosity solution]
\label{def:MAViscosity}
A function $u\in C(\bar\Omega)$
is called a viscosity subsolution (resp., supersolution) of \eqref{MA}
on the set of convex functions if $u$ is convex and if for all convex
$\varphi\in C^2(\Omega)$ such that $u-\varphi$ has a local 
maximum (resp., minimum) at $x_0\in \Omega$, we have
$\det(D^2 \varphi(x_0))\ge f(x_0)$ (resp., $\det(D^2 \varphi(x_0))\le f(x_0)$).
A function $u\in C(\bar\Omega)$ is a viscosity solution 
to \eqref{MA} on the class of convex functions if it is simultaneously a viscosity
subsolution and super solution on the set of convex functions.
\end{definition}

\begin{rem}[$f$ is nonnegative]
Note that the definition of viscosity solution implicitly requires
the source function $f$ to be continuous and nonnegative.  
\end{rem}

Before stating  existence and uniqueness results for viscosity
solutions let us, following \cite{FroeseOberman11}, give an example of a non-classical
viscosity solution.  Note that this example shows
that $f\in C^{\alpha}(\Omega)$ and smooth $\Omega$
is not sufficient to guarantee the existence of a classical solution.
%
%


\begin{ex}[viscosity solution]
\label{ex:viscositysolution}
  Let $\dm = B_2 (0) \subset \mathbb R^2$, and consider the function 
  \begin{align}\label{example2}
u(x)  = \left\{
\begin{array}{ll}
0  \quad &\text{in $|x| \leq 1$},  
\\
\frac 12 \big( |x| -1 \big)^2 \quad &\text{in $ 1 \leq |x| \leq 2$}.
\end{array} \right.
\end{align}
Then $u$ is a viscosity solution of \eqref{MA} with $g=1/2$ and
\begin{align*}
f(x) = \left\{
\begin{array}{ll}
0  \quad &\text{in $|x| \leq 1$},  
\\
1 - |x|^{-1} \quad &\text{in $ 1 \leq |x| \leq 2$}.
\end{array} \right.
\end{align*}
\end{ex}

\begin{thm}[existence]
Assume that $\Omega$ is convex. Assume further that $g\in C(\p\Omega)$, $f\in C(\bar\Omega)$ and 
$f\ge 0$.  Then there exists a unique viscosity solution $u\in C(\bar\Omega)$
to \eqref{MA} in the class of convex functions.
\end{thm}

Apart from the viscosity solution, another  notion of weak solution of the \MA equation is based on geometric arguments, which we 
now describe.
To motivate it,
suppose for the moment
that $f$ is uniformly positive in $\Omega$
and that $u\in C^2(\bar{\Omega})$
is a strictly convex function satisfying \eqref{MA} pointwise.
Let $\p u(x)$ denote
the subdifferential of $u$ at
the point $x\in \Omega$ given
by Definition \ref{def:subdiff},
i.e., $\p u(x)$ is the set of slopes of the supporting 
hyperplanes at $x$.  From the convexity and smoothness
assumptions we infer that $\p u(x) =\{ D u(x) \}$
and that  the subdifferential, viewed
as a map from $\Omega$ to $\bbR^d$, is injective.
%
Consequently, a change of variables reveals that
\begin{align*}
\int_D f  = \int_D \det(D^2 u) = \int_{\p u(D)}  = |\p u(D)|\quad \text{for all Borel sets }D\subset \Omega,
\end{align*}
where we recall that $\p u(D) = \cup_{x\in D} \p u(x)$
and $|\p u(D)|$ is the $d$-dimensional Lebesgue measure of $\p u(D)$.
Since $\p u$ is well-defined for non-smooth convex functions,
the above identity allows us to widen the class of admissible solutions.

\begin{definition}[Alexandrov solution]\label{def:AlexSolution}
A convex function $u\in C(\bar{\Omega})$
is an {\it Alexandrov solution} to \eqref{MA} if $u|_{\p \Omega} = g$
and
\begin{align}\label{Alek}
|\p u(D)| = \int_D f
\end{align}
for all Borel sets $D\subset \Omega$.
\end{definition}
Note that the continuity of the source term $f$ 
is no longer required for \eqref{Alek} to be well-defined.
The following example 
illustrates this feature.

\begin{ex}[Alexandrov solution]
 Let $\dm = B_1 (0) \subset \mathbb R^2$.  Then the function 
\[
u (x) = |x| - 1
\]
 is a Alexandrov solution of Monge-Amp\`{e}re equation 
\[
\det (D^2 u) (x) = \pi \delta_{(x=0)},
\]
where $\delta_{(x=0)}$ is the Dirac measure at origin. 
However,  $u$ is not a viscosity solution 
because the right hand side is not a (continuous) function.
\end{ex}

The existence and uniqueness 
of Alexandrov solutions is summarized 
in the next theorem, see \cite[Theorem 1.6.2]{Gutierrez01}.

\begin{thm}[existence of Alexandrov solutions]
Let $\Omega\subset \bbR^d$ be an open, bounded, and strictly convex
domain.  Suppose that the data satisfies $g\in C(\p\Omega)$
and $\int_\Omega f<\infty$.  Then there
exists a unique Alexandrov solution $u\in C(\bar\Omega)$
to \eqref{MA} in the class of convex functions.
\end{thm}

The next result states situations
where viscosity solutions and Alexandrov
solutions coincide \cite[Propositions 1.3.4 and 1.7.1]{Gutierrez01}.  The result
also shows that the notion of Alexandrov
solution is strictly weaker than that of a
viscosity solution. 
\begin{prop}[equivalence]
Suppose that $u\in C(\bar\Omega)$
is an Alexandrov solution to \eqref{MA}.  Then if $f$ is continuous, 
$u$ is a viscosity solution to \eqref{MA}.
Conversely, if $u$ is a viscosity
solution to \eqref{MA} and if  $f\in C(\bar\Omega)$ with $f>0$ in $\bar\Omega$,
then $u$ is an Alexandrov solution. 
\end{prop}

The rest of our presentation will deal separately with numerical methods for \eqref{MA} depending on the type of solution we aim to approximate: viscosity or Alexandrov solutions.
Before we discuss them, it is useful to review a property of the determinant. For convenience of notation we define
the set of symmetric nonnegative definite matrices as
\[
\bbS^d_+ :=\{A\in \bbS^d:\ A \geq 0\}.
\]

\begin{prop}[determinant inequality]
\label{prop:determinantisoperimetricinequality}
The following inequality holds:
\[
  \det (A)^{1/d} \det (B)^{1/d} \leq  \frac1d \tr AB  = \frac 1d (A : B)  \qquad \forall A,B\in \bbS^d_+
\]
with equality holding if and only if $B^{1/2}AB^{1/2}  = c I$ for some scalar $c \geq 0$.
\end{prop}
\begin{proof}
Let $C = B^{1/2} A B^{1/2}$, and note that $C \in \polS^d_+$. The arithmetic-geometric inequality yields
\[
  \det (C)^{1/d} \leq \frac 1d \tr C,
\]
with equality holding if and only if $B^{1/2}A B^{1/2} = C = c I$.
Since $A,B\in \bbS_+^d$ we find that $c\ge 0$.
The inequality follows from the algebraic identities 
$\det C = \det (A B) = \det A \det B$ and $\tr C = \tr A B = A : B$. 
\end{proof}

\begin{rem}[\MA is concave]\label{rmk:determinantconvex}
The determinant inequality of Proposition~\ref{prop:determinantisoperimetricinequality} implies that for any positive definite matrix $A$,
\[
  \det (A)^{1/d} = \inf_{B\in \bbS^d_+} \left\{ \frac 1d \tr AB : \;  \det B = 1 \right\}.
\]
Therefore, for any symmetric positive definite $A$ and $B\in \bbS_+^d$, and for $0 < t \leq 1$, we have
\begin{align*}
  \det (t A + (1-t) B)^{1/d} = &\; \inf_{K \in \bbS^d_+} \{ \frac 1d \tr  \left( (t A + (1-t) B) K \right) : \;  \det K = 1 \}
  \\
  \ge &\; t \inf_{K \in \bbS^d_+} \{ \frac 1d \tr  (A K ) : \;  \det K = 1 \}
  \\
   &\; +  (1- t) \inf_{K\in \bbS^d_+} \{ \frac 1d \tr (B K)  : \;  \det K = 1 \}
  \\
  = &\; t \det (A)^{1/d}  + (1-t)\det( B)^{1/d}.
\end{align*}
The inequality $ \det (t A + (1-t) B)^{1/d}\ge t \det (A)^{1/d}  + (1-t)\det( B)^{1/d}$
is also satisfied if $t=0$, and if $\det(A) = \det(B) =0$.
Thus the function $A \to \det A^{1/d}$ is concave over $\polS^d_+$.
\end{rem}



\subsection{Approximation of viscosity solutions: Hamilton-Jacobi-Bellman reformulation}\label{subsec:FengJensen}
In this section we summarize the results of \cite{FengJensen}, where a numerical method 
based on a HJB reformation of the \MA problem is developed.
Before stating the method we first note that Remark 6.15 implicitly gives such a HJB reformation, 
\ie if $u$ satisfies the first equation in \eqref{MA}, then formally
\begin{align*}
 f(x)^{1/d} =  \big(\det(D^2 u)(x)\big)^{1/d} = \frac1d \mathop{\inf_{B\in \bbS^d_+}}_{\det(B)=1}   D^2 u(x):B.
\end{align*}
Note that the constraint in the control set is nonlinear, and furthermore, the control 
set is not compact; these two features make the PDE and numerical constructions less obvious.
Rather, the method proposed in \cite{FengJensen} is based on the following result
\cite{KrylovBook88}.

\begin{prop}[HJB reformulation of MA]\label{prop:KrylovSmart}
Define
\[
\bbS^d_1 = \{A\in \bbS^d_+:\ {\rm tr}\,A=1\},
\]
and let $f\in \bbR$ with $f\ge 0$.
Then a matrix $A\in \bbS^d$ satisfies
\begin{align}\label{HJBFixed}
  H(A,f):=\sup_{B\in \bbS^d_1} \Big(-\frac 1d B:A+ f^{1/d} \det(B)^{1/d} \Big)=0
\end{align}
if and only if $\det(A) = f$ and $A\in \bbS_+^d$.
Moreover, there exists a maximizer $B_*\in \bbS^d_1$
that commutes with $A$.
\end{prop}
\begin{proof}
If $A \in \bbS^d_+$ and  $\det(A) = f$, then Proposition \ref{prop:determinantisoperimetricinequality} implies 
\[
  -\frac 1d B:A+ f^{1/d} \det(B)^{1/d} \leq (f^{1/d} - \det (A)^{1/d}) \det(B)^{1/d} = 0
\]
and equality holds if and only if $B^{1/2} A B^{1/2} = c I$ for some $c \geq 0$. 
If $A>0$ then we take $B = A^{-1}/\left(\tr A^{-1}\right)$.
  Otherwise, if $\det(A)=0$, we
use the eigendecomposition of a matrix
to construct $B\in \bbS_+^d$ that commutes with $A$, $\tr B=1$,
and $B^{1/2} A B^{1/2} =0$.
In either case we conclude 
that $A$ satisfies \eqref{HJBFixed}. 

Now our goal is to show that \eqref{HJBFixed} implies that $\det(A) = f$ and $A\in \bbS_+^d$. 
We prove the statement in two cases.

{\it Case I.}
If $A > 0$, 
then take $B = c A^{-1}$ with $c = 1/(\tr A^{-1})$ in \eqref{HJBFixed}.
Using the identities $\frac1d B:A = c$ and $\det(B) = c^d \det(A)^{-1}$, we find that
\[
  c(f^{1/d} - \det (A)^{1/d})  \det(A)^{-1/d} = \frac {-1}d B:A+ f^{1/d} \det(B)^{1/d} \leq 0
\]
and therefore $\det(A) \geq f$. On the other hand
Proposition \ref{prop:determinantisoperimetricinequality} implies that
\[
 0  = \sup_{B\in \bbS^d_1}\Big(  -\frac 1d B:A+ f^{1/d} \det(B)^{1/d}\Big)
 \leq  \sup_{B\in \bbS^d_1} (f^{1/d} - \det (A)^{1/d}) \det(B)^{1/d}.
\]
Now if $\det(A)>f$, then we must have $\det(B)=0$ for the matrix $B$ to attain the supremum. 
However, since $A >0$, Proposition \ref{prop:determinantisoperimetricinequality} implies 
the inequality above is strict if $B \neq c A^{-1}$. This leads to a contradiction.
Therefore $\det(A) = f$.
This proves the first case.

{\it Case II:} If $A$ is not strictly positive definite, then 
 without loss of generality, we we assume that $A$ is of the diagonal form
\[
A = \lambda_1 \bv_1 \otimes \bv_1 + \cdots + \lambda_d \bv_d \otimes \bv_d
\qquad
\lambda_1 \leq \cdots \leq \lambda_d,
\]
where $\{\bv_i\}_{i=1}^d$ are the eigenvectors of $A$  with unit length, and $\lambda_1 \le 0$.

Consider the matrix
\[
  B = (1-\eps) \bv_1\otimes \bv_1 + \frac{\eps}{d-1} \sum_{i=2}^d \bv_i\otimes \bv_i\in \bbS^d_1
\]
for some parameter $\eps\in (0,1)$.
We then have
\begin{align*}
  -\frac 1d B:A+ f^{1/d} \det(B)^{1/d}
   &= -\frac1d \Big( (1-\eps)\lambda_1 + \frac{\eps}{d-1}\sum_{i=2}^d \lambda_i\Big)\\
&\qquad   +(1-\eps)^{1/d} \frac{\eps^{(d-1)/d}}{(d-1)^{(d-1)/d}} f^{1/d}\\
&\ge \frac{-1}{d} \left( (1-\epsilon) \lambda_1 +  \epsilon \lambda_d \right) +  \left[ \frac{ (1-\eps) \eps^{(d-1)}}{(d-1)^{(d-1)}} f \right]^{1/d}.
\end{align*}
Note that if $f>0$ and for sufficiently small $\epsilon$, the dominating term is $ \left[ \frac{ (1-\eps) \eps^{(d-1)}}{(d-1)^{(d-1)}} f \right]^{1/d}$.
Thus, if $f > 0$, then 
we deduce  
\[
 -\frac 1d B:A+ f^{1/d} \det(B)^{1/d} > 0
\]
which contradicts \eqref{HJBFixed}. 
If $f = 0$, then 
\[
  -\frac 1d B:A+ f^{1/d} \det(B)^{1/d}
   = -\frac1d \Big( (1-\eps)\lambda_1 + \frac{\eps}{d-1}\sum_{i=2}^d \lambda_i\Big).
\]
Equation \eqref{HJBFixed} holds for any $\epsilon>0$ if and only if $\lambda_1 =0$. Thus, $A \geq 0$ and $\det(A) = f =0$. 
This completes the proof.
\end{proof}

Lemma \ref{prop:KrylovSmart} leads to the following result, showing that viscosity solutions to the 
Monge-Amp\`ere equation are viscosity solutions of a HJB problem. 
We refer the reader to \cite[Theorem 3.3]{FengJensen} for a proof.
\begin{thm}[equivalence of viscosity solutions]
Let $f\in C(\Omega)$ 
be a non-negative function,
and let $u$ be a viscosity
subsolution (resp., supersolution) 
of the Monge-Amp\`ere problem \eqref{MA}
on the set of convex functions.
Then $u$ is a viscosity subsolution (resp., supersolution)
to the HJB problem
\begin{equation}
\label{eqn:MAHJB}
  \begin{dcases}
    H(D^2 u,f) = 0 & \text{in }\Omega,\\
    u  = g & \text{on }\p\Omega.
  \end{dcases}
\end{equation}
\end{thm}

\begin{rem}[convexity of solution]
We emphasize that the convexity
of the solution to \eqref{eqn:MAHJB}
is not assumed a priori; it occurs
implicitly  from the structure of the HJB problem.
\end{rem}

\begin{rem}[$\polS^d_1$ is not Cordes]
It is worth mentioning
that matrices in $\bbS_1^d$
do not necessarily satisfy
the Cordes condition if $d\ge 3$
(\cf Definition \ref{def:Cordes}).
To see this, consider $B\in \bbS_1^3$
with eigenvalues $\lambda_1 = 1-\tau$,
$\lambda_2 = \lambda_3 = \frac{\tau}2$
for some $\tau\in (0,\frac13]$.  We then find
\begin{align*}
\frac{|B|^2}{({\rm tr}\,B)^2} = \lambda_1^2 + \lambda_2^2 +\lambda_3^2 = \frac32 \tau^2 - 2\tau+1 \ge \frac12,
\end{align*}
and thus, $B$ does not satisfy the Cordes condition.
\end{rem}

\subsubsection{Discretization of the HJB-reformulation}

To describe a discretization method for problem 
\eqref{eqn:MAHJB}, we first note that if 
 $B\in \bbS^d_1$, then 
$B = YD Y^\intercal$, where 
$D\in \bbS^d_1$ 
is a 
diagonal matrix and $Y = [\by_1, \cdots, \by_d]\in \mathbb{SO}^d$,
where $\mathbb{SO}^d$ is the special orthogonal group.
Thus, denoting by $\mathbb{D}^d_1$ the space
of diagonal nonnegative definite matrices with unit trace, 
we find that 
\begin{align*}
H(A,f)  = \mathop{\sup_{Y\in \mathbb{SO}^d}}_{D\in \mathbb{D}_1^d}
\Big( -\frac1d Y D Y^\intercal: A + f^{1/d} \det(D)^{1/d}\Big).
\end{align*}
Let $\by_i\in \bbR^d$ denote the
$ith$ column of $Y\in \mathbb{SO}^d$ 
and let $\lambda_i = D_{i,i}$. 
Then the solution to the \MA
equation satisfies
\begin{align*}
\mathop{\sup_{Y\in \mathbb{SO}^d}}_{D\in \mathbb{D}_1^d}
\Big( -\frac1d \sum_{i=1}^d \lambda_i \by_i\otimes \by_i: D^2 u(x)+ f(x)^{1/d} \Big(\prod_{i=1}^d \lambda_i\Big)^{1/d}\Big)=0
\end{align*}
in the viscosity sense.

Let $\Th$ be a quasi-uniform and shape regular mesh of $\dm$ and $\lc$ 
be the continuous piecewise linear polynomials over $\Th$
defined by \eqref{eqn:LagrangeSpace}.
We denote by $\Omega_h^I$ the set of interior 
vertices of $\mct$, and by $\Omega_h^B$ the set of boundary vertices.
Let $v_h \in \lc$. Recall that, for a discretization parameter $k>0$ 
and vector $y$, the second order difference operator is defined by
\begin{align*}
\delta_{y,k}^2 v_h(z) = \frac1{k^2} \big( v_h(z+k y)-2v_h(z)+v_h(z-k y)\big)
\end{align*}
provided that $z$ is sufficiently far from the boundary, namely, $z\pm k y\in \bar\Omega$.
Otherwise we set
\begin{align*}
\delta_{y,k}^2 v_h (z) =&\; \frac{2}{k^+ + k^-} \left( \frac{ v_h (z + k^+ y) -  v_h (z) } { k^+}
 + \frac{ v_h (z - k^- y) - v_h (z)}{k^-} \right),   
\end{align*}
where $k^{\pm}$ is the only element in $(0, k]$ such that $z \pm k^{\pm} y \in \bdry$.

With these definitions at hand we introduce, on $\lc$, the 
following approximation of the HJB reformulation 
\[
  H_h^k[u_h,f](z) = \mathop{\sup_{Y\in \mathbb{SO}^d}}_{D\in \mathbb{D}_1^d} 
  \left( \frac{-1} d \sum_{i=1}^d \lambda_i \delta_{y,k}^2 u_h (z) + f(z)^{1/d} \det(D)^{1/d} \right),
\]
where $z \in \Omega_h^I$.

The numerical method based on the HJB reformulation seeks a piecewise linear function $u_h \in \lc$ such that 
\begin{equation}
\label{FengJensenmethod}
  \begin{dcases}
    H^k_h[u_h, f] = 0 & \text{in }\Omega_h^I,\\
    u_h = g & \text{on }\Omega_h^B.
  \end{dcases}
\end{equation}
Note that, the boundary conditions uniquely determine $u_h$ on $\p \Omega$, and therefore 
the problem is well-defined. The scheme \eqref{FengJensenmethod} is a
 semi-Lagrangian method because the points $z \pm k y$, which are used to 
 evaluate $\delta_{y,h}^2 u_h(z)$, may not belong to $\bar\Omega_h$. Thus, additional effort is needed to evaluate $u_h(z \pm k y)$.

It is straightforward to check that $\delta_{y,k}^2u_h(z)$ is monotone for any direction $y$,
implying that operator $H^k_h[ u_h, f]$ is monotone. 
The monotonicity also leads to 
stability of the method \cite[Lemmas 6.2 and 6.4]{FengJensen}.
\begin{lem}[stability]
Problem \eqref{FengJensenmethod} is stable in the sense that
there exists a unique solution $u_h\in \lc$ to \eqref{FengJensenmethod}
and an $h$-independent constant $C>0$ such that
$\|u_h\|_{L^\infty(\Omega)}\le C$.  Moreover if we set
\begin{align*}
\bar{u}(x) := \mathop{\limsup_{y\to x}}_{h\to 0^+} u_h(y),
\quad
\underline{u}(x) := \mathop{\liminf_{y\to x}}_{h\to 0^+} u_h(y),\quad x\in \bar\Omega,
\end{align*}
then $\bar{u}(x) = \underline{u}(x) = g(x)$ for all $x\in \p\Omega$
provided that $\Omega$ is strictly convex.
\end{lem}

To show the consistency of the method, we recall that $I_h^{fe}:C(\bar\Omega)\to \lc$
denotes the nodal interpolant.
We then have the following result.

\begin{lem}[consistency]
Let $\phi\in C^{2,\alpha}(\Omega)$, then there is a constant $C$ such that, for every $z \in \Omega_h^I$, we have
\begin{align*}
  \left| 
    H(D^2 \phi (z),f(z)) - H_h^k[I_h^{fe} \phi,f](z)
  \right| \leq C \left(  k^{\alpha} + \frac{h^2}{k^2_{\min}} \right),
\end{align*}
where $k_{\min} = \min\{k^+,k^-\}$.
Consequently, the method is consistent if $k \to 0$ and $\frac h{k_{\min}} \to 0$. 
\end{lem}
\begin{proof}
It suffices to show that
\begin{align*}
 | \delta_{y,k}^2 I_h^{fe} \phi(z) - (y \otimes y): D^2 \phi(z) | \leq C\left(  k^{\alpha} + \frac{h^2}{k^2_{\min}} \right) .
\end{align*}
First, by Taylor's Theorem, we have
\begin{align}\label{eqn:phiTT}
\big|\delta_{y,k}^2 \phi (z) - (y \otimes y): D^2 \phi(z)\big|\le C k^\alpha \|\phi\|_{C^{2,\alpha}(\Omega)}.
\end{align}
Recalling that \cite{CiarletBook},
\[
\|\phi-I_h^{fe} \phi\|_{L^\infty(\Omega)}\le C h^2 \|\phi\|_{W^{2,\infty}(\Omega)},
\]
we have
\begin{align}\label{eqn:phiInterpolation}
\big|\delta_{y,k}^2 \phi(z) - \delta_{y,k}^2 I_h^{fe} \phi(z)\big|\le C \frac{h^2}{k_{\min}^2} \|\phi\|_{W^{2,\infty}(\Omega)}.
\end{align}
The desired result now follows from \eqref{eqn:phiTT}--\eqref{eqn:phiInterpolation}
and the triangle inequality.
\end{proof}

%

\begin{rem}[discrete controls]
To implement the method \eqref{FengJensenmethod}, it remains to specify a discrete set $\bbS_{1,h}^d \subset \bbS_1^d$ of symmetric positive definite matrices with unit trace. 
To ensure consistency of the method, we require the discrete set $\bbS_{1,h}^d$ to be dense as $h \to 0$, that is, 
for any $B \in \bbS_1^d$, there is $B_h \in \bbS_{1,h}^d$ such that $B_h \to B$ as $h \to 0$.
\end{rem}
%

Since the method is monotone and consistent, the convergence of the numerical solution now follows from the Barles-Souganidis  theory (cf. Theorem \ref{thm:BSTHM1Alt} and \cite[Theorem 6.5]{FengJensen}).

\begin{thm}[convergence]
Let $\Omega\subset \bbR^d$ be a strictly
convex domain.  Assume that $f\in C(\Omega)$
with $f\ge 0$ and $g\in C(\p\Omega)$.
Then as $h\to 0$, $h/k_{\min}\to 0$, 
the solutions $u_h\in \lc$ of \eqref{FengJensenmethod}
converge uniformly to the unique viscosity solution
on the set of convex functions of the Monge-Amp\`ere problem \eqref{MA}.
\end{thm}

We conclude the discussion on semi-Lagrangian schemes by commenting that rates of convergence for a general semi-Lagrangian scheme for \eqref{eq:HJBbvp} have been obtained in \cite[Corollary 7.3]{MR3042570}.

\subsection{Approximation of Alexandrov solutions}

Now we discuss numerical methods based
on the Alexandrov solution concept presented in Definition \ref{def:AlexSolution}.
Essentially, this class of numerical methods are finite dimensional analogues 
of \eqref{Alek}.

Let $\{\omega_i\}_{i=1}^N$ be an open, disjoint partition
of the domain, i.e., $\omega_i\cap \omega_j = \emptyset$
for $i\neq j$ and $\bar{\Omega} = \cup_{i=1}^N \bar{\omega}_i$.
Let $\Omega_h^I := \{z_i\}_{i=1}^N$ be a collection of points with the property that $z_i\in \omega_j$ if and only if $i=j$,
and let $\Omega_h^B := \{z_i\}_{i=N+1}^{M+N}$ be a set of distinct
points on $\p\Omega$. As before, we denote $\bar\Omega_h = \Omega_h^I \cup \Omega_h^B$ 
and we will call its elements nodes or grid points.

Recall that for a nodal function $v_h\in \fd$,
its subdifferential at a grid point $z\in \bar\Omega_h$
is given by \eqref{eqn:NodalSubDiff}.
Now, a natural generalization of \eqref{Alek},
and the discrete problem proposed in \cite{Oliker88} reads: 
Find a convex nodal function $u_h\in \fd$ satisfying
\begin{equation}
\label{OP}
  \begin{dcases}
    |\p u_h(z_i)| =  \int_{\omega_i} f, & \forall z_i \in \Omega_h^I, \\
    u_h(z_i) =  g(z_i),  & \forall z_i \in \Omega_h^B.
  \end{dcases}
\end{equation}
Note that since  the partition $\{\omega_i\}_{i=1}^N$ is non-overlapping,
for all Borel sets $D\subset \Omega$ we have
\begin{align*}
|\p u_h(D)| = \sum_{z_i\in D} f_i, \qquad f_i = \int_{\omega_i} f.
\end{align*}
Thus, the scheme is obtained
by replacing $f$ in \eqref{Alek} by a 
family of Dirac measures supported at the nodes and by replacing $g$
by its nodal interpolant on the boundary.

One special case of this method is when  
the interior nodal set is a lattice, \ie for some basis $\{\tilde \be_j\}_{j=1}^d$ of $\Real^d$ we have
\[
  \Omega_h^I = \left\{ z = h\sum_{j=1}^d z^j \tilde \be_j: z^j \in \polZ \right\} \cap \Omega.
\]
We remark that this property applies to interior nodes only. For the boundary nodes, we only require that their spacing is of order $h$, namely, $\bdry \subset \cup_{z \in \Omega_h^B} B_{h/2}(z)$.
In this case, the partition $\{\omega_i\}$ of the domain 
can be taken as parallelotopes
\begin{align}\label{partition_MA}
  \omega_i = \left\{ z_i + \sum_{j=1}^d h^j \tilde \be_j : \ h^j \in \Real, \ |h^j| \leq \frac h2 \right\} \cap \Omega.
\end{align}
The length of the coordinate vectors $\{\tilde \be_j\}_{j=1}^d$ is such that the parallelotope $\omega_i$ is inside the ball $B_{h}(z_i)$ centered at $z_i$ and of radius $h$. Notice that, by construction, $\omega_j = z_j - z_i + \omega_i$; consequently, the radius of the largest ball inscribed in $\omega_i$ does not depend on $i$ and we denote it by $\rho$. We define the shape-regularity of the nodal set as
\begin{align}\label{shaperegularity_MA}
  \sigma = h\rho^{-1}.
\end{align}

\begin{rem}[meshless nodal function]
It is worth mentioning again that the solution $u_h$ is only defined at the nodes $\bar\Omega_h$. Its convex envelope 
induces a triangulation of the domain $\dm$ and a piecewise linear
function. However, this triangulation  is not known a priori.  
See Remark \ref{rem:discreteconvex} and Examples \ref{ex:FEMCE1}--\ref{ex:FEMCE2}.
\end{rem}

In view of Example \ref{ex:anisotropy}, if the solution of the \MA equation is nearly degenerate, wide stencils are needed to compute the subdifferential. Thus, method \eqref{OP} may be regarded as a wide stencil finite difference scheme when the solution is nearly degenerate.
On the other hand, if the solution to \eqref{MA} is strictly convex, \ie for $0<\lambda \leq \Lambda$ we have $\lambda I \leq D^2 u \leq \Lambda I$, then the sub-differential of $u_h$ at node $z$ depends only on the values of $u_h$ at the adjacent nodes of $z$. 
To make this last statement rigorous, we state a definition.

\begin{definition}[adjacent set]
Let $u_h\in \fd$ and {$z\in \Omega^I_h$}.
The {\it adjacent set} $A_z$ of $u_h$ at $z$
is the collection of nodes $z_i \in \bar\Omega_h$ 
such that there exists a supporting hyperplane $\ell$ of $u_h$ at $z$ and $\ell(z_i) = u_h(z_i)$.
\end{definition}

Note that $A_z$ is the set of nodes of a star of $z$ which is induced by the {discrete} convex envelope
of $u_h$; see Figure \ref{fig:anisotropy}.  In particular, we have that
the subdifferential $\p u_h(z)$ is determined by the values $u(z_i)$ for $z_i \in A_z$.

Let us now estimate the size of $A_z$.

\begin{lem}[size of $A_z$]
\label{lem:estimate_adjacent_set}
Let $v\in C^2(\bar\Omega)$ be a strictly convex function with  $ \lambda I \leq D^2 v (x) \leq \Lambda I$ 
for some constants $0<\lambda \leq  \Lambda$.
Assume that $E:=\{\tilde \be_j\}_{j=1}^d$ is a basis of $\bbR^d$ such that $\Omega_h^I$ is a lattice 
spanned by $E$ with shape regularity constant $\sigma$ defined in \eqref{shaperegularity_MA}. 
Let the nodal function $ v_h \in \fd$ be defined by $v_h= I_h^{fd} v$.
Then, if  $A_z$ is the adjacent set of $v_h$ at $z\in \Omega_h^I$,   
\[
  A_z \subset B_{Rh}(z),
\]
where
\begin{equation}
\label{eq:defofRAz}
  R \geq \bar R :=\frac\Lambda\lambda \sigma^2 \left| \sum_{j=1}^d \tilde{\be}_j \right|^2.
\end{equation} 
\end{lem}
\begin{proof}
Without loss of generality,  we may assume that $z=0$, $v(z)=0$, and $D v(z) = \boldsymbol0$.
Let $\hat z \in \Omega_h^I \cap \p \conv (A_z)$ and $\omega$ be the parallelotope defined as in \eqref{partition_MA} with center $z = 0$. By convexity of $\omega$, there is a $c \in (0,1)$ such that $c \hat z \in \partial \omega$. Thus, we can express $\hat z$ as a multiple of a convex combination of $\{\zeta_j\}_{j=1}^{2^d}$, the vertices of $\omega$. In other words, for $R = 1/c$, we have
\[
  \hat z = R \sum_{j=1}^{2^d} k_j \zeta_j,    \quad  k_j \geq 0, \quad  \sum_{j=1}^{2^d} k_j = 1.
\]
This representation shows that $\abs{\hat z}\leq R h$; thus, to obtain the result, it remains to estimate $R$.

Since $c\hat z \in \p \omega$ we have, using \eqref{shaperegularity_MA}, that $|c \hat{z}| \ge \rho = h \sigma^{-1}$ which can be rewritten as
$|\hat{z}| \ge R h \sigma^{-1}$. Using that $D^2 v \geq \lambda I$ we estimate
\[
  v_h(\hat z) =   v(\hat z) \geq \frac 12 \lambda R^2 \sigma^{-2} h^2.
\]

Let us now obtain an upper bound for $v_h(\hat z)$. 
To do so, let us introduce $\hat \omega$ as the (unique) smallest parallelotope with vertices $\{\bar z_m \}_{m=1}^{2^d} \subset \Omega_h^I$ such that $z \in \{\bar z_m\}_{m=1}^{2^d}$ is a vertex and $c\hat z \in \hat \omega$. 
This parallelotope can be thought of as belonging to the {\it dual mesh}. We now invoke Caratheodory's 
theorem \cite[Theorem 2.13]{MR2361288} 
to conclude that there is a subset of $\{\bar z_m\}_{m=1}^{2^d}$, of cardinality $d+1$, for which 
$c \hat z$ can be expressed as a convex combination of these vertices. In other words, 
up to a permutation in $\{1, \ldots, 2^d\}$, we have
\begin{equation}
\label{eq:hatziscarath}
  \hat z = R \sum_{m=1}^{d+1} \alpha_m \bar z_m, \quad \alpha_m \geq 0, \quad \sum_{m=1}^{d+1} \alpha_m = 1.
\end{equation}
We now invoke that $\hat z \in A_z$. This implies that there is an affine function $\ell$ that verifies
\[
  \ell(z) = v_h(z) = 0, \qquad \ell(\hat z) = v_h(\hat z), \qquad \ell(\bar z_m) \leq v_h(\bar z_m), \ m=1,\ldots, d+1.
\]
Using representation \eqref{eq:hatziscarath} of $\hat z$ and that $D^2 v \leq \Lambda I$ we then obtain
\[
  v_h(\hat z) = R \sum_{m=1}^{d+1} \alpha_m \ell(\bar z_m) \leq \frac12 \Lambda R \sum_{m=1}^{d+1} \alpha_m |\bar z_m|^2.
\]
It remains to observe now that $\bar z_m = z + \sum_{j=1}^d \eps_j h \tilde \be_j$ with $\eps_j \in \{-1,0,1\}$ and, therefore,
\[
  |\bar z_m| \leq h \left| \sum_{j=1}^d \tilde \be_j\right|.
\]

A combination of the obtained upper and lower bounds for $v_h(\hat z)$ yields
\[
  \frac12 \lambda R^2 h^2 \sigma^{-2} \leq \frac12 \Lambda R h^2 \left| \sum_{j=1}^d \tilde \be_j\right|^2,
\]
from which \eqref{eq:defofRAz} follows.
\end{proof}

\begin{rem}[Cartesian lattice]
In the setting of Lemma~\ref{lem:estimate_adjacent_set}, if $E$ is the canonical basis of $\Real^d$ then \cite{Mirebeau15,BCM16} have improved estimate \eqref{eq:defofRAz} to $\bar R = \frac\Lambda\lambda \sigma^2$.
\end{rem}

If  $\Omega^I_h$ is a lattice, then
we are able to show consistency of the method 
\eqref{OP} in the following sense. 

\begin{lem}[consistency]\label{lem:consistency_OP}
Let $E = \{ \tilde \be_j\}_{j=1}^d$ be a basis of $\Real^d$ and $\Omega_h^I$ be a lattice spanned by $E$. 
Let $p$ be a strictly convex quadratic polynomial
with $\lambda I\le D^2 p\le \Lambda I$. 
If $z\in \Omega_h^I$ is such that $\dist(z,\p \Omega)\ge \bar R h$, with $\bar R$ as in \eqref{eq:defofRAz}, 
then we have
\begin{align}\label{consistency_OP}
     |\p I_h^{fd} p (z)| = \det(D^2p) |\omega_z|, 
\end{align}
where $\omega_z$ is the parallelotope defined by \eqref{partition_MA}.
\end{lem}
\begin{proof}
We divide the proof in two steps.

{\it Step 1.} We first show that \eqref{consistency_OP} holds when the domain is $\dm = \bbR^d$. 
Without loss of generality, we assume that $z = 0$ and $p(x) = \frac12 x \cdot M x$ for some $M\ge \lambda I$.
For a vector $\bq \in \Real^d$ we define the norm $|\bq |_M^2:=\bq \cdot M \bq$,  and define the set
\begin{align*}
  V &:= \{ \bq \in \bbR^d : |\bq |_M  \leq  |\bq - z_j |_M,\ \forall z_j \in \bar\Omega_h\} \\
  &= \left\{ \bq \in \bbR^d : \bq \cdot M z_j \leq  \frac 12 z_j \cdot M z_j,\  \forall z_j \in \bar\Omega_h\right\} .
\end{align*}
It can be shown, see \cite[Lemma 2.3]{Mirebeau15} for details, that translations of $V$ tile $\bbR^d$ and that $|V| = |\omega_z|$.
Moreover, by a simple algebraic manipulation,
\[
V = \{ M^{-1} \bq \in \bbR^d : \bq \cdot z_j \leq I_h^{fd} p(z_j),\ \forall z_j\in \bar\Omega_h \}.
\]
Thus, $V = M^{-1} [ \p I_h^{fd} p(0)] $, \ie it is the image of subdifferential $\p I_h^{fd} p(0)$ under the linear map $M^{-1}$. 
Taking measure on both sides yields
\[
  |\omega_z| = |V| = \det(M)^{-1} |\p I_h^{fd} p(0)|.
\]
The proof of step 1 is now completed by rearranging terms.

{\it Step 2.} We now consider a bounded domain and show that \eqref{consistency_OP} holds for nodes sufficiently 
far away from the boundary , \ie $\dist(z, \bdry) \geq \bar R h$. To do so, we observe that the 
subdifferential $\p I_h^{fd} p(z)$ is determined only by the function values in the adjacent set $A_z$.
Since, as shown in Lemma \ref{lem:estimate_adjacent_set}, $A_z \subset B_{\bar Rh}(z)$  we deduce that if the node 
$z$ is bounded away from the boundary with $\dist(z, \bdry) \geq \bar Rh$, then $A_z \subset \dm$. Thus, \eqref{consistency_OP} holds. 

This concludes the proof.
\end{proof}

\subsubsection{A truncated version}
In the case $\Omega_h^I$ is a Cartesian lattice, scheme \eqref{OP} is is closely related with the finite difference method of \cite{Mirebeau15,BCM16} which we now describe. For simplicity, suppose that $d=2$ and that the interior grid points are given by
\begin{align*}
\Omega_h^I =\Omega \cap \bbZ^2_h.
\end{align*}
Let $ S \subset \bbZ^2\backslash \{0\}$
denote a stencil.
For any $ y \in S$ and $z \in \Omega_h^I$
sufficiently far from $\p\Omega$, we recall
that the second-order difference operator in the direction $y$
is given by
\[
 \delta_{y,h}^2 v(z)
  =  \frac{v(z+hy)-2v(z)+v(z-hy)}{h^2}.
\]
When $z \in \Omega_h^I$ is close to $\bdry$, the point $z \pm hy$ may not belong to $\bar\Omega_h$. 
In such cases, we define 
\[
  \delta^2_{y,h} v(z):=\frac{2}{h^++h^-}\left( \frac{v(z+h^+y)-v(z)}{h^+} + \frac{v(z-h^-y) - v(z)}{h^-} \right),
\]
where $h^{\pm}$ is the only element in $[0,h]$ such that $z\pm h^{\pm}y \in \bdry$.
This construction implicitly defines the set of boundary points $\Omega_h^B$.

We define the set of {\it superbases} of $S$ as 
\begin{align*}
Y_h := \left\{(y_0,y_1,y_2)\in S^3:\ |\det(y_0,y_1,y_2)|=1,\ y_0+y_1+y_2=\boldsymbol0 \right\}.
\end{align*}
Note that for $z\in \Omega_h^I$ and $\by = (y_0,y_1,y_2)\in Y_h$, the convex hull $\calH_{z,\by} = \conv\{z\pm h y_i\}_{i=0}^3$ is a hexagon.
Given a nodal function $v_h$,  superbasis $\by = (y_0,y_1,y_2)\in Y_h$, 
and a point $z\in \Omega_h^I$, we denote by $\Gamma_{z,\by} (v_h)$ the maximal convex map bounded above by $v_h$ at the points $z$ and $\{z \pm h y_i\}_{i=0}^3$.
As before $\Gamma_{z,\by}(v_h)$, restricted to $\calH_{z,\by}$, is a piecewise linear function with respect to some triangulation of $\calH_{z,\by}$, which depends on the values of $v_h$ on the extreme points of $\calH_{z,\by}$ and $z$.

We define the discrete \MA operator 
\begin{align*}
\gamma(\Delta^+_{y_0} v_h(z),\Delta^+_{y_1} v_h(z),\Delta^+_{y_2} v_h(z)),
\end{align*}
with $\Delta_y^+ v(z) = \max\{\delta_{y,h}^2 v_h(z),0\}$ and
\begin{align*}
\gamma(\delta_0,\delta_1,\delta_2) := 
\left\{
\begin{array}{ll}
\delta_{i+1} \delta_{i+2} & \text{if } \delta_i \ge \delta_{i+1}+\delta_{i+2}, \ i = 0,\ldots,2 \mod 3, \\
\gamma_1(\delta_0,\delta_1,\delta_2)
 & \text{otherwise},
\end{array}
\right.
\end{align*}
with $\gamma_1(\delta_0,\delta_1,\delta_2):=\frac12 (\delta_0\delta_1 +\delta_1\delta_2+\delta_0\delta_2)-\frac14 (\delta_0^2+\delta_1^2+\delta_2^2)$.
As shown in \cite[Remark 1.10]{BCM16}, from the definition of subdifferential and some geometric arguments it follows that
\[
  \gamma(\Delta^+_{y_0} v_h(z),\Delta^+_{y_1} v_h(z),\Delta^+_{y_2} v_h(z))
=
h^2 |\p \Gamma_{z,\by}(v_h)(z)|. 
\]

The scheme proposed in \cite{BCM16} reads: Find the nodal function $u_h\in \fd$ such that
\begin{equation}
\label{BCM}
  \begin{dcases}
    \min_{\by\in Y_h} \gamma(\Delta^+_{y_0} u_h(z),\Delta^+_{y_1} u_h(z),\Delta^+_{y_2} u_h(z)) =  f(z), & \forall z\in \Omega_h^I, \\
    u_h(z) = g(z), & \forall z \in \Omega_h^B. 
  \end{dcases}
\end{equation}

\begin{lem}[consistency]
\label{lem:consistency_BCM}
Let $M$ be a positive definite matrix and $p(x) = \frac12 x\cdot Mx$ be a convex quadratic polynomial. Then 
\[
  \min_{\by\in Y_h} |\p \Gamma_{z,\by}(I^{fd}_h p)(z)| = \det(M)
\]
if and only if there is a $M$-obtuse superbasis $(y_0, y_1, y_2)$, that is,
\[
y_i \cdot  M y_j \leq 0 \quad \forall i \neq j.
\]
Moreover, if $B_{R} \subset \conv S$ with $R^2 = 2 |{M} | |{M^{-1}} |$, then such a $M$-obtuse basis exists in $\St$.
\end{lem}

We refer to \cite[Propositions 1.12 and 2.2]{BCM16} for a proof. Notice that the previous result shows that if the matrix $M$ is anisotropic, \ie 
$ |M | |M^{-1} |$ is large, then a wide stencil is needed to ensure the existence of a $M$-obtuse superbasis.

\begin{rem}[three dimensions]
In three space dimensions, to the best of our knowledge, there is no explicit formula to compute $|\p \Gamma_{z,{\by}}(v_h)(z)|$. As shown in \cite{Mirebeau15},
if $p(x) = \frac 12 x \cdot Mx$ and  the stencil $S$ is such that $B_R \subset \conv S$, for some $R$ that depends on $ |M | |M^{-1} |$, then we have that
\[
  \det(M) = h^{3} |\partial \Gamma_{z,\by}(I^{fd}_h p)(z)|.
\]
This result is consistent with Lemma \ref{lem:consistency_OP}.
\end{rem}

\subsubsection{Stability of \eqref{OP}}

While the convergence of monotone and consistent schemes, like \eqref{OP}, can be obtained using the framework described in Section~\ref{sub:LaxEquiv},
few results are known on the rate of convergence of such approximations. Here we discuss, following  \cite{NochettoZhang16A}, some recent results on the  $L^{\infty}$-rate of convergence of scheme \eqref{OP}.

The derivation of these error estimates involves, in addition to the discrete Alexandrov estimates of Lemmas~\ref{lem:FDAE} and \ref{Alexandroff}, a discrete barrier argument as in Section~\ref{sub:Wujunnondiv} and the Brunn-Minkowski inequality, which we now state.

Let $D$ and $E$ be two nonempty compact subsets of $\mathbb R^d$. We define their (Minkowski) sum as
\[
  D + E := \left\{ \bv + \bw \in \mathbb{R}^d:  \bv \in D \  \bw \in E \right\}.
\]

\begin{prop}[Brunn-Minkowski]
\label{prop:BM}
Let $A$ and $B$ be two nonempty compact subsets of $\mathbb R^d$. Then the following inequality holds:
\[
  |A + B|^{1/d} \ge |A|^{1/d} + |B|^{1/d}
\]
where $|\cdot|$ denotes the Lebesgue measure on $\bbR^d$.
\end{prop}

\begin{rem}[concavity]
  The Brunn-Minkowski inequality implies that the function $D \to |D|^{1/d}$ is concave, in the sense that for $0 \leq t \leq 1$,
\[
  |t A + (1-t) B|^{1/d} \ge t |A|^{1/d} + (1-t) |B|^{1/d}.
\]
\end{rem}

The discrete Alexandrov estimate Lemma \ref{lem:FDAE} shows that the $L^{\infty}$-norm of nodal or piecewise linear function $v_h$ is controlled by the measure of the subdifferential $|\partial v_h|$. 
Now suppose $u_h$ and $w_h$ are two nodal functions. The following stability estimate shows that the difference $v_h - w_h$ measured in the $L^{\infty}$-norm is controlled by the difference of the measure of their subdifferentials. This can be recast as a stability estimate for scheme \eqref{OP}.

\begin{prop}[stability]\label{prop:stability_MA}
Let $v_h$ and $w_h$ be two nodal functions with $v_h\ge w_h$ on $\Omega_h^B$. Then 
\[
  \sup_{\bar\Omega_h} (v_h - w_h)^- \leq C \left( \sum_{z\in \mathcal{C}_h^-(v_h - w_h)} \Big( \abs{ \partial v_h (z)}^{1/d} - \abs{ \partial w_h (z)}^{1/d} \Big )^d  \right)^{1/d}.
\]
\end{prop}
\begin{proof}
Consider the discrete convex envelope of the difference $v_h - w_h$, which we denote by $\Gamma_h (v_h - w_h)$ and its lower nodal contact set 
\[
\mathcal{C}^-_h (v_h - w_h) = \{z \in \Omega_h^I : \Gamma_h  (v_h - w_h)(z) =  (u_h - w_h)(z) \}.
\]
By Lemma \ref{lem:FDAE} (finite difference Alexandrov estimate), we have 
\begin{align}\label{stability-alexandroff}
 \sup_{\bar\Omega_h} ( v_h - w_h)^- \le C \left( \sum_{z \in \mathcal{C}_h(v_h - w_h)^-} |\partial  \Gamma_h (v_h - w_h)(z) | \right)^{1/d}.
  \end{align}
Thus, we only need to estimate $|\partial  \Gamma_h (v_h - w_h)(z) |$ for all $z$ in the contact set, which we do as follows.
We first note that Lemma~\ref{Msum} implies that
\begin{align}
\label{WZClaim}
\partial w_h(z) + \partial \Gamma_h (v_h - w_h)(z) \subset \partial v_h(z).
\end{align}
From this, and the Brunn-Minkowski inequality (Proposition \ref{prop:BM}), we obtain 
\begin{equation}
\label{WZClaim2}
\begin{aligned}
  |\partial w_h(z)|^{1/d} &+ |\partial \Gamma_h (v_h - w_h)(z)|^{1/d} \\
  &\leq
    |\partial w_h(z) + \partial \Gamma_h (v_h - w_h)(z) |^{1/d}
  \leq | \partial v_h(z)|^{1/d},
\end{aligned}
\end{equation}
which clearly implies that 
\begin{align*}
  |\partial \Gamma_h (v_h - w_h)(z)| \leq &\; \left(| \partial v_h(z)|^{1/d} - |\partial w_h(z)|^{1/d}   \right)^{d} .
\end{align*}
This is the desired estimate for $ |\partial \Gamma_h (u_h - w_h)(z)|$. Inserting it into \eqref{stability-alexandroff} yields the claimed result.
\end{proof}

A direct consequence of this stability result is a maximum principle for nodal functions, which we state below.

\begin{col}[maximum principle]\label{MP}
Let $v_h$ and $w_h$ be two nodal functions associated with $\bar\Omega_h$. 
If $v_h \geq w_h$ on $\Omega_h^B$ and $|\partial v_h(z)| \leq |\partial w_h(z)|$ at all $z \in \Omega_h^I$, then 
\[
w_h(z) \leq v_h(z) \quad \forall z \in \bar\Omega_h.
\]
\end{col}
\begin{proof}
By \eqref{WZClaim} and \eqref{WZClaim2}, 
we have
$|\partial w_h(z)| \leq |\partial v_h(z)|$
for any $z\in \mathcal{C}_h^-(v_h-w_h)$.
 Since $|\partial v_h(z)| \leq |\partial w_h(z)|$ by assumption, 
 we have  $|\partial v_h(z)| = |\partial w_h(z)|$ at contact points. Thus, by Proposition \ref{prop:stability_MA}, we get
\[
\sup_{\bar\Omega_h} (v_h - w_h)^- = 0.
\]
Consequently, $v_h - w_h \geq 0$ which proves the result.
\end{proof}

\subsubsection{Error estimates for \eqref{OP}}
\label{subsec:rates_MA}

Let us now to derive rates of convergence in 
the $L^{\infty}$-norm for method \eqref{OP}. To do so, we will build upon all the tools we have developed in previous sections; namely, the discrete Alexandrov estimate of Proposition~\ref{lem:FDAE}, the Brunn-Minkowski inequality of Proposition~\ref{prop:BM} and the stability result of Proposition~\ref{prop:stability_MA}.

Owing to Proposition \ref{prop:stability_MA} we only need to study the consistency error.
Since the method is consistent for convex quadratic polynomials at nodes bounded away from the boundary $\bdry$ (\cf Lemma \ref{lem:consistency_OP}), if we expect that $u$ can be well approximated by quadratic polynomials, then the consistency error will also be small. Let us make this intuition rigorous.

\begin{lem}[interior consistency]
\label{lem:C2alpha_consistency_MA}
Let $E = \{ \tilde \be_j\}_{j=1}^d$ be a basis of $\Real^d$ and $\Omega^I_h$ be a lattice spanned by $E$. Given $u$, strictly convex, let $z \in \Omega_h^I$ with $\dist(z,\p\Omega) \geq \bar R h$, where $\bar R$ is defined in \eqref{eq:defofRAz}, and set $\bar B = B_{\bar R h}(z)$. If $u\in C^{2,\alpha}(\bar{B})$ , then we have
\[
  \left| {\abs{ \partial I_h^{fd} u (z)}} - \int_{\omega_z} \det (D^2 u) \right| \leq C h^{\alpha} |\omega_z|,
\]
where the constant $C$ depends on $|u|_{C^{2,\alpha}(\bar{B})}$, and $\omega_z$ is defined in \eqref{partition_MA}.
\end{lem}
\begin{proof}
Let us show that 
\begin{align*}
| \partial I_h^{fd} u (z) | 
\leq  \int_{\omega_z} \det (D^2 u)+ Ch^{\alpha} |\omega_z|.
\end{align*}
The other inequality can be obtained in a similar fashion. 

Since $u \in C^{2, \alpha}(\bar B)$,  there is a convex quadratic polynomial $p \in \polP_2$ that satisfies $p(z) = u(z)$, $D p(z) = D u(z)$, $D^2 p = D^2 u(z)$ and, moreover,
\[
  u(x) \leq  p(x) + |u|_{C^{2,\alpha}(\bar{B})} h^{2+\alpha} \quad \forall x \in \bar B.
\]
Define $q(x) = p(x) + h^{\alpha} |u|_{C^{2,\alpha}(\bar{B})} |x-z|^2$ and notice that, by construction, $q(z) = u(z)$ and, for all nodes $z_j \in \bar B \cap \Omega_h^I$ we have $u(z_j) \leq q(z_j)$. Since $q$ is convex its nodal interpolant $I_h^{fd} q$ is also convex (\cf Definition \ref{def:convexnodalfcn}). Thus, we can apply  Lemma \ref{monotonicity} (monotonicity), to get $|\partial I_h^{fd} u (z)| \leq |\partial I_h^{fd} q (z)|$.
From these considerations we see that it is sufficient to show that
\[
  |\partial I_h^{fd} q (z)| \leq  \int_{\omega_z} \det (D^2 u) + Ch^{\alpha} |\omega_z|.
\]
Since $\lambda + Ch^{\alpha} \leq D^2 q \leq \Lambda + Ch^{\alpha}$ and
\[
  \frac{\Lambda + Ch^{\alpha}} {\lambda + Ch^{\alpha}} \leq \frac {\Lambda}{\lambda},
\]
we invoke the consistency result of Lemma \ref{lem:consistency_OP} 
and the regularity $u\in C^{2,\alpha}(\bar B)$ to obtain
\begin{align*}
  | \partial I_h^{fd} q (z) | &=  \det (D^2 q) |\omega_z|
    \leq  \left( \det (D^2 p) + Ch^{\alpha} \right) |\omega_z| \\
  &\leq  \int_{\omega_z}  \det (D^2 u)  + Ch^{\alpha}|\omega_z|.
\end{align*} 
This concludes the proof.
\end{proof}

The previous result establishes a consistency error at nodes bounded away from the boundary. 
For nodes close to the boundary, we have the following estimate.

\begin{lem}[boundary consistency]
\label{lem:estimate_bdry_MA}
Let $E = \{ \tilde \be_j\}_{j=1}^d$ be a basis of $\Real^d$ and $\Omega^I_h$ be a lattice spanned by $E$. Given $u$, strictly convex, let $z \in \Omega_h^I$ satisfy $\dist(z,\partial\Omega) \leq \bar R h$, where $\bar R$ is defined in \eqref{eq:defofRAz} and set $\bar B = B_{\bar Rh}(z) \cap \dm$. If $u \in C^{1,1}(\bar B)$, then 
\begin{align}\label{eqn:estimate_bdry_MA}
  \left| \abs{\partial I_h^{fd} u (z)} -  \int_{\omega_z} \det D^2 u \right| \leq C |\omega_z|,
\end{align}
where the constant $C$ depends only on $|u |_{C^{1,1}(\bar{B})}$. 
\end{lem}
\begin{proof}
As in Lemma~\ref{lem:C2alpha_consistency_MA}, it suffices to show
\[
  | \p I_h^{fd} u(z) | \leq \int_{\omega_z} \det D^2 u  + C |\omega_z|.
\]
Since $\lambda I \leq D^2 u \leq \Lambda I$, 
Lemma \ref{lem:estimate_adjacent_set} yields $A_z \subset B_{\bar R h}(z) \cap \dm $.
Recall now that $\Gamma(I_h^{fd} u)$ is piecewise linear with respect to a triangulation that has as nodes $\bar\Omega_h$.
The $C^{1,1}$-regularity assumption of $u$ implies that, if $K \subset \omega_z \subset \bar B$ is an element of this triangulation, we have
\[
  D \Gamma(I_h^{fd} u)|_K = D u(z) + \bv_K 
  \quad 
 | \bv_K|  \leq Ch |u|_{C^{1,1}(\bar{B})}.
\]
Thus, by Lemma \ref{char_subdifferential} (characterization of subdifferential), we deduce that the piecewise gradient $D\Gamma(I_h^{fd} u)|_K $ is contained in a ball centered at $D u(z)$ and with radius 
$Ch |u|_{C^{1,1}(\bar{B})}$. Thus, we have 
\[
  | \partial I_h^{fd} u(z) | \leq \int_{\omega_z} \det D^2 u + C |u|_{C^{1,1}(\bar{B})}^d |\omega_z|.
\]
This completes the proof.
\end{proof}

To control the $L^{\infty}$ error caused by the $\calO(1)$ error near the boundary, we construct a discrete barrier function below.  We refer to \cite{NochettoZhang16A} for a proof.

\begin{lem}[discrete barrier]\label{barrier}
Let $\dm$ be uniformly convex and $\Omega_h^I$ be a translation invariant nodal set in $\dm$. 
Given a constant $M > 0$, for each node $z \in \Omega_h^I$ 
with $\dist(z , \bdry) \leq \bar R h$, there exists a
convex nodal function $p_z\in \fd$ such that 
$
\abs{ \partial p_z (z_i) } \geq M |\omega_z|
$
for all $z_i \in \bar\Omega_h$, $p_z(z_i) \leq 0$ on $z_i \in \Omega_h^B$ and 
\[
  \abs{ p_z (z) } \leq C R M^{1/d} h ,
\]
for sufficiently small $h$.
\end{lem}

Now we are ready to prove the $L^{\infty}$-error estimate.

\begin{thm}[rate of convergence I]
\label{thm:rateC2alpha}
Assume that $\Omega$ is uniformly convex, and let
$u$ be the strictly convex (Alexandrov) solution of \MA equation \eqref{MA} with $f \geq \lambda^d > 0$.
Suppose that the  nodes
$\Omega_h^I$ are translation invariant,  and let $u_h\in \fd$ be the solution of \eqref{OP}. If
$
\lambda I \leq D^2 u \leq \Lambda I
$
and
$
u \in C^{2, \alpha}({\dm}),
$
then 
\[
  \|u - \Gamma(u_h) \|_{L^\infty(\Omega)}  \leq C h^{\alpha},
\]
where the constant $C = C(d, \dm, \lambda, \Lambda) \big( |u|_{C^{2, \alpha}(\bar\dm)} + |u|_{C^{1,1}(\bar\dm)} \big)$. 
\end{thm}
\begin{proof}
We begin by constructing a piecewise linear approximation of $u$. Recall that $\Gamma (I_h^{fd}u)$, the convex envelope of the nodal function $I_h^{fd} u$, is a piecewise linear function over a mesh that has $\bar\Omega_h$ as nodes.
Thus, classic interpolation theory yields
\[
  \|{\Gamma (I_h^{fd} u) - u}\|_{L^{\infty}(\dm)} \leq C h^2 |u|_{C^{1,1}(\bar\Omega)}.
\]
Therefore, we only need to prove that 
\begin{align}\label{thm:goal}
  \sup_{\bar\Omega_h} (I_h^{fd} u - u_h)^- \leq C h^{\alpha}.
\end{align}
A similar inequality, which controls the positive part of $I_h^{fd} u - u_h$, 
can be derived in an analogous fashion.

{\it Step 1.} We first show that for all $z \in \bar\Omega_h$ such that ${\rm dist}(z, \bdry) \leq \bar R h$,
\begin{equation}\label{eqn:MAConvStep1}
  (I_h^{fd} u - u_h)(z) \geq - C h|u|_{C^{1,1}(\bar\Omega)}.
\end{equation}
Let $p_z$ be the discrete barrier defined in Lemma \ref{barrier}
with free parameter $M$ and consider the function $ u_h + p_z $.
Since Lemma \ref{Msum} (addition inequality) implies
\[
\partial u_h (z_i) + \partial p_z (z_i) \subset  \partial (u_h + p_z) (z_i),
\]
by Lemma \ref{prop:BM} (Brunn-Minkowski inequality), we obtain
\begin{align*}
\abs {\partial (u_h + p_z) (z_i)} \geq &\; \left( \abs{\partial u_h (z_i)}^{1/d} + \abs{ \partial p_z (z_i)}^{1/d} \right)^d.
\end{align*}
Therefore, by Lemmas \ref{lem:estimate_bdry_MA} and \ref{barrier}, we have
\begin{align*}
|\p (u_h+ p_z)(z_i)|
&\ge 
\left(\Big( \int_{\omega_{z_i}} \det(D^2 u) \Big)^{1/d} + \big(M |\omega_z|\big)^{1/d}\right)^d\\
&\ge |\p I_h^{fd} u(z_i)|\quad \forall z_i\in \Omega_h^I
\end{align*}
provided that $M = C |u|^d_{C^{1,1}(\bar B)}$ 
Since $p_z \le 0$ on $\Omega_h^B$, we have $u_h + p_z \leq  I_h^{fd} u$ on $\Omega_h^B$. Moreover, because $\abs {\partial ( u_h + p_z) (z_i)} \geq \abs {\partial I_h^{fd} u (z_i)}$ for all $z_i \in \Omega_h^I$, we have, by the maximum principle of Corollary~\ref{MP}
\[
 u_h(z_i) + p_z(z_i) \leq I_h^{fd} u(z_i) \quad \forall z_i \in \bar\Omega_h.
\]
The estimate on the discrete barrier function, given in Lemma \ref{barrier}, yields
\begin{align}\label{ineq:estimate_boundary}
  u_h(z) - C h |u|_{C^{1,1}(\bar\Omega)} \leq u_h(z) +  p_z(z) \leq I_h^{fd} u(z),
\end{align}
thus proving \eqref{eqn:MAConvStep1}.

{\it Step 2.} For all nodes $z$ with ${\rm dist}(z, \bdry)\geq \bar R h$, thanks to  the consistency of the method, Lemma \ref{lem:C2alpha_consistency_MA}, we have
\begin{align}\label{eqn:consistencyError}
\Big| \abs{\partial I_h^{fd} u(z)} - \abs{\p u_h(z)}\Big|
\le C h^{\alpha} |u|_{C^{2,\alpha}(\bar{\dm})} |\omega_z|.
\end{align}
We apply the stability result of Proposition \ref{prop:stability_MA} on a smaller domain 
\begin{align}\label{eqn:SmallerDomain}
  \bar\dm_{*,h} = \{z \in \Omega_h^I: {\rm dist}(z, \bdry) \geq \bar Rh\}
\end{align}
and on the nodal functions 
$I_h^{fd} u$ and $u_h-Ch|u|_{C^{1,1}(\bar\Omega)}$, 
where $C$ is the constant in \eqref{eqn:MAConvStep1} so that $I_h^{fd} u\ge u_h-Ch|u|_{C^{1,1}(\bar\Omega)}$ on the boundary nodes of $ \bar\dm_{*,h}$.  Upon denoting by $\calC_{h,*}^-$ the contact set of $I_h^{fd} u - (u_h-Ch|u|_{C^{1,1}(\bar\Omega)}) $ with respect to $\bar\Omega_{h,*}$ we get
\begin{multline}
\label{eqn:MAConvStep2}
\sup_{\bar\Omega_{*,h}} \left( I_h^{fd} u - (u_h-Ch|u|_{C^{1,1}(\bar\Omega)}) \right) ^- \\
\leq C \left( \sum_{z\in \mathcal{C}_{h,*}^-} \Big( |\p I_h^{fd} u(z)|^{1/d} - |\p u_h(z)|^{1/d} \Big)^d \right)^{1/d}.
\end{multline}
Note now that $t\to  t^{1/d}$ is a concave function and, therefore, $(t+\delta)^{1/d} \le  t^{1/d} + \frac 1d t^{\frac{1-d}{d}} \delta$.  Thus, by setting $t = |\p u_h(z)|$, $\delta = |\p I_h^{fd} u(z)| - |\p u_h(z)|$ and applying \eqref{eqn:consistencyError}, we find
\begin{align*}
 |\p I_h^{fd} u(z)|^{1/d} - |\p u_h(z)|^{1/d} 
&\le \frac1d |\p u_h(z)|^{\frac{1-d}{d}}\big(|\p u_h(z)|- |\p I_h^{fd} u(z)|\big)\\
&\le C h^\alpha |u|_{C^{2,\alpha}(\bar{\dm})} |\omega_z| \Big(\int_{\omega_z} f\Big)^{\frac{1-d}{d}} 
\\
& \leq C h^{\alpha}|u|_{C^{2,\alpha}(\bar{\dm})}|\omega_z|^{1/d}
\end{align*}
because $f \geq \lambda^d > 0$. 
Inserting this estimate into \eqref{eqn:MAConvStep2} yields
\begin{align*}
\sup_{\bar\Omega_{h,*}}
(I_h^{fd} u - (u_h-Ch|u|_{C^{1,1}(\bar\Omega)}))^- 
&\le C h^\alpha  |u|_{C^{2,\alpha}(\bar{\dm})}
\Big(\sum_{z\in \mathcal{C}_{h,*}^-}
|\omega_z| \Big)^{1/d}
\\
&\le C h^{\alpha} |u|_{C^{2,\alpha}(\bar{\dm})}|\dm|^{1/d}
\end{align*}
or, equivalently,
\[
\sup_{\bar\Omega_{h,*}} (I_h^{fd} u - u_h)^- 
\le Ch|u|_{C^{1,1}(\bar\Omega)} + C h^\alpha  |u|_{C^{2,\alpha}(\bar{\dm})}.
\]
This inequality, together with \eqref{ineq:estimate_boundary},  proves the lower bound for $I_h^{fd} u - u_h$. The upper bound can be proved in a similar fashion. 
\end{proof}

Note that Proposition \ref{prop:stability_MA} controls $L^{\infty}$-error 
by the $L^d$-norm of the consistency error. Thus, if large 
consistency errors occur only in regions with small measure, 
we may still derive a rate of convergence. This observation may be used 
for solutions that are not $C^2(\dm)$ regular.

To state the result, we first introduce the Minkowski-Bouligand dimension. Let $U$ be a subset of $\dm$. Let $\{\omega_i\}_{z_i \in \Omega_h^I}$ be a translation invariant partition covering $\dm$ where $\omega_i$ is as in \eqref{partition_MA}. Define $m=m(h)$ to be the number of elements of $\{\omega_{z_i}\}$ required to cover $U$. We define the (Minkowski-Bouligand) dimension of $U$ as
\[
  \dim U = - \lim_{h\to0}  \frac{\log m(h)}{\log h}.
\]
For example, it is easy to check that $\partial B_1$, the discontinuity set of $D^2u$ in Example~\ref{ex:viscositysolution}, is of dimension one.
Note that in this example the solution $u \in C^{1,1}(\bar\Omega) \setminus C^2(\bar\Omega)$. 

The following result addresses the rate of convergence for piecewise smooth solutions such that the discontinuity set of $D^2 u $ is of low dimension.

\begin{thm}[rate of convergence II]
\label{C11estimate}
Let $u \in C^{1,1}(\bar\Omega)$ be strictly convex with $\lambda I \leq D^2u \leq \Lambda I$ and solve \eqref{MA}. Assume that $D^2u$ is piecewise H\"older continuous (with exponent $\alpha >0$) and its discontinuity set $U$ has dimension $n < d$. Let $u_h$ be the solution of \eqref{OP}, over a lattice $\Omega_h^I$ of translation invariant nodes. Then 
\[
  \inftynorm{u - \Gamma(u_h) } \leq Ch^\alpha |u|_{C^{2,\alpha}(\bar{\dm}\setminus U)} + C h^{\frac{d-n}{d}} |u|_{C^{1,1}(\bar{\dm})}.
\]
\end{thm}
\begin{proof}
  Following the estimate of Theorem \ref{thm:rateC2alpha}, we first note that for $z\in \bar\Omega_h$ 
with ${\rm dist}(z,\p\Omega)\le \bar R h$,
\begin{align}\label{eqn:estimateCloseBoundary}
   \abs{ (u_h - I_h^{fd} u)(z) } \leq C h .
\end{align}
Since $D^2 u$ is H\"older continuous except on the set $U$ and the aspect ratio $ \frac{\Lambda}{\lambda}$ of $D^2 u$ is bounded we have, by Lemmas \ref{lem:C2alpha_consistency_MA} and \ref{lem:estimate_bdry_MA}, that
\[
  \left| \p I_h^{fd} u(z) \right| \leq \begin{dcases}
    \int_{\omega_{z}} \det D^2 u  + C h^\alpha  |\omega_z| |u|_{C^{2,\alpha}(\bar{\dm}\setminus U)} & z \in \bar{\dm}_{h,*} \setminus U_{\bar R h}, \\
    \int_{\omega_{z}} \det D^2 u  + C |\omega_z| |u|_{C^{1,1}(\bar{\dm})} & z \in U_{\bar Rh},
                      \end{dcases}
\]
where $\bar\Omega_{h,*}$ is given by \eqref{eqn:SmallerDomain} and
\begin{align*}
  U_{\bar Rh} = &\; \{z \in \bar\Omega_{h,*}:  {\rm dist}(z, U) \leq \bar Rh \}.
\end{align*}

Now, the stability estimate of Proposition \ref{prop:stability_MA} yields 
\begin{multline}\label{ineq:stability}
  \sup_{\bar\Omega_{h,*}} \left(I_h^{fd} u - (u_h-Ch|u|_{C^{1,1}(\bar{\dm})}) \right)^-  \\
  \leq C \left( \sum_{z \in \mathcal{C}^-_{h,*}} \left( \abs{\partial I_h^{fd} u(z)}^{1/d} - \abs{\partial u_h(z)}^{1/d} \right)^d  \right)^{1/d},
\end{multline}
where $C>0$ is the constant in \eqref{eqn:estimateCloseBoundary}.

For $z\in \bar\Omega_{h,*} \setminus U_{\bar Rh}$, we apply the same arguments as in 
the proof of Theorem \ref{thm:rateC2alpha} to get
\begin{align*}
  \abs{ \partial I_h^{fd} u(z)}^{1/d} - \abs{ \partial u_h(z)}^{1/d}\le C h^\alpha |u|_{C^{2,\alpha}(\bar{\dm}\setminus U)} |\omega_z|^{1/d}.
\end{align*}
For $z\in U_{\bar R h}$ similar arguments (essentially take $\alpha=0$) yield
\begin{align*}
  \abs{\partial I_h^{fd} u(z)}^{1/d} - \abs{ \partial u_h(z)}^{1/d}  \leq  C |u|_{C^{1,1}(\bar{\dm})} |\omega_z|^{1/d}.
\end{align*}
Inserting these estimates into \eqref{ineq:stability} we deduce that
\begin{align*}
  \sup_{\bar\Omega_{h,*}} (I_h^{fd} u - u_h)^- 
  \leq &\;
  C \Big( h^{\alpha d} |u|^d_{C^{2,\alpha}(\bar{\dm}\setminus U)} |\dm| + |u|_{C^{1,1}(\bar \Omega)}^d \sum_{z \in U_{\bar Rh}} |\omega_z|  \Big)^{1/d}
  \\
  \le &\; C h^\alpha |u|_{C^{2,\alpha}(\bar{\dm}\setminus U)} + C |u|_{C^{1,1}(\bar\Omega)} \Big(\sum_{z\in U_{\bar R h}} |\omega_z|\Big)^{1/d}.
\end{align*}
Since $\dim U = n<d$, we have 
\[
  \sum_{z \in U_{\bar R h}} |\omega_z| = N |\omega_z| \leq C N h^d
\]
with $N \leq C h^{-n}$. Thus, we conclude that 
\[
  \sup_{\bar\Omega_{h,*}} (I_h^{fd} u - u_h)^- \leq C\big(h^\alpha |u|_{C^{2,\alpha}(\bar{\dm}\setminus U)} + h^{\frac{d - n}{d}} |u|_{C^{1,1}(\bar{\dm})} \big).
\]
This completes the proof.
\end{proof}


\begin{rem}[extensions]
The developments of this section have found the following extensions:
\begin{enumerate}[1.]
\item The  error analysis developed for method \eqref{OP} has been recently applied to wide-stencil schemes of \cite{Oberman08,FroeseOberman11,NochettoNtogkasZhang}. 


\item The error estimate stated
in Theorems \ref{thm:rateC2alpha} and \ref{C11estimate} applies only to structured nodes (lattices). 
It  remains an open problem how to extend the analysis to unstructured nodes
and to the degenerate case, \ie when $f=0$ in some region.

\item The error analysis of monotone schemes for the \MA equation in other norms, such as the $H^1$-norm, remains an open problem.



\end{enumerate}
\end{rem}




%
%
%


\section{Discretizations of non-convex second-order elliptic equations}\label{sec:nonconvex}

This section is a continuation
of the developments
of Section \ref{sec:convex},
where we apply the discretizations
and results for uniformly elliptic
linear PDEs to fully nonlinear problems.
However, in contrast to Section 
\ref{sec:convex}, we shall not assume
convexity (or concavity) of the differential
operator, but rather, only assume it is
uniformly elliptic. As explained in Section 
\ref{subsec:Equivalence}, it suffices
to consider numerical approximations
of the Isaacs equation given in Example \ref{ex:Isaacs}.
For simplicity and to communicate the essential
points in the discussion, we shall
assume that the nonlinear problem
does not have lower-order terms,
and in addition has homogenous Dirichlet boundary
conditions.  Thus we consider
numerical approximations for the problem
\begin{equation}
\label{eq:Isaacbvp}
  \begin{dcases}
    F(x,D^2u) := \inf_{\beta\in \calB} \sup_{\alpha \in \calA} \left[ {\calL}^{\alpha,\beta} u(x) - f^{\alpha,\beta}(x) \right] = 0, & \text{in } \Omega, \\
    u = 0, & \text{on } \partial \Omega,
  \end{dcases}
\end{equation}
with $\calL^{\alpha,\beta} u(x) = A^{\alpha,\beta}(x):D^2 u(x)$,
and the coefficient matrices satisfy
\[
\lambda I \le   A^{\alpha,\beta} \le \Lambda I , \quad \forall \alpha \in \calA,\ \beta\in \calB,
\]
so that $F$ is uniformly elliptic.  We  assume
that for each $x\in \Omega$, the mapping $M \mapsto F(x,M)$ is continuous
and locally Lipschitz continuous on $\polS^d$,
and that $f^{\alpha,\beta}\in C(\bar\Omega)$ 
for each $\alpha\in \calA$ and $\beta\in \calB$.
Further assumptions may be made in subsequent
developments.

As in the convex case, we can roughly classify numerical schemes as finite difference, finite element and semi-Lagrangian methods. The construction and analysis of finite difference schemes is detailed in Section~\ref{subsec:NCFD}, in principle, follows the convex case of Section~\ref{sec:convex} and the Barles-Souganidis theory as presented in Theorems~\ref{thm:BSTHM1} and \ref{thm:BSTHM1Alt}, but it is clouded by numerous technicalities that, for many years, prevented researchers to obtain rates of convergence. In fact, this was considered, as expressed in the Introduction of \cite{KuoTrudinger92}, an important unsolved open problem for several years.

The heart of the issue can be captured by examining the proof of Theorem~\ref{thm:rateFDHJB}. An important step there is the construction of a smooth subsolution to the equation (scheme) which, in the convex case, can be obtained by mollification of a subsolution to a perturbed equation. Convexity of the operator allows us to claim that this is a subsolution to the original problem and, thus, can be used to carry out the program detailed at the beginning of Section~\ref{sub:FDHJB}. However, without convexity, it is not clear how to construct a smooth approximation to the solution of \eqref{eq:Isaacbvp}, which can be used to invoke the consistency of the scheme. This is particularly important in the nonconvex case since, as shown in Example~\ref{ex:Nadirashvili}, one cannot assume smoothness of the solution to \eqref{eq:Isaacbvp}.

For many years, all of the available results were rather specialized. For instance, \cite{Jakobsen04ABC} considers a one dimensional problem and clearly shows that the arguments do not extend to more dimensions.
 A particular case of an Isaacs equation --- an obstacle problem for an HJB equation --- is discussed in \cite{MR2270893}, where this special structure is exploited.

The derivation of rates of convergence for general schemes remained an unsolved problem until \cite{CaffarelliSoug08} showed how to obtain a rate of convergence, within the Barles-Souganidis framework, for approximations of \eqref{eq:Isaacbvp} in the case that, for all $\alpha \in \calA$ and $\beta \in \calB$ the matrices $A^{\alpha,\beta}$ do not depend on $x$. We detail these results in Section~\ref{subsec:NCFDRates}, where we also comment on extensions and variations to these estimates. Simply put, the estimates assert that there exists an algebraic rate of convergence. An explicit rate, however, is not available at the moment.

Rather recently, in \cite{SalgadoZhang16}, the authors have extended the results of the 
finite element method in Section~\ref{sub:Wujunnondiv} to the case of \eqref{eq:Isaacbvp} and obtained an algebraic rate of convergence which, as in the finite difference case, is not explicit. These developments are detailed in Section~\ref{sub:FEMIsaacs}.

We also comment that, as of this writing, no rates of convergence are available for semi-Lagrangian schemes. The only known result is convergence, as obtained in \cite{MR3042570}.

We conclude our discussion on nonconvex equations in Section~\ref{sub:solschemesIsaacs} by describing how to solve the nonlinear system of equations that results after discretization, be it by any of the schemes discussed before.

\subsection{Finite difference methods}\label{subsec:NCFD}

Here, following the framework given in
Section \ref{subsec:monoFD} and in \cite{KuoTrudinger92}, 
we {construct} finite difference approximations
to the fully nonlinear problem \eqref{eq:Isaacbvp}
and study the stability and convergence
of {these} discretizations.   We consider the problem:
Find $u_h\in \fd$ satisfying
\begin{equation}
\label{eqn:KTNonlinearProblem}
  \begin{dcases}
F_h[u_h]=0 &\text{in }\Omega_h^I,\\
u_h =0 & \text{on }\Omega_h^B,
\end{dcases}
\end{equation}
where the interior and boundary 
nodes are as in Definition \ref{def:IBNodes}.
We assume that $F_h$ is consistent
and that
the scheme is of the form $F_h[u_h](z) = \mathcal{F}_h(z,\delta_h^2 u_h(z))$
with $\delta_h^2u_h(z) = \{\delta_{y,h}^2 u_h(z):\ y\in \St\}$.
We further assume that
 $\mathcal{F}_h$ is of positive type in the sense of
Definition \ref{def:OperationPositiveType}, so that $F_h$ is monotone,
and that
\begin{align*}
\frac{\p \mathcal{F}_h}{\p s_y}\le \Lambda_0,
\end{align*}
for some $\Lambda_0>0$.
For example, schemes that 
satisfy these properties, and the ones that we have
in mind are
\begin{equation*}
F_h[u_h] =  \inf_{\beta\in \mathcal{B}}\sup_{\alpha\in \mathcal{A}}
\big(\mathcal{L}_h^{\alpha,\beta} u_h - f^{\alpha,\beta}\big),
\end{equation*}
where each (linear) discrete operator
is given by
\begin{align*}
\mathcal{L}_h^{\alpha,\beta} u_h(z) = \sum_{y\in \St } a^{\alpha,\beta}_y(z) \delta_{y,h}^2 u_h(z),
\end{align*}
and is of positive type and consistent with $\mathcal{L}^{\alpha,\beta}$;
 Section \ref{sub:ConstructingFD} 
describes how to construct linear operators with these properties.

Before discussing the solvability of the finite difference
scheme \eqref{eqn:KTNonlinearProblem}, let us show first that
solutions to the scheme are uniformly bounded.
Let $v_h,w_h$ be two grid functions, and 
consider the linearization
\begin{equation}\label{nonconvex_linearization}
\begin{aligned}
  F_h[v_h](z) - F_h[w_h](z) & = \mathcal{L}_h(v_h-w_h)(z) \\
  &:=\sum_{y\in S} a_y(z) \delta_{y,h}^2 (v_h-w_h)(z),
\end{aligned}
\end{equation}
with
\begin{align*}
a_y(z) = \int_0^1 \frac{\p \mathcal{F}_h}{\p s_y}(z,\delta^2_{h,y} q_t(z))\, \diff t,\quad q_t = tv_h+(1-t)w_h.
\end{align*}
In particular, if we set $v_h=u_h$, $w_h = 0$,
and $f= -F(x,0)$, then the solution to \eqref{eqn:KTNonlinearProblem}
satisfies
\begin{align}\label{eqn:LinearizationU}
\mathcal{L}_{h,u_h} u_h = f\qquad \text{in }\Omega_h^I,
\end{align}
where the coefficients in the operator $\mathcal{L}_{h,u_h}$
depend on $u_h$.  
Since $\mathcal{L}_{h,u_h}$ is a positive operator, 
 Theorem \ref{thm:KTTHM1} yields the follow
stability result.

\begin{thm}[uniqueness and stability]
In this setting,
solutions to \eqref{eqn:KTNonlinearProblem} 
are unique and satisfy
\begin{align*}
\|u_h\|_{L^\infty(\bar\Omega_h)} \le C \Big(\sum_{z\in \Omega_h^I} h^d |f(z)|^d\Big)^{1/d}\le C.
\end{align*}
\end{thm}
In addition to uniqueness
and stability, 
these a priori estimates also
imply the existence of solutions.

\begin{thm}[existence]
Under these conditions, problem \eqref{eqn:KTNonlinearProblem}
has a  unique solution.
\end{thm}
\begin{proof}
Consider the map $Q_h:\fd\to \fd$ satisfying
\begin{equation*}
\begin{dcases}
\mathcal{L}_{h,v_h} Q_h(v_h) = f &\text{in }\Omega_h^I,\\
v_h =0 &\text{on }\Omega_h^B.
\end{dcases}
\end{equation*}
Theorem \ref{thm:KTTHM1} ensures
that $Q_h$ is well-defined and that $\| Q_h(v_h) \|_{L^\infty(\bar\Omega_h)} \le C$.
Brouwer's fixed point theorem shows
that $Q_h$ has a fixed point $u_h\in \fd$,
which is a solution to \eqref{eqn:LinearizationU},
and hence \eqref{eqn:KTNonlinearProblem}.
\end{proof}

\begin{rem}[other proofs]
Other (constructive) proofs
of existence of solutions can
be found in \cite[Section 4]{KuoTrudinger92}
and in Section \ref{sub:solschemesIsaacs}.
\end{rem}
Finally we apply Theorem \ref{thm:KTLinearHolder}
to deduce
that solutions to \eqref{eqn:KTNonlinearProblem}
are H\"older continuous.

\begin{prop}[discrete H\"older continuity]\label{prop:KTNLHolder}
Suppose that $u_h\in \fd$ solves \eqref{eqn:KTNonlinearProblem}.
Then there exists $\eta\in (0,1)$ and $C>0$
depending on the data such that for $z_1,z_2\in \bar\Omega_h$,
there holds
\begin{align*}
|u_h(z_1)-u_h(z_2)|\le C |z_1-z_2|^\eta.
\end{align*}
\end{prop}

Proposition \ref{prop:KTNLHolder} implies
that the sequence of solutions $\{u_h\}_{h>0}$
is equicontinuous.  Thus, applying Theorem 
\ref{thm:BSTHM1Alt} we obtain 
convergence to the viscosity solution.
\begin{thm}[convergence]
The solutions $u_h$ to \eqref{eqn:KTNonlinearProblem} converge locally uniformly
to the viscosity solution of \eqref{eq:Isaacbvp}.
\end{thm}

\subsection{Rates of convergence for finite 
difference schemes with constant coefficients}\label{subsec:NCFDRates}

Here we summarize the results
regarding rates of convergence
of finite difference schemes 
for fully nonlinear, nonconvex PDEs  as orginally shown by \cite{CaffarelliSoug08}.
As a starting point, we focus
on the case where the operator in \eqref{eq:Isaacbvp} is 
such that, for every  $\alpha \in \calA$ and $\beta \in \calB$, the 
matrices $A^{\alpha,\beta}$ are independent of $x$; and, moreover, we have $f^{\alpha,\beta} = f \in C^{0,1}(\Omega)$.
Under these conditions, problem \eqref{eq:Isaacbvp} 
satisfies a comparison principle (cf.~Theorem \ref{thm:comparison}),
and  there exists a unique viscosity
solution with regularity $u\in C^{1,s}(\bar\Omega)$ 
for some universal constant $s>0$, see Theorem \ref{thm:ggtthhmmreg}.

In order to derive rates, 
we shall, with an abuse
of notation, extend
the domain of $F_h$
to continuous functions
and to all $z\in \Omega^I:=\{x\in \Omega:\ {\rm dist}\{x, \p \Omega\}\ge  mh\}\supset \Omega_h^I$ where $m$ is the stencil size.
We also assume, in this section, 
the following (strengthened) consistency criterion:
\begin{align}\label{eqn:ConsistencyAgain}
|F(z,D^2 \phi(z))-F_h[\phi](z)|\le C h\|D^3\phi\|_{L^\infty(\Omega)}
\end{align}
for all smooth $\phi$ and $z\in \Omega^I_h$.

Now, as in the convex case, the key idea
to obtain rates of convergence is to construct
a smooth function $u_\eps$ that is a subsolution to \eqref{eq:Isaacbvp}
and then apply the consistency and monotonicity
of the finite difference operator to get a one sided bound
of the error; see the proof of Theorem \ref{thm:rateFDHJB}
for details.  However, unlike the convex case, it is not immediate how to construct
a sufficiently smooth subsolution to carry out this program; 
for example, the standard mollification $u^\eps = u_\eps \star \rho_\eps$
used in Theorem \ref{thm:rateFDHJB} is no longer
a subsolution due to the lack of convexity of $F$ (and $F_h$).

Instead, 
we employ the so-called {\it the sup- (inf-) convolutions} of the viscosity solution $u$.

\begin{definition}[sup-convolution]
Let $u \in C(\bar\Omega)$ and $\tau >0$, The {\em sup- (inf-) convolution} $u^+_{\tau}$ of $u$ is
\begin{align*}
u^+_\tau(x) = & \sup_{y\in \bar\Omega} \left[ u(y) - \frac1{h^\tau} |x-y|^2 \right],
\\
u^-_\tau(x) = & \inf_{y\in \bar\Omega} \left[ u(y) +  \frac1{h^\tau} |x-y|^2 \right].
\end{align*}
\end{definition}

\begin{rem}[alternative definition]
Notice that the definition we give here is tied to a mesh size $h$. We do so because this is the scale that suits our needs. In general the literature defines the sup-convolution of a function by
\[
  u_\varrho^+ = \sup_{y \in \bar\Omega} \left[ u(y) - \frac1{2\varrho} |x-y|^2 \right], \qquad \varrho > 0.
\]
The change of variables $\varrho = \tfrac12 h^\tau$ shows the relation between these two. A similar reasoning can be used for the inf-convolution.
\end{rem}

We show two examples of sup-convolution of functions to motivate the introduction of this useful concept.

\begin{ex}[sup-convolution I]
Let $u = |x|^2$ be defined in $\bbR^d$ and consider its sup-convolution 
\[
  u^+_{\tau}(x) = \sup_{y\in \bbR^d} \left[ |y|^2 - \frac1{h^\tau} |x - y|^2 \right].
\]
Let $y_* = y_*(x)$ be a point where the supremum is attained, then a simple calculation shows that
\[
  2 y_* + 2 h^{-\tau} (x - y_*) =0 
\]
which implies that
$y_* = (1 - h^{\tau})^{-1} x$. 
Inserting $y_*$ into the definition of $u^+_{\tau}(x)$, we obtain
\begin{align*}
  u^+_{\tau}(x) = &\; \left( \frac{1} {1 - h^{\tau}} \right)^2 |x|^2 - h^{-\tau} \left( \frac{ h^{\tau} } {1 - h^{\tau}} \right)^2 |x|^2
 =\frac{|x|^2}{(1 - h^{\tau})} .
\end{align*}
See Figure \ref{fig:supConvo1} and note that $D^2 u^+_{\tau} > D^2 u$. 
\end{ex}

Next, we consider a function that is less smooth.

\begin{ex}[sup-convolution II]
Let $u = -|x|$ be defined in $\bbR^d$. The maximum is attained at $y_* = y_*(x)$ if and only if 
\[
  \boldsymbol0 \in - \partial |y_*| + 2 h^{-\tau} (x - y_*)
\]
or, equivalently,
\[
  x \in y_* + \frac { h^{\tau}} 2 \partial |y_*|,
\]
where $\partial |y_*|$ denotes the subdifferential of $|\cdot|$ at $y_*$. 
We note that if $ |x| \leq  h^{\tau}/2$, then $y_* =\boldsymbol0 $ because $\partial |\boldsymbol0| = B_1(0)$. 
Otherwise, we have,
\[
  x = y_* + \frac{h^{\tau} y_*}{2 |y_*|} 
\]
because $\partial |y_*| = \left\{ \frac{y_*}{|y_*|} \right\}$ for $|y_*|>0$. Therefore, we conclude that
\begin{align*}
  u^+_{\tau}(x) =
  \begin{dcases}
    -h^{-\tau} |x|^2,      
    \quad & |x| \leq \frac{h^{\tau}}2,
    \\
    -|x| + \frac{h^{\tau}}{4}
    \quad & \text{otherwise}.
  \end{dcases}
\end{align*}
Note that $u^+_{\tau}$ is $C^{1,1}$, while $u$ is only Lipschitz. Moreover, near the singularity $x =0$, $u^+_{\tau}$ behaves like a paraboloid with $D^2 u_\tau^+ = -2h^{-\tau}I$; see Figure~\ref{fig:supConvo2}.
\end{ex}

The previous example shows that, intuitively speaking, the sup-convolution $u^+_{\tau}$ ``opens up'' the kinks of $u$. 

\begin{figure}
\centering
\begin{tikzpicture}
\begin{axis}[
axis lines = middle,
scaled ticks=false,
xtick={-0.5,0.5},
yticklabels={,,}
]
\addplot[domain=-0.5:0.5,blue] {x*x};
\addplot[domain=-0.5:0.5,red] {x*x/(1.-0.1)};
\end{axis}
\end{tikzpicture}
\begin{center}
\caption{\label{fig:supConvo1}The graph of the function $u(x) = x^2$ (\textcolor{blue}{\texttt{blue}})
and its sup-convolution (\textcolor{red}{\texttt{red}}) with parameters $h=0.1$ and $\tau=1$.}
\end{center}
\end{figure}
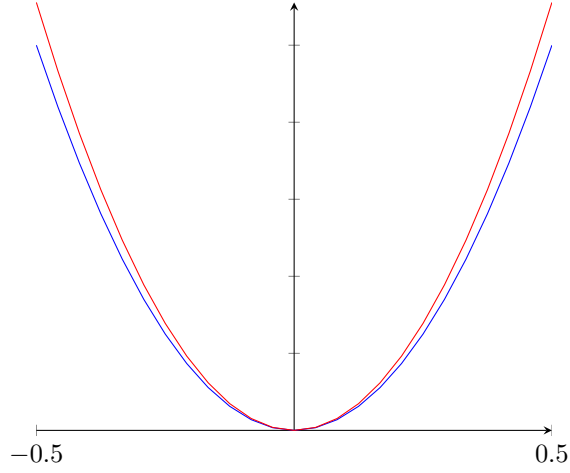

\begin{figure}
\begin{tikzpicture}
\begin{axis}[
axis lines = middle,
scaled ticks=false,
xtick={-0.5,0.5},
yticklabels={,,}
]
\addplot[domain=-0.5:0,blue] {x};
\addplot[domain=0:0.5,blue] {-x};
\addplot[domain=0:0.05,red] {-10*x*x};
\addplot[domain=-0.05:0,red] {-10*x*x};
\addplot[domain=0.05:0.5,red] {-x+1/40};
\addplot[domain=-0.5:-0.05,red] {x+1/40};
\end{axis}
\end{tikzpicture}
\caption{\label{fig:supConvo2}The graph of the function $u(x) = -|x|$ (\textcolor{blue}{\texttt{blue}})
and its sup-convolution (\textcolor{red}{\texttt{red}}) with parameters $h=0.1$ and $\tau=1$.}
\end{figure}

The examples above illustrate some of the general properties of sup- and inf-convolutions. To concisely state them, we begin with a definition.

\begin{definition}[opening $t$]
We say that $p \in \polP_2$ is a paraboloid of {\it opening $t$} if, for some $\ell\in \mathbb{P}_1$, we have
\[
  p(x) = \ell(x)\pm \frac{t}{2}|x|^2.
\]
\end{definition}

The main properites of sup- and inf-convolutions are as follows. To simplify the presentation, we denote be $\bar\p u(x)$ the superdifferential of $u$ at the point $x$, that is
\[
  \bv \in \bar\p u(x) \Leftrightarrow -\bv \in \p(-u)(x).
\]

\begin{prop}[properties of $u^+_{\tau}$]\label{prop:propertydeltaconvolution}
Let $u \in C^{0,1}(\bar\Omega)$ and, for $x \in \Omega$, let $y_* = y_*(x)$ denote the point where the supremum in the definition of $u_\tau^+$ is attained. The following statements hold:
\begin{enumerate}[1.]
  \item \label{[(i)]} $x = y_* +  \frac12 h^\tau \bv$ for some vector $\bv \in \bar\partial u(y_*)$ and, therefore, $|x-y_*|\le C h^\tau$.
  \item \label{[(ii)]} $\|u - u^+_{\tau} \|_{L^\infty(\Omega)} \leq C h^{\tau}$.
  \item \label{[(iv)]} There exists a paraboloid of opening $2h^{-\tau}$, that touches $u^+_{\tau}$ (resp., $u^-_{\tau}$) from below (resp., above) at $x$. 
  
  \item \label{[(iii)]} If $u$ is the viscosity solution to \eqref{eq:Isaacbvp}, then $|x_1 - x_2| \leq C |y_*(x_1) - y_*(x_2)|$.
\end{enumerate}
\end{prop}

\begin{proof} Let us prove each statement separately.

{\it Proof of \ref{[(i)]}:}
For any fixed $x$, if $u(y) - h^{-\tau} |x-y|^2$ attains its maximum at $y_*$, then
\[
  \boldsymbol0 \in \bar\partial u(y_*) - 2 h^{-\tau}  (x - y_*) .
\]
Thus, we have $x = y_* + \frac12 h^\tau \bv$ for some $\bv \in \bar\partial u(y_*)$.

{\it Proof of \ref{[(ii)]}:} 
By definition $u_\tau^+ \geq u$. Moreover, property \ref{[(i)]} implies
\[
0 \leq u^+_{\tau} (x) - u(x)  = u(y_*) - u(x)  - h^{-\tau} |\frac 12 h^{\tau} \bv |^2
\leq  |u(x) - u(y_*)| + \frac 14 h^\tau |\bv |^2
\]
for some $\bv \in \bar\partial u(y_*)$. Since $u$ is Lipschitz continuous, we conclude that, for every $x \in \Omega$,
\[
| u(x) - u^+_{\tau} (x) | \leq C h^{\tau}.
\]

{\it Proof of \ref{[(iv)]}:}
Since, for any $x \in \Omega$, $y_* = y_*(x)$ is the point where the supremum is attained, let us define the paraboloid
\[
p(z) = u(y_*) - \frac1{h^{\tau}}|z - y_*|^2,
\]
and notice that $p(x) = u^+_{\tau}(x)$. Moreover,
\[
p(z) \leq \sup_{y\in\bar\Omega} \left[   u(y) - \frac1{h^{\tau}} |z - y|^2 \right] =: u^+_{\tau}(z), \quad \forall z \in \bar\Omega.
\]
Thus, the paraboloid $p$ touches the graph of $u^+_{\tau}$ at point $x$ from below.

{\it Proof of \ref{[(iii)]}:}
Since, by definition,
\[
  u_\tau^+(x) = u(y_*(x)) -\frac1{h^\tau}|x-y_*(x)|^2,
\]
upon defining the paraboloid $p(z) = u_\tau^+(x) + \frac1{h^\tau}|x-z|^2$ we have that
\[
  p(y_*(x)) = u(y_*(x)), \qquad p(z) \geq u(z), \ \forall z \in \bar \Omega. 
\]
In other words, at $y_*(x)$, the function $u$ is touched from below by a paraboloid with opening $2h^{-\tau}$. On the other hand, since $u$ solves \eqref{eq:Isaacbvp}, the Harnack inequality of Theorem~\ref{thm:Harnack} implies that, at $y_*(x)$, the function $u$ can be touched from above by a paraboloid of opening $Ch^{-\tau}$ with $C$ depending only on $d$, $\lambda$ and $\Lambda$. Invoking \cite[Proposition 1.2]{MR1351007} we conclude that, for every $x_1,x_2 \in \bar\Omega$ the function $u$ is differentiable at $y_*(x_i)$, $i=1,2$ and, moreover,
\[
  |D u(y_*(x_1)) - D u(y_*(x_2))| \le Ch^{-\tau} |y_*(x_1)-y_*(x_2)|.
\]
We now invoke \ref{[(i)]}, to obtain that
\[
  |x_1 - x_2| = \left| y_*(x_1) - y_*(x_2) + \frac12 h^\tau D u(y_*(x_1)) - \frac12 h^\tau D u(y_*(x_2)) \right|.
\]
The previous two inequalities yield the claim.

This completes the proof.
\end{proof}

Next we establish a lower bound for $F_h[u_\tau^+](\cdot)$. We define
\begin{align*}
\Omega_h^{I,\tau} &= \{z\in \Omega_h^I:\ \dist(z,\p \Omega)\ge Ch^\tau\},
\end{align*}
where the constant $C>0$ depends on the stencil
size $m$ and is chosen so that
$F_h[u](x)$ is well-defined for
$x$ satisfying $|x-z|\le C h^\tau$ for some $z\in \Omega_h^{I,\tau}$; see Figure \ref{fig:IheartTikz}.

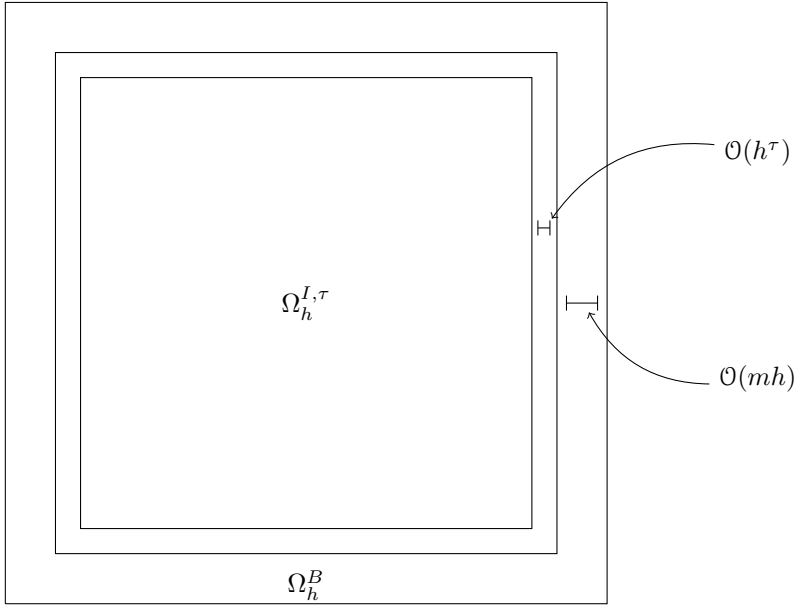
\begin{figure}[h]
\centering
\begin{tikzpicture}[scale = 2]
\draw[-](0,0)--(4,0)--(4,4)--(0,4)--(0,0);
\draw[-](0.3333,0.3333)--(3.66667,0.3333)--(3.66667,3.66667)--(0.3333,3.66667)--(0.3333,0.3333);
\draw[-](0.5,0.5)--(3.5,0.5)--(3.5,3.5)--(0.5,3.5)--(0.5,0.5);  
\draw (2,2) node {$\Omega_h^{I,\tau}$};
\draw (2,0.125) node {$\Omega_h^{B}$};
\draw[-](3.73,2)--(3.9366,2);
\draw[-](3.73,2.05)--(3.73,1.95);
\draw[-](3.9366,2.05)--(3.9366,1.95);
\draw[-](3.54,2.5)--(3.62,2.5);
\draw[-](3.54,2.55)--(3.54,2.45);
\draw[-](3.62,2.55)--(3.62,2.45);
\draw (5.,1.5) node (n1) {$\mathcal{O}( m h)$};
\draw (5.,3) node (m1) {$\mathcal{O}(h^\tau)$};
\draw (3.8333,2) node (n2) {};
\draw (3.58,2.5) node (m2) {};
\path[->](n1) edge [bend left] (n2);
\path[->](m1) edge [bend right] (m2);
\end{tikzpicture}
\caption{A pictorial description of the set $\Omega_h^{I,\tau}$. The values $F_h[u](x)$ are well defined for $x$ within a distance of $Ch^\tau $ of this set.}
\label{fig:IheartTikz}
\end{figure}

\begin{prop}[{bound on $F_h[u_{\tau}^+](\cdot)$}]\label{prop:FutauEst}
Assume that $u \in C^{0,1}(\bar\Omega)$ and that $z\in \Omega^{I,\tau}_h$. Let $y_* = y_*(z)\in \Omega$ 
be chosen so that 
\begin{align}
\label{eqn:uPlusu}
u_\tau^+(z) = u(y_*) - \frac1{h^{\tau}}|z-y_*|^2,
\end{align}
we then have
\[
  F_h[u_{\tau}^+](z) \ge F_h[u](y_*)+\big(f(y_*)-f(z)\big).
\]
Moreover, for all $z\in \Omega_h^{I}$, it holds that
\begin{equation}\label{eqn:uDeltaLower}
  F_h[u_{\tau}^+](z)\ge - C h^{-\tau}.
\end{equation}
\end{prop}
\begin{proof}
First note that, by Proposition~\ref{prop:propertydeltaconvolution} item \ref{[(i)]} we have that $|z-y_*|\le C h^\tau$ and, therefore,
$F_h[u](y_*)$ is well-defined.
Now, owing to the monotonicity of $F_h$, to show the first estimate it suffices to show that
\[
  \delta_{y,h}^2 u_{\tau}^+ (z) \geq \delta_{y,h}^2 u(y_*)\quad y\in S.
\]

Define the function 
\[
  v(x) = u(x - z + y_*) - \frac1{h^{\tau}} |z - y_*|^2,
\]
and note that $v(z) = u (y_*) - \frac1{h^\tau} |z - y_*|^2 = u_{\tau}^+(z) $. 
The change of variables $\zeta = x - z + y_*$ reveals that
\begin{align*}
  v(x) = & u(\zeta) - \frac1{h^\tau}|\zeta-x|^2 \le u_\tau^+(x),
\end{align*}
and therefore we conclude that, for any $y \in S$, $\delta_{y,h}^2 v(z)\le \delta_{y,h}^2 u_\tau^+(z)$.
This, together with the fact that, for any $y \in S$, $\delta_{y,h}^2 v(z) = \delta_{y,h}^2 u(y_*)$ yields
\[
  \delta_{y,h}^2 u(y_*)  = \delta_{y,h}^2 v(z) \leq \delta_{y,h}^2 u_{\tau}^+ (z). 
\]
Since $f$ is independent of $\alpha$ and $\beta$, this proves the estimate $F_h[u_\tau^+](z) + f(z) \ge F_h[u](y_*) + f(y_*)$.

To prove \eqref{eqn:uDeltaLower}, let
\[
w(x) = u(y_*) - \frac1{h^\tau} |x-y_*|^2.
\]
Clearly $w(z)  = u_\tau^+(z)$ and $w(x)\le u_\tau^+(x)$.
Therefore, monotonicity and 
the identity $\delta_{y,h}^2 w(z) = -\frac2{h^\tau} |y|^2$
yields 
\[
F_h[u_\tau^+](z) \ge F_h[w](z) \ge - C h^{-\tau},
\]
which concludes the proof.
\end{proof}

Our goal now is to compare $u^+_{\tau}$ with the numerical solution $u_h$ through $F_h[u_h]$ and $F_h[u^+_{\tau}]$ and a comparison principle
for the discrete operator $F_h$. From Proposition~\ref{prop:FutauEst} it follows that $F_h[u^+_{\tau}]$ at $z$ depends on the consistency error of $F_h[u]$ at $y_*:=y_*(z)$.
The consistency of $F_h$, as stated in \eqref{eqn:ConsistencyAgain}, implies that
\[
  F_h[p](y_*) = F(y_*, D^2p), \quad \forall p \in \polP_2.
\]
Consequently, the consistency error $F_h[u](y_*)$ depends on the quality of the approximation of the solution $u$ by a quadratic polynomial at $y_*$. This, in turn, depends on the regularity of $u$.

The next result asserts that in our setting, outside sets of arbitrarily small measure, viscosity solutions to \eqref{eq:Isaacbvp} have pointwise second-order Taylor expansions with an error that is controlled by the size of the singular set. For a proof, the reader is referred to \cite[Theorem A]{CaffarelliSoug08}.

\begin{thm}[quadratic approximation]
\label{thm:BowDownToMightyCaff}
Assume that $F$ has constant coefficients, is uniformly elliptic with ellipticity constants $\lambda,\Lambda>0$, and that $f\in C^{0,1}(\bar\Omega)$.
Let $u$ be a Lipschitz continuous solution of $F(x,D^2 u) = 0$ in $B_{2r}(x_0)$.
There exist positive constants $\sigma$, $t_0$ and $C$ that depend on $\lambda$, $\Lambda$ and $d$ such that, for all $t>t_0$, it is possible to find an open set $A_t \subset \Omega$ such that if $x\in A_t$, there exists a quadratic polynomial $p_x \in \polP_2$ of opening less than $t$ that satisfies
\[
F(x,D^2 p_x(x)) =0.
\]
Moreover, 
\begin{align}
\label{eqn:CSTaylor}
\big|u(y)-u(x)-p_x(y-x)\big|\le C  t |x-y|^3\quad \forall y\in B_{2r}(x_0),
\end{align}
and
\begin{align*}
| \Omega \backslash A_t |\le C t^{-\sigma}.
\end{align*}
\end{thm}

This theorem essentially says that, for every $t$, 
except in a singular set of measure $t^{-\sigma}$,
which we denote by $\mathcal{S}_t$,
the function $u$ has a second order Taylor expansion.
We define the regular set $\mathcal{R}_t :=\bar\Omega\backslash \mathcal{S}_t$.

Let us now explain how Theorem~\ref{thm:BowDownToMightyCaff} can be used to obtain a rate of convergence. By consistency, \eqref{eqn:ConsistencyAgain}, we have that, for every $p\in \mathbb{P}_2$, 
\[
  F_h[p](z) = F(z,D^2p(z))\qquad \forall z\in \Omega.
\]
Therefore, if $t>0$, $z \in \Omega$ and $B_{h}(z) \cap \mathcal{R}_t \neq \emptyset$,
we choose $y \in B_{h}(z) \cap \mathcal{R}_t$ and set $p = p_y$, the second order expansion of $u$ at $y$ from Theorem \ref{thm:BowDownToMightyCaff}.
We then estimate the consistency error as follows
\begin{align*}
  \big|F_h[u](z)\big| 
  \leq &\;
  \big|F_h[u](z)- F_h[p_y](z) \big| + \big| F_h[p_y](z) - F_h[p_y](y) \big| + |F_h [p_y] (y)|.
\end{align*}
From estimate \eqref{eqn:CSTaylor} and the fact that $F_h[p_y](z) - F_h[p_y](y) = f(y) - f(z)$ and $F_h [p_y] (y) =0$, it follows that
\begin{align*}
\big|F_h[u](z)\big| 
\le &\; C h t + |f(z) - f(y)| 
\le C_1 ht + C_2 h
\end{align*}
where $t > 1$, $C_1$ depends on the constant in \eqref{eqn:CSTaylor} and stencil size $m$ and $C_2$ depends on the Lipschitz constant of $f$. 
%
Thus, the consistency error is of the order $\mathcal{O}(t h)$
on $\Omega_h^I$ except for a set that has measure $\mathcal{O}(t^{-\sigma})$.
To be more precise, we introduce the discrete regular and singular sets
\begin{equation}
\label{eqn:discreteSingularSet}
\mathcal{R}_{t,h} :=\{z\in \Omega^{I,\tau}_h:\ B_{h}(y_*(z)) \cap \mathcal{R}_t\neq \emptyset \},\ 
\mathcal{S}_{t,h} :=\{z\in \Omega^{I,\tau}_h: z \notin \mathcal{R}_{t,h} \}.
\end{equation}

We now  estimate the size of the discrete singular set.
\begin{lem}[cardinality of $\mathcal{S}_{t,h}$]\label{lemma:nonconvex_cardinality}
Let $\mathcal{S}_{t,h}$ be defined by \eqref{eqn:discreteSingularSet}. Then
\[
  \# \mathcal{S}_{t,h} \le C (h^{-d} t^{-\sigma}+h^{1-d}).
\]
\end{lem}
\begin{proof}
Notice that,
\begin{align*}
C h^d (\# \mathcal{S}_{t,h}) \leq \sum_{z \in \mathcal{S}_{t,h}} |B_h(y_*(z))| \leq C \left|\bigcup_{z \in \mathcal{S}_{t,h} } B_h(z) \right|  \leq C(t^{-\sigma}+h).
\end{align*}
Indeed, the first inequality is obvious; the second one follows from Proposition~\ref{prop:propertydeltaconvolution} item \ref{[(iii)]}, which shows that, for $z_1,z_2\in \bar\Omega_h$ with $z_1\neq z_2$, we have $|y_*(z_1)-y_*(z_2)|\ge C h$ so the overlap of the balls can be controlled independently of $h$; finally, the last inequality follows from Theorem~\ref{thm:BowDownToMightyCaff}. Rearranging terms in this last inequality yields the result.
\end{proof}

We now estimate the consistency error of $F_h[u^+_{\tau}]$ and have the following corollary of Proposition \ref{prop:FutauEst}.
\begin{col}[consistency]\label{col:SREstimate}
For every $z\in \Omega_h^{I,\tau}$ we have
\[
F_h[u_\tau^+](z) \ge -C
  \begin{dcases}
    \big(h t +h^\tau\big) & z\in \mathcal{R}_{t,h},\\
    h^{-\tau} &  z\in \mathcal{S}_{t,h}.
  \end{dcases}
\]
\end{col}
\begin{proof}
For $z \in \mathcal{R}_{t,h}$, since $|z-y_*|\le C h^\tau$ and $f$ is Lipschitz continuous,
we have by Proposition \ref{prop:FutauEst},
\begin{align*}
F_h[u_\tau^+](z)\ge F_h[u](y_*)-C h^\tau \ge -C\big( h t +h^\tau\big).
\end{align*}
The estimate on the singular set $\calS_{t,h}$ was already established in \eqref{eqn:uDeltaLower}.
\end{proof}

The previous result controls the consistency error in the interior of the domain. To obtain a rate of convergence for the scheme, we also need to control this error near the boundary. This is accomplished with the help of the following discrete barrier function.

\begin{lem}[discrete barrier]
\label{lem:discbarrierIsaacs}
Let the domain $\Omega$ satisfy an exterior ball condition. If $h$ is sufficiently small,
then for any point $z\in \Omega_h^I$ with $\dist(z, \partial \Omega) \leq \delta$, 
there exists a discrete function $b_{z,h} \in X_h^{fd}$ such that 
\[
  F_h[b_{z,h}](\zeta) \leq 0 \ \forall \zeta \in \Omega_h^I, 
  \qquad 
  b_{z,h}(\zeta) \geq 0 \ \forall \zeta \in \Omega_h^B.
\]
Moreover, we have
\[
| b_{z,h} (z)| \leq C \delta,
\]
where $C$ depends only on $\lambda, \Lambda, d$ and $\Omega$.
\end{lem}
\begin{proof}
Let $z_* \in \partial \Omega$ such that $|z- z_*| = \dist(z , \partial \Omega)$.  
Since $\Omega$ satisfies an exterior ball condition, there exists a open ball 
$B_r(z')$ centered at some point $z'$ such that
\[
z_* \in \partial B_r(z') 
\quad
\text{ and }
\quad
B_r(z') \cap \Omega = \emptyset.
\]

We construct a barrier function 
\[
b_z(x) =  B (r^{-q} - |x - z'|^{-q} ),
\]
where $B$ and $q$ are some positive constant to be determined later. 
We observe that 
\begin{align*}
D b_z(x) = & qB  |x - z'|^{-q -2} (x - z'),
\\
D^2 b_z(x) =& qB  |x - z'|^{-q -2} 
\left( I - (2+q)  \frac{(x - z')}{|x - z'|} \otimes \frac{(x - z')}{|x - z'|}  \right).
\end{align*}
Note that $D^2 b_z(x)$
is a diagonal perturbation of 
a rank-one matrix that
has eigenvalues $0$ (with multiplicity $d-1$)
and $-q(2+q) B |x-z'|^{-q-2}$ (with multiplicity $1$).
Therefore $D^2 b_z$ has eigenvalues
$qB |x-z'|^{-q-2}$ (with multiplicity $d-1$)
and 
$-q(1+q)B/|x - z'|^{q +2}$
(with multiplicity $1$).
In particular the smallest eigenvalue of $D^2 b_z(x)$
is $-q(1+q)B/|x - z'|^{q +2}$
and the largest is $qB |x-z'|^{-q-2}$.

Since $\lambda I \leq A^{\alpha, \beta} \leq \Lambda I$ we have
\begin{align*}
\inf_{\beta \in \calB} \sup_{\alpha \in \calA} \left[ A^{\alpha, \beta} : D^2 b_z(x) \right]
\leq 
\frac{B q}{ |x-z'|^{q +2}} \left[ (d-1)\Lambda - (1 +q) \lambda \right].
\end{align*}
Taking $q = 2(d-1) \frac{\Lambda}{\lambda} -1$ we obtain that
\begin{align*}
  \inf_{\beta \in \calB} \sup_{\alpha \in \calA}\left[ A^{\alpha, \beta} : D^2 b_z(x) \right] &\leq 
    \frac{B q}{|x-z'|^{q+2}} (1-d)\Lambda \\
  &\leq \frac{B q}{ |r + \diam(\Omega)|^{q+2}} (1-d)\Lambda.
\end{align*}
Since the right-hand side
of the last inequality is negative,
we conclude that, for $B$ sufficiently large,
\[
\inf_{\beta \in \calB} \sup_{\alpha \in \calA}\left[ A^{\alpha, \beta} : D^2 b_z(x) \right] \leq -2 \|f\|_{L^\infty(\Omega)}.
\]

The discrete barrier function is then defined as $b_{h,z} = I_h^{fd} b_z$. It is easy to check that
\[
  b_{h,z}(\zeta) \geq 0 \ \forall \zeta \in \Omega_h^B
  \qquad
  |b_{h,z}(z)| \leq C \delta,
\]
so that it remains to verify that $F_h[b_h](\zeta) \leq 0 $ for all $\zeta \in \Omega_h^I$. 

Since $b_z$ is a smooth function, the consistency condition \eqref{eqn:ConsistencyAgain} implies
\[
  F_h[b_{h,z}](\zeta) - F(\zeta, D^2b_z)(\zeta) = \calO(h).
\]
Finally, because $F(\zeta, D^2b_z)(\zeta) \leq - \|f\|_{L^\infty(\Omega)}$ we deduce
that, for $h$ sufficiently small,
\[
  F_h[b_{h,z}](\zeta) \leq 0, \quad \forall \zeta \in \Omega_h^I.
\]
This completes the proof.
\end{proof}

\begin{rem}[bound on $u_h$]
\label{rem:boundaryestiamteIssac}
Corollary \ref{col:SREstimate} implies that 
for any $z \in \Omega^I_h$ with $\dist(z, \partial \Omega) \leq \delta$, we have
\[
|u_h(z)| \le C \delta.
\]
Indeed, since $b_{z,h} \ge u_h = 0$ on $\Omega^B_h$ and $F_h[b_{z,h}](z) \le F_h [u_h](z)$ for all $z \in \Omega^I_h$, 
invoking a comparison principle for $F_h$, we obtain
\[
u_h(z) \le b_{z,h}(z) \leq C \delta.
\]
Similarly we can show that $u_h(z) \geq - C \delta$. 
\end{rem}

Combining Corollary \ref{col:SREstimate}
and Lemma \ref{lemma:nonconvex_cardinality} 
we obtain a rate of convergence
for the finite difference scheme \eqref{eqn:KTNonlinearProblem}.

\begin{thm}[rate of convergence]
Let the domain $\Omega$ satisfy an exterior ball condition and $f \in C^{0,1}(\bar{\Omega})$. 
Let $u_h\in \fd$
be the solution to \eqref{eqn:KTNonlinearProblem}
and $u\in C^{0,1}(\bar\Omega)$
the viscosity solution to \eqref{eq:Isaacbvp}. 
Then there holds
\[
\|I_h^{fd} u-u_h \|_{L^\infty(\bar\Omega_h)} \le  C h^{ \sigma/(2d+\sigma)}.
\]
where $\sigma$ is the exponent in Theorem \ref{thm:BowDownToMightyCaff}.
\end{thm}
\begin{proof}
By Proposition~\ref{prop:propertydeltaconvolution} item \ref{[(ii)]} we have that $\| u-u_\tau^+\|_{L^\infty(\Omega)} \leq C h^\tau$. Thus it is sufficient to
bound
$v_h = u_h - I_h^{fd} u^+_\tau$.
%
We shall show that $\sup_{\bar\Omega_h} v_h^-\le C h^{\sigma/(2d+\sigma)}$;
the proof of the complementary estimate is similar.
We divide the proof into two steps.

{\it Step 1 (boundary estimate).}
We first show that, for any $z\in \bar\Omega_h\backslash \Omega_h^{I,\tau}$,
\[
| v_h(z) | \le C h^{\tau}
\]
for some constant $C > 0$.
By the definition of $v_h$, we have
\begin{align*}
| v_h(z) | 
\le &\; |u_h(z) - u(z)| + |u(z) - I_h^{fd} u^+_{\tau}(z)|
\\
\le &\; |u_h(z)| + |u(z)| + |u(z) - I_h^{fd} u^+_{\tau}(z)|.
\end{align*}
Since $u$ is Lipschitz continuous and $u = 0$ on $\partial \Omega$, we have that $|u(z)|\leq C h^\tau$. In addition,
owing to Remark~\ref{rem:boundaryestiamteIssac} and Proposition~\ref{prop:propertydeltaconvolution} property \ref{[(ii)]}, we conclude
that
\[
|v_h(z)| \leq C h^{\tau}.
\]


{\it Step 2 (interior estimate).} 
Now 
\[
- F_h[u_\tau^+](z) = F_h[u_h](z) - F_h[u_\tau^+](z) \geq  \mathcal{L}_h (u_h-I_h^{fd}u_\tau^+)(z) = \mathcal{L}_h  v_h(z),
\]
where $\mathcal{L}_h$ is the linear(ized) operator \eqref{nonconvex_linearization} with coefficients
that depend on $u_\tau^+$ and $u_h$. 
Thus, setting $t = h^{\tau-1}$, we have by Corollary \ref{col:SREstimate}
\begin{equation}
\label{eqn:LhvEstimate}
\mathcal{L}_h v_h(z)\le 
C 
\begin{dcases}
  h^\tau &  z\in \mathcal{R}_{h^{\tau-1},h},\\
  h^{-\tau} & z\in \mathcal{S}_{h^{\tau-1},h},
\end{dcases}
\end{equation}
and, by Lemma \ref{lemma:nonconvex_cardinality},
\begin{align*}
\# \mathcal{S}_{h^{\tau-1},h} \le C \big(h^{-d+\sigma (1-\tau)}+h^{1-d}\big)\le C h^{-d+\sigma(1-\tau)}.
\end{align*}
Applying the ABP estimate 
given in Theorem \ref{thm:KTTHM1} (with $\Omega_h^I$ replaced 
with $\Omega_h^{I,\tau}$), to \eqref{eqn:LhvEstimate} yields
\begin{align*}
  \sup_{\Omega_h^{I,\tau}} v_h^-
  &\le C \Big(\sum_{z\in \mathcal{R}_{h^{\tau-1},h}} h^{d+d \tau}
  +\sum_{z\in \mathcal{S}_{h^{\tau-1},h}} h^{d- d\tau}\Big)^{1/d}\\
&\le C \Big(h^{d\tau} + h^{-d \tau+\sigma(1-\tau)}\Big)^{1/d}.
\end{align*}
Thus, setting $\tau$ so that $d \tau = -d \tau+\sigma(1-\tau)$, i.e., 
$\tau = \sigma/(2d+\sigma),$
we obtain
\begin{align*}
\sup_{\Omega_h^{I,\tau}} v_h^- \le C h^{\sigma/(2d+\sigma)},
\end{align*}
and therefore
\begin{align*}
\sup_{\Omega_h^{I,\tau}} (u_h-I_h^{fd} u_\tau^+)^-\le C h^{\sigma/(2d+\sigma)} .
\end{align*}

Gathering the obtained bounds implies the result.
\end{proof}

\begin{rem}[extensions]
The works \cite{Krylov15,Tura15} have obtained rates of convergence for finite difference schemes for nonconvex PDEs with variable coefficients and low order terms.
\end{rem}

\subsection{Finite element methods}
\label{sub:FEMIsaacs}

Let us now focus, following \cite{SalgadoZhang16}, on the development of a finite element scheme for \eqref{eq:Isaacbvp} based on the integro-differential approximation presented in Section~\ref{sub:Wujunnondiv}. The idea is, after choosing an $\epsilon>0$, to replace each one of the operators $\calL^{\alpha,\beta}$ by their integro-differential approximations $\calL^{\alpha,\beta}_\epsilon$, as defined in \eqref{eqn:IDEReg}. In doing so we obtain, according to \cite{CafSil10} a smooth approximation $u^\epsilon$ of $u$, the solution to \eqref{eq:Isaacbvp}. Moreover, the difference $u-u^\epsilon$ can be controlled in terms of $\epsilon$. We can now proceed to approximate $u^\epsilon$ by, simply put, taking the inf--sup over discrete problems of the form \eqref{eqn:NZMethod}, thus obtaining a discrete function $u_h^\epsilon \in \lco$. Using the regularity of $u^\epsilon$ we can control the difference $u^\epsilon - u_h^\epsilon$ in terms of the mesh size and, possibly, negative powers of $\epsilon$. Optimizing with respect to $\epsilon$ yields a rate of convergence.

Let us now proceed with the details. To set ideas we will assume that, for every $\alpha \in \calA$ and $\beta \in \calB$ we have $f^{\alpha,\beta} = f \in C^{0,1}(\bar\Omega)$ and that the matrices $A^{\alpha,\beta} $ are constant. The approximate problem is then
\begin{equation}
\label{eq:Isaacseps}
  \begin{dcases}
    \frac\lambda2 \Delta u^\eps + \inf_{\beta \in \calB}\sup_{\alpha \in \calA} I_\epsilon^{\alpha,\beta} u^\epsilon = f & \text{in } \Omega_\epsilon, \\
    u = 0, & \text{on } \omega_\epsilon,
  \end{dcases}
\end{equation}
where, in analogy to \eqref{eqn:IeDef}, the integral operators are defined by
\[
  I_\epsilon^{\alpha,\beta} w(x) = \frac1{\epsilon^{d+2} \det A_\lambda^{\alpha,\beta} } \int_{\Real^d} |y|^2 \delta^2_{\theta y, \theta} w(x)
    \varphi\left( \frac1\epsilon \left(A_\lambda^{\alpha,\beta}\right)^{-1} y \right),
\]
with
\[
  A_\lambda^{\alpha,\beta} = \left( A^{\alpha,\beta} - \frac\lambda2 I \right)^{1/2},
\]
and the domains $\Omega_\eps$ and $\omega_\eps$ are given by \eqref{eqn:WJNDomains}.

Let us now state the existence, uniqueness, smoothness properties of $u^\epsilon$ and its rate of convergence to $u$. For a proof see Theorem 3.5, Theorem 4.8 and Theorem 6.1 of \cite{CafSil10}.

\begin{prop}[properties of $u^\eps$]
\label{prop:porpueps}
Problem \eqref{eq:Isaacseps} has a unique classical solution $u^\epsilon$. Moreover, there exists $s \in (0,1)$ that depends on $\lambda$, $\Lambda$ and $d$ but not on $\eps$ such that, for every $\omega \Subset \Omega$, we have
\[
  \| u^\epsilon \|_{C^{1,s}(\omega)} + \| u^\epsilon \|_{C^{0,1}(\bar\Omega)} \leq C\left ( \| u^\eps \|_{L^\infty(\Omega)} + \| f \|_{L^\infty(\Omega)} \right),
\]
where the constant $C>0$ depends on the distance between $\omega$ and $\p\Omega$. Additionally, there exists a $\gamma>0$ such that
\[
  \| u - u^\eps \|_{L^\infty(\Omega)} \leq C \eps^\gamma \| f \|_{C^{0,1}(\Omega)},
\]
where the constant $C>0$ depends only on $\lambda$, $\Lambda$, $d$ and $\Omega$.
\end{prop}
  
With the properties of $u^\eps$ at hand we introduce a finite element scheme to approximate it. 
Namely, given a quasiuniform triangulation for which \eqref{eqn:FEMMMatrix}
holds we seek $u_h^\vare \in \lco$ such that (compare with \eqref{eqn:NZMethod})
\begin{equation}
\label{eq:Isaacsepsh}
  F_h^\eps[u_h^\eps](z_i) := \frac\lambda2 \Delta_h u_h^\eps (z_i) 
  + \inf_{\beta \in \calB} \sup_{\alpha \in \calA} I_\epsilon^{\alpha,\beta} u_h^\eps(z_i) - f_i=0, \quad \forall z_i \in \Omega_h^I,
\end{equation}
where, as in \eqref{eqn:NZMethod}, $f_i = \int_\Omega f \phi_i$, 
and $\phi_i$ are the normalized hat functions.

Let us now address the existence, uniqueness and approximation properties of $u_h^\eps$. We begin by noticing that, since each one of the operators
\[
  w_h \mapsto \frac\lambda2 \Delta_h w_h + I_\eps^{\alpha,\beta} w_h
\]
is monotone (\cf Lemma \ref{lem:IEMonotone}), 
the same holds for $F_h^\eps$. This is the content of the next result.

\begin{col}[monotonicity]
\label{cor:Fhemonotone}
Assume that $\mct$ is such that $\Delta_h$ is monotone. Then the operator $F_h^\eps$, defined in \eqref{eq:Isaacsepsh} is monotone.
\end{col}

\begin{rem}[flexibility]
\label{rem:kewl}
Notice that, in Corollary~\ref{cor:Fhemonotone}, all that is necessary is a comparison principle 
{\it for the finite element Laplacian} $\Delta_h$. A sufficient condition for this is given in Lemma~\ref{lem:FEMMmatrix}. This is in stark contrast with 
the methods of Section~\ref{sub:FEMHJB} where, either 
a restrictive mesh
condition is needed (see Remark~\ref{rem:BScondition}) or where the coefficients are assumed to be isotropic as in \eqref{eqn:paraHJB}.
\end{rem}

In addition to monotonicity we must also consider the consistency of the scheme which, owing to the consistency of $\Delta_h$ (see Lemma~\ref{lem:EllipticProjectionConsistency}) reduces to the study of the consistency of the inf--sup of integral transforms when applied to Lipschitz functions (recall that $u^\eps \in C^{0,1}(\bar\Omega)$ uniformly in $\eps$). To measure this we define, for $z_i \in \Omega_h^I$,
\[
  \calR_{h,\eps}[w](z_i) = \int_\Omega \left[ \inf_{\beta \in \calB} \sup_{\alpha \in \calA} I_\eps^{\alpha,\beta}I_h^{ep}w(z_i) -
\inf_{\beta \in \calB} \sup_{\alpha \in \calA} I_\eps^{\alpha,\beta}w(x)   \right] \phi_i(x).
\]
The consistency of the scheme is then encoded
in the following result whose proof mainly follows Lemma~\ref{lem:RateConvIT} and Theorem~\ref{thm:NZErrorEstimateC2} but must take into account the reduced regularity of the functions. This is the reason for the reduced rate.

\begin{lem}[consistency]
\label{lem:cons}
Let $w \in C^{0,1}(\bar\Omega)$.  Then for every $z \in \Omega_h^I$ we have
\[
  |\calR_{h,\eps}[w](z)| \leq C \frac{h}{\eps^2}|\log h| \| w\|_{C^{0,1}(\bar\Omega)},
\]
where the constant $C>0$ is independent of $z$, $h$, $\eps$ and $w$.
\end{lem}

Monotonicity and consistency allow us to conclude existence and uniqueness of solutions. The proof of the following result is either a corollary of Proposition~\ref{prop:convTLHoward} below or follows 
from a discrete version of Perron's method.

\begin{thm}[existence and uniqueness]
\label{thm:existuniqueIsaacsepsh}
Let the family of meshes be such that $\Delta_h$ is monotone. Then, for every $h>0$ and $\eps>0$, the scheme \eqref{eq:Isaacsepsh} has a unique solution.
\end{thm}

To conclude, we study the rates of convergence. Since Proposition~\ref{prop:porpueps} provides a rate for $\| u - u^\eps\|_{L^\infty(\Omega)}$ and Proposition~\ref{prop:SchatzWahlbin} one for $\| u^\eps - I_h^{ep} u^\eps\|_{L^\infty(\Omega)}$ it remains to compare the discrete solution $u_h^\eps$ to the Galerkin projection $I_h^{ep} u^\eps$. Upon denoting $e_h = I_h^{ep} u^\eps - u_h^\eps \in \lco$, and after some manipulations, it turns out that the error $e_h$ satisfies, for every $z \in \Omega_h^I$,
\begin{equation}
\label{eq:ISerreq}
  \frac\lambda2 \Delta_h e_h(z) + \inf_{\beta \in \calB}\sup_{\alpha \in \calA} I_\eps^{\alpha,\beta}I_h^{ep} u^\eps(z)
     - \inf_{\beta \in \calB} \sup_{\alpha \in \calA} I_\eps^{\alpha,\beta} u_h^\eps(z) 
     = \calR_{h,\eps}[u^\eps](z).
\end{equation}
With this equation at hand we are in place to establish a rate of convergence.

\begin{thm}[rate of convergence]
\label{thm:bigthm}
Let $u \in C^{1,s}(\Omega)\cap C^{0,1}(\bar\Omega)$ be the solution of problem \eqref{eq:Isaacbvp} and $u_h^\eps \in \lco$ be the solution to scheme \eqref{eq:Isaacsepsh} with $\eps \geq C h^{1/2} |\log h|$. There is $\gamma >0 $ such that
\[
  \| u - u_h^\eps \|_{L^\infty(\Omega)} \leq C \left( \eps^\gamma + \frac{h}{\eps^2}|\log h| \right) \| f \|_{C^{0,1}(\bar\Omega)},
\]
where the constant $C>0$ depends only on $\lambda$, $\Lambda$, $\Omega$ and $d$.
\end{thm}
\begin{proof}
As already discussed, the rates of convergence reduce to estimating the difference $e_h$. To do so we employ the error equation \eqref{eq:ISerreq} and notice that, if $z$ belongs to $\calC_h^-(e_h)$, the contact set of $e_h$, then we must have
\[
  \inf_{\beta \in \calB}\sup_{\alpha \in \calA} I_\eps^{\alpha,\beta}I_h^{ep} u^\eps(z)
     \geq \inf_{\beta \in \calB} \sup_{\alpha \in \calA} I_h^{\alpha,\beta} u_h^\eps(z).
\]
In other words, for $z \in \calC_h^-(e_h)$, the error $e_h$ satisfies the inequality
\[
  \frac\lambda2 \Delta_h e_h(z) \leq \calR_{h,\eps}[u^\eps](z).
\]
The discrete ABP estimate of Theorem~\ref{thm:FEMABP123} then implies that
\[
  \sup_{\Omega} e_h^- \leq C \left( \sum_{z \in \calC_h^-(e_h) } |\omega_{z}| |\calR_{h,\eps}[u^\eps](z)|^d \right)^{1/d}.
\]
We must now invoke the consistency estimates of Lemma~\ref{lem:cons} to obtain an upper bound for $e_h^-$. A similar argument will yield the reverse bound. The theorem is thus proved.
\end{proof}

\begin{rem}[algebraic rate]
\label{rem:chooseps}
If the actual value of $\gamma>0$ is known, then one can choose $\eps^{\gamma+2} = h |\log h|$ to obtain
\[
  \| u - u_h^\eps \|_{L^\infty(\Omega)} \leq C h^{\frac{\gamma}{\gamma+2}} 
  |\log h|^{\frac{\gamma}{\gamma+2}},
\]
which in the best scenario, $\gamma =1 $, would yield $\calO(h^{1/3}|\log h|^{1/3})$.
\end{rem}

\subsection{Solution of the discrete problems}
\label{sub:solschemesIsaacs}

Having discussed discretization schemes for Isaacs equations \eqref{eq:Isaacbvp} and detailed their convergence properties let us concentrate, to finalize our discussion, on how to solve the ensuing nonlinear systems of equations. We will see, as it should be by now clear to the reader, that the approaches are extensions of the convex case, which we discussed in Section~\ref{sub:solschemes}, but the results are much more modest.

After discretization, in complete analogy to \eqref{eq:HJBdiscrete}, we must solve the problem: find $\bx \in \Real^N$, such that
\begin{equation}
\label{eq:Isaacsdiscrete}
  \bF(\bx) = \inf_{\beta \in \calB} \sup_{\alpha \in \calA} \left[ \bK^{\alpha,\beta} \bx - \bef^{\alpha,\beta} \right] = \boldsymbol0,
\end{equation}
where the inf--sup is computed component-wise, $\{ \bK^{\alpha,\beta}: \alpha \in \calA, \beta \in \calB \}$ and $\{ \bef^{\alpha,\beta}: \alpha \in \calA, \beta \in \calB\}$ are, respectively, discretizations of $\{\calL^{\alpha,\beta}: \alpha \in \calA, \beta \in \calB \}$ and $\{ f^{\alpha,\beta}: \alpha \in \calA, \beta \in \calB\}$. Once again, the dimension $N$ equals the number of degrees of freedom in the discretization.

One would be tempted, following Algorithm~\ref{alg:Howard}, to define, for $\by \in \Real^N$, the indices $\balpha(\by) $ and $\bbeta(\by)$ by conditions similar to the ones that led to \eqref{eq:defofweirdmatrixvecs} and use them in each iteration of a possible extension of Howard's algorithm. However, this will not produce a convergent method; see, for instance, \cite[Remark 5.8]{BMSZ09}. The reason for this, simply put, is that the map $\by \mapsto \min\{ \bg, \max\{\by,\bh\}\}$ is not slant differentiable and, moreover, since the map $\bF$, defined in \eqref{eq:Isaacsdiscrete}, is neither convex nor concave any generalized notion of derivative for this function will not possess any monotonicity properties.

For the reasons outlined above, we now describe, following \cite[Section 5]{BMSZ09} a convergent scheme that cannot be cast as a Newton-type method but, instead, can be understood as a two-level Howard's algorithm. We begin by defining, for $\by \in \Real^N$ and $i \in \{1, \ldots,N\}$, the element $\beta(\by,i) \in \calB$ by the condition
\begin{equation}
\label{eq:defalphayi}
  \bF(\by)_i = \sup_{\alpha \in \calA} \left[ \bK^{\alpha,\beta(\by,i)} \by - \bef^{\alpha,\beta(\by,i)} \right]_i.
\end{equation}
With this at hand we define $\bbeta(\by) \in \calB^N$ by $\bbeta(\by)_i = \bbeta(\by,i)$ and the mapping $\bF^{\bbeta(\by)} : \Real^N \to \Real^N$ is such that
\begin{equation}
\label{eq:defFalpha}
  \bF^{\bbeta(\by)}(\bw)_i = \sup_{\alpha \in \calA} \left[ \bK^{\alpha,\beta(\by,i)} \bw - \bef^{\alpha,\beta(\by,i)} \right]_i, \quad i = 1,\ldots,N. 
\end{equation}
The generalization of Howard's method is described in Algorithm~\ref{alg:TLHoward}.

\begin{algorithm}
  \SetKwInOut{Input}{input}
  \SetKwInOut{Output}{output}
  \SetKw{Return}{return}
  
  \Input{Sets $\calA$ and $\calB$. \\
          Matrices $\{\bK^{\alpha,\beta}: \alpha \in \calA, \beta \in \calB\} \subset \Real^{N\times N}$. \\
          Right hand sides $\{\bef^{\alpha, \beta}: \alpha \in \calA, \beta \in \calB \} \subset \Real^N$.}
  
  \Output{Vector $\bx \in \Real^N$, solution of \eqref{eq:Isaacsdiscrete}.}
  \BlankLine
  Initialization: Choose $\bx_{-1} \in \Real^N$ \;
  \BlankLine
  \For{$k\geq0$}{
    Set $\bbeta_k = \bbeta(\bx_{k-1})$ \;
    Find: $\bx_k \in \Real^N$ that solves 
    \begin{equation}
    \label{eq:HJBStep}
      \bF^{\bbeta_k}(\bx_k) = \boldsymbol0
    \end{equation}  \;
    \If{ $\bF(\bx_k) = \boldsymbol0$ }{
       \Return{$\bx_k$} \;
    }
  }

  \caption{Two-level Howard's algorithm for nonconvex problems.}
  \label{alg:TLHoward}
\end{algorithm}

Notice that, owing to the definition of $\bF^{\bbeta(\by)}$ given in \eqref{eq:defalphayi}--\eqref{eq:defFalpha}, every step of Algorithm~\ref{alg:TLHoward} requires the solution of a discrete Hamilton-Jacobi-Bellman equation \eqref{eq:HJBStep}. This can be done by  applying Algorithm~\ref{alg:Howard} and justifies calling this method a {\it two-level} one. The convergence properties of Algorithm~\ref{alg:TLHoward} are described below.

\begin{prop}[convergence of two-level Howard]
\label{prop:convTLHoward}
Let $\calA$ and $\calB$ be finite sets. Assume that, for every $\bbeta \in \calB^N$, the map 
$\bF^\bbeta$ is monotone  in the sense of Definition~\ref{def:discreteMonotone}. Then the sequence of iterates, given by Algorithm~\ref{alg:TLHoward} satisfies $\bx_k \geq \bx_{k+1}$ and converges, in a finite number of steps, to $\bx$, the solution of \eqref{eq:Isaacsdiscrete}.
\end{prop}
\begin{proof}
Using equations \eqref{eq:defalphayi}--\eqref{eq:defFalpha}, line \textbf{3} of Algorithm~\ref{alg:TLHoward} and identity \eqref{eq:HJBStep} we observe that the iterates $\{\bbeta_k\}_{k \in \polN}$ and $\{\bx_k\}_{k \in \polN}$ satisfy
\[
  \bF^{\bbeta_{k+1}}(\bx_k) = \bF^{\bbeta(\bx_k)}(\bx_k) = \bF(\bx_k) \leq \bF^{\bbeta_k}(\bx_k) = \boldsymbol0 =
  \bF^{\bbeta_{k+1}}(\bx_{k+1}),
\]
so that, from the monotonicity of $\bF^{\bbeta_{k+1}}$, we conclude that $\bx_k \geq \bx_{k+1}$.

As in Theorem~\ref{thm:convHoward}, the fact that $\# \calB$ is finite together with the monotonicity of iterates imply that Algorithm~\ref{alg:TLHoward} converges in a finite number of steps.
\end{proof}

We conclude the discussion on Algorithm~\ref{alg:TLHoward} by commenting that the case when $\calA$ and $\calB$ are compact spaces and when the equation \eqref{eq:HJBStep} is only solved approximately are also discussed by \cite{BMSZ09}.

Let us, in addition to Algorithm~\ref{alg:TLHoward}, present the Richardson-type iterative scheme of \cite[Section 4.2]{KuoTrudinger92}: Starting from $\bx_0 \in \Real^N$ the iterates are computed via
\begin{equation}
\label{eq:JacobiKT}
  \bx_{k+1} = \bG(\bx_k) := \bx_k - \frac1{\Lambda_N} \bF(\bx_k),
\end{equation}
where $\Lambda_N >0$ is a sufficiently large constant that depends on $N$. It is shown there that, for $\Lambda_N \geq C N^{2/d}$, the mapping $\bG$ is a contraction, \ie there is a constant $C>0$ such that
\[
  | \bG(\bv) - \bG(\bw) |
  \leq \left( 1-\frac{N^{-2/d}}C \right) | \bv - \bw |,
\]
so that \eqref{eq:JacobiKT} is convergent. Some strategies that combine Algorithm~\ref{alg:TLHoward} and \eqref{eq:JacobiKT} are suggested in \cite{MR1182321}.

\section{Outlook}
\label{sec:Outlook}

\epigraph{\emph{Overall we know too much about linear PDE and too little about nonlinear PDE.}
}{L.C.~Evans}

In this paper, we summarized some of the recent
trends and advancements in the discretizations
and convergence analysis 
for strongly nonlinear PDEs, with an emphasis
of convex and nonconvex fully nonlinear equations.
While these two classes of equations
and  discretization types have fundamentally
different structure conditions,
common themes permeate
the analysis; for example, 
consistency, monotonicity and comparison principles,
Alexandrov-Bakelman-Pucci estimates, wide-stencils,
and smooth approximations.
It is also evident from our discussion
that, due to the pointwise definition 
of viscosity solutions and, correspondingly, the monotonicity criterion
given in the Barles-Souganidis
framework, there is a stark contrast
of results between finite difference
schemes and finite element methods
for such problems; only within the last 10 years
have {\it any} significant advances been
made in the convergence analysis of {finite element}-type
methods.

Despite the recent flurry of results for 
numerical fully nonlinear PDEs, there
still remain fundamental open problems in the field.
Probably the most pressing, at least on the finite element front,
is an alternative framework that bypasses or
relaxes the monotonicity requirement  
found in the Barles-Souganidis theory.
Other completely open problems include, but are not limited to
\begin{enumerate}[(i)]
\item Derivation of {\it sharp} rates of 
convergence for nonconvex fully nonlinear equations.

\item Rates of convergence for finite difference
schemes on {\it unstructured} grids.

\item Rates of convergence of finite element methods 
for convex and nonconvex PDEs in 
other norms, for example Sobolev norms such as $H^1$.

\item A posteriori error estimation and adaptive methods for fully nonlinear problems.

\item Rates of convergence for nonconvex degenerate problems.

\end{enumerate}

In light of these questions we feel that the quote from the Preface of \cite{Evans} given above nicely summarizes 
the current state of numerical PDEs. It is our hope that this overview motivates current and future 
researchers to work on the numerical approximation of nonlinear problems.

\section*{Acknowledgements}
The first two authors were supported
in part by the National Science Foundation
grants DMS-1417980 (Neilan) and DMS-1418784 (Salgado),
and the Alfred Sloan Foundation (Neilan).
This work was initiated 
at the conference, {\it Nonlinear PDEs,
Numerical Analysis, and Applications}
at the University of Pittsburgh, 
which was supported by the NSF
 grant DMS-1541585,
an IMA conference grant, and the Mathematical Research Center,
University of Pittsburgh.


\bibliographystyle{abbrv}
\bibliography{nlreview}

\end{document}